\documentclass[12pt]{amsart}
\usepackage{amssymb, amsmath, pinlabel, graphicx, amscd, multirow}

\newtheorem{thm}{Theorem}[section]
\newtheorem{prp}[thm]{Proposition}
\newtheorem{cor}[thm]{Corollary}
\newtheorem{lem}[thm]{Lemma}

\newtheorem*{qst}{Question}
\newtheorem{lma}[thm]{Lemma}
\theoremstyle{definition}
\newtheorem{defn}[thm]{Definition}
\theoremstyle{remark}
\newtheorem{rem}[thm]{Remark}
\newenvironment{rmk}{\begin{rem}\rm}{\end{rem}}

\newenvironment{pf}{\begin{proof}}{\end{proof}}

\numberwithin{equation}{section}

\numberwithin{equation}{section}
\newcommand{\R}{{\mathbb{R}}}
\newcommand{\C}{{\mathbb{C}}}

\newcommand{\Z}{{\mathbb{Z}}}

\newcommand{\Aa}{{\mathbf{A}}}
\newcommand{\Bb}{{\mathbf{B}}}
\newcommand{\Cc}{{\mathbf{C}}}
\newcommand{\Ee}{{\mathbf{E}}}
\newcommand{\Mm}{{\mathbf{M}}}
\newcommand{\Ll}{{\boldsymbol{\lambda}}}

\newcommand{\vA}{\bar{A}}

\newcommand{\ch}{{\mathcal{Q}}}
\newcommand{\ri}{{\mathcal{R}}}

\newcommand{\conf}{{\mathcal{C}}}

\newcommand{\A}{{\mathcal{A}}}
\newcommand{\BB}{{\mathcal{B}}}
\newcommand{\T}{{\mathcal{T}}}

\newcommand{\II}{{\mathcal{I}}}
\newcommand{\M}{{\mathcal{M}}}

\newcommand{\Ordo}{{\mathcal{O}}}

\newcommand{\sblv}{{\mathcal{H}}}

\newcommand{\bull}{\cdot} % used to be \bullet

\newcommand{\la}{\langle}
\newcommand{\ra}{\rangle}
\newcommand{\pa}{\partial}
\newcommand{\id}{\operatorname{id}}

\newcommand{\ind}{\operatorname{index}}

\newcommand{\ix}{\operatorname{index}}
\newcommand{\krn}{{\operatorname{ker}}}
\newcommand{\cokrn}{{\operatorname{coker}}}

\newcommand{\sign}{\operatorname{sign}}

\newcommand{\con}{{\mathrm{con}}}
\newcommand{\sol}{{\mathrm{sol}}}
\newcommand{\flow}{{\mathrm{flow}}}

\newcommand{\ordo}{{\mathbf{o}}}

\newcommand{\gdim}{\operatorname{gdim}}

\newcommand{\area}{\operatorname{Area}}

\newcommand{\action}{\mathfrak{a}}

\def\dfn#1{{\em #1}}

\title{Knot Contact Homology}
\author[T.~Ekholm]{Tobias Ekholm}
 \address{Uppsala University, Box 480,
  751 06 Uppsala, Sweden}
  \email{tobias@math.uu.se}
\author[J.~Etnyre]{John Etnyre}
\address{School of Mathematics, Georgia Institute of Technology,
Atlanta, GA 30332-0160, USA}
\email{etnyre@math.gatech.edu}
\author[L.~Ng]{Lenhard Ng}
\address{Mathematics Department, Duke University, Durham, NC 27708-0320, USA}
\email{ng@math.duke.edu}
\author[M.~Sullivan]{Michael Sullivan}
\address{Department of Mathematics, University of Massachusetts,
 Amherst, MA 01003-9305, USA}
\email{sullivan@math.umass.edu}

\begin{document}

\begin{abstract}
The conormal lift of a link $K$ in $\R^3$ is a Legendrian submanifold
$\Lambda_K$ in the unit cotangent bundle $U^* \R^3$ of $\R^3$
with contact structure equal to the kernel of the Liouville form. Knot
contact homology, a topological link invariant of $K$, is defined as the
Legendrian homology of $\Lambda_K$, the homology of a differential
graded algebra generated by Reeb chords whose differential
counts holomorphic disks in the symplectization $\R \times U^*\R^3$ 
with Lagrangian boundary condition
$\R \times \Lambda_K$.

We perform an explicit and complete computation of the Legendrian
homology of $\Lambda_K$ for arbitrary links $K$ in terms of a braid
presentation of $K$, confirming a conjecture that this invariant
agrees with a previously-defined combinatorial version of knot contact
homology. The computation uses a double degeneration: the braid
degenerates toward a multiple cover of the unknot which in turn
degenerates to a point. Under the first degeneration, holomorphic
disks converge to gradient flow trees with quantum corrections. The combined degenerations give rise to a new generalization of
flow trees called multiscale flow trees. The theory of multiscale flow trees is the key tool in our
computation and is already proving to be useful for other computations as well.
\end{abstract}
\maketitle

%\tableofcontents

% **************************************************
% **************************************************
% **************************************************
\section{Introduction}\label{sec:intr}

% **************************************************
% **************************************************
\subsection{Knot contact homology}\label{ssec:kch}
Let $K\subset \R^{3}$ be a link. The {\em conormal lift} $\Lambda_K\subset U^{\ast}\R^{3}$ of $K$ in the unit cotangent bundle $ U^{\ast}\R^{3}$ of $\R^{3}$ is the sub-bundle of $ U^{\ast}\R^{3}$ over $K$ consisting of covectors which vanish on $TK$. The submanifold $\Lambda_K$ is topologically a union of $2$-tori, one for each component of $K$. The unit cotangent bundle carries a natural contact $1$-form $\alpha$: if $p\,dq$ denotes the Liouville form on $T^{\ast}\R^{3}$ then $\alpha$ is the restriction of $p\,dq$ to ${U^{\ast}\R^{3}}$. The conormal lift $\Lambda_K$ is a Legendrian submanifold with respect to the contact structure induced by $\alpha$. Furthermore, if $K_{t}$, $0\le t\le 1$, is a smooth isotopy of links then $\Lambda_{K_t}$ is a Legendrian isotopy of tori. Consequently, Legendrian isotopy invariants of $\Lambda_{K}$ give topological isotopy invariants of $K$ itself.

Contact homology, and Legendrian (contact) homology, is a rich source of deformation invariants in contact topology. Legendrian homology associates a differential graded algebra (DGA) to a Legendrian submanifold $\Lambda\subset Y$ of a contact manifold $Y$, in which the algebra is generated by the Reeb chords of $\Lambda$ and the Reeb orbits in $Y$, and the differential is defined by a count of holomorphic curves in the symplectization $\R\times Y$ with Lagrangian boundary condition $\R\times\Lambda$. The main result of this paper is a complete description of the DGA of the conormal lift $\Lambda_K$ of a link $K$ in terms of a braid presentation of $K$, see Theorem~\ref{thm:mainlink} below.

Before presenting Theorem~\ref{thm:mainlink}, we discuss some consequences of it. Theorem~\ref{thm:mainlink} confirms the conjecture that the combinatorial knot invariant defined and studied by the third author in \cite{Ng05, Ng05a, Ng08}, called ``knot contact homology'' in those works,
 indeed equals the Legendrian homology of the conormal lift. (In fact,
 the Legendrian homology described here is a nontrivial extension of
 the previously-defined version of knot contact homology; see
 Section~\ref{ssec:mainthm}.)
 We note that Legendrian homology can be expressed combinatorially in
 other circumstances, for example by Chekanov \cite{Chekanov}
 and Eliashberg \cite{Eliashberg} for $1$-dimensional Legendrian knots
 in $\R^3$ with the standard contact structure. Also, computations of
 Legendrian homology in higher dimensions have been carried out in
 some circumstances; see {\em e.g.\ }\cite{Rizell,EkholmEtnyreSullivan05a,EkholmKalman}.

The results of the present paper
 constitute one of the first complete and reasonably involved computations of Legendrian homology in higher dimensions (in the language of Lagrangian Floer theory, the computation roughly corresponds to the calculation of the differential and all higher product operations). Extensions of our techniques have already found applications, see {\em e.g.\ }~\cite{BEE2}, and we expect that further extensions may be used in other higher-dimensional situations in the future.

A more concrete consequence of Theorem~\ref{thm:mainlink} is its application to an interesting general question in symplectic topology: to what extent do symplectic- and contact-geometric objects naturally associated to objects in smooth topology remember the underlying smooth topology? More specifically, how much of the smooth topology is encoded by holomorphic-curve techniques on the symplectic/contact side? The construction that associates the (symplectic) cotangent bundle to any (smooth) manifold has been much studied recently in this regard, see {\em e.g.\ }~\cite{Abouzaid08,Abouzaid10,FukayaSeidelSmith, Nadler}. In our setting, the general question specializes to the following.

\begin{qst}
How much does the Legendrian-isotopy class of the conormal $\Lambda_K$ remember about the smooth-isotopy class of $K$? In particular, if $\Lambda_{K_1}$ and $\Lambda_{K_2}$ are Legendrian isotopic, are $K_1$ and $K_2$ necessarily smoothly isotopic?
\end{qst}

At this writing, it is possible that the answer to the second question is ``yes'' in general. One consequence of Theorem~\ref{thm:mainlink} that the answer is ``yes'' if $K_1$ is the unknot: the conormal lift detects the unknot. See Corollary~\ref{cor:unknotdetect} below.

Another geometric application of our techniques is the development of a filtered version of the Legendrian DGA associated to $\Lambda_K$ when $K$ is a link transverse to the standard contact structure $\ker(dz-y\,dx)$ in $\R^3$. This is carried out in \cite{EkholmEtnyreNgSullivan10}, which relies heavily on the computations from the present paper. There is a related combinatorial treatment in \cite{Ng10}, where it is shown that this filtered version (``transverse homology'') constitutes a very effective invariant of transverse knots. We remark that although \cite{Ng10} can be read as a standalone paper without reference to Legendrian homology or holomorphic disks, the combinatorial structure of transverse homology presented in \cite{EkholmEtnyreNgSullivan10,Ng10} was crucially motivated by the holomorphic-disk enumerations we present here.

% **************************************************
% **************************************************
\subsection{Main result}\label{ssec:mainthm}
We now turn to a more precise description of our main result. To this
end we first need to introduce some notation.
Let $K$ be an arbitrary $r$-component link in $\R^3$, given by the closure of a braid $B$ with $n$ strands. Then $\Lambda_K$ is a disjoint union of $r$ $2$-tori, which (after identifying the tangent bundle with the cotangent bundle and the normal bundle to $K$ with its tubular neighborhood) we may view as the boundaries of tubular neighborhoods of each component of $K$. We can then choose a set of generators $\lambda_1,\mu_1,\ldots,\lambda_r,\mu_r$ of $H_1(\Lambda_K)$, with $\lambda_i$ and $\mu_i$ corresponding to the longitude (running along the component of $K$) and meridian (running along a fiber of $\Lambda_K$ over a point in $K$) of the $i$-th torus.

We proceed to an algebraic definition of the DGA that 
computes the Legendrian homology of $\Lambda_K$.
The DGA is $(\A_n,\pa)$, where the algebra $\A_n$ is
the unital graded algebra over $\Z$ generated by the group ring $\Z[H_1(\Lambda_K)] = \Z[\lambda_1^{\pm
  1},\mu_1^{\pm 1},\dots,\lambda_r^{\pm 1},\mu_r^{\pm 1}]$
in degree $0$, along with the following generators:
\begin{alignat*}{2}
&\{a_{ij}\}_{1\le i,j\le n;\,i\ne j} & \text{ in degree }0,\\
&\{b_{ij}\}_{1\le i,j\le n;\,i\ne j} & \text{ in degree }1,\\
&\{c_{ij}\}_{1\le i,j\le n} & \text{ in degree }1,\\
&\{e_{ij}\}_{1\le i,j\le n} & \text{ in degree }2.
\end{alignat*}
Note that in $\A_n$, the generators $a_{ij},b_{ij},c_{ij},e_{ij}$ do not commute with each other or with nontrivial elements of $\Z[H_1(\Lambda_K)]$.

We next define a differential $\pa$ on $\A_n$.
%Temporarily
Introduce variables $\widetilde{\mu}_1,\ldots,\widetilde{\mu}_n$ of degree $0$, and write $\widetilde{\A}_n^{0}$ for the free unital algebra over $\Z$  generated by the %group
ring $\Z[\widetilde{\mu}_1^{\pm 1},\ldots,\widetilde{\mu}_n^{\pm 1}]$ and the $a_{ij}$. Now if $\sigma_k$ is a standard generator of the braid group $B_n$ on $n$ strands, then define the automorphism $\phi_{\sigma_k}\colon \widetilde{\A}_n^{0} \to \widetilde{\A}_n^{0}$ by
\[
\begin{aligned}
\phi_{\sigma_k}(a_{ij})&=
a_{ij}  \quad & i,j\ne k, k+1\\
\phi_{\sigma_k}(a_{k\,k+1})&=
-a_{k+1\,k} \quad & \\
\phi_{\sigma_k}(a_{k+1\,k})&=
-\widetilde{\mu}_{k}a_{k\,k+1}\widetilde{\mu}_{k+1}^{-1}\quad &\\
\phi_{\sigma_k}(a_{i\,k+1})&=
a_{ik}\quad & i\not= k, k+1\\
\phi_{\sigma_k}(a_{k+1\, i})&=
a_{ki}\quad & i\not=k, k+1\\
\phi_{\sigma_k}(a_{ik})&=
a_{i\,k+1} - a_{ik}a_{k\, k+1}\quad& i < k\\
\phi_{\sigma_k}(a_{ik})&=
a_{i\,k+1} - a_{ik}\widetilde{\mu}_{k}a_{k\, k+1}\widetilde{\mu}_{k+1}^{-1}\quad& i > k+1\\
\phi_{\sigma_k}(a_{ki})&=
a_{k+1\,i} - a_{k+1\,k}a_{ki}\quad &i \ne k, k+1\\
\phi_{\sigma_k}(\widetilde{\mu}_k^{\pm 1}) &= \widetilde{\mu}_{k+1}^{\pm 1} \\
\phi_{\sigma_k}(\widetilde{\mu}_{k+1}^{\pm 1}) &= \widetilde{\mu}_{k}^{\pm 1} \\
\phi_{\sigma_k}(\widetilde{\mu}_i^{\pm 1}) &= \widetilde{\mu}_{i}^{\pm 1} \quad& i\ne k,k+1.
\end{aligned}
\]
The map  $\phi$ induces a homomorphism from $B_n$ to the automorphism
group of $\widetilde{\A}^{0}_n$: if $B=\sigma_{k_1}^{\epsilon_1}\cdots
\sigma_{k_m}^{\epsilon_m}\in B_n$, with $\epsilon_l=\pm 1$ for
$l=1,\dots,m$, then $\phi_B=\phi_{\sigma_{k_1}}^{\epsilon_1}\circ
\cdots\circ\phi_{\sigma_{k_m}}^{\epsilon_m}$. See Proposition~\ref{prop:homom}.

Now let $\A^0_n$ denote the subalgebra of $\A_n$ of elements of degree $0$, generated by the $\mu_j$, $\lambda_j$, and $a_{ij}$.
For $1\leq j \leq n$, let $\alpha(j)\in\{1,\ldots,r\}$ be the number of the link component of $K$ corresponding to the $j$-th strand of $B$. Then $\phi_B$ descends to an automorphism of $\A^{0}_n$, which we also denote by $\phi_B$, by setting $\widetilde{\mu}_j = \mu_{\alpha(j)}$ for all $j$, and having $\phi_B$ act as the identity on $\lambda_j$ for all $j$.

For convenient notation we assemble the generators of $\A_n$ into $(n\times n)$-matrices.
Writing $\Mm_{ij}$ for the element in position $ij$ in the $(n\times n)$-matrix $\Mm$, we define the $(n\times n)$-matrices
\[
\begin{aligned}
\Aa:\quad &
\begin{cases}
\Aa_{ij}=a_{ij} &\text{ if }i<j,\\
\Aa_{ij}=a_{ij} \mu_{\alpha(j)} &\text{ if }i>j,\\
\Aa_{ii}=1+\mu_{\alpha(i)},
\end{cases} &\quad
\Bb:\quad &
\begin{cases}
\Bb_{ij}=b_{ij} &\text{ if }i<j,\\
\Bb_{ij}=b_{ij} \mu_{\alpha(j)} &\text{ if }i>j,\\
\Bb_{ii}=0,
\end{cases}\\
\Cc:\quad &
\begin{cases}
\Cc_{ij}=c_{ij},
\end{cases}&
\Ee:\quad &
\begin{cases}
\Ee_{ij}=e_{ij}.
\end{cases}
\end{aligned}
\]
To the braid $B$, we now associate two $(n\times n)$-matrices $\Phi_B^L,\Phi_B^R$ with coefficients in $\A_n^0$ as follows. Set $\widehat{B}$ to be the $(n+1)$-strand braid obtained by adding to $B$ an extra strand labeled $0$ that does not interact with the other strands. Then $\phi$ gives a map $\phi_{\widehat{B}}$ acting on the algebra generated by homology classes and $\{a_{ij}\}_{0\leq i,j\leq n;i\neq j}$, and we can define $\Phi_B^L,\Phi_B^R$ by
\[
\phi_{\widehat{B}}(a_{i0}) = \sum_{j=1}^n \left( \Phi^L_B \right)_{ij} a_{j0} \quad\text{and}\quad
\phi_{\widehat{B}}(a_{0j}) = \sum_{i=1}^n a_{0i} \left( \Phi^R_B \right)_{ij}.
\]
(Note that by the above formula for $\phi_{\sigma_k}$, any monomial
contributing to $\phi_B(a_{i0})$, respectively $\phi_B(a_{0j})$,
must begin with a generator of the form $a_{l0}$, respectively end
with a generator $a_{0l}$.
Also, since the $0$-th strand does not interact with the others,
$\tilde{\mu}_0$ does not appear anywhere in the expressions for
$\phi_{\widehat{B}}(a_{i0}),\phi_{\widehat{B}}(a_{0j})$, and so
$\Phi^L_B,\Phi^R_B$ have coefficients in $\A^0_n$.)

Finally, define a diagonal $(n\times n)$ matrix $\Ll$ as follows. Consider the strands $1,\dots,n$ of the braid. Call a strand \emph{leading} if it is the first strand in this ordering belonging to its component.
Define
\[
\Ll:\quad
\begin{cases}
\Ll_{ij}=0 &\text{ if }i\ne j,\\
\Ll_{ii}=\lambda_{\alpha(i)} &\text{ if the $i^{\rm th}$ strand is leading},\\
\Ll_{ii}= 1 &\text{otherwise}.
\end{cases}
\]

We can now state our main result.
\begin{thm}\label{thm:mainlink}
Let $K \subset \R^3$ be an oriented link given by the closure of a braid $B$ with $n$ strands.
After Legendrian isotopy, the conormal lift $\Lambda_K \subset U^{\ast}\R^3$ has Reeb chords in graded one-to-one correspondence with the generators of $\A_n$. Consequently, the Legendrian algebra of
\[
\Lambda_K\subset J^1(S^{2})\approx U^{\ast}\R^{3}
\]
is identified with $\A_n$. Under this identification, the differential of the Legendrian DGA is the map $\pa\colon \A_n\to\A_n$ determined by the following matrix equations:
\begin{align*}
\pa \mathbf{A} &= 0, \\
\pa \mathbf{B} &= -\Ll^{-1}\cdot\Aa\cdot\Ll\,\, +\,\, \Phi^L_B \cdot \Aa \cdot \Phi^R_B, \\
\pa \mathbf{C} &= \Aa\cdot\Ll\,\, +\,\, \Aa\cdot\Phi^{R}_B, \\
\pa \mathbf{E} &= \Bb\cdot(\Phi_B^{R})^{-1}\,\, +\,\, \Bb\cdot\Ll^{-1}\,\, -\,\,
\Phi^L_B\cdot\Cc\cdot\Ll^{-1}\,\, +\,\, \Ll^{-1}\cdot\Cc \cdot (\Phi^R_B)^{-1},
\end{align*}
where if $\Mm$ is a matrix, the matrix $\pa \Mm$ is defined by $(\pa \Mm)_{ij}=\pa(\Mm_{ij})$.
\end{thm}

Theorem \ref{thm:mainlink} has the following corollary, which establishes a conjecture of the third author \cite{Ng06,Ng08}.

\begin{cor}
For a knot $K$, the combinatorial framed knot DGA of $K$,
\label{cor:knotDGA}
as defined in \cite{Ng08}, is isomorphic to the Legendrian DGA of $\Lambda_K$.
\end{cor}

In \cite{Ng08}, the framed knot DGA is shown to be a knot invariant via a purely algebraic but somewhat involved argument that shows its invariance under Markov moves. We note that Theorem~\ref{thm:mainlink} provides another proof of invariance.

\begin{cor}
The DGA given in Theorem~\ref{thm:mainlink}, and consequently the framed knot DGA of \cite{Ng08}, is a link invariant: two braids whose closures are isotopic links produce the same DGA up to equivalence (stable tame isomorphism).
\end{cor}

\begin{proof}
Direct consequence of the fact that the Legendrian DGA is invariant under Legendrian isotopy, see Theorem~\ref{thm:LHinvariance}.
\end{proof}

Quite a bit is known about the behavior of combinatorial knot contact homology, and Corollary~\ref{cor:knotDGA} allows us to use this knowledge to deduce results about the geometry of conormal lifts. For instance, we have the following results.

\begin{cor}
A knot $K$ is isotopic to the unknot $U$ if and only if the conormal lift $\Lambda_K$ of $K$ is Legendrian isotopic to the conormal lift $\Lambda_U$ of $U$.
\label{cor:unknotdetect}
\end{cor}

\begin{proof}
The degree $0$ homology of the framed knot DGA detects the unknot \cite[Proposition~5.10]{Ng08}.
\end{proof}

\begin{cor}
If $K_1,K_2$ are knots such that $\Lambda_{K_1}$ and $\Lambda_{K_2}$ are Legendrian isotopic, then $K_1$ and $K_2$ have the same Alexander polynomial.
\label{cor:Alexander}
\end{cor}

\begin{proof}
The degree $1$ linearized homology of the framed knot DGA with respect to a certain canonical augmentation encodes the Alexander polynomial \cite[Corollary~4.5]{Ng08}.
\end{proof}

We close this subsection by discussing a subtlety hidden in the statement of Corollary~\ref{cor:knotDGA}.
The DGA $(\A_n,\pa)$ defined in Theorem~\ref{thm:mainlink} is actually an \textit{extension} of the combinatorial knot DGA introduced by the third author in \cite{Ng05,Ng05a,Ng08}, in two significant ways.\footnote{There are also some inconsequential sign differences between our formulation and the knot DGA of \cite{Ng05,Ng05a,Ng08}; see \cite{Ng10} for a proof that the different sign conventions yield isomorphic DGAs.} First, the original combinatorial knot DGA assumed that $K$ is a one-component knot, whereas our $(\A_n,\pa)$ works for general links, associating separate homology variables $\lambda_j,\mu_j$ to each component.
Second, in the original combinatorial knot DGA as presented in
\cite{Ng08}, and in the filtered version for transverse knots presented in
\cite{EkholmEtnyreNgSullivan10}, homology variables
commute with the generators $a_{ij}$, etc., while they do not commute
here. We may thus think of the original knot DGA as a
``homology-commutative'' quotient of our ``full'' DGA; see Section~\ref{sssec:Legalg} for further discussion.

Although we do not pursue this point here, it seems quite possible that the full DGA introduced here constitutes a stronger link invariant, or otherwise encodes more information, than the homology-commutative knot DGA from \cite{Ng08}. For example, the proof in \cite{Ng08} that the framed knot DGA detects the unknot uses work of Dunfield and Garoufalidis \cite{DG},  building on some deep gauge-theoretic results of Kronheimer and Mrowka \cite{KronheimerMrowka}. However, if we use the full DGA rather than the homology-commutative quotient, then there is an alternate proof that knot contact homology detects the unknot, in unpublished work of the first and third authors along with Cieliebak and Latschev, which uses nothing more than the Loop Theorem.

% **************************************************
% **************************************************
\subsection{Strategy and outline of paper}\label{ssec:outline}
We conclude this introduction with a description of the strategy of
our proof of Theorem~\ref{thm:mainlink}. A loose sketch of this
approach has previously been summarized in \cite{EkholmEtnyre05}.

The unit cotangent bundle $U^{\ast}\R^{3}$ with contact form the
restriction of the Liouville form $p\,dq$ is contactomorphic to the
1-jet space $J^1(S^2)\approx T^{\ast}S^{2}\times\R$ of $S^2$, with contact form $dz-p\,dq$ where $z$ is the coordinate along the $\R$-factor. To find the rigid holomorphic disks that contribute to Legendrian homology for a Legendrian submanifold of any $1$-jet space, one can use gradient flow trees \cite{Ekholm07}. In our case, rather than directly examining the gradient flow trees for the conormal lift $\Lambda_K \subset J^1(S^2)$ of a link $K$, we break the computation down into three steps.
First, we compute the differential for the conormal lift
$\Lambda=\Lambda_U$ of the unknot $U$ which we represent as a planar
circle. This is done by calculating gradient flow trees for a
particular small perturbation of $\Lambda$.

Second, given an arbitrary link $K$, let $B$ be a braid whose closure is $K$. We can view the closure of $B$ as lying in the solid torus given as a small tubular neighborhood of $U$, and we can thus realize $K$ as a braid that is $C^1$-close to $U$. Then $\Lambda_K$ lies in a small neighborhood of $\Lambda$, and by the Legendrian neighborhood theorem we can choose this neighborhood to be contactomorphic to the $1$-jet space $J^1(\Lambda) = J^1(T^2)$:
\[
\Lambda_K \subset J^1(\Lambda) = J^1(T^2) \subset J^1(S^2).
\]
We can use gradient flow trees to find the rigid holomorphic disks in $J^1(T^2)$ with boundary on $\Lambda_K$. This turns out to be easier than directly calculating the analogous disks in $J^1(S^2)$ because $\Lambda_K$ is everywhere transverse to the fibers of $J^1(T^2)$ ({\em i.e.\ }~its caustic is empty). We can assemble the result, which is computed in terms of the braid $B$, as the Legendrian DGA of $\Lambda_K \subset J^1(T^2)$, which is a subalgebra of the Legendrian DGA of $\Lambda_K \subset J^1(S^2)$.

Finally,  we prove that there is a one-to-one correspondence between
rigid holomorphic disks in $J^1(S^2)$ with boundary on $\Lambda_K$,
and certain objects that we call \textit{rigid multiscale flow trees}
determined by $\Lambda$ and $\Lambda_K$, which arise as follows. As we
let $K$ approach $U$, $\Lambda_K$ approaches $\Lambda$ and each
holomorphic disk with boundary on $\Lambda_K$ approaches a holomorphic
disk with boundary on $\Lambda$ with partial flow trees of
$\Lambda_K\subset J^{1}(\Lambda)$ attached along its boundary. Here
the flow trees correspond to the thin parts of the holomorphic disks
before the limit; in these parts, the energy approaches zero. In a multiscale flow tree we substitute the holomorphic disk with boundary on $\Lambda$ in this limit by  the corresponding flow tree computed in the first step, and obtain a flow tree of $\Lambda\subset J^{1}(S^{2})$ with portions of flow trees of $\Lambda_K\subset J^1(\Lambda)$ attached along its boundary.

Here is the plan for the rest of the paper. In Section~\ref{sec:background}, we present background material, including the definitions of the conormal construction, Legendrian homology, and gradient flow trees. In Section~\ref{sec:diffthrflowtree}, we use gradient flow trees to accomplish the first two steps in the three-step process outlined above: calculating holomorphic disks for $\Lambda \subset J^1(S^2)$ and $\Lambda_K \subset J^1(T^2)$. We extend these calculations to multiscale flow trees in Section~\ref{sec:combdiff} to count holomorphic disks for $\Lambda_K \subset J^1(S^2)$, completing the proof of Theorem~\ref{thm:mainlink}.

The computations in Section~\ref{sec:combdiff} rely on some technical results about multiscale flow trees whose proofs are deferred to the final two sections of the paper. In Section~\ref{sec:diskandgentree}, we establish Theorem~\ref{thm:diskandgentree}, which gives a one-to-one correspondence between holomorphic disks and multiscale flow trees. In Section~\ref{sec:orientations}, we prove Theorems~\ref{thm:signunknottree} and \ref{thm:combsign}, which deal with combinatorial signs arising from choices of orientations of the relevant moduli spaces of flow trees and multiscale flow trees.

% **************************************************
% **************************************************
\subsection*{Acknowledgments}
TE was partially supported by the G\"oran Gustafsson Foundation for Research in Natural Sciences and Medicine.
JBE was partially supported by NSF grant DMS-0804820.
LLN was partially supported by
NSF grant DMS-0706777 and NSF CAREER grant DMS-0846346.
MGS was partially supported by NSF grants DMS-0707091 and
DMS-1007260. The authors would also like to thank MSRI for hosting them in spring 2010 during part of this collaboration. Finally, we thank Cecilia Karlsson for useful remarks on coherent orientations.

% **************************************************
% **************************************************
% **************************************************
\section{Conormal Lifts, Legendrian Homology, and Flow Trees}\label{sec:background}
In this section we review  background material. We begin by discussing
the conormal lift construction for links in $\R^{3}$ in
Section~\ref{ssec:conormal}, and place it in the context of $1$-jet
spaces of surfaces in Section~\ref{ssec:leginjet}.  In
Section~\ref{ssec:conhom}, we review
the definition of Legendrian homology. For our purposes, the
holomorphic disks counted in Legendrian homology are replaced by flow
trees, which we discuss in Section~\ref{ssec:gft}. Vector splitting along
flow trees, which is needed later when assigning signs to rigid flow trees,
is discussed in Section~\ref{ssec:vectorsplit}.
We end with a compilation in
Section~\ref{ssec:alg} of algebraic results
about the map $\phi_B$ and the matrices $\Phi^L_B$ and $\Phi^L_R$ that
are used in the proof of our main result, Theorem~\ref{thm:mainlink}.

% **************************************************
% **************************************************
\subsection{The conormal construction}\label{ssec:conormal}
Fixing the standard flat metric on $\R^3$ we consider the unit cotangent bundle $U^*\R^3\approx\R^3\times S^2$ of $\R^3$. The Liouville form on the cotangent bundle $T^*\R^3$ is
\[
\theta=\sum_{i=1}^3 p_i\,dq_i,
\]
where $q=(q_1,q_2,q_3)$ are coordinates on $\R^3$ and $p=(p_1,p_2,p_3)$ are coordinates in the
fibers of $T^*\R^3$. The restriction $\theta|_{U^{\ast}\R^{3}}$ is a
contact $1$-form. We denote its associated contact structure $\xi=\krn\,
\theta$.

The standard contact $1$-form on the $1$-jet space $J^{1}(S^{2})=T^{\ast}S^{2}\times\R$ of $S^{2}$ is given by
\[
\alpha=dz-\theta,
\]
where $z$ is a coordinate in the $\R$-factor and where $\theta$ is the Liouville form on $T^{\ast}S^{2}$. Using the standard inner product $\la\cdot\,,\cdot\ra$ on $\R^{3}$ to identify vectors and covectors we may write $T^{\ast}S^{2}$ as follows:
\[
T^{\ast}S^{2}=\left\{(x,y)\in\R^{3}\times \R^{3}\colon |x|=1\,,\,\la x,y\ra=0\right\}.
\]
We define the map $\phi\colon U^{\ast}\R^{3}\to J^{1}(S^{2})=T^{\ast}S^{2}\times\R$ by
\begin{equation}\label{eq:maincontacto}
\phi(q,p)=\left(p\,\,,\,\,q-\la q,p\ra p\,\,,\,\, \la q,p\ra\right)
\end{equation}
and notice that $\phi$ is a diffeomorphism such that $\phi^{\ast}\alpha=\theta$. Thus $U^{\ast}\R^{3}$ and $J^{1}(S^{2})$ are contactomorphic.

Let $K$ be a knot or link in $\R^3$. We associate to $K$ its conormal
lift
\[
\Lambda_K=\left\{u\in U^*\R^3|_K\colon u(v)=0 \text{ for all } v\in TK\right\},
\]
which topologically is a union of tori, one for each component of $K$.
By definition $\theta|_{\Lambda_K}=0$, {\em i.e.}, $\Lambda_{K}$ is Legendrian. Furthermore,
smooth isotopies of $K$ induce Legendrian isotopies of $\Lambda_K$. In particular the Legendrian
isotopy class of $\Lambda_K$ (and consequently any invariant thereof) is an invariant of the isotopy class of $K$. For more on this construction see \cite{EkholmEtnyre05}.

% **************************************************
% **************************************************
\subsection{Legendrian submanifolds in the $1$-jet space of a surface}\label{ssec:leginjet}
Let $S$ be a surface and
consider a Legendrian submanifold $\Lambda\subset
J^{1}(S)=T^{\ast}S\times\R$. After a small perturbation of $\Lambda$
we can assume that $\Lambda$ is in general position with respect to
the {\em Lagrangian projection} $\Pi\colon J^1(S)\to T^*S$, the {\em
  front projection} $\Pi_F\colon J^1(S)\to S\times\R$, and the {\em
  base projection} $\pi\colon J^1(S)\to S$.

General position for $\Pi$ means that the self intersections
of the Lagrangian immersion $\Pi(L)$ consists of transverse double points. General position for $\Pi_F$ means that singularities of $\Pi_{F}|_{\Lambda}$ are of two types: {\em cusp edges} and {\em swallow tails}. For a more precise description we first introduce some notation: the {\em caustic} $\Sigma\subset \Lambda$ is the critical set of $\Pi_F|_\Lambda$. General position for $\Pi_F$ first implies that $\Sigma$ is a closed $1$-submanifold along which the rank of
$\Pi_F|_{\Lambda}$ equals $1$. The kernel field $\ker(d\Pi_F|_{\Lambda})$ is then a line field $l$ along
$\Sigma$ and general position for $\Pi_F$ implies that
$l$ has only transverse tangencies with $T\Sigma$ ({\em i.e.\ }~all tangencies have order one). This gives a stratification of the caustic:
the $1$-dimensional top stratum, called the {\em cusp edge}, consisting of points where $l$ is transverse to $\Sigma$, and the $0$-dimensional complement, called the set of {\em swallow tail points}, where $l$ is tangent to $T\Sigma$ of order $1$.
Finally, general position for $\pi$ means that the image of the caustic
under $\pi$ is stratified:
$\pi(\Sigma)=\Sigma_1\supset(\Sigma_2\cup\Sigma_2^{\rm sw})$, where
$\Sigma_2$ are transverse self-intersections of $\pi(\Sigma)$ and
$\Sigma_{2}^{\rm sw}$ is the image under $\pi$ of the swallow tail
points around which $\pi(\Sigma)$ has the form of a semi-cubical cusp.

For $\Lambda$ in general position we obtain the following local descriptions:
\begin{itemize}
\item If $p\in \pi(\Lambda)-\Sigma_1$ then $p$ has a neighborhood
  $U\subset S$ such that $\pi^{-1}(U)\cap \Lambda\subset J^{1}(S)$ is
  the
union of a finite number of $1$-jet graphs of functions $f_1,\ldots, f_n$ on $U$. We call such
functions \dfn{local defining functions of $\Lambda$}.
\begin{figure}[htb]
\centering
\includegraphics{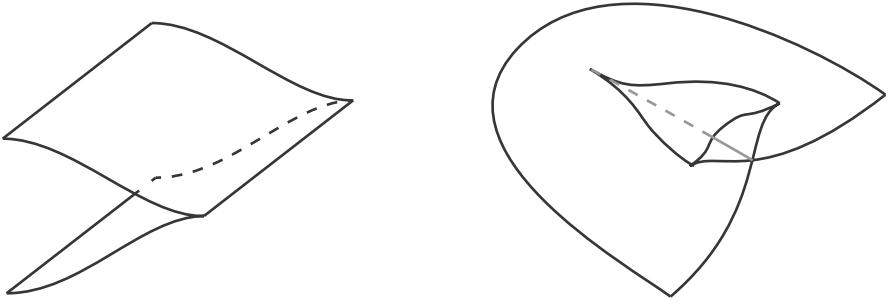}
\caption{A local model for the cusp edge is shown on the left and a local model for the swallow tail singularity is shown on the right.}
\label{fig:locmod}
\end{figure}
\item  If $p\in\Sigma_1-(\Sigma_2\cup\Sigma_{2}^{\rm sw})$ then there is a
neighborhood $U\subset S$ of $p$ which is diffeomorphic to the open unit disk, $U=U_+\cup U_-$,
where $U_+$ (respectively $U_-$) corresponds to the upper (respectively lower) half disk and
where $\pi(\Sigma)\cap U$ corresponds to $U_+\cap U_-$, with the
following properties. The intersection $\pi^{-1}(U)\cap \Lambda\subset
J^{1}(S)$ consists of $n\ge 0$ smooth sheets given by the $1$-jet
graphs of local defining functions and one cusped sheet given by the
$1$-jet graph of two functions $h_0,h_1\colon U_+\to \R$. We can
choose coordinates $(x_1,x_2)$ near $\pa U_+$ so that
$h_0$ and $h_1$ have the normal form
\begin{align*}
h_0(x_1, x_2)&= \frac 13 (2x_1)^{\frac 32} + \beta x_1 + \alpha x_2\\
h_1(x_1,x_2) &= -\frac 13 (2x_1)^{\frac 32} + \beta x_1 + \alpha x_2
\end{align*}
for some constants $\alpha$ and $\beta$, see \cite[Equation
(2--1)]{Ekholm07}. In particular $dh_0$ and $dh_1$ agree near the
boundary, see the left diagram in Figure~\ref{fig:locmod}. We also call $h_0$ and $h_1$ local defining functions (for a cusped sheet).
\item If $p\in\Sigma_2^{\rm sw}$ then there is a neighborhood $U$ in
  which $n\ge 0$ sheets are smooth and one sheet is a standard swallow
  tail sheet, see \cite[Remark 2.5]{Ekholm07} and the right diagram in
  Figure~\ref{fig:locmod}.
\item If $p\in\pi(\Sigma_{2})$ there is a neighborhood $U$ diffeomorphic to the unit disk over which the Legendrian consists of $n\ge 0$ smooth sheets and two cusped sheets, one defined over the upper half-plane and one over the right half-plane.
\end{itemize}
The following simple observation concerning double points of $\Pi$ will be important in the definition of the Legendrian homology algebra.
\begin{lem}\label{lem:critpts}
A point $s\in \Pi(\Lambda)$ is a double point if and only if $p=\pi(\Pi^{-1}(s))$ is a critical point of the difference of two local defining functions for $\Lambda$ at $p$.
\end{lem}

% **************************************************
% **************************************************
\subsection{Legendrian homology in the $1$-jet space of a surface}\label{ssec:conhom}
In this subsection we briefly review Legendrian (contact) homology in
$J^{1}(S)$, where $S$ is an orientable surface. In fact, for our
applications it is sufficient to consider $S\approx S^{2}$ and
$S\approx T^{2}$. We refer the reader to
\cite{EkholmEtnyreSullivan05b} for details on the material presented
here.

% **************************************************
\subsubsection{Geometric preliminaries}\label{sssec:geomprelim}
As in Section \ref{ssec:conormal}, we consider $J^{1}(S)=T^{\ast} S\times\R$, with contact form $\alpha=dz-\theta$, where $z$ is a coordinate in the $\R$-factor and $\theta$ is the Liouville form on $T^{\ast}S$.  
The Reeb vector field $R_{\alpha}$ of $\alpha$ is given by
$R_{\alpha}=\pa_z$ and thus Reeb flow lines are simply $\{p\}\times\R$
for any $p\in T^{\ast}S$. In particular, if $\Lambda\subset J^{1}(S)$
is a Legendrian submanifold then Reeb chords of $\Lambda$ (Reeb flow
segments which begin and end on $\Lambda$) correspond to pairs of
distinct points $y_1,y_2\in \Lambda$ such that $\Pi(y_1)=\Pi(y_2)$, where $\Pi\colon J^1(S)\to T^{\ast}S$ is the Lagrangian projection. As
noted in Lemma~\ref{lem:critpts}, for such Legendrian submanifolds
there is thus a bijective correspondence between Reeb chords of
$\Lambda\subset T^{\ast}S\times\R$ and double points of $\Pi(\Lambda)$
and in this case we will sometimes use the phrase ``Reeb chord'' and
``double point'' interchangeably. As mentioned in Section~\ref{ssec:leginjet}, after small perturbation
of $\Lambda$, $\Pi|_{\Lambda}$ is a self-transverse Lagrangian
immersion with a finite number of double points.

Consider an almost complex structure on $T^{\ast} S$ which is compatible with $d\theta$, by which we can view $T(T^{\ast}S)$ as a complex vector bundle. Since $T(T^{\ast}S)|_{S}$ is the complexification of $T^{\ast}S$ we find that $c_1(T(T^{\ast}S))=0$ and hence there is a complex trivialization of
$T(T^{\ast}S)$. The orientation of $S$ induces a trivialization of the
real determinant line bundle (second exterior power) $\Lambda^{2}T^{\ast}M$ which in turn gives a
trivialization, called the \emph{orientation trivialization}, of the complex determinant line bundle
$\Lambda^{2}T(T^{\ast}S)$ (second exterior power over $\C$). The complex trivialization of
$T(T^{\ast}S)$ is determined uniquely up to homotopy by the condition
that it induces a complex trivialization of
$\Lambda^{2}T(T^{\ast}S)$ which is homotopic to the orientation trivialization. We will work throughout with complex trivializations of $T(T^{\ast}S)$ that satisfy this condition.

Let $\Lambda\subset T^{\ast}S\times\R$ be a Legendrian submanifold and
let $\gamma$ be a loop in $\Lambda$. Then the tangent planes of
$\Pi(\Lambda)$ along $\Pi(\gamma)$ constitute a loop of Lagrangian
subspaces. Using the trivialization of $T(T^{\ast}S)$ we get a loop of
Lagrangian subspaces in $\C^{2}$; associating to each loop $\gamma$
the Maslov index of this loop of Lagrangian subspaces gives a
cohomology class $\mu\in H^{1}(\Lambda;\Z)$ called the {\em Maslov
  class} of $\Lambda$.   We assume throughout that the Maslov class of
$\Lambda$ vanishes so that the Maslov index of any such loop equals
$0$. (Note that this implies that $\Lambda$ is orientable.) In order
to orient the moduli spaces of holomorphic disks with boundary on
$\Lambda$ we will further assume that $\Lambda$ is spin and equipped
with a spin structure. See \cite{EkholmEtnyreSullivan05c} for
more details. For future use, we note that the Maslov index of the
loop of Lagrangian tangent spaces as described can be computed by
first making the loop generic with respect to fibers of $T^{\ast}S$
and then counting (with signs) the number of instances where the
tangent space has a $1$-dimensional intersection with the tangent
space of some fiber. In terms of the front projection, once
$\gamma$ is generic (in particular, transverse to the cusp edges and
disjoint from the swallow tail points), one counts the number of
times the curve transversely intersects cusp edges going down (that is
with the $\R$ coordinate of $S\times \R$ decreasing)  minus the number
of times it transversely intersects cusp edges going up.

% **************************************************
\subsubsection{The Legendrian algebra}\label{sssec:Legalg}

In the remainder of the present subsection, Section~\ref{ssec:conhom}, we
describe the differential graded
algebra $(LA(\Lambda),\pa)$, which we call the \textit{Legendrian
  DGA}, associated to a
Legendrian submanifold $\Lambda \subset J^1(S)$, whose homology is the
\textit{Legendrian homology} of $\Lambda$. This description is divided
into three parts: the algebra, the grading, and the differential.

The algebra $LA(\Lambda)$ is simple to describe; in particular, the Legendrian algebra is simpler in our setting than for general contact manifolds since $J^1(S)$ has no closed Reeb orbits. Assume that $\Lambda$
is in general position so that $\Pi|_{\Lambda}$ is a Lagrangian
immersion with transverse self-intersections. Let $\ch$ denote the
set of Reeb chords of $\Lambda$. Then $LA(\Lambda)$ is the
noncommutative unital algebra over $\Z$ generated by:
\begin{itemize}
\item
elements of $\ch$ (Reeb chords) and
\item
$\Z[H_1(\Lambda)]$ (homology classes).
\end{itemize}
Thus a typical generator of $LA(\Lambda)$ viewed as a $\Z$-module is a monomial of the form
\[
\gamma_0 q_1 \gamma_1 q_2 \cdots q_m \gamma_m
\]
where $q_j \in \ch$ and $\gamma_j \in H_1(\Lambda)$, and
multiplication of generators is the obvious multiplication (with $H_1(\Lambda)$ viewed as a multiplicative group). Note that
homology classes do not commute with Reeb chords; $LA(\Lambda)$ is
more precisely defined as the tensor algebra over $\Z$ generated by
elements of $\ch$ and elements of $H_1(\Lambda)$, modulo the relations
given by the relations in $H_1(\Lambda)$.

To simplify notation, for the remainder of the paper we will assume that
$\Lambda$ is a disjoint union of oriented $2$-tori
$\Lambda_1,\ldots,\Lambda_r$. We will further assume that each
component $\Lambda_j$ is equipped with a fixed symplectic basis
$(\lambda_j,\mu_j)$ of $H_1(\Lambda_j)$ (for conormal lifts, these
correspond to the
longitude and meridian of the link component in $\R^3$). Then
$\Z[H_1(\Lambda)] = \Z[\lambda_1^{\pm 1},\mu_1^{\pm 1},\ldots,
\lambda_r^{\pm 1},\mu_r^{\pm 1}]$.

\begin{rmk}
In the subject of Legendrian homology, it is often customary to
quotient by commutators between Reeb chords and homology classes, to
obtain the \textit{homology-commutative} algebra. This quotient can
also be described as the tensor algebra over $\Z[H_1(\Lambda)]$ freely
generated by elements of $\ch$. The homology-commutative algebra is
the version of the Legendrian algebra considered in many sources, in
particular the combinatorial formulation of knot contact homology in
\cite{Ng08} and the transverse version in \cite{EkholmEtnyreNgSullivan10,Ng10}.

From the geometric viewpoint of the present paper, there is no reason
to pass to the homology-commutative quotient, and we will adhere to
the rule that homology classes do not commutate with Reeb chords in
the Legendrian algebra. There are indications that the fully
noncommutative Legendrian DGA may be a stronger invariant than the
homology-commutative quotient; see the last paragraph of Section~\ref{ssec:mainthm}.
\end{rmk}

\begin{rmk}
When $\Lambda$ has more than one component, there is an algebra
related to the Legendrian algebra $LA(\Lambda)$, called the
\textit{composable algebra}, which is sometimes a more useful object
to consider than $LA(\Lambda)$. See for instance \cite{BEE,BEE2}. We do not
need the composable algebra in this paper, but we briefly describe its
definition here and note that 
one can certainly modify our
definition of knot contact homology to the composable setting; that is,
the Legendrian DGA for knot contact homology descends to a differential on the composable algebra.

Suppose $\Lambda$ has $r$ components $\Lambda_1,\ldots,\Lambda_r$. Let $\ri$ denote the ring
\[
\ri=\bigoplus_{j=1}^{r} \Z[H_1(\Lambda_j)]
\]
with multiplication given as follows: if $\gamma_1 \in
\Z[H_1(\Lambda_{j_1})]$ and $\gamma_2 \in \Z[H_1(\Lambda_{j_2})]$,
then $\gamma_1 \cdot \gamma_2$ is $0$ if $j_1 \neq j_2$, or
$\gamma_1\gamma_2 \in \Z[H_1(\Lambda_{j_1})] \subset \ri$ if $j_1=j_2$.
 (Note that
$\ri$ is nearly but not quite a quotient of $\Z[H_1(\Lambda)]$: the
identity in $\Z[H_1(\Lambda_j)]$ is an idempotent in $\ri$ distinct from $1$.) Let
$\ch_{ij}$ denote the set of Reeb chords beginning on $\Lambda_i$ and
ending on $\Lambda_j$. A
\textit{composable monomial} in $LA(\Lambda)$ is a monomial of the
form
\[
\gamma_0 q_1 \gamma_1 q_2 \cdots q_l \gamma_l
\]
for some $l\geq 0$, such that there exist
$i_0,\ldots,i_l\in\{1,\ldots,m\}$ for which $\gamma_{i_j} \in
H_1(\Lambda_{i_j})$, viewed as an element of $\ri$, and $q_j \in \ch_{i_{j-1}i_j}$ for all $j$.

The \textit{composable algebra} is the $\Z$-module
generated by composable monomials, with multiplication given in the
obvious way; note that the product of two composable monomials is either $0$ or another composable monomial. This algebra, which is naturally the path algebra of a quiver with vertices given by components of $\Lambda$ and edges given by Reeb chords, is almost but not quite a quotient of
$LA(\Lambda)$: if we replace $\Z[H_1(\Lambda)]$ by $\ri$ in the definition of $LA(\Lambda)$, then the composable algebra is the quotient setting non-composable monomials to $0$.
\end{rmk}

% **************************************************
\subsubsection{Grading in the Legendrian algebra}
In order to define the grading on $LA(\Lambda)$ we fix a point $p_j\in\Lambda_j$ on each component $\Lambda_j$ of $\Lambda,$ and for each Reeb chord endpoint in $\Lambda_j$ we choose an {\em endpoint path} connecting the endpoint to $p_j$. Furthermore, for $j=1,\ldots,r$, we choose paths $\gamma_{1j}$ in $T^{\ast}S$ connecting $\Pi(p_1)$ to $\Pi(p_j)$ and symplectic trivializations of $\gamma_{1j}^{\ast}T(T^{\ast}S)$ in which the tangent space $\Pi(T_{p_1}\Lambda)$ corresponds to the tangent space $\Pi(T_{p_j}\Lambda)$; for $j=1$, $\gamma_{11}$ is the trivial path. For any $i,j \in \{1,\ldots,r\}$, we can then define $\gamma_{ij}$ to be the path $-\gamma_{1i} \cup \gamma_{1j}$ joining $\Pi(p_i)$ to $\Pi(p_j)$, and $\gamma_{ij}^{\ast}T(T^{\ast}S)$ inherits a symplectic trivialization from the trivializations for $\gamma_{1i}$ and $\gamma_{1j}$.

The grading $|\cdot|$ in $LA(\Lambda)$ is now the following. First, homology variables have degree $0$:
\[
|\lambda_j|=|\mu_j|=0\quad\text{ for }j=1,\dots,r.
\]
(Recall that the Maslov class $\mu$ of $\Lambda$ is assumed to vanish; in general a homology variable $\tau\in H_1(\Lambda)$ would be graded by $-\mu(\tau)$.)
Second, if $q$ is a Reeb chord we define the grading by associating a path of Lagrangian subspaces to $q$. We need to consider two cases according to whether the endpoints of the chord lie on the same component of $\Lambda$ or not.
In the case of equal components, both equal to $\Lambda_j$, consider the path of tangent planes along the endpoint path from the final point of $q$ to $p_j$ followed by the reverse endpoint path from $p_j$ to the initial point of $q$. The endpoint of this path is the tangent space of $\Pi(\Lambda)$ at initial point of $q$. We close this path of Lagrangian subspaces to a loop $\widehat\gamma_q$ by a ``positive rotation'' along the complex angle between the endpoints of the path (see \cite{EkholmEtnyreSullivan05a} for details). In the case of different components, we associate a loop of Lagrangian subspaces $\widehat \gamma_{q}$ to $q$ in the same way except that we insert the path of Lagrangian subspaces induced by the trivialization along the chosen path, $\gamma_{ij},$ connecting the components of the endpoints in order to connect the two paths from Reeb chord endpoints to chosen points. The grading of $q$ is then
\[
|q|=\mu(\widehat\gamma_q)-1,
\]
where $\mu$ denotes the Maslov index.

\begin{rmk}\label{rmk:gradingmixed}
Note that the grading of a pure Reeb chord (a chord
whose start and endpoint lie on the same component of $\Lambda$) is well-defined ({\em i.e.\ }~independent of choice of paths to the base point and symplectic trivializations) because the Maslov class vanishes. The grading of mixed chords however depends on the choice of trivializations along the paths $\gamma_{1j}$ connecting base points. Changing the trivializations changes the gradings as follows: for some fixed $(n_1,\ldots,n_r) \in \Z^r$, for all $i,j$, $|q|$ is replaced by $|q|+n_i-n_j$ for all Reeb chords $q$ beginning on component $i$ and ending on component $j$. If we choose orientations on each component of $\Lambda$, and stipulate that the base points $p_j$ are chosen such that $\pi\colon \Lambda \to S$ is an orientation-preserving local diffeomorphism at each $p_j$ (and that the symplectic trivializations on $\gamma_{1j}$ preserve orientation), then the mod $2$ grading of the mixed chords is well-defined, independent of the choice of base points.
\end{rmk}

For computational purposes we mention that the grading can be computed
in terms of the front projection $\Pi_F\colon J^1(S)\to S\times\R$, compare \cite{EkholmEtnyreSullivan05b}.  By Lemma~\ref{lem:critpts}, a Reeb chord $q$
corresponds to a critical point $x$ of the difference of two local defining functions $f_1$ and
$f_2$ for $\Lambda\subset J^1(S)$. 
We make the following assumptions, which hold generically: $\Lambda$ is in general position with respect to the front projection; the critical point is
non-degenerate with index denoted by $\ix_x(f_1-f_2)$, where $f_1$ defines the upper local sheet (the sheet with the larger $z$-coordinate) of $\Lambda$; and the base points $p_j$ do not lie on cusp edges. 
\begin{lem}\label{lem:indexcritpt}
There is a choice of trivializations for which the following grading formula holds for all Reeb chords.
Let $q$ be a Reeb chord with final point (respectively initial point)
in component $\Lambda_j$ (respectively $\Lambda_i).$
Let $\gamma$ be the union of the endpoint path from
the chord's final point to the base point $p_j,$ and the reverse  endpoint path from base point $p_i$ to the chord's initial point. Assume $\gamma$ is in general position with respect to the stratified caustic $\Sigma\subset\Lambda$. Then
\begin{equation}\label{eq:frontnu}
    |q| = D(\gamma) - U (\gamma) + \ix_{x}(f_1-f_2)-1,
\end{equation}
where $D(\gamma)$ (respectively $U(\gamma)$) is the number of cusps that $\gamma$ traverses in the
downward (respectively upward) $z$-direction.
\end{lem}

% **************************************************
\subsubsection{Differential in the Legendrian algebra}
In general, the Legendrian algebra differential of a Legendrian submanifold $\Lambda$ in a contact manifold $Y$ is defined using moduli spaces of holomorphic curves in the symplectization $\R\times Y$ of $Y$ with Lagrangian boundary condition $\R\times\Lambda$. In our case, $Y =J^1(S),$ one can instead use holomorphic disks in $T^{\ast}S$ with boundary on $\Pi(\Lambda)$. We give a brief description. See \cite{EkholmEtnyreSullivan07} for details and \cite{EkholmEtnyreSullivan05b} for the relation to curves in the symplectization.

The differential $\pa\colon LA(\Lambda)\to LA(\Lambda)$ is defined on generators and then extended by linearity over $\Z$ and the signed Leibniz rule,
\[
\partial (vw)= (\partial v)w\,+\,(-1)^{|v|}v(\pa w).
\]
We set $\pa\lambda_j=\pa\mu_j=0$ for $j=1,\dots,m$. It thus remains to
define the differential on Reeb chords.

To do this, we begin by fixing an almost complex structure $J$ on $T(T^{\ast}S)$ that is compatible with $d \theta$. Let $q_0,q_1,\dots,q_k$ be Reeb chord generators of $LA(\Lambda)$. Let $D_{k+1}$ be the unit disk in $\C$ with $k+1$ boundary punctures $z_0,\ldots,z_k$ listed in counterclockwise order. We
consider maps
\[
u:(D_{k+1},\partial D_{k+1}-\{z_0,\dots,z_k\})\to (T^{\ast}S, \Pi(\Lambda))
\]
such that $u|_{\pa D_{k+1}-\{z_0,\dots,z_m\}}$ lifts to a continuous map $\widetilde u$ into $\Lambda\subset T^{\ast}S$. Call a puncture $z$ mapping to the double point $q$
\dfn{positive} (respectively \dfn{negative}) if the lift of the arc just clockwise of $z$ in $\partial
D_{k+1}$ is a path in $\Lambda$ approaching the upper Reeb chord endpoint $q^+$ (respectively the lower endpoint $q^-$) and the arc just counterclockwise of $z$
lifts to a path approaching $q^-$ (respectively $q^+$). For a $(k+1)$-tuple of homology classes $\vA=(A_0,\dots,A_k)$, $A_j\in H_1(\Lambda;\Z)$, $j=0,\dots,k$, we let
\[
\M_{\vA}(q_0;q_1,\dots,q_k)
\]
denote the moduli space of $J$-holomorphic maps
\[
u\colon (D_{k+1},\pa D_{k+1}-\{z_0,\dots,z_k\})\to(T^{\ast}S,\Pi(\Lambda))
\]
with the following properties: $u|_{\pa D_{k+1}-\{z_0,\dots,z_k\}}$ lifts to a continuous map $\widetilde u$ into $\Lambda$, $z_0$ is a positive puncture mapping to $q_0$, $z_j$ is a negative puncture mapping to $q_j$, $j=1,\dots,k$, and when $\widetilde u|_{(z_j,z_{j+1})}$ is completed to a loop using the endpoint paths then it represents the homology class $A_j$. Here $(z_j,z_{j+1})$ denotes the boundary interval in $\pa D_{k+1}$ between $z_j$ and $z_{j+1}$ and we use the convention $z_{k+1}=z_0$.

For a generic $\Lambda$ and $J$ the following holds, see \cite{EkholmEtnyreSabloff09,EkholmEtnyreSullivan05a} for details:
\begin{itemize}
\item $\M_{\vA}(q_0; q_1,\ldots, q_k)$ is a manifold of dimension (recall that the Maslov class of $\Lambda$ is assumed to vanish)
\[
\dim\left(\M_{\vA}(q_0; q_1,\ldots, q_k)\right)=
|q_0|-\sum_{i=1}^k |q_i| -1
\]
which is transversely cut out by the $\overline{\partial}_J$-operator. Furthermore it admits a compactification as a manifold with boundary with corners in which the boundary consists of broken disks. Consequently, if the dimension equals 0 then the manifold is compact.
\item The moduli spaces $\M_{\vA}(q_0;q_1,\ldots,q_k)$ determined by $\Lambda$ and $J$ can be ``coherently'' oriented, see Section~\ref{sec:orientations}. (Note that the assumption that the components of $\Lambda$ are tori and thus admit spin structures is used here. The moduli space orientations depend on the choice of spin structure on $\Lambda$.)
\end{itemize}

The differential of a Reeb chord generator $q_0$ is then defined as follows:
\[
\pa q_0=
\sum_{{\text{\tiny $\begin{matrix}
&\vA=(A_0,\dots,A_k),\\
&\sum_{j=1}^k|q_j|=|q_0|-1\\
\end{matrix}$}}
}
\left|\M_{\vA}(q_0;q_1\ldots q_k)\right|\, A_0q_1A_1q_2A_2\ldots A_{k-1}q_kA_k,
\]
where $|\M|$ denotes the algebraic number of points in the oriented compact $0$-manifold $\M$.

% **************************************************
\subsubsection{Invariance of Legendrian homology}
The main properties of $LA(\Lambda)$ are summarized in the following
theorem.
\begin{thm}[Ekholm, Etnyre and Sullivan 2007, \cite{EkholmEtnyreSullivan07}]
The map $\pa\colon LA(\Lambda)\to LA(\Lambda)$ is a differential, that is $\pa^{2}=0$. The stable tame isomorphism class of $LA(\Lambda)$ is an invariant of $\Lambda$ up to Legendrian isotopy; in particular, its homology $L{H}(\Lambda)$ is a Legendrian isotopy invariant.
(For the notion of stable tame isomorphism in this setting, see \cite{EkholmEtnyreSullivan05b}.)
\label{thm:LHinvariance}
\end{thm}
This result is stated and proven in \cite{EkholmEtnyreSullivan07} for the
homology-commutative quotient, but the proof there extends verbatim to the full
noncommutative algebra.

% **************************************************
% **************************************************
\subsection{Flow trees}\label{ssec:gft}
Consider a Legendrian submanifold $\Lambda\subset
J^{1}(S)=T^{\ast}S\times\R$ as in Section \ref{ssec:leginjet}. There is a Morse-theoretic description of the differential in the Legendrian DGA of $\Lambda$ via flow trees, as developed in \cite{Ekholm07}, which we will describe in this subsection (in less generality than \cite{Ekholm07}). The motivation is as follows.
For
$\sigma>0$, the map
\[
\phi_\sigma\colon T^{\ast}S\times \R\to T^{\ast}S\times\R:(q, p, z) \mapsto (q,\sigma p, \sigma z),
\]
satisfies $\phi_{\sigma}^{\ast}(dz -\theta)=\sigma(dz-\theta)$. Hence $\Lambda_\sigma=\phi_\sigma(\Lambda)$ is a Legendrian submanifold that is Legendrian isotopic to $\Lambda$. For $\sigma>0$ small enough there are regular almost complex structures for which there is a one-to-one correspondence between rigid holomorphic disks with boundary on $\Lambda_\sigma$ with one positive puncture and rigid flow trees determined by $\Lambda$, see \cite{Ekholm07}.

We now define flow trees.
Fix a metric $g$ on $S$. Then two local defining functions $f_0$ and $f_1$ for $\Lambda$ defined on the same open set in $S$ define a local vector field on $S$:
\[
-\nabla(f_1-f_0),
\]
where $\nabla$ denotes the $g$-gradient. The \dfn{$1$-jet lift} of a path $\gamma:(-\epsilon, \epsilon)\to S$ is a pair of continuous paths
$\gamma_i:(-\epsilon,\epsilon)\to \Lambda$, $i=0,1$ with the following properties:  $\pi\circ\gamma_i=\gamma$ and either $\gamma_0(t)\ne\gamma_1(t)$ or $\gamma_0(t)=\gamma_1(t)$ is a point in $\Sigma$. A path $\gamma\colon(a,b)\to S$ is called a \dfn{flow line of $\Lambda$} if it has a $1$-jet lift $\gamma_i, i=0,1$, such that for each
$t\in (a,b)$ there are local defining functions $f_0,f_1$ defined near $\gamma(t)$ such that $\gamma_i$ lies in the sheet determined by $f_i$, $i=0,1$, and
\[
\dot\gamma(t)=-\nabla(f_1-f_0)(\gamma(t)).
\]
See Figure~\ref{fig:flowline}.
If $\gamma$ is a flow line with $1$-jet lift $\gamma_0,\gamma_1$ we define its \dfn{cotangent lift} as $\Pi\circ\gamma_0,\Pi\circ\gamma_1$.

\begin{figure}[htb]
\labellist
\small\hair 2pt
\pinlabel $S\times \R$ [Br] at 328 136
\pinlabel $\Gamma_{f_1}$ [Br] at  265 169
\pinlabel $\Gamma_{f_0}$ [Br] at 265 118
\pinlabel $\gamma$ [Br] at 155 25
\pinlabel $\gamma_0$ [Br] at 155 99
\pinlabel $\gamma_1$ [Br] at 158 150
\pinlabel $s$ [Br] at 122 26
\pinlabel $S$ [Br] at 257 39
\pinlabel $s_0$ [Br] at 120 96
\pinlabel $s_1$ [Br] at 137 183
\endlabellist
\centering
\includegraphics{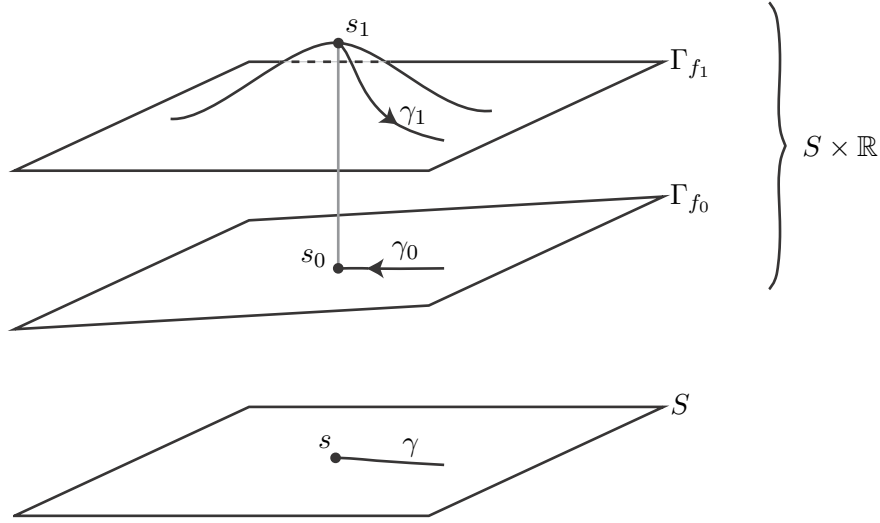}
\caption{Graphs $\Gamma_{f_i}$ of the local defining functions $f_i$ for two sheets of $\Pi_F(\Lambda)$. A flow line $\gamma$ and critical point $s$ in $S$ and their lifts to $\Pi_F(\Lambda)$. Notice $s$ is a positive puncture.}
\label{fig:flowline}
\end{figure}

Let $\gamma$ be a flow line of $\Lambda$ with $1$-jet lift $\gamma_0,\gamma_1$ in sheets with local defining functions $f_0,f_1$. Then the \dfn{flow orientation} of $\gamma_0$ (respectively $\gamma_1$) is given by the lift of the vector
\[
-\nabla(f_0-f_1)\quad\text{(respectively $-\nabla(f_1-f_0)$)}.
\]
If $\gamma\colon (-\infty,b)\to S$ (respectively $\gamma\colon (a,\infty)\to S$) is a flow line as above such
that
\[
\lim_{t\to-\infty}\gamma(t)=s\in S\quad(\text{respectively }\lim_{t\to\infty}\gamma(t)=s)
\]
is a critical point of $f_1-f_0$, then we say $s$ is a \dfn{puncture} of the flow
line $\gamma$. Let $s_i$ be the point in the sheet of $f_i$ with $\pi(s_i)=s$, $i=0,1$. Choose notation so that $f_1(s_1)>f_0(s_0)$; then $s$ is a
\[
\begin{cases}
\text{\dfn{positive puncture}} &\text{if the flow orientation of $\gamma_0$ points toward $s_0$}\\
&\text{and that of $\gamma_1$ points away from $s_1$},\\
\text{\dfn{negative puncture}} &\text{if the flow orientation of $\gamma_1$ points toward $s_1$}\\
&\text{and that of $\gamma_0$ points away from $s_0$}.
\end{cases}
\]
If $s$ is a puncture of a flow tree then the {\em chord at $s$} is the vertical line segment oriented in the direction of increasing $z$ that connects $s_0$ and $s_1$.

A {\em flow tree} of $\Lambda\subset J^{1}(S)$ is a map into $S$ with domain a finite tree $\Gamma$ with extra structure consisting of a cyclic ordering of the edges at each vertex and with the following three properties.
\begin{enumerate}
\item The restriction of the map to each edge is a flow line of $\Lambda$.
\item If $v$ is a $k$-valent vertex with cyclically ordered 
edges $e_1,\dots, e_k$ and $\bar e_j^0,\bar e_j^1$ is the cotangent lift of $e_j$, $1\le j\le k$, then there exists a pairing of lift components such that for every $1\le j\le k$ (with $k+1=1$)
\[
\bar e_j^1(v)=\bar e_{j+1}^0(v)= p\in\Pi(\Lambda)\subset T^\ast S,
\]
and such that the flow orientation of $\bar e_j^1$ at $p$ is directed toward $p$ if and only if the flow orientation of $\bar e_{j+1}^0$ at $p$ is directed away from $p$. Thus the cotangent lifts of the edges of $\Gamma$ then fit together as an oriented curve $\bar\Gamma$ in $\Pi(\Lambda)$.
\item This curve $\bar\Gamma$ is closed.
\end{enumerate}

We first notice that vertices may contain punctures. We will be
interested in flow trees with only one positive puncture. Such flow
trees can have only one puncture above each vertex, see \cite[Section
2]{Ekholm07}. Thus for such flow trees $\Gamma$ we divide the vertices
into three sets: the set of positive punctures $P(\Gamma)$, the set of negative punctures $N(\Gamma)$ and the set of other vertices $R(\Gamma)$. Recall at a (non-degenerate) puncture $v$ the corresponding difference between the local defining functions has a non-degenerate critical point. Denote its index by $I(v)$.

If $\Gamma$ is a flow tree as above then its {\em formal dimension} (see \cite[Definition 3.4]{Ekholm07}), which measures the dimension of the space of flow trees with $1$-jet lift near the $1$-jet lift of $\Gamma$, is
\begin{equation}\label{eq:dimoftrees}
\text{dim}(\Gamma)=\left(\sum_{v\in P(\Gamma)} (I(v)-1)-\sum_{v\in N(\Gamma)} (I(v)-1) +\sum_{v\in R(\Gamma)} \mu(v)\right) -1
\end{equation}
where $\mu(v)$ is the {\em Maslov content} of $v$ and is defined as
follows. For a vertex $v\in R(\Gamma)$ let $x\in \pi^{-1}(v)$ be a
cusp point that lies in the 1-jet lift of $\Gamma$ (if such a point
exists). If $\gamma_0$ and $\gamma_1$ are two 1-jet lifts of an edge
of $\Gamma$ adjacent to $v$ that contain $x$ and for which the flow
orientation of $\gamma_0$ is pointed towards $x$ and that of
$\gamma_1$ is pointed away from $x$, then we set
$\widetilde\mu(x)=+1$, respectively $-1$, if $\gamma_0$ is on the
upper, respectively lower, local sheet of $\Lambda$ near $x$ and
$\gamma_1$ is on the lower, respectively upper, local sheet. Otherwise
define $\widetilde\mu(x)=0$. We can now define $\mu(v)=\sum
\widetilde\mu(x)$ where the sum is taken over all $x\in \pi^{-1}(v)$
that are cusp points in the 1-jet lift of $\Gamma$.

There is also a notion of \emph{geometric dimension} of $\Gamma$, $\gdim(\Gamma)$, see \cite[Definition 3.5]{Ekholm07}, which measures the dimension of the space of flow trees near $\Gamma$ that have the exact same geometric properties as $\Gamma$. In \cite[Lemma 3.7]{Ekholm07} it is shown that $\gdim(\Gamma)\le \dim(\Gamma)$ for any flow tree $\Gamma$, and a characterization of the vertices of trees for which equality holds is given. In combination with transversality arguments for Morse flows, this leads to the following result, see \cite[Lemma 3.7]{Ekholm07} and \cite[Proposition 3.14]{Ekholm07}.

\begin{thm}[Ekholm 2007, \cite{Ekholm07}]\label{thm:gentree}
Given a number $n>0$ then after a small perturbation of $\Lambda$ and the metric $g$ on $S$ we may assume that for any flow tree $\Gamma$ having one positive puncture
and $\dim(\Gamma)\leq n$, the space of flow trees with the same geometric properties in a neighborhood of $\Gamma$ is a transversely cut out manifold of dimension $\gdim(\Gamma)$. In particular, trees $\Gamma$ with $\dim(\Gamma)=0$ form a transversely cut out $0$-manifold.  Furthermore, such rigid trees satisfy $\gdim(\Gamma)=\dim(\Gamma)=0$ and have vertices only of the following types, see Figure~\ref{fig:gentree}.
\begin{figure}[htb]
\labellist
\small\hair 2pt
\pinlabel $S$ [Bl] at 316 24
\pinlabel $S$ [Bl] at 138 24
\pinlabel $S$ [Bl] at 338 204
\pinlabel $S$ [Bl] at 138 187
\pinlabel $S$ [Bl] at 138 375
\pinlabel $S$ [Bl] at 338 375
\pinlabel $S\times\R$ [Bl] at 360 100
\pinlabel $S\times\R$ [Bl] at 360 275
\pinlabel $S\times\R$ [Bl] at 360 450
\endlabellist
\centering
\includegraphics{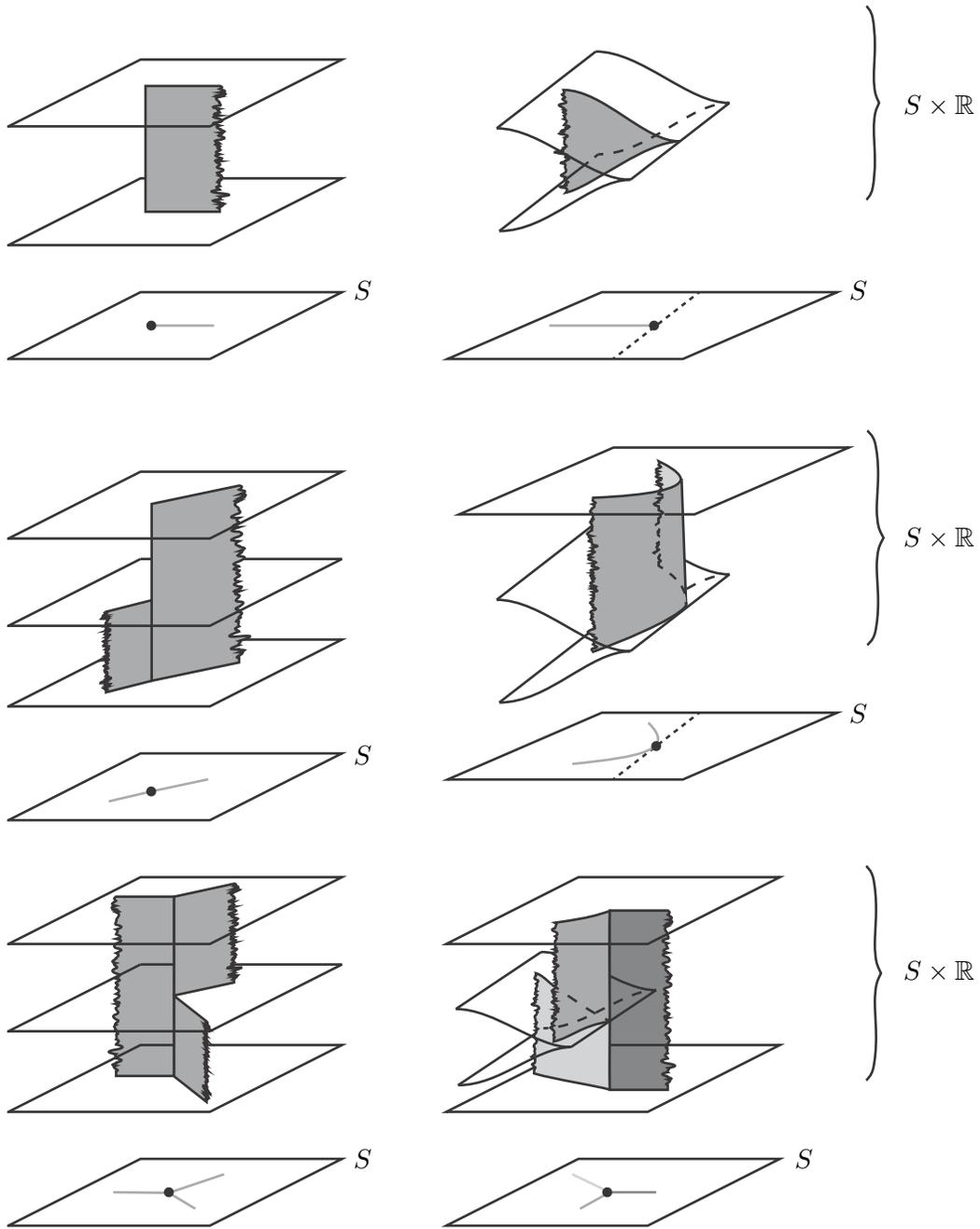}
\caption{Lifts of neighborhoods, in $\Gamma$, of vertices into the
  front space $S\times \R$. Dashed lines in $S$ are cusp edges.}
\label{fig:gentree}
\end{figure}
\begin{enumerate}
\item Valency one vertices that
\begin{enumerate}
\item are positive punctures with Morse index $\ne 0$,
\item are negative punctures with Morse index $\ne 2$, or
\item lift to a cusp edge and have Maslov content $+1$, called a cusp end.
\end{enumerate}
\item Valency two vertices that
\begin{enumerate}
\item are positive punctures with Morse index $0$,
\item are negative punctures with Morse index $2$, or
\item have order one tangencies with a cusp edge and Maslov content $-1$, called a switch.
\end{enumerate}
\item Valency three vertices that
\begin{enumerate}
\item are disjoint from the projection $\pi(\Sigma)$ of the singular locus to $S$, called a $Y_0$ vertex,  or
\item lie on the image of the cusp locus and have Maslov content $-1$, called a $Y_1$ vertex.
\end{enumerate}
\end{enumerate}
\end{thm}

Notice that we may endow a flow tree $\Gamma$ that has exactly one positive puncture with a {\em flow orientation}: any edge $e$ is oriented by the negative gradient of the positive difference of its defining functions.

When working with flow trees it will also be useful to consider the
symplectic area of a flow tree $\Gamma$. Given a flow tree $\Gamma$,
let $\widehat\Gamma$ denote its 1-jet lift (which, in previous
notation, projects to $\overline{\Gamma} \subset \Pi(\Lambda)$), and
define the {\em symplectic area} of $\Gamma$ to be
\[
A(\Gamma)=-\int_{\widehat \Gamma} p\, dq=-\int_{\widehat\Gamma} dz.
\]
The name comes from the connection between flow trees and holomorphic curves. The important features of the symplectic area are summarized in the following lemma. Before stating it we introduce some notation. A puncture $a$ of a flow tree lifts to a double point in the Lagrangian projection $\Pi(\Lambda)$ and hence corresponds to a Reeb chord. Thus in the 1-jet lift of $\Gamma$ there will be two points that project to $a$. We denote them $a^+$ and $a^-$ where $a^+$ has the larger $z$-coordinate.
\begin{lem}[Ekholm 2007, \cite{Ekholm07}]\label{lem:action}
For any flow tree $\Gamma$ the symplectic action is positive: $A(\Gamma)>0$. The symplectic action can be computed by the formula
\[
A(\Gamma)=\sum_{p} \left(z(p^+)-z(p^-)\right) - \sum_q \left( z(q^+)-z(q^-)\right),
\]
where the first sum is over positive punctures of $\Gamma$, the second sum is over negative punctures of $\Gamma$, and $z(a)$ denotes the $z$-coordinate of the point $a$.
\end{lem}

For our applications two further types of flow trees will be
needed. First, a {\em partial flow tree} is a flow tree $\Gamma$ for
which we drop the condition that the cotangent lift $\bar\Gamma$ is
closed and allow $\Gamma$ to have $1$-valent vertices $v$ such that
$\bar\Gamma$ intersects the fiber over $v$ in two points. We call such
vertices {\em special punctures}. In the dimension formula,
Equation~\eqref{eq:dimoftrees}, $I(v)=3$ for a positive special
puncture. Theorem~\ref{thm:gentree} holds as well for partial flow
trees with at least one special puncture, see \cite{Ekholm07}.
Second, in Section~\ref{ssec:basicft} we consider {\em constrained flow} trees: if $p_1,\ldots, p_r$ are points in $\Lambda$ then a
flow tree constrained by $p_1,\ldots, p_r$ is a flow tree $\Gamma$ with $1$-jet lift which passes through the points $p_1,\ldots, p_r$.

% **************************************************
% **************************************************
\subsection{Vector splitting along flow trees}\label{ssec:vectorsplit}
In this subsection we describe a combinatorial algorithm for transporting normal vectors in a flow tree to all its vertices which will combinatorially determine the sign of a flow tree, see Section \ref{ssec:thecount}. Specifically we will be concerned only
with flow trees that do not involve cusp edges and only have punctures
at critical points of index 1 and 2.

Suppose $\Lambda$ is a Legendrian submanifold in $J^1(S)$ that does not have cusp edges and only has Reeb chords corresponding to critical points of index 1 and 2.
Let $\Gamma$ be a partial flow tree with positive special puncture determined by $\Lambda \subset J^{1}(S)$.

Consider the local situation at a $Y_0$-vertex of $\Gamma$ at $t\in
S$. In the flow orientation of $\Gamma$ one edge adjacent to $t$ is
pointing toward it (we call this edge {\em incoming}) and the other
two edges are pointing away from it (we call them {\em
  outgoing}). Furthermore, the natural orientation of the $1$-jet lift
of the tree induces a cyclic order on the three edges adjacent to $t$
and thus an order of the two outgoing edges. (Specifically, the edge
$e_1$ will have the same upper sheet as the incoming edge while the
edge $e_2$ will have the same lower sheet.) If $v_0$ denotes the
negative gradient of the incoming edge at $t$ and $v_1, v_2$ the
negative gradients of the two outgoing edges, all pointing according
to the flow orientation of $\Gamma$, then $v_1$ and $v_2$ are linearly
independent and the following balance equation holds:
$v_0=v_1+v_2$. (This follows from the fact that the difference between
the local defining
functions along the incoming edge is the sum of the function differences
along the outgoing edges.)

We next define vector splitting. Let $p$ denote the special positive puncture of $\Gamma$, let $t_1,\dots,t_{k-1}$ denote its trivalent vertices, and let $q_1,\dots,q_k$ denote its negative punctures. Vector splitting along $\Gamma$ is a function
\[
N_{p}\Gamma\,\,\,\to\,\,\, \Pi_{j=1}^{k-1} F_{t_j} S\,\,\times\,\, \Pi_{i=1}^{k} N_{q_j}\Gamma,
\]
where $N_x\Gamma$ denotes the normal bundle of the edge of $\Gamma$
containing the point $x\in\Gamma$ (which is assumed not to be a
trivalent vertex) and where $FS$ denotes the frame bundle of
$S$. 
It is defined as follows.

Translating $n\in N_x\Gamma$ as a normal vector along the edge in the direction of the flow orientation of the tree, it eventually arrives at a trivalent vertex $t$. At this vertex, the translated vector vector $n(t)$ is perpendicular to the incoming edge at $t$ and determines
two unique vectors $w_{1}(t)$ and $w_{2}(t)$ in $T_tS$ perpendicular to the first and second outgoing edges at $t$, respectively, by requiring that  $n(t)=w_1(t)+w_2(t)$, see Figure \ref{fig:vectorsplit}. The frame at $t$ is $(w_1(t),w_2(t))$.
\begin{figure}[htb]
\labellist
\small\hair 2pt
\pinlabel $n$ at 68 85
\pinlabel $1$ at 181 110
\pinlabel $2$ at 181 -4
\pinlabel $w_1$ at 140 110
\pinlabel $w_2$ at 169 49
\endlabellist
\centering
\includegraphics{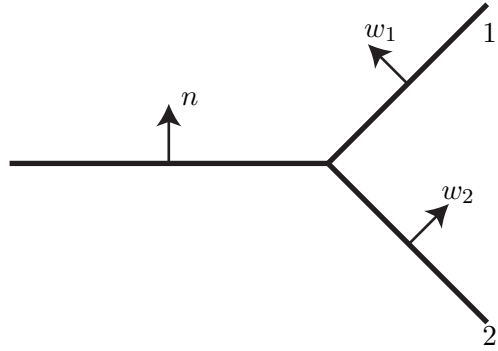}
\caption{Vector splitting at a trivalent $Y_0$-vertex.}
\label{fig:vectorsplit}
\end{figure}
Regarding $w_1(t)$ and $w_2(t)$ as normal vectors of their respective edges, the above construction can be applied with $n$ replaced by $w_j(t)$, $j=1,2$, and $\Gamma$ replaced by the partial tree $\Gamma_j$ which is obtained by cutting $\Gamma$ in the $j^{\rm th}$ outgoing edge at $t$ and taking the component which does not contain $t$. When a subtree $\Gamma_j$ contains no trivalent vertex we translate $w_j(t)$ along its respective edge to the negative puncture at the vertex. Continuing in this way until the cut-off partial trees do not have any trivalent vertices, we get
\begin{itemize}
\item two vectors $w_1(t)$ and $w_2(t)$
perpendicular to the first and second outgoing edge at $t$, for each trivalent vertex $t$, and
\item a vector $w(q)$ perpendicular to the edge ending at $q$, for each negative puncture $q$.
\end{itemize}
The {\em vector splitting of $n$ along $\Gamma$} is
\[
n\mapsto \bigl((w_1(t_1),w_2(t_1)),\dots,(w_1(t_{k-1}),w_2(t_{k-1}));w(q_1),\dots,w(q_k)\bigr).
\]

% **************************************************
% **************************************************
\subsection{Algebraic results}\label{ssec:alg}

In this subsection, we collect for convenience
some algebraic results about the map
$\phi$ and the matrices $\Phi^L$ and $\Phi^R$, which were defined in
Section~\ref{ssec:mainthm} and play an essential role in the combinatorial
formulation of the Legendrian algebra for knot contact homology,
Theorem~\ref{thm:mainlink}. These results, which will occasionally be needed
in the remainder of the paper, have essentially been established in
previous work of the third author
\cite{Ng05,Ng05a,Ng08,Ng10}. However, we repeat the
caveat from Section~\ref{ssec:mainthm} that in our present context, homology
classes ($\mu$ and $\lambda$) do not commute with Reeb chords
($a_{ij}$), while the previous papers deal with the
homology-commutative quotient
(see Section~\ref{sssec:Legalg}). Nevertheless, all existing proofs
extend in an obvious way to the present setting.

\begin{prp}
The map $\phi : \widetilde{\A}_n^0 \to \widetilde{\A}_n^0$, defined in
Section~\ref{ssec:mainthm} for braid generators $\sigma_k$ of the braid $B_n$,
respects the relations in the $B_n$ and thus extends to a homomorphism
from $B_n$ to $\operatorname{Aut}\widetilde{\A}_n^0$.
\label{prop:homom}
\end{prp}

\begin{proof}
Direct computation, {\em cf.\ }\cite[Proposition 2.5]{Ng05}.
\end{proof}

\begin{prp}
Let $B\in B_n$, and let $\phi_B(\Aa)$ be the $(n\times n)$-matrix
defined by $\left(\phi_B(\Aa)\right)_{ij} = \phi_B(\Aa_{ij})$. Then we
have the matrix identity
\label{prop:PhiLPhiR}
\[
\phi_B(\Aa) = \Phi^L_B \cdot \Aa \cdot \Phi^R_B.
\]
\end{prp}

\begin{proof}
Induction on the length of the braid word representing $B$,
{\em cf.\ }\cite[Proposition 4.7]{Ng05} and \cite[Lemma 2.8]{Ng10}. The latter
reference proves the result stated here ($\hat{\Aa}$ and $\check{\Aa}$
there correspond to $\Aa$ here once we set $U=V=1$), but in the
homology-commutative
quotient, for the case of a single-component knot, and with slightly
different sign conventions. Nevertheless, the inductive proof given there works
here as well; we omit the details.
\end{proof}

We remark that Proposition~\ref{prop:PhiLPhiR} can be given a more natural,
geometric proof via the language of ``cords'' \cite{Ng05a}. This
approach also provides an explanation for the precise placement of the
homology classes $\tilde{\mu}$ in the definition of $\phi_{\sigma_k}$
from the Introduction. We refer the interested reader to \cite[Section
3.2]{Ng08}, which treats the homology-commutative single-component
case, and leave the straightforward extension to the general case to
the reader.

One consequence of Proposition~\ref{prop:PhiLPhiR} is that the
differential for knot contact homology presented in
Theorem~\ref{thm:mainlink} is well-defined. More precisely, the
differential for the $b_{ij}$ generators is given in matrix form by
\[
\pa \mathbf{B} = -\Ll^{-1}\cdot\Aa\cdot\Ll\,\, +\,\, \Phi^L_B \cdot
\Aa \cdot \Phi^R_B.
\]
However, since $\mathbf{B}$ has $0$'s along the main diagonal, it is
necessary that the right hand side has $0$'s along the diagonal as
well. This is indeed the case: the $(i,i)$ entry of
$\Ll^{-1}\cdot\Aa\cdot\Ll$ is $1+\mu_{\alpha(i)}$, while the $(i,i)$
entry of $\Phi^L_B \cdot \Aa \cdot \Phi^R_B$ is
$\phi_B(1+\mu_{\alpha(i)}) = 1+\mu_{\alpha(i)}$.

% **************************************************
% **************************************************
% **************************************************
\section{The Differential and Flow Trees}\label{sec:diffthrflowtree}

In Section~\ref{sec:combdiff}, we will prove
Theorem~\ref{thm:mainlink} by using multiscale flow trees to compute the differential of $\Lambda_K$ in $J^1(S^2)=U^\ast \R^3$. These multiscale flow trees combine two types of flow trees, which are the focus of this section. Specifically we will see that if $\Lambda$ is the conormal lift of the unknot $U$ then by thinking of $K$ as a braid about $U$ we can isotop $\Lambda_K$ into an arbitrarily small neighborhood of $\Lambda =\Lambda_U$, which can be identified with a neighborhood of the zero section in $J^1(\Lambda)$. Thus we may think of $\Lambda_K$ as a subset of $J^1(\Lambda)$.

To use multiscale flow trees to compute the differential of $\Lambda_K$ in $J^1(S^2)$ we will combine flow trees of $\Lambda\subset J^1(S^2)$ and $\Lambda_K \subset J^1(\Lambda)$.
The content of this section is a computation of these flow trees; in Section~\ref{sec:combdiff}, we then combine the flow trees to complete the computation of the Legendrian DGA for $\Lambda_K \subset J^1(S^2)$.

Here is a more detailed summary of this section. In Section~\ref{ssec:flowtreesonU}, we discuss the Legendrian torus $\Lambda$ and describe a generic front projection for it. In Section~\ref{ssec:flowtreesforLambda}, we compute the rigid flow trees for $\Lambda$ as well as the 1-parameter families of flow trees. In Section~\ref{ssec:conormalliftofK}, we give an explicit identification of a neighborhood of the zero section in $J^1(\Lambda)$ with a neighborhood of $\Lambda$ in $J^1(S^2)$, use this identification to explicitly describe $\Lambda_K$ in $J^1(\Lambda)$, and find all the Reeb chords of $\Lambda_K \subset J^1(\Lambda)$. Section~\ref{ssec:thecount} computes the rigid flow trees of $\Lambda_K$ in $J^1(\Lambda)$ modulo some technical considerations concerning ``twist regions'' that are handled in Section~\ref{ssec:twistslice}. We comment that Section~\ref{ssec:thecount} produces an invariant of braids (cf.\ \cite{Ng05}) using only existing flow tree technology and not multiscale flow trees.

% **************************************************
% **************************************************
\subsection{A generic front for $\Lambda$}\label{ssec:flowtreesonU}
Let $U$ be the round unknot given by the unit circle in the $xy$-plane
in $\R^3$. We give a description of the conormal lift
$\Lambda=\Lambda_U$ in $J^1(S^2)=U^*\R^3$ by describing its front
projection in $S^2\times\R$. In the figures below, we draw
$S^2\times\R$ as $\R^3-\{0\}$ and identify the zero section $S^2\times\{0\}$ with the
unit sphere. If $C$ is the circle in the conormal bundle of $U$ lying over a point $x\in U$ then its image in $S^2\times\{0\}\subset S^2\times \R$ is a great circle running through the north and south poles of $S^2$. See Figure~\ref{fig:unknotbasepicture}.
By the contactomorphism $\phi$ between $U^\ast\R^3$ and $J^1(S^2)$ from Section~\ref{ssec:conormal},
the image of $C$ in $S^2\times \R$ is the graph of $\langle x, y\rangle$ where
$y\in C$. This is shown in the leftmost picture in Figure~\ref{fig:unknotfront}.
\begin{figure}[htb]
\labellist
\small\hair 2pt
\pinlabel $U$ at 92 -46
\pinlabel $p$ at 74 -63
\endlabellist
\centering
\includegraphics{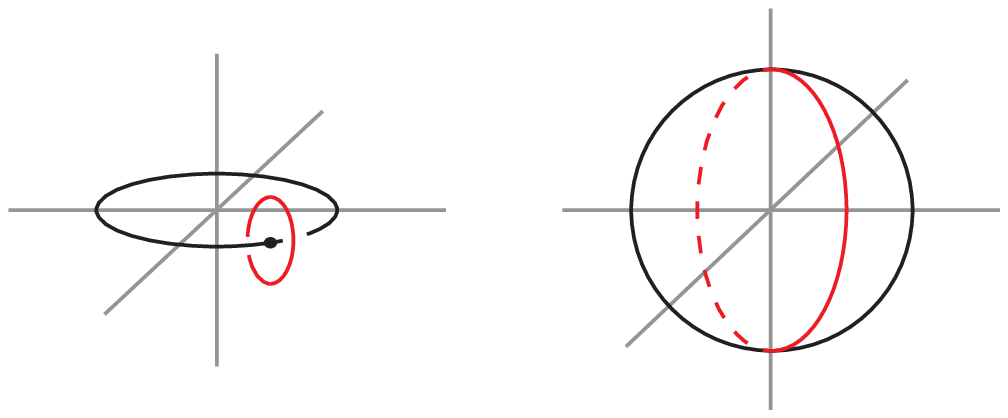}
\caption{On the left is the unknot $U$ in the $xy$-plane with a point $p$ on $U$ labeled and its unit (co-)normal bundle shown. On the right is the tangent space $\R^3=T_p\R^3$ at $p$ with the unit sphere indicated along with the image of the unit (co-)normal bundle to $U$ at $p$.}
\label{fig:unknotbasepicture}
\end{figure}
\begin{figure}[htb]
\labellist
\small\hair 2pt
\pinlabel $c$ at 304 -50
\pinlabel $e$ at 380 -50
\endlabellist
\centering
\includegraphics{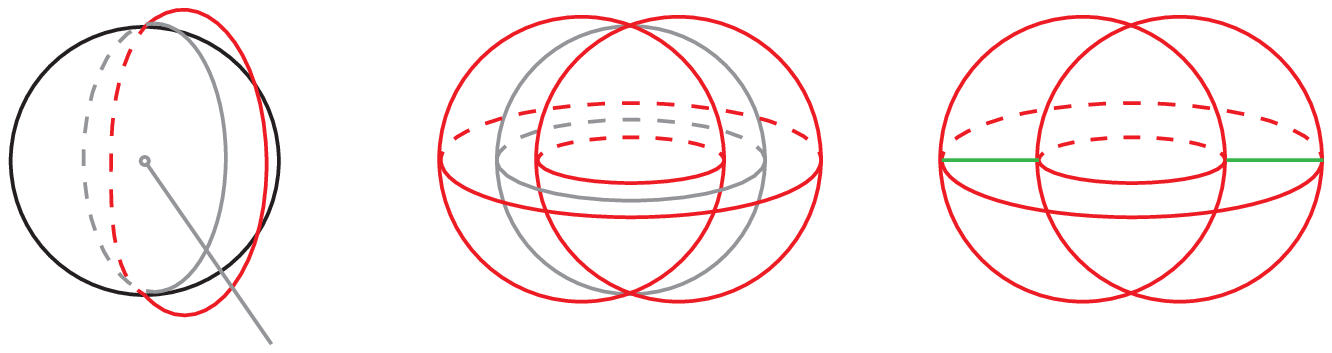}
\caption{The front of the unknot. On the left is the image of the circle shown in Figure~\ref{fig:unknotbasepicture} in the front projection in $J^0(S^2)=S^2\times \R=\R^3-\{0\}$. The center image shows the entire front projection (with the unit sphere $S^2\times\{0\}$ shown in light grey) and the image on the right shows the two Reeb chords after perturbation.}
\label{fig:unknotfront}
\end{figure}
By symmetry we get the entire front projection by simply rotating this image of $C$
about the axis through the north and south pole as shown in the middle picture of Figure~\ref{fig:unknotfront}.
Using Lemma~\ref{lem:critpts} we see that this representative of $\Lambda$ has an $S^1$'s worth of Reeb chords over the equator. We perturb $\Lambda$ using a Morse function on the equator with one maximum and one minimum, so that only two Reeb chords $c$ and $e$, as indicated in the rightmost picture in Figure~\ref{fig:unknotfront}, remain. Here $e$ and $c$ correspond to the maximum and minimum, respectively, of the perturbing function, and both correspond to transverse double points of $\Pi(\Lambda)$.

This perturbation does not suffice to make the front of $\Lambda$ generic with respect to fibers of $T^{\ast}S^{2}$: over the poles of $S^2$ we see that $\Lambda$ consists of Lagrangian cones. We first describe these in local coordinates and then show how to perturb them to become front generic. Let $x=(x_{1},x_{2})$ be local coordinates near the pole in $S^{2}$ and let $(x,y)=(x_1,y_1,x_2,y_2)$ be corresponding Darboux coordinates in $T^{\ast}S^{2}$, with symplectic form given by $dx\wedge dy=dx_1\wedge dy_1+ dx_2\wedge dy_2$.

Consider $S^{1}=\left\{\xi\in\R^{2}\colon |\xi|=1\right\}$. The {\em Lagrangian cone} is the exact Lagrangian embedding $C\colon S^{1}\times\R\to\R^{4}$ given by
\[
C(\xi,r)=(r\xi,\xi).
\]
 As mentioned above, the Lagrangian cone is not front generic: the front projection $\Pi_{F}\circ C$ is regular for $r\ne 0$ but maps all of $S^{1}\times\{0\}$ to the origin. In order to describe perturbations of $C$ we first find a cotangent neighborhood of it.
See Figure~\ref{fig:resolution}.
\begin{figure}[htb]
\centering
\includegraphics{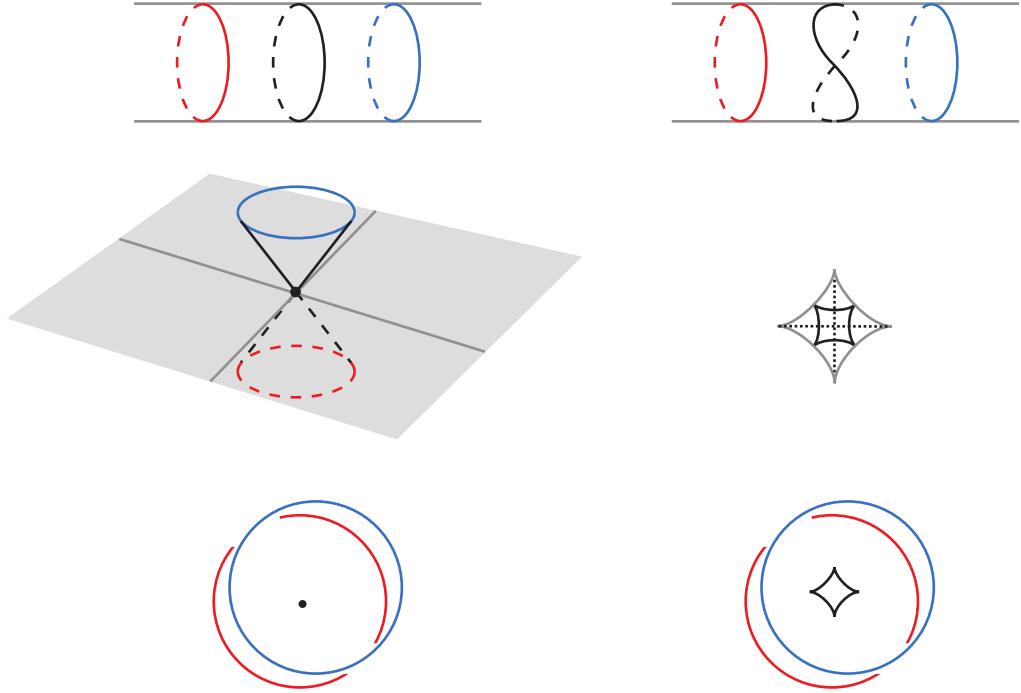}
\caption{A perturbed Lagrangian cone. Along the top of the figure is an annular neighborhood of the circle in $\Lambda = \Lambda_U$ that maps to the fiber above the north pole.
On the middle left, we see the image of this annulus near
the north pole in the front projection,
a cone whose boundary is two circles.
On the bottom left is the image of this annulus near the north pole in $S^2$ (that is, the top view of the cone where we have slightly offset the circles so that they are both visible).
On the middle right, we see the top view of the cone after it has been perturbed to have a generic front projection. More specifically, the lighter outer curve is the image of the cusp curves, the dotted lines are the image of double points in the front projection and the darkest inner curve is the image of the circle that mapped to the cone point before the perturbation.
On the bottom right, we see the image in $S^2$ of the cusp curve and
the two boundary circles on $\Lambda_U$.}
\label{fig:resolution}
\end{figure}

Consider $T^{\ast}(S^{1}\times\R)=T^{\ast}S^{1}\times T^{\ast}\R$. Let
$(\xi,\eta,r,\rho)$ be coordinates on this space where $\rho$ is dual
to $r$ and $\eta$ is dual to $\xi$. Here we think of $\eta$ as a
(co)vector perpendicular to $\xi\in S^1\subset \R^2$, that is $\eta\in \R^2$ and $\eta \bull \xi = 0.$
Consider the map
\[
\Phi\colon T^{\ast} S^{1}\times T^{\ast}\R\to T^{\ast}\R^{2}=\R^{4}
\]
given by
\[
\Phi(\xi,\eta,r,\rho)=\left(r\xi - \frac{1}{1+\rho}\eta,(1+\rho)\xi\right).
\]
Then $\Phi^{\ast} (dx\wedge dy)=d\xi\wedge d\eta+d r\wedge d\rho$. To see this we compute
\begin{align*}
dx &=(dr)\xi + rd\xi - \frac{1}{1+\rho}d\eta +\frac{d\rho}{(1+\rho)^{2}}\eta,\\
dy &= (d\rho)\xi + (1+\rho)d\xi
\end{align*}
and hence
\begin{align*}
dx\wedge dy &=(dr\wedge d\rho)(\xi\bull\xi)\\
&+((1+\rho)dr-rd\rho)\wedge (\xi\bull d\xi)
+r(1+\rho)d\xi\wedge d\xi\\
&+\frac{d\rho}{1+\rho}(\xi\bull d\eta+\eta\bull d\xi)\\
&=dr\wedge d\rho + d\xi\wedge d\eta,
\end{align*}
since $\xi\bull\xi=1$ and $\xi\bull\eta=0$.

Thus $\Phi$ is a symplectic neighborhood map extending the Lagrangian
cone. It follows that  exact Lagrangian submanifolds $C^1$-near $C$ can be described by
$\Phi(\Gamma_{df})$ where $f$ is a smooth function on $S^{1}\times\R$ and where $\Gamma_{df}$ denotes the graph of its differential.

We will consider specific functions of the form
\[
f(\xi,r)=\alpha(r)g(\xi),
\]
where $\alpha(r)$ is a cut off function equal to $0$ for $|r|>2\delta$ and equal to $1$ for $|r|<\delta$ for some small $\delta>0$. Write $C_f=\Phi(\Gamma_{df})$.

We first give a description of the caustic of $C_f$. Note that $\Pi_{F}\circ C$ is an immersion for $|r|>\delta$. The same therefore holds true for $C_f$ provided $f$ is small enough. In order to describe the caustic we thus focus on the region where $|r|<\delta$ and hence $\alpha(r)=1$. Here
\[
C_{f}(\xi,r)=\left(r\xi+dg,\xi\right)=\left(C_{f}^{1}(\xi,r),C_{f}^{2}(\xi,r)\right).
\]

The caustic is the set where $C_{f}$ has tangent lines in common with the fiber. Consequently a point $p\in S^{1}\times\R$ belongs to the fiber provided the differential of the first component $C_{f}^{1}$ of the map $C_{f}$ has rank $<2$.  Write $\xi=(\cos \theta,\sin \theta)\in S^{1}$ and take $g(\xi)=g_1(\xi)=\epsilon\cos 2\theta$. The caustic is then the image of the locus $r=4\epsilon\cos 2\theta$ under the first component of the map
\[
\Phi(\theta,r)=\bigl(r(\cos\theta, \sin\theta)
-2\epsilon\sin\theta(-\sin\theta, \cos\theta), (\cos\theta,\sin\theta)\bigr),
\]
see Figure \ref{fig:resolution}.

\begin{lma}\label{lem:gradingunknot}
The Maslov class of $\Lambda$ vanishes and consequently the grading of any Reeb chord of $\Lambda$ is independent of choice of capping path. Let $e$ and $c$ denote the Reeb chords of $\Lambda$, as described above; then
\[
|e|=2\quad\text{ and }\quad|c|=1.
\]
\end{lma}

\begin{pf}
To see that the Maslov class vanishes we need to check that the Maslov index of any generator of $H_1(\Lambda)$ vanishes. We compute the Maslov index of a curve as described at the end of Section~\ref{sssec:geomprelim}. Take one generator as a curve in $\Lambda$ over the equator; since this curve does not intersect any cusp edge its Maslov index vanishes. Take the other generator as a curve perpendicular to the equator going to the poles and then back; since such a curve has two cusp edge intersections, one up-cusp and one down-cusp, its Maslov index vanishes as well.

Finally choose the capping path of $e$ and $c$ which goes up to the north pole and then back. This capping path has one down-cusp and the Morse indices of $e$ and $c$ are $2$ and $1$, respectively. The index assertions now follow from Lemma~\ref{lem:indexcritpt}.
\end{pf}

% **************************************************
% **************************************************
\subsection{Flow trees of $\Lambda$}\label{ssec:flowtreesforLambda}
We next determine all flow trees of $\Lambda$.
\begin{lma}\label{lma:treesofU}
There are exactly six rigid flow trees of $\Lambda$: four with positive puncture at $c$: $I_N$, $Y_N$, $I_S$, and $Y_S$, and two with positive puncture at $e$ and negative puncture at $c$: $E_1$ and $E_2$. Furthermore, if $p$ is a point of $\Lambda$ lying over a point where the front of $\Lambda$ has $2$ sheets then there are exactly two constrained rigid flow trees with positive puncture at $e$ with 1-jet lift passing through $p$.
\end{lma}

Before proving Lemma~\ref{lma:treesofU}, we make a couple of remarks.

\begin{rmk}
In fact one can show the following: There are exactly four $1$-parameter families of trees with positive puncture at $e$: $\widetilde I_N$, $\widetilde Y_N$, $\widetilde I_S$, and $\widetilde Y_S$. The boundaries of these $1$-parameter families are as follows:
\begin{align*}
\pa\widetilde I_N = (E_1\,\#\, I_N) \cup (E_2\,\#\, I_N),&\quad
\pa\widetilde Y_N = (E_1\,\#\, Y_N) \cup (E_2\,\#\, Y_N),\\
\pa\widetilde I_S = (E_1\,\#\, I_S) \cup (E_2\,\#\, I_S),&\quad
\pa\widetilde Y_S = (E_1\,\#\, Y_S) \cup (E_2\,\#\, Y_S).
\end{align*}
Here $E_1\,\#\, I_N$ denotes the broken tree obtained by adjoining $I_N$ to $E_1$, etc. See Figures~\ref{fig:rigidtreeunknot} and~\ref{fig:onedimtreeunknot}. Furthermore, the  $1$-jet lifts of the flow trees in each of these families sweep the part of the torus lying over the corresponding hemisphere ($N$ or $S$) once.

The formal proof of this result about $1$-parameter families would require a more thorough study of flow trees in particular including a description of all vertices of flow trees that appear in generic 1-parameter family. This is fairly straightforward, see \cite[Section 7]{Ekholm07}. For the purposes of this paper it suffices to work with constrained rigid trees rather than 1-parameter families so details about 1-parameter families of flow trees will be omitted.
\end{rmk}
\begin{figure}[htb]
\labellist
\small\hair 2pt
\pinlabel $I_N$ at 42 32
\pinlabel $c$ at 9 12
\pinlabel $c$ at 120 12
\pinlabel $Y_N$ at 153 32
\pinlabel $c$ at 231 12
\pinlabel $e$ at 317 89
\pinlabel $E_0$ at 307 32
\pinlabel $E_1$ at 253 83
\endlabellist
\centering
\includegraphics{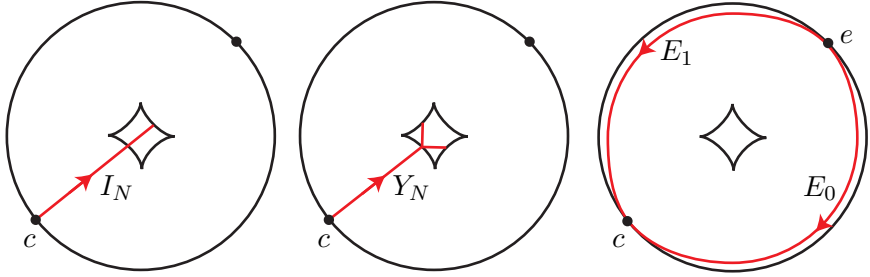}
\caption{Rigid flow trees for $\Lambda$ on the northern hemisphere of $S^2$.}
\label{fig:rigidtreeunknot}
\end{figure}
\begin{rmk}\label{rmk:1dfam}
We will not need a precise expression of the one-dimensional families $\widetilde I_N$, $\widetilde Y_N$, $\widetilde I_S$, and $\widetilde Y_S$ in our computations, but we do need a rough understanding of them. To see the family of disks, start with the symmetric picture of $\Lambda$ coming from the conormal lift of $U$. Now make a small perturbation of the north and south poles as shown in Figure~\ref{fig:resolution}. Then we see an $I$ and $Y$ flow tree from each point on the equator into the northern hemisphere and another into the southern hemisphere. Now perturbing slightly so the equator is no longer a circle of critical points but contains only the critical points $c$ and $e$ and two flow lines between them, we will see that each of the $I$ and $Y$ disks will become part of one of the one-dimensional families of disks $\widetilde I_N$, $\widetilde Y_N$, $\widetilde I_S$, and $\widetilde Y_S$. See Figure~\ref{fig:onedimtreeunknot}.

It is also useful to see these trees as arising from the Bott-degenerate conormal lift of the round unknot. Here there are four holomorphic disks emanating from each Reeb chord. The corresponding trees are just flow lines from the equator to the pole. The 1-jet lift of such a flow line can then be completed by one of the two half circles of the circle in $\Lambda$ which is the preimage of the pole. The Bott-degenerate $1$-parameter family then consists of a flow segment starting at the maximum in the Bott-family and ending at some point where a disk emanating at that point (corresponding to a flow line from that point to a pole) is attached.
\end{rmk}
\begin{figure}[htb]
\labellist
\small\hair 2pt
\pinlabel $e$ at 86 190
\pinlabel $e$ at 203 190
\pinlabel $e$ at 320 190
\pinlabel $e$ at 134 82
\pinlabel $e$ at  250 82
\endlabellist
\centering
\includegraphics{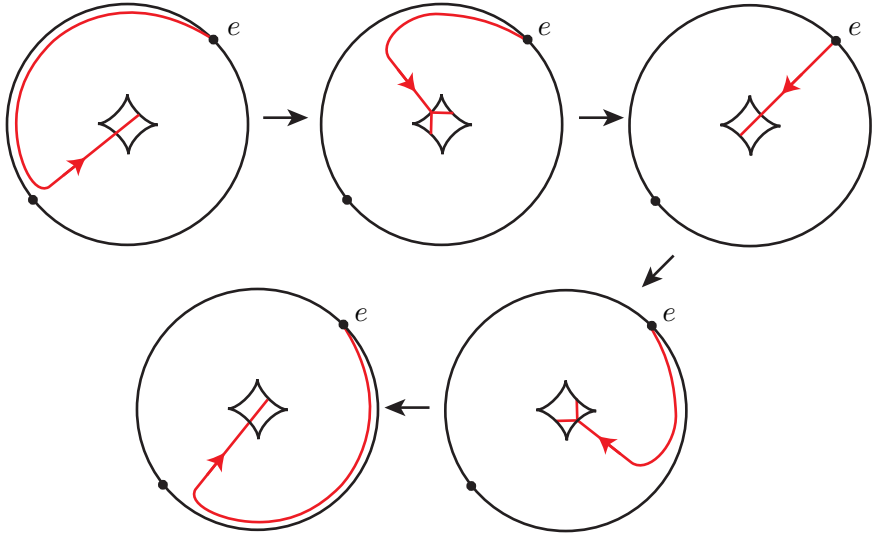}
\caption{One of the $1$-dimensional flow tree families for $\Lambda$ lying over the northern hemisphere of $S^2$.}
\label{fig:onedimtreeunknot}
\end{figure}
\begin{pf}[Proof of Lemma~\ref{lma:treesofU}]
We are only considering flow trees with exactly one positive puncture. First consider a rigid flow tree $\Gamma$ with a positive puncture at $c$. Using Lemma~\ref{lem:action} we see that there can be no negative punctures since the symplectic area must be positive. Thus all vertices of $\Gamma$  are $Y_0$ or $Y_1$ vertices, switches, or cusp ends.

We can rule out switches as follows.
Since the front is a small perturbation of a Lagrangian cone, it is possible to arrange the following: along a cusp edge in the image $\pi(\Sigma)$ of the caustic, the gradient vector fields for the function difference between a sheet meeting the cusp edge and a sheet not meeting the cusp edge are transverse to the caustic except at the swallow tail points, where they are tangent to the caustic (or zero). This is a non-generic situation, generically tangencies occur only at smooth points of the caustic, and after small perturbation tangency points lie close to the swallowtails. Switches of flow trees lie at tangency points in the caustic and thus a switch could only occur near a swallow tail point. Now, using again that the front is a small perturbation of a Lagrangian cone,
we can also arrange that near a swallow tail, the function difference between any one of the three sheets involved in the swallow tail and the fourth sheet is larger than any function difference in the small region bounded by $\pi(\Sigma)$ away from swallow tails. Then since the unique flow line that leaves $c$ hits $\pi(\Sigma)$ away from the swallow tails, by positivity of symplectic area (in this case, the fact that the function difference decreases along a flow line), there is a neighborhood of the swallow tail points which the the flow tree cannot reach and in particular it cannot have any switches.

Thus the vertices of $\Gamma$ are all $Y_0$ or $Y_1$ vertices or cusp ends.  By the dimension formula, Equation~\ref{eq:dimoftrees}, we see that for the flow tree to be rigid there is some $n\geq 0$ such that there are $n$ type $Y_1$ vertices and $n+1$ cusp ends (each with $\mu=+1$). To have a $Y_1$ vertex a flow line must intersect a cusp edge so that when traveling along the flow the number of sheets used describing $\Lambda$ increases (as one passes the cusp). Around the north and south pole of $S^2$ the cusp edges are arranged so that only flow lines traveling towards the poles could possibly have a $Y_1$ vertex. Thus for $n>1$ one of the edges in the flow tree will be a flow line that travels from near the north pole to near the south pole. Since this clearly does not exist, as there are two flow lines connecting $e$ to $c$ along the equator and $\Lambda$ has only two sheets along the equator, we must have $n=0$ or $1$.

There are only two flow lines leaving $c$ (that is two flow lines that could have $c$ as a positive puncture), one heading towards the north pole and one heading towards the south pole. When $n=0$ we clearly get $I_N$ and $I_S$ from these flow lines when they do not split at a $Y_1$ vertex and when $n=1$ we get $Y_N$ and $Y_S$ when they do split at a $Y_1$ vertex.

Notice for future reference that the above analysis shows that there are no $1$-parameter families of flow trees with $c$ as a positive puncture.

Now consider a flow tree $\Gamma$ with $e$ as a positive puncture. Noting that the Reeb chord above $c$ is only slightly shorter than the chord above $e$ we see, using Lemma~\ref{lem:action}, that $\Gamma$ can have either no negative punctures or just one negative puncture at $c$. If $\dim(\Gamma)=0$, then by the dimension formula there must be a negative puncture at $c$. Then, since $\Lambda$ is defined by only two functions away from neighborhoods of the poles, the only vertices of $\Gamma$ are the punctures $e$ and $c$. Thus $\Gamma$ is simply a flow line from $e$ to $c$ and there are precisely two: $E_1$ and $E_2$.

The argument for constrained rigid trees with positive puncture at $e$ follows from the argument used for rigid flow trees with positive puncture at $c$ above.
\end{pf}

% **************************************************
% **************************************************
\subsection{Conormal lifts of general links}\label{ssec:conormalliftofK}
In order to describe the conormal lift of a general link $K\subset \R^{3}$ we first represent it as the closure of a braid around the unknot $U$. More precisely, $K$ lies in a small neighborhood $N=S^1\times D^2\subset \R^{3}$ of $U$ and is transverse to the fiber disks $\{\theta\}\times D^2$ for all $\theta\in S^1$. Note that $B$ is a braid on $n$ strands if and only if $K$ intersects any fiber disk $n$ times. We write $B$ for the closed braid corresponding to $K$ considered as lying in $S^{1}\times D^{2}$.

We represent a closed braid $B$ on $n$ strands as the graph $\Gamma_{f_B}$ of a
multi-section $f_B\colon S^1\to D^2$, where $f_B(s)$ consists of $n$
distinct points in $D^2$ varying smoothly with $s$, so that
$\Gamma_{f_{B}}$ is a smooth submanifold. Representing $S^1$ as
$[0,2\pi]$ with endpoints identified, we can express $f_B$ as a
collection of $n$ functions $\{f_1(s),\ldots, f_n(s)\}$ where
$f_i\colon[0,2\pi] \to D^2$ are smooth functions, $i=1,\dots,n$. (The
sets $\{f_1(0),\dots,f_n(0)\}$ and $\{f_1(2\pi),\dots,f_n(2\pi)\}$ are
equal but it does not necessarily hold that $f_i(0)=f_i(2\pi)$.) Note
that the distance between $\Lambda$ and $\Lambda_K$ is controlled by
the $C^{1}$-distance from $f_B$ to the trivial multi-section which
consists of $n$ points at the origin. In particular if $N$ is a fixed
neighborhood of $\Lambda$ then $\Lambda_K\subset N$ provided that $f_B$
is sufficiently $C^1$-small.

% **************************************************
\subsubsection{A $1$-jet neighborhood of $\Lambda$}
 In order to describe the multiscale flow trees on $\Lambda$ determined by $\Lambda_K$ we need to identify some neighborhood of $\Lambda$ with $J^{1}(\Lambda)\approx T^{\ast}T^{2}\times\R$. Thinking of $T^2$ as $S^1\times S^1$
we use two different versions of the cotangent bundle of $S^{1}$ to describe $T^{\ast}T^2$. First, we think of $S^{1}$ as $[0,2\pi]$ with endpoints identified, we let $s$ be a coordinate on $[0,2\pi]$ and $\sigma\in\R$ be a fiber coordinate in $T^{\ast}S^{1}=S^{1}\times\R$. We write $T^{\ast} S^{1}_{\lambda}$ for the cotangent bundle with these coordinates:
\[
T^{\ast}S^{1}_{\lambda}=\left\{(s,\sigma)\colon s\in[0,2\pi]\,,\,\sigma\in\R\right\}.
\]
We denote the second version $T^{\ast}S^{1}_{\mu}$ and define it as
\[
T^{\ast} S^{1}_{\mu}=\left\{(\xi,\eta)\in\R^{2}\times\R^{2}\colon |\xi|=1,\,\, \xi\bull\eta=0\right\}.
\]

Let $r(s)=(\cos s, \sin s, 0)$, $0\le s\le2\pi$ denote a unit vector in the $x_1x_2$-plane at an angle $s$ from the $x_1$-axis and  $\theta(s)=(-\sin s, \cos s, 0)$ the vector $r(s)$ rotated $\pi/2$ counter-clockwise in the $x_1x_2$-plane (we can think of it as the standard angular vector at $r(s)$ translated back to the origin).
Let $\zeta$ be the coordinate on $\R$. Consider the map
\[
\Phi\colon T^{\ast}S^{1}_{\lambda}\times T^{\ast} S^{1}_{\mu}\times\R\to U^{\ast}\R^{3}=\R^{3}\times S^{2},
\]
where $\Phi=(\Phi_1,\Phi_2)$ is defined by
\begin{align}\label{eq:almost1jetcoord}
\Phi_1(s,\sigma,\xi,\eta,\zeta)&=
r(s)+\frac{1}{\sqrt{1-\sigma^{2}}}(\eta_1r(s)+\eta_2(0,0,1))\\\notag
&+\zeta\left(-\sigma{\theta(s)}+\sqrt{1-\sigma^{2}}(\xi_1r(s)+\xi_2(0,0,1))\right),\\\notag
\Phi_2(s,\sigma,\xi,\eta,\zeta)&=
-\sigma{\theta(s)}+\sqrt{1-\sigma^{2}}(\xi_1r(s)+\xi_2(0,0,1)).
\end{align}
Then $\Phi|_{S^{1}_{\lambda}\times 0 \times S^{1}_{\mu} \times 0\times 0}$ is a parametrization of $\Lambda$ and its restriction to a small neighborhood of the $0$-section is an embedding. Furthermore, since $\xi\bull\eta=0$, we know that $\eta\bull d\xi=-\xi\bull d\eta$. Using this and the fact that $r(s)\bull \theta(s)=0$ and $r'(s)=\theta(s)$ we can compute
\begin{align*}
\Phi^{\ast}(p\, dq)&=d\zeta-\sigma\,ds-\xi\bull d\eta -\frac{1}{\sqrt{1-\sigma^{2}}}\sigma\eta_1\,ds.
\end{align*}

We introduce the following notation:
\begin{align}\label{eq:alphabeta}
\beta_0&=\sigma\,ds+\xi\bull d\eta,\\
\alpha &=\frac{1}{\sqrt{1-\sigma^{2}}}\sigma\eta_1\,ds,\\
\beta_t&=\beta_0-t\alpha,\quad 0\le t\le 1.
\end{align}
Note that $d\beta_t$ is symplectic in a neighborhood of the $0$-section, for all $t$.
Using Moser's trick we define a time-dependent vector field $X_t$ by
\begin{equation}\label{eq:Moser}
-\alpha=d\beta_{t}(X_t,\cdot)
\end{equation}
and find that if $\psi_{t}$ denotes the time $t$ flow of $X_t$ then
\[
\psi_{t}^{\ast}d\beta_t=d\beta_0.
\]
In particular,
\[
d(\psi_{t}^{\ast}\beta_t-\beta_0)=0.
\]
Equation \eqref{eq:Moser} and the definition of $\alpha$ imply
that $X_t=0$ along the $0$-section and thus $\psi_{t}^{\ast}\beta_t-\beta_0=0$ along the $0$-section. By the homotopy invariance of de Rham cohomology, the closed form $\psi_{t}^{\ast}\beta_t-\beta_0$ is exact. Let the function $h$ be such that $\beta_0=\psi_{1}^{\ast}\beta_1+dh$ and such that $h=0$ on the $0$-section.

We can bound the growth of $\alpha,X_t,h$ in terms of distance $r := (|\eta|^2+|\sigma|^2)^{1/2}$ from the origin.
From the explicit expression for $\alpha$ in Equation~\eqref{eq:alphabeta}, we have $\alpha = \Ordo(r^2)$ and thus from Equation~\eqref{eq:Moser},
\begin{equation}\label{eq:vfestimate}
|X_t|=\Ordo(|\eta|^2+|\sigma|^2) \quad\text{and}\quad
|dX_t|=\Ordo((|\eta|^2+|\sigma|^2)^{1/2}).
\end{equation}
Then from the definition of $h$, $|dh| = \Ordo(r^2)$ and so
\begin{equation}\label{eq:hestimate}
|d^{(k)}(h)| = \Ordo((|\eta|^2+|\sigma|^2)^{(3-k)/2}),\quad k=0,1,2.
\end{equation}
We will use these estimates in the proof of Lemma~\ref{lma:braidfront} below.

Finally, if $\Psi$ is the diffeomorphism of $T^{\ast}S^{1}_{\lambda}\times T^{\ast}S^{1}_{\mu}\times\R$ given by
\[
\Psi((s,\sigma,\xi,\eta),\zeta)=(\psi_1(s,\sigma,\xi,\eta),\zeta+h(s,\sigma,\xi,\eta)),
\]
then
\[
\Psi^{\ast}\Phi^{\ast}(p\,dq)=d\zeta-\sigma\,ds-\xi\bull d\eta.
\]

The map $\Theta=\Phi\circ\Psi$ is the $1$-jet neighborhood map we will use. It is an embedding from a neighborhood of the $0$-section in $J^{1}(T^{2})=T^{\ast}T^{2}\times\R$ to a neighborhood of $\Lambda\subset U^{\ast}\R^{3}=J^{1}(S^{2})$ such that $\Theta^{\ast}(p\,dq)=d\zeta-\theta$, where $\zeta$ is a coordinate in the $\R$-factor and where $\theta$ is the Liouville form on $T^{\ast}T^{2}$.

\begin{rmk}\label{rem:shortchords}
Notice that the contactomorphism $\Theta$ sends the Reeb flow of $J^1(T^2)$ to the Reeb flow of $J^1(S^2)$. (Here, as throughout the rest of the paper, we are identifying $J^1(S^2)$ with $U^{\ast}\R^{3}$ using the contactomorphism in Equation~\eqref{eq:maincontacto}.) Thus any Reeb chord of $\Lambda_K$ in $J^1(T^2)$ corresponds to a Reeb chord of $\Lambda_K$ in $J^1(S^2)$ and any Reeb chord of $\Lambda_K$ in $J^1(S^2)$ that lies entirely in $N$ corresponds to a Reeb chord of $\Lambda_K$ in $J^1(T^2)$.
\end{rmk}

% **************************************************
\subsubsection{Conormal lifts of closed braids as multi-sections}
Consider $T^{2}=S^{1}_{\lambda}\times S^{1}_{\mu}$ as above. A
multi-section of $J^{0}(T^{2})$ is a smooth map $F\colon T^{2}\to
J^{0}(T^{2})$ such that $\pi\circ F$ is an immersion ({\em i.e.}, a covering
map). In particular, a multi-section can be thought of as the graph of
a multi-function $F\colon T^2\to \R$. (Context will designate whether
$F$ refers to a (multi-)function or its graph, a
(multi-)section.)   The $1$-jet extension of a generic multi-section is
a Legendrian submanifold. We denote it $\Gamma_{j^{1}(F)}$. Let
$K\subset\R^{3}$ be a link and let $B$ be a closed braid representing
$K$ with corresponding multi-section $f_B=\{f_1,\dots,f_n\}$,
$f_j\colon[0,2\pi]\to D^{2}$, as described at the beginning of
Section~\ref{ssec:conormalliftofK}.
\begin{lma}\label{lma:braidfront}
The conormal lift $\Lambda_K$ is Legendrian isotopic to the Legendrian submanifold $\Theta(\Gamma_{j^{1}(F_B)})$ where $F_{B}$ is the multi-section given by the functions $F_{j}\colon[0,2\pi]\times S^{1}_{\mu}\to\R$, $j=1,\dots,n$, where
\[
F_j(s,\xi) = f_j(s) \bull \xi.
\]
Here we think of $S^{1}_{\lambda}$ as $[0,2\pi]$ with endpoints identified and we identify $S^{1}_\mu$ with the unit circle in the plane of the disk where $f_j\colon [0,2\pi]\to D^{2}$ takes values.
\end{lma}

\begin{proof}
Let $N\subset U^{\ast}\R^{3}$ denote a $\delta$-neighborhood of $\Lambda$ in which $\Theta$  and $\Phi$ give local coordinates. Take the $C^{2}$-norm of $f_B$ sufficiently small so that $\Lambda_{K}$ will be in $N$. We first show that $\Phi^{-1}(\Lambda_K)$ is given by the 1-jet lift of $F_B$. To see this notice that the image of the braid $B$ is given by the image of the maps
\[
(f_j(s))_x r(s) + (f_j(s))_y (0,0,1)
\]
where $(f_j)_x$ and $(f_j)_y$ are the $x$ and $y$-coordinates of $f_j$, and $r(s)$ is as in Equation~\eqref{eq:almost1jetcoord}. Thus the normal component in $J^1(S^2)=T^\ast S^2\times \R$ ({\em i.e.}, the $\R$-component) of the braid at the point $(f_j(s))_x r(s) + (f_j(s))_y (0,0,1))$ is given by
\[
\left((f_j(s))_x r(s) + (f_j(s))_y (0,0,1)\right) \bull (\xi_1r(s)+\xi_2(0,0,1))=f_j(s)\bull \xi.
\]
According the the definition of $\Phi$ we see that the $\R$-factor of $T^\ast T^2\times\R=J^1(T^2)$ maps to the $\R$-factor in $T^\ast S^2\times \R$ by $\zeta\mapsto \zeta \sqrt{1-\sigma^2}$. Thus the multi-section of $J^0(T^2)$ corresponding to $\Phi^{-1}(\Lambda_K)$ is given by $\frac 1{\sqrt{1-\sigma^2}} f_j(s)\bull \xi$, but in $J^0(T^2)$ the $\sigma$-coordinate is always equal to zero. Thus the multi-function $F_B$ does indeed describe $\Phi^{-1}(\Lambda_K)$ as claimed.

Now $\Theta^{-1}(\Lambda_{K})$ is the $1$-jet graph $\Gamma_{j^{1}(G_{B})}$ of some multi-section $G_B$. In general, Legendrian submanifolds of $J^1(T^2)$ will be given by cusped multi-sections, but since each point in $\Theta^{-1}(\Lambda_{K})$ has a neighborhood in $\Theta^{-1}(\Lambda_{K})$ that can be made $C^1$ close to the zero section in $J^1(T^2)$, we see that $\Theta^{-1}(\Lambda_{K})$ has empty caustic and hence is the 1-jet extension of a multi-section.

From the above discussion we see that $\Gamma_{j^1(G_B)}$ is the same as $\Psi^{-1}(\Gamma_{j^1(F_B)})$. The estimates~\eqref{eq:vfestimate} and~\eqref{eq:hestimate} then imply that the $C^1$-distance between $F_B$ and $G_B$ is $\Ordo(\delta^{2})$. Consequently, for $\delta>0$ sufficiently small, $\Gamma_{j^{1}(G_B)}=\Lambda_K$ and $\Gamma_{j^{1}(F_B)}$ are Legendrian isotopic.
\end{proof}

% **************************************************
\subsubsection{Reeb chords and grading}
Let $K\subset\R^{3}$ be a link. We assume that $K$ is braided around the unknot as a braid on $n$ strands and we represent $\Lambda_{K}$ as $\Gamma_{j^{1}(F_B)} \subset J^1(T^2)$ as in Lemma \ref{lma:braidfront}. Then the Reeb chords of $\Lambda_{K}$ in $J^1(S^2)$ are of two types: {\em short chords}, which are entirely contained in the neighborhood $N$ of $\Lambda,$ and {\em long chords}, which are not. According to Remark~\ref{rem:shortchords} we see that the short chords of $\Lambda_K$ correspond to chords of $\Lambda_K$ in $J^1(T^2)$. As with $\Lambda$, one can use the techniques of Section~\ref{sssec:geomprelim} to conclude that the Maslov class of $\Lambda_{K}$ vanishes, and thus the grading of any Reeb chord of $\Lambda_{K}$ with both endpoints on the same component is independent of capping path. For Reeb chords with endpoints on distinct components the grading depends on the chosen transport along the paths connecting the components, {\emph cf.\ } Remark~\ref{rmk:gradingmixed}. Here we use the following choice throughout. Fix a point $x\in \Lambda$ and let the base points of the components of $\Lambda_K$ all lie in the intersection  $J^{1}_{x}(\Lambda)\cap \Lambda_K$ of $\Lambda_K$ and the fiber of $J^{1}(\Lambda)$ at $x$. Take the connecting paths as straight lines in $T^{\ast}_{\Pi^{\Lambda}(x)}\Lambda$ and use parallel translation in the flat metric followed by a rotation along the complex angle or the complementary complex angle, whichever is less than $\frac{\pi}{2}$, as transport. Note that as $K$ gets closer to $U$ the angle of this rotation approaches $0$.
\begin{lma}\label{lma:braidnicechords}
Up to smooth isotopy, we can choose the link $K \subset \R^3$
so that $\Lambda_K$ has exactly $2n(n-1)$ short Reeb chords:
\begin{align*}
\{a_{ij}\}_{1\le i,j\le n,\,\,i\ne j}, &\quad|a_{ij}|=0,\\
\{b_{ij}\}_{1\le i,j\le n,\,\,i\ne j}, &\quad|b_{ij}|=1,
\end{align*}
and exactly $2n^{2}$ long Reeb chords:
\begin{align*}
\{c_{ij}\}_{1\le i,j\le n}, &\quad|c_{ij}|=1,\\
\{e_{ij}\}_{1\le i,j\le n}, &\quad|e_{ij}|=2.
\end{align*}
Here $c_{ij}$ (respectively $e_{ij}$) lie in small neighborhoods of the Reeb chords $c$ (respectively $e$) of $\Lambda$ for all $i,j$. Furthermore all Reeb chords can be taken to correspond to transverse intersection points in $T^{\ast} S^{2}$.
\end{lma}

\begin{rmk}\label{rmk:Reebchordtypes}
We will use the following notation for the Reeb chords of $\Lambda_K$. We say that the short Reeb chords are of type $\mathbf{S}$ and the long of type $\mathbf{L}$. We sometimes specify further and say that a short (long) Reeb chord of grading $j$ is of type $\mathbf{S}_j$ ($\mathbf{L}_j$).
\end{rmk}

\begin{pf}[Proof of Lemma~\ref{lma:braidnicechords}]
The statement in the Lemma for long chords is immediate from the fact that $\Lambda_K$ is the $1$-jet graph of a multi-section with $n$ sheets over $\Lambda$ which has two Reeb chords: $e$ with $|e|=2$ and $c$ with $|c|=1$. To prove the statement on short chords we note that we may choose the multi-section $f_{B}$ so that for any $i\ne j$, $|f_i-f_j|$ has a maximum at $2\pi-\delta\in[0,2\pi]$, a minimum at $2\pi-2\delta$, and no other critical points. Parameterizing $S^{1}_{\mu}$ by $\xi=(\cos t,\sin t)$, $t\in[0,2\pi]$, we find that the difference between two local functions of $\Lambda_{K}$ is
\[
F_{ij}(s,t)=(f_i(s)-f_{j}(s))\bull (\cos t,\sin t).
\]
Now $dF_{ij}=0$ if and only if
\begin{align}%\label{eq:braidcrit1}
&(f_{i}'(s)-f_{j}'(s))\bull (\cos t,\sin t)=0 \text{ and}\\\label{eq:braidcrit2}
&(f_{i}(s)-f_{j}(s))\bull (-\sin t,\cos t)=0,
\end{align}
which in turn happen if and only if $s$ is critical for $|f_i(s)-f_j(s)|$ and $t$ takes one of the two values, say $t_0$ and $t_0+\pi$, for which Equation~\eqref{eq:braidcrit2} holds. We take $a_{ij}$ to be the chord corresponding to the minimal distance between strands and $b_{ij}$ to the maximal distance.

In order to compute the gradings of $a_{ij}$ and $b_{ij}$ we note that the front of $\Lambda_{K}$ in $J^0(\Lambda)=\Lambda\times\R$ has no singularities and that the chords $a_{ij}$ (respectively $b_{ij}$) correspond to saddle points (respectively maxima) of positive function differences of local defining functions. The grading statement then follows from Equation~\eqref{eq:frontnu}.
\end{pf}

% **************************************************
% **************************************************
\subsection{Counting flow trees of $\Lambda_K$ in $J^1(\Lambda)$}\label{ssec:thecount}
In this subsection we determine all the flow trees for $\Lambda_K \subset J^1(\Lambda)$. Our enumeration relies on a particular and fairly technical choice of position for $\Lambda_K$ over the regions where the braid twists, whose details we defer to Section~\ref{ssec:twistslice}. For our current computational purposes, we only need a few qualitative features of these twist regions, as described in Sections~\ref{sssec:setup} through \ref{sssec:endpointpath} below, which will serve to motivate the more technical parts of our discussion of twist regions in Section~\ref{ssec:twistslice}.

Given these qualitative features, we perform the actual combinatorial computation of flow trees in Sections~\ref{sssec:partial} through~\ref{sssec:countflowtree} (after first presenting a scheme for calculating signs for flow trees in Section~\ref{sssec:signrules}), culminating in Lemma~\ref{lma:dB}, which presents a purely algebraic formula for the Legendrian DGA of $\Lambda_K \subset J^1(\Lambda)$. This comprises an important subalgebra of the Legendrian DGA of $\Lambda_K \subset J^1(S^2)$, the rest of which is computed in Section~\ref{sec:combdiff}.

% **************************************************
\subsubsection{Basic setup}\label{sssec:setup}
Recall that $K$ is the closure of a braid in $S^{1}\times D^2$ given by a collection $\{f_1,\dots,f_n\}$ of functions $f_j\colon [0,2\pi]\to D^{2}$, $j=1,\dots,n$. We use Lemma \ref{lma:braidfront} to represent $\Lambda_{K}\subset J^{1}(\Lambda)$ as the $1$-jet graph of the functions $F_j\colon[0,2\pi]^{2}\to\R$ given by
\[
F_j(s,t)=f_j(s)\cdot(\cos t,\sin t).
\]
Here $(s,t)\in [0,2\pi]^{2}$ are coordinates on $\Lambda$, with $s$ corresponding to the parameter along the unknot (represented by the unit circle in the $x_1x_2$-plane) and $t$ to the parameter along a unit circle in the normal fiber of the unknot with $0$ corresponding to the positive outward normal of the unit circle in the $x_1x_2$-plane. Furthermore, recall from Lemma \ref{lma:braidnicechords} that Reeb chords of $\Lambda_K\subset J^{1}(\Lambda)$ correspond to points $(s,t)\in\Lambda$ where $s$ is a critical point of $|f_i(s)-f_j(s)|$ and where $t$ is such that the vector $(\cos t,\sin t)$ is parallel to the vector $f_i(s)-f_j(s)$.

Let $[s_0,s_1]\subset[0,2\pi]$. We refer to the part of $\Lambda_K$ lying over an interval $[s_0,s_1]$,
\[
\Lambda_K\cap J^{1}([s_0,s_1]\times S^{1})\subset J^{1}(\Lambda),
\]
as the \emph{$[s_0,s_1]$-slice} of $\Lambda_{K}$. We will represent the braid as follows: the actual twists will take place in an $[s_0^{\mathrm{br}},s_1^{\mathrm{br}}]$-slice, where $[s_0^{\mathrm{br}},s_1^{\mathrm{br}}]\subset(0,\frac{\pi}{2})$. Inside the $[s_0^{\mathrm{br}},s_1^{\mathrm{br}}]$-slice the braid is given by sub-slices where it twists so that two strands are interchanged, separated by slices where the braid is trivial. Outside the $[s_0^{\mathrm{br}},s_1^{\mathrm{br}}]$-slice the braid is trivial.

We will choose perturbations so that the following holds: Reeb chord endpoints of the Reeb chords $e$ and $c$ of $\Lambda$ have the following coordinates:
\[
e^{+}=(0,0),\quad e^{-}=(\pi,\pi),\quad c^{+}=(\pi,0),\quad\text{and}\quad c^{-}=(0,\pi),
\]
see Lemma \ref{lem:gradingunknot}. Consequently, all Reeb chords $e_{ij}$ and $c_{ij}$, $1\le i,j\le n$, of type $\mathbf{L}$ are located near these points. The Reeb chords $a_{ij}$ and $b_{ij}$ of type $\mathbf{S}$ are located near $s=\underline{s}$ and $s=\overline{s}$ where $\frac{\pi} {2}<\underline{s}<\overline{s}<2\pi$, see the proof of Lemma \ref{lma:braidnicechords}.

% **************************************************
\subsubsection{Reeb chords and trivial slices}\label{sssec:flowtreesKinLambda}
Consider an interval $[s_0,s_1]$ where the braid is trivial. We will describe a model for the $[s_0,s_1]$-slice of the conormal lift of a trivial braid which we will use in two ways: to control Reeb chords and to define a normal form of a slice of the trivial braid in which there are no Reeb chords. We first describe a somewhat degenerate $[s_0,s_1]$-slice of $\Lambda_K$: we represent the trivial braid by
\[
f_j(s)=(0,j\,\psi(s)),\quad j=1,\dots,n,
\]
where $\psi(s)$ is a positive function that has a non-degenerate local minimum at $\underline{s}$ and a non-degenerate local maximum at $\overline{s}$. Here $s_0<\underline{s}<\overline{s}<s_1$ and $\psi$ has no other critical points.

We call the $1$-jet graph of the function
\[
F_j(s,t)=f_j(s)\bull(\cos t,\sin t)
\]
the $j^{\rm th}$ sheet of $\Lambda_K$ and denote it by $S_j$, $j=1,\dots,n$.
Writing
\begin{align*}
F_{ij}(s,t)&=F_i(s,t)-F_j(s,t)=(f_i(s)-f_j(s))\bull(\cos t,\sin t)\\
&=(i-j)\psi(s)\sin t,
\end{align*}
Reeb chords correspond to critical points of $F_{ij}$ and are all located in the fibers over the following points in $\Lambda$:
\[
\left(\underline{s},\frac{\pi}{2}\right),\,\left(\underline{s},\frac{3\pi}{2}\right),\,\left(\overline{s},\frac{\pi}{2}\right),\,\text{and}\,\left(\overline{s},\frac{3\pi}{2}\right).
\]
The Reeb chords lying in the fibers over the first (respectively second) two points correspond to saddle points (respectively extrema) of the functions $F_{ij}.$ We denote them by $a_{ij}$ (respectively $b_{ij}$). Here the labeling is such that $a_{ij}$ and $b_{ij}$ begin on $S_j$ and end on $S_i$. Thus, if $i>j$ then the $t$-coordinate of $a_{ij}$ and $b_{ij}$ equals $\frac{\pi}{2}$ whereas if $i<j$ it equals $\frac{3\pi}{2}$.

The flow trees that we study all have their negative punctures at
$a_{ij}$ and thus we must understand the unstable manifold $W^{\rm
  u}(a_{ij})$ of $a_{ij}$ as a critical point of $F_{ij}$. (Recall that we are using the positive gradient flow when discussing stable and unstable manifolds.)
It is straightforward to check that
\[
W^{\rm u}(a_{ij})=
\begin{cases}
\left\{(s,t)\colon t=\frac{\pi}{2}\right\}\quad &\text{if }i>j\\
\left\{(s,t)\colon t=\frac{3\pi}{2}\right\}\quad &\text{if }i<j.\\
\end{cases}
\]

Note that the $[s_0,s_1]$-slice of $\Lambda_{K}$ as defined above is degenerate: Reeb chords are not disjoint, the unstable manifolds $W^{\rm u}(a_{ij})$ are not mutually transverse, and their stable counterparts are not mutually transverse either. The following lemma describes the $[s_0,s_1]$-slice of $\Lambda_{K}$ after a small perturbation which makes it generic. In particular, the perturbation is so small that there is a natural one-to-one correspondence between Reeb chords before and after perturbation and we will keep the notation $a_{ij}$ and $b_{ij}$ from above. There are of course many perturbations which makes $\Lambda_K$ generic. The particular choice studied here is designed to make counting flow trees as simple as possible.
\begin{lma}\label{lma:braidReeb+triv}
For $\epsilon_0>0$ arbitrarily small, there exist a Legendrian isotopy of $\Lambda_{K}$ and a collection of functions $\epsilon_{ij},\epsilon'_{ij},\epsilon''_{ij} :\thinspace [s_0,s_1] \to (0,\epsilon_0)$ for each $n\geq i>j\geq 1$, so that the following conditions hold:
\begin{itemize}
\item
$\epsilon_{ij}$ are lexicographically ordered in $(i,j)$:
for any $s,s'$, $\epsilon_{ij}(s) > \epsilon_{i'j'}(s')$ if $i>i'$, or if $i=i'$ and $j>j'$;
\item
similarly, $\epsilon'_{ij}$ and $\epsilon''_{ij}$ are lexicographically ordered in $(i,j)$;
\item
the unstable manifolds $W^{\rm u}(a_{ij})$ are curves of the following form:
\begin{itemize}
\item if $i>j$ then $W^{\rm u}(a_{ij})=\{(s,t)\colon t=\frac{\pi}{2}-\epsilon_0+\epsilon_{ij}(s)\}$ and
\item if $i<j$ then $W^{\rm u}(a_{ij})=\{(s,t)\colon t=\frac{3\pi}{2}+\epsilon'_{ji}(s)\}$;
\end{itemize}
\item
the Reeb chords corresponding to critical points in these unstable manifolds satisfy the following:
\begin{itemize}
\item if $i>j$ then the $s$-coordinate of $a_{ij}$ (respectively $b_{ij}$) equals $\underline{s}+\epsilon''_{ij}(\underline{s})$
(respectively $\overline{s}+\epsilon_{ij}''(\overline{s})$).
\item if $i<j$ then the $s$-coordinate of $a_{ij}$ (respectively $b_{ij}$) equals $\underline{s}+\epsilon''_{ji}(\underline{s})$
(respectively $\overline{s}+\epsilon_{ji}''(\overline{s})$).
\end{itemize}
\end{itemize}
\end{lma}

\begin{rmk}
The functions $\epsilon_{ij},\epsilon_{ij}',\epsilon_{ij}''$ are chosen so that the
unstable manifolds $W^{\rm u}(a_{ij})$, $i>j$, appear in the
lexicographical order on $\{(i,j)\}_{1\le j<i\le n}$ if read in the
increasing $t$-direction, so that the $W^{\rm u}(a_{ij})$, $i<j$,
appear in the lexicographical order of $\{(j,i)\}_{1\le i<j\le n}$ if
read in the positive $t$-direction, and so that no
$W^{\mathrm{u}}(a_{ij})$ lies in the region $\frac{\pi}{2}\le
t\le\frac{3\pi}{2}$, see Figure \ref{fig:ordering}.
\end{rmk}
\begin{figure}[htb]
\labellist
\small\hair 2pt
\pinlabel $a_{ij}$ at 168 154
\pinlabel $b_{ij}$ at 205 154
\pinlabel $a_{ij}$ at 168 24
\pinlabel $b_{ij}$ at 205 24
\pinlabel $i<j$ at 0 143
\pinlabel $i>j$ at 0 38
\pinlabel $t=\frac{3\pi}2$ at 0 128
\pinlabel $t=\frac\pi2$ at 0 51
\pinlabel $t$ at -4 85
\pinlabel $s$ at 121 -2
\pinlabel {lex.\ order} at 278 140
\pinlabel $(j,i)$ at 271 128
\pinlabel {lex.\ order} at 278 50
\pinlabel $(i,j)$ at 271 38
\endlabellist
\centering
\includegraphics{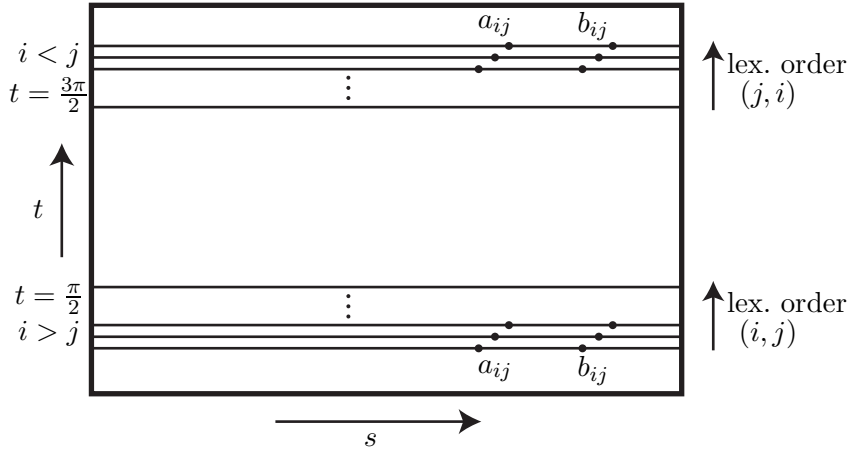}
\caption{Ordered unstable manifolds after perturbation.}
\label{fig:ordering}
\end{figure}
\begin{pf}[Proof of Lemma~\ref{lma:braidReeb+triv}]
Throughout the proof we use the notational conventions above for Reeb chords. We will consider two different perturbations of the braid representative for the trivial braid given at the beginning of this subsection and then
combine them to give the desired perturbation.

We first choose the scaling function $\psi(s)$ from the beginning of this subsection
to additionally satisfy $\psi(s)=\frac M2 (s-\underline{s})^2+c$ in a $2\epsilon_0$ neighborhood of  $\underline{s}$ and $\psi(s)=-\frac M2 (s-\overline{s})^2+c'$ in a $2\epsilon_0$ neighborhood of $\overline{s}$, where $c$ and $c'$ are constants, $M$ is some large constant and $\epsilon_0$ is some small positive constant.

Now choose constants $\delta_1,\ldots, \delta_n$ so that
\[
0\leq \delta_1< \delta_2<2\delta_2<\delta_3<3\delta_3< \ldots<\delta_n< n\delta_n<\epsilon_0.
\]
For $i>j$, define $\epsilon_{ij}'' = \frac{i\delta_i-j\delta_j}{i-j}$. The above inequalities imply that the $\epsilon_{ij}''$ are positive, less than $\epsilon_0$, and lexicographically ordered in $(i,j)$: indeed, we have
\[
\epsilon_{ij}'' \in \left( \frac{i}{i-j+1}\delta_i,\frac{i}{i-j}\delta_i \right) \subset (\delta_i,i\delta_i).
\]

Set $\psi_i(s)=\psi(s-\delta_i)$. Defining the trivial braid by the multi-function determined by $f_j(s)=(0,j\psi_j(s))$, we see that the
Reeb chord $a_{ij}$ has $s$-coordinate determined
by the solution near $\underline{s}$ to
\[
iM(s-\delta_i-\underline{s}) =i \psi_i'(s) = j\psi_j'(s)=jM(s-\delta_j-\underline{s}),
\]
which is precisely $s = \underline{s}+\epsilon''_{ij}$, and similarly for $b_{ij}$ and $\overline{s}$.
Thus the Reeb chords $a_{ij}$ are at
\[
\begin{cases}
\left(\underline{s},\frac{\pi}{2}\right)+(\epsilon''_{ij},0) &\text{if }i>j,\\
\left(\underline{s},\frac{3\pi}{2}\right)+(\epsilon''_{ji},0) &\text{if }i<j,
\end{cases}
\]
and the Reeb chords $b_{ij}$ are at
\[
\begin{cases}
\left(\overline{s},\frac{\pi}{2}\right)+(\epsilon''_{ij},0) &\text{if }i>j,\\
\left(\overline{s},\frac{3\pi}{2}\right)+(\epsilon''_{ji},0) &\text{if }i<j.
\end{cases}
\]

Now consider a different representation of the trivial braid. In particular, returning to the original representation of the trivial braid given at the beginning of this subsection, we can replace the curve $\{x=0\}$ where the functions
$f_i(s)$ take values, with a curve family $\{x=h_t(y)\}$ where
$h_t\colon\R\to\R$ is a smoothly varying family of functions such
that:
\begin{itemize}
\item
$h_t$ is constant in $t$ near $t=\frac{\pi}{2}$ and
$t=\frac{3\pi}{2}$,
\item
$h_{\pi/2}(j) =  j(\delta_n - \delta_j)$ for $j=1,\ldots,n$,
\item
$h_{3\pi/2}(j) = - j\delta_j$ for $j=1,\ldots,n$,
\end{itemize}
where the $\delta_j$ are as before. (We can assume that the $t$
dependence of $h_t$ is supported in an arbitrarily small neighborhood
of $\pi$ and $0$. While this is not necessary here, it will be
important when describing the lift of a non-trivial braid in Section~\ref{ssec:twistslice}.)
Then let
\[
F_j(s,t)=\psi(s)(h_t(j), j)\bull (\cos t,\sin t)
\]
be the function of the $j^{\rm th}$ sheet. The critical points for
$F_{ij} = F_i-F_j$ are at $s = \underline{s}$ or
$s=\overline{s}$, with $t$ given by
\[
\cot(t) = \frac{h_t(i)-h_t(j)}{i-j+(\partial h/\partial t)(i) -
(\partial h/\partial t)(j)}.
\]
Since $h_t$ is small (and can be made arbitrarily small by the appropriate choice of $\delta_i$),
this equation can only hold for $t$ near
$\pi/2$ or $3\pi/2$, whence
$\cot(t) = \frac{h_{\pi/2}(i)-h_{\pi/2}(j)}{i-j}$ or $\cot(t) =
\frac{h_{3\pi/2}(i)-h_{3\pi/2}(j)}{i-j}$, respectively.

In the former case, the expression for $h_{\pi/2}$ implies that the
Reeb chords $a_{ij},b_{ij}$ for $i>j$ have $t$-coordinate near
$\pi/2$ and are given by the solution to
$\cot(t) = 
\delta_n-\epsilon_{ij}''
>0$. If $t=t_{ij}$ is the solution to this equation, then $t_{ij} <
\pi/2$ and $t_{ij}$ is ordered in lexicographic order on $(i,j)$,
and so we can write
$t_{ij} = \pi/2-\epsilon_0+\epsilon_{ij}$ with $\epsilon_{ij}$ ordered
lexicographically in $(i,j)$. Similarly, the expression for $h_{3\pi/2}$ implies that the Reeb chords $a_{ij},b_{ij}$ for $i<j$ have $t$-coordinate near $3\pi/2$ and are given by the solution to $\cot(t) = -\epsilon_{ji}'' < 0$, which yields $t = 3\pi/2+\epsilon_{ji}'$ with $\epsilon_{ji}'$ ordered lexicographically on $(i,j)$.
Furthermore, the unstable manifolds
$W^{\rm u}(a_{ij})$ and $W^{\rm u}(a_{ji})$ are horizontal (constant in $t$)
since the $\partial_t$ component of $\nabla F_{ij}$ is zero at $t=t_{ij}$.

Combining the perturbations, we find that the location and ordering of the critical points and unstable manifolds of
\[
F_j(s,t)=\psi_j(s)(h_t(j), j)\bull (\cos t,\sin t)
\]
is as desired.
\end{pf}

\begin{rmk}\label{rmk:poscharstabmfd}
Recall that our notation for (un)stable manifolds refers to the positive gradient flow of positive function differences and note that the unstable manifolds $W^{\rm u}(a_{ij})$ can be characterized as the only flow line determined by sheets $S_i$ and $S_j$ along which the local function difference stays positive for all time under the negative gradient flow: along any other (non-constant) flow lines of the negative gradient, the local function difference eventually becomes negative.
\end{rmk}

\begin{rmk}\label{rmk:standardtrivbraid}
Consider next an $[s',s'']$-slice where the braid is trivial, {\em e.g.}, the slices mentioned above that separate the slices where twists of the braid occur. In each such slice we will take the braid to look much like in the $[s_0,s_0+\delta]$-slice of the braid in Lemma \ref{lma:braidReeb+triv}, where $\delta>0$ is small enough so that the $s$-coordinate of any Reeb chord $a_{ij}$ is larger than $s_0+\delta$. More precisely, we require that the functions $f_j(s)$, $j=1,\dots,n$ take values in a family of graphical curves $\{x=h_t(y)\}$ (with $y=j$ for $f_j$), where $h_t$ is a small function which is independent of $t$ near $\frac{\pi}{2}$ and $\frac{3\pi}{2}$.
In order to make sure that the $[s',s'']$-slice does not have any Reeb chords we let the points $f_j(s)$ move away from each other along the curves $\{x=h_t(y)\}$ as we go through $[s',s'']$ \emph{from right to left}, so that $|f_i(s)-f_j(s)|$ decreases with $s$ for all $i,j$. We call a braid with these properties a \emph{standard trivial braid} and we number the functions $f_1,\dots,f_n$ according to the order in which they appear along the curve $\{x=h_t(y)\}$ with orientation induced from the positively oriented $y$-axis.
\end{rmk}

% **************************************************
\subsubsection{A model for $\Lambda_K$, endpoint paths, and homology indicators}\label{sssec:endpointpath}
Here we describe the qualitative features of the gradient flows
associated to $\Lambda_K$ in $J^1(\Lambda)$ that will be necessary for
our computation of the rigid flow trees. Some of these features have
been discussed in Sections~\ref{sssec:setup} through \ref{sssec:endpointpath}.

For the braid $B=\sigma_{k_1}^{\epsilon_1}\dots\sigma_{k_m}^{\epsilon_m}$, $\epsilon_l=\pm 1$, we can build a braiding slice $[s_0^{\mathrm{br}},s_1^{\mathrm{br}}]$ where $B$ lives, by starting with a twist slice corresponding to $\sigma_{i_m}^{\epsilon_m}$ and then attaching slices corresponding to the other twists $\sigma_{i_l}^{\epsilon_l}$ consecutively, working backwards in $l$. We then extend $B$ to the complement of the braiding slice by closing it with a trivial braid as in Lemma \ref{lma:braidReeb+triv}. This is the model of $\Lambda_K$ that we will use below. More precisely, we will construct the braid model to have the following properties:
\begin{itemize}
\item 
 A slice $[s',2\pi]$ of a trivial braid contains all of the Reeb chords of $\Lambda_B \subset J^1(\Lambda)$, as in Lemma~\ref{lma:braidReeb+triv}. The Reeb chords $a_{ij}$ (respectively  $b_{ij}$) of $\Lambda_K$ lie near $s=2\pi-4\delta$ (respectively $s=2\pi-2\delta$), which is in the slice which is complementary to the braiding region.
The manifolds $W^{\mathrm{u}}(a_{ij})$ in this slice lie just below $t=\frac{\pi}{2}$ and are ordered in the $t$-direction according to the lexicographically order on $\{(i,j)\}_{1\le j<i\le n}$. The manifolds $W^{\mathrm{u}}( a_{ji})$ lie just above $t=\frac{3\pi}{2}$ and are ordered in the $t$-direction according to the lexicographically order on $\{(i,j)\}_{1\le j<i\le n}$.
\item All twists of the braid occur in the braiding region $2\delta <s<4\delta$, where $\Lambda_K$ looks like twist slices separated by standard trivial braid slices.
\item There are points $s_1=2\delta<s_2<\cdots<s_{2m}=4\delta$ so that the $[s_{2l-1}, s_{2l}]$-slice contains the twist region for $\sigma_{i_l}^{\epsilon_l}$. For a slice corresponding to the braid crossing $\sigma_{i_l}^{\epsilon_l}=\sigma_k^{\pm 1}$, the unstable manifolds for $a_{k\, k+1}$ and $a_{k+1\, k}$ are shown in Figure~\ref{fig:twistslices}.  The unstable manifolds for the other $a_{ij}$ are as for the trivial braid.
\begin{figure}[htb]
\vspace{12pt}
\labellist
\small\hair 2pt
\pinlabel {\large $\sigma_k$} [Bl] at 66 164
\pinlabel {\large $\sigma_k^{-1}$} [Bl] at 241 164
\pinlabel $t=2\pi$ [Br] at 0 156
\pinlabel $t=\frac{3\pi}{2}$ [Br] at 0 121
\pinlabel $t=\frac{\pi}{2}$ [Br] at 0 54
\pinlabel $t=0$ [Br] at 0 18
\pinlabel $W^u(a_{k+1,k})$ [Bl] at 112 62
\pinlabel $W^u(a_{k,k+1})$ [Bl] at 112 130
\pinlabel $W^u(a_{k+1,k})$ [Bl] at 298 62
\pinlabel $W^u(a_{k,l+1})$ [Bl] at 298 130
\pinlabel $s$ [Bl] at 53 -5
\pinlabel $s$ [Bl] at 244 -5
\endlabellist
\centering
\includegraphics{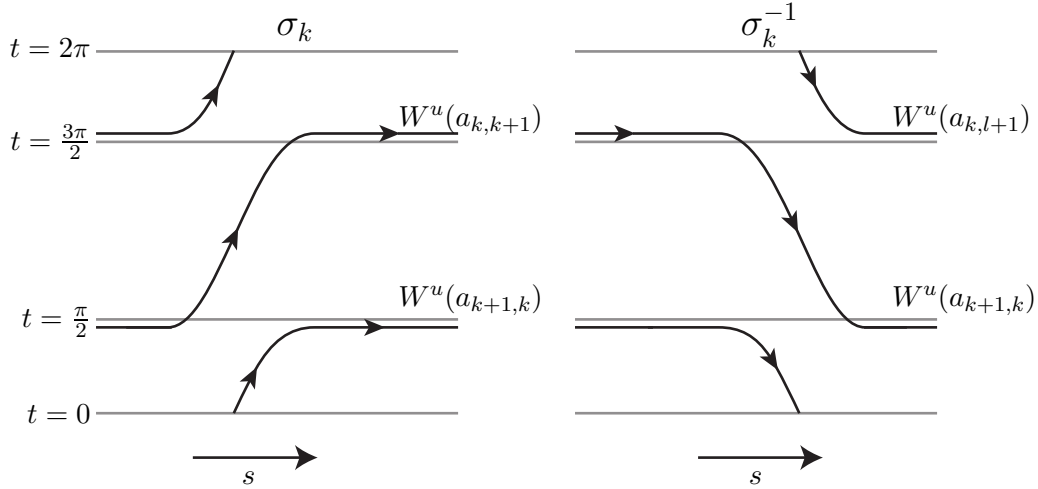}
\caption{Flow lines in positive and negative twist slices.}
\label{fig:twistslices}
\end{figure}
\item
For each $i \neq j$ and $l\in\{1,\ldots,m\}$, there is an interval neighborhood $J_{ij}^{2l} \subset \{s=s_{2l}\}$ of $W^\mathrm{u}(a_{ij})\cap\{s=s_{2l}\}$ such that for fixed $l$ the intervals $J_{ij}^{2l}$ are disjoint, and if we consider the set of all  negative gradient flow lines of $F_{ij}$ in the $[s_{2l-1},s_{2l}]$ slice that start on $J^{2l}_{ij}$ for all $i,j$, then any pair of distinct flow lines from this set are transverse. See Figure~\ref{fig:intervals}.
\begin{figure}[htb]
\labellist
\small\hair 2pt
\pinlabel {$J^{2l-2}_{ij}$'s} [Br] at -2 114
\pinlabel {$J^{2l-2}_{ij}$'s} [Br] at -2 19
\pinlabel $s_{2l-2}$ [Bl] at -1 -6
\pinlabel $s_{2l-1}$ [Br] at 163 -6
\pinlabel $s_{2l-1}$ [Bl] at 187 -6
\pinlabel $s_{2l}$ [Bl] at 324 -6
\pinlabel {$J^{2l}_{ij}$'s} [Bl] at 334 114
\pinlabel {$J^{2l}_{ij}$'s} [Bl] at  334 19
\pinlabel {$J^{2l-1}_{3\pi/2}$} [Bl] at 141 114
\pinlabel {$J^{2l-1}_{\pi/2}$} [Bl] at 141 19
\endlabellist
\centering
\includegraphics{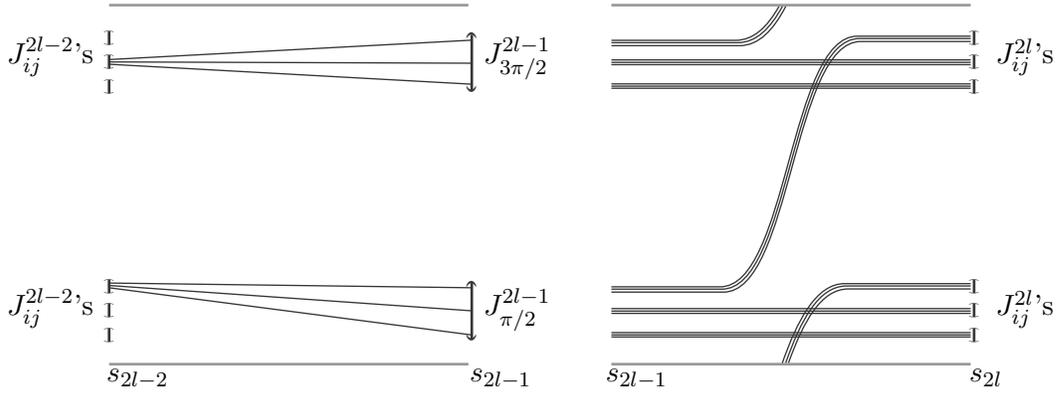}
\caption{The intervals $J_{ij}^{2l}$, $J_{\pi/2}^{2l-1}$, and $J_{3\pi/2}^{2l-1}$, and some representative gradient flow lines.}
\label{fig:intervals}
\end{figure}
\item
On each $\{s=s_{2l-1}\}$ there are two intervals $J^{2l-1}_{\pi/2}$ and $J^{2l-1}_{3\pi/2}$ such that
\[
J^{2l-1}_{\pi/2} \supset \cup_{i>j} (W^\mathrm{u}(a_{ij})\cap \{s=s_{2l-1}\})~~~~~\text{and}~~~~~
J^{2l-1}_{3\pi/2} \supset \cup_{i<j} (W^\mathrm{u}(a_{ij})\cap \{s=s_{2l-1}\}).
\]
Moreover, any negative gradient flow line of $F_{ij}$ ($i>j$) that intersects $J^{2l-1}_{\pi/2}$ also intersects $J^{2l-2}_{ij}$, and similarly for $J^{2l-1}_{3\pi/2}$ for $i<j$.
\end{itemize}

Assume that $\Lambda_K$ has $r$ components $\Lambda=\Lambda_{K;1}\cup\dots\cup\Lambda_{K;r}$.
We will keep track of homology classes of cycles in $H_1(\Lambda_K)$ by counting intersections with certain fixed cycles. On each component $\Lambda_{K;j}$, fix the curve $\mu_j'$ which is the preimage under the base projection map $\pi\colon\Lambda_{K;j}\to\Lambda$ of the curve
$t=\frac{\pi}{2}-\epsilon_1$ for $\epsilon_1$ positive and extremely small (for now, the line
$t=\frac{\pi}{2}$ will suffice; in Section~\ref{sec:combdiff}, we will
need $\mu_j'$ to lie below $t=\frac{\pi}{2}$ but above the unstable
manifolds of the $a_{ij}$ for all $i>j$). Also, fix the curve $\lambda_j'$ which is the preimage of the curve $s=2\pi-3\delta$ (a vertical curve between the $a_{ij}$'s and the $b_{ij}$'s) in the leading sheet of $\Lambda_{K;j}$, where we recall from the introduction that ``leading'' refers to the first of the $n$ sheets of $\Lambda_K$ that belongs to component $\Lambda_{K;j}$. Intersections of a cycle in $H_1(\Lambda_K)$ with $\mu_j'$ and $\lambda_j'$ then count the multiplicity of the $j$-th meridian and longitude class in the cycle.
See Figure \ref{fig:setup}.
\begin{figure}[htb]
\labellist
\small\hair 2pt
\pinlabel $2\pi$ [Br] at  -3 140
\pinlabel $\frac{3\pi}{2}$ [Br] at -3 110
\pinlabel $\pi$ [Br] at -3 70
\pinlabel $\frac{\pi}{2}$ [Br] at -3 35
\pinlabel $0$ [Br] at -3 -2
\pinlabel $0$ [Bl] at -2 -10
\pinlabel $2\pi$ [Bl] at 222 -10
\pinlabel $a_{ij}$ [Bl] at 155 122
\pinlabel $a_{ij}$ [Bl] at 155 19
\pinlabel $b_{ij}$ [Bl] at 195 122
\pinlabel $b_{ij}$ [Bl] at 195 19
\pinlabel {$\leftarrow$ braiding region} [Bl] at 40 94
\pinlabel $\lambda'$ [Bl] at 162 132
\pinlabel $\mu'$ [Bl] at 104 22
\pinlabel {$s$} [Bl] at 104 -10
\pinlabel $t$ [Br] at -3 90
\endlabellist
\centering
\includegraphics{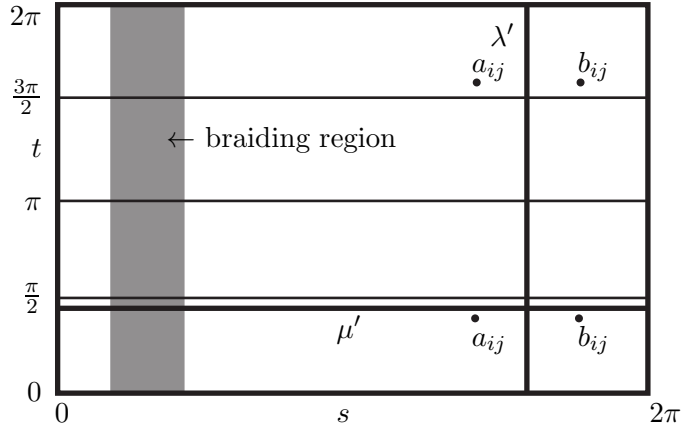}
\caption{Schematic picture of $\Lambda_K$, as projected to the torus $\Lambda \approx T^2$.}
\label{fig:setup}
\end{figure}

We choose a base point in each component over $(s,t)=\left(\frac{\pi}{2},\frac{9\pi}{8}\right)$ and endpoint paths for each Reeb chord endpoint which are disjoint from $\lambda_j'$ and $\mu_j'$. (Since the complement of $\lambda_j'\cup \mu_j'$ in $\Lambda_{K;j}$ is a disk such paths exist.)

% ***********************************************
\subsubsection{Orientation choices and sign rules for $\Lambda_K\subset J^{1}(\Lambda)$}\label{sssec:signrules}
Before we proceed with the computation of flow trees for $\Lambda_K \subset J^1(\Lambda)$, we discuss the general method we use to assign signs to flow trees. These signs come from a fairly elaborate orientation scheme which depend on certain initial choices,
some of which are of global nature which we call {\em basic
  orientation choices}, and others which are local, more specifically,
the choice of orientations of determinant lines of a capping operator
associated to each Reeb chord.  Here we will simply state a
combinatorial rule that comes from one particular set of choices. The
derivation of the combinatorial rule and the effect of orientation
choices is discussed in detail in Section~\ref{sec:orientations}. (For
later computations, we will also need signs for multiscale flow trees;
this is discussed in Section~\ref{ssec:signs}.)

We will discuss signs of rigid flow trees with one positive puncture or partial flow trees of dimension 1 with a special positive puncture of $\Lambda_K\subset J^{1}(\Lambda)$. Cutting a rigid flow tree close to its positive puncture we obtain a partial flow tree of dimension $1$ with special positive puncture so it suffices to consider this case.

We first discuss how orientation choices for the ``capping operators'' corresponding to the Reeb chords $a_{ij}$ and $b_{ij}$ of $\Lambda_K\subset J^{1}(\Lambda)$ are encoded geometrically:
\begin{itemize}
\item Consider a Reeb chord $a_{ij}$, which is of type $\mathbf{S}_0$ (with notation from Remark~\ref{rmk:Reebchordtypes}). Let $W^{\mathrm{u}}(a_{ij})$
  denote the unstable manifold of the positive local function
  difference defining $a_{ij}$.
Fix a vector $v^{\krn}(a_{ij})$ perpendicular to $W^{\mathrm{u}}(a_{ij})$, see Figure \ref{fig:orientationdata}.   (This choice corresponds to the choice of an orientation of the capping operator of $a_{ij}$.)
\item Consider a Reeb chord $b_{ij}$, which is of type $\mathbf{S}_1$ and note that $b_{ij}\in W^{\mathrm{u}}(a_{ij})$. Fix vectors $v^\krn(b_{ij})$ parallel to $W^{\mathrm{u}}(a_{ij})$ and $v^\cokrn(b_{ij})$ perpendicular to $W^{\mathrm{u}}(a_{ij})$, see Figure \ref{fig:orientationdata}. (This choice corresponds to the choice of an orientation of the capping operator of $b_{ij}$.)
\item If $t$ is a trivalent vertex of a partial flow tree $\Gamma$ then we let $v^\con(t)$ be a vector tangent to the incoming edge at $t$ and pointing into this edge, see Figure \ref{fig:vcon}.
(This is a reflection of a chosen orientation on the space of conformal structures on the disk with boundary punctures.)
\end{itemize}

\begin{figure}[htb]
\labellist
\small\hair 2pt
\pinlabel $W^{\mathrm{u}}(a_{ij})$ at 0 13
\pinlabel $a_{ij}$ at 51 -5
\pinlabel $b_{ij}$ at 167 -5
\pinlabel $v^\krn(a_{ij})$ at 75 37
\pinlabel $v^\cokrn(b_{ij})$ at 191 37
\pinlabel $v^\krn(b_{ij})$ at 204 -5
\endlabellist
\centering
\includegraphics{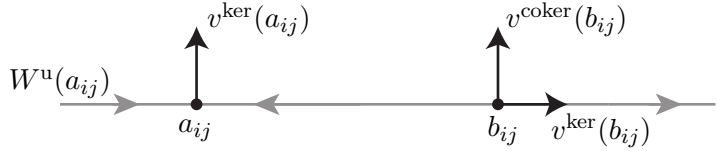}
\caption{Orientation data at $b_{ij}$ and $a_{ij}$. (The unstable manifold $W^{\mathrm{u}}(a_{ij})$ is in grey.)}
\label{fig:orientationdata}
\end{figure}

\begin{figure}[htb]
\labellist
\small\hair 2pt
\pinlabel $t$ at 120 50
\pinlabel $v^\con(t)$ at 98 72
\endlabellist
\centering
\includegraphics{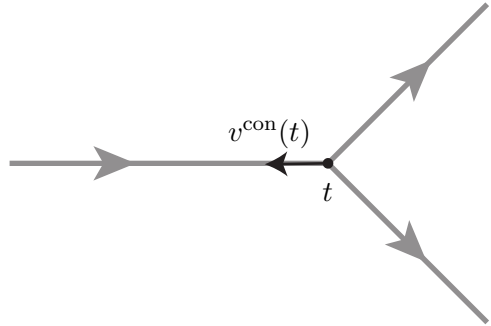}
\caption{Orientation data at a trivalent $Y_0$-vertex. (Gradient flow tree is in grey.)}
\label{fig:vcon}
\end{figure}

We next define two functions which are central to our definition of signs of rigid multiscale flow trees.
Let $\la\,,\ra$ denote a Riemannian metric on $S$ (which we will take to be the flat metric on the torus) and let
\[
\sign\colon\R-\{0\}\to\R
\]
be the function which maps negative numbers to $-1$ and positive
numbers to $1$. First consider a flow tree $\Gamma$ of $\Lambda\subset J^{1}(S)$. Let $b_{ij}$ be its positive puncture,  $v^\flow(\Gamma)$ denote the vector field of the flow orientation of $\Gamma$ and define
%\begin{equation}\label{eq:possign}
\[
\sigma_{\rm pos}(\Gamma)=\sign\left(\la v^\flow(\Gamma),v^\krn(b_{ij})\ra\right).
\]
%\end{equation}

Next consider a partial flow tree $\Gamma$ of $\Lambda\subset J^{1}(S)$ with positive special puncture $p$, trivalent vertices $t_1,\dots,t_{k-1}$, and negative punctures $q_1,\dots,q_k$ and let $n$ be a normal vector of $\Gamma$ at $p$. Denote the result of vector splitting of $n$ along $\Gamma$ by
\[
\bigl((w_1(t_1),w_2(t_1)),\dots,(w_1(t_{k-1}),w_2(t_{k-1}));w(q_1),\dots,w(q_k)\bigr)
\]
and define
\begin{alignat}{2}\label{eq:trivialsign}
\sigma_{n,\Gamma}(t)&=\sign\Bigl(\bigl\la w_2(t)-w_1(t),v^\con(t)\bigr\ra\Bigr),\quad
& t\in\{t_1,\dots,t_{k-1}\}\\%\label{eq:negpunctsign}
\sigma_{n,\Gamma}(q)&=\sign\Bigl(\bigl\la w(q_j),v^\krn(q_j)\bigr\ra\Bigr),\quad
& q\in\{q_1,\dots,q_{k}\}.
\end{alignat}
Finally we define
%\begin{equation}\label{eq:dfnsigntree}
\[
\sigma(n,\Gamma)=\Pi_{j=1}^{k}\sigma_{n,\Gamma}(q_j)\,\,\Pi_{j=1}^{k-1}\sigma_{n,\Gamma}(t_j).
\]
%\end{equation}

We can now assign orientations to trees according to the following theorem.
\begin{thm}\label{thm:combsign-fortrees}
There exists a choice of basic orientations and of orientations of capping operators for all Reeb chords of type $\mathbf{L}$ such that for a rigid flow tree $\Gamma$ of $\Lambda$ in $J^1(S)$ with positive puncture at $b_{ij}$ of type $\mathbf{S}_1$, the sign of $\Gamma$ is
\[
\sigma_{\rm pos}(\Gamma)\,\sigma(v^\cokrn(b_{ij}),\Gamma).
\]
\end{thm}

\begin{pf}
Theorem~\ref{thm:combsign-fortrees} is a special case of
Theorem~\ref{thm:combsign}, which is proved in Section~\ref{ssec:multisigns}.
\end{pf}

% **************************************************
\subsubsection{Partial flow trees in twist slices}\label{sssec:partial}
We now begin our enumeration of flow trees for $\Lambda_K \subset J^1(\Lambda)$. Consider a braid whose closure is $K$, and assume that the braid is in the form given in Section~\ref{sssec:endpointpath}. Then flow trees for $\Lambda_K$ decompose nicely into pieces in each twist region $[s_{2l-2},s_{2l}]$ which we call \textit{partial flow trees}.
Using the notation from Section~\ref{sssec:endpointpath}, we focus on one of these twist regions, an interval $[s_{2l-2}, s_{2l}]\times [0,2\pi]$ (for fixed $l\in\{1,\ldots,m\}$) containing the $l$-th twist $\sigma_k^{\pm 1}$ from the braid.
We first define a special type of partial flow tree in such a slice which we call a \emph{slice tree}.

Let $n\ge i>j\ge 1$. Fix symbols $\bar a_{ij}$ and $\bar a_{ji}$ for negative punctures and $\bar b_{ij}$ and $\bar b_{ji}$ for positive punctures; at the moment, these are just symbols and do not correspond to actual punctures or special punctures. Now for each $i>j$ choose one point $\bar a'_{ij}\in \{s_{2l}\}\times J^{2l}_{ij}$ and think of it as a special puncture connecting the sheet $S_j$ to $S_i$, where the sheets are numbered by the order of the braid strands at $s=s_{2l}$. Similarly choose $\bar a'_{ji}\in \{s_{2l}\}\times J^{2l}_{ji}$ and think of it as a special puncture connecting the sheet $S_i$ to $S_j$. Once this choice is made we will frequently conflate the variables $\bar a_{ij}$ and $\bar a'_{ij}$. We will only use the primes when we need to refer to specific special punctures.

Given these choices any flow tree lying entirely in the $[s_{2l-2},s_{2l}]$-slice that has a positive special puncture connecting sheets $S_j$ to $S_i$ at any point in $J^{2l-2}_{ij}$ and negative punctures at the $\bar a'_{pq}$ will be called a \textit{slice tree with positive special puncture at $\bar b_{ij}$} (think of $\bar b_{ij}$ as the intersection of the flow tree with $J_{ij}^{2l-2}$). We let $\T^{\pm}(\bar{b}_{ij})$ denote the set of slice trees in the $[s_0,s_1]$-slice with positive special puncture at $\bar b_{ij}$.

\begin{rmk}
Note that the formal dimension of slice trees $\T^{\pm}(\bar{b}_{ij})$ is $0$ (recall we have fixed the locations of all the $\bar{a}_{pq}$). Note also that the set $\T^{\pm}(\bar{b}_{ij})$ is independent of the specific choice of the $\bar a'_{pq}$, {\em i.e.}, a different choice of the $\bar a'_{pq}$ induces a one-to-one correspondence of $\T^{\pm}(\bar{b}_{ij})$. This follows from the choice of the intervals $J^{2l}_{ij}$, $J^{2l-1}_{\pi/2}$, $J^{2l-1}_{3\pi/2}$ in Section~\ref{sssec:endpointpath}.
\end{rmk}

In order to keep track of orientation signs of flow trees, we will decorate the special punctures with arrows which should be thought of as normal vectors to flow lines through the special punctures. We write $\bar{a}_{kl}^{\,\uparrow}$ and $\bar{b}_{kl}^{\,\uparrow}$ for the relevant special puncture decorated with a vector $\nu$ normal to the slice tree at that point with $\la \nu, \pa_t\ra>0$; similarly, we write $\bar{a}_{kl}^{\,\downarrow}$ and $\bar{b}_{kl}^{\,\downarrow}$ for the same special puncture decorated with a normal vector $\nu'$ such that $\la \nu',\pa_t\ra<0$.

Furthermore, as the twist is part of a larger braid, each strand belongs to a link component and we need to keep track of the corresponding homology coefficients. We write $\mu_A$ (respectively  $\mu_B$) for the homology variable of the meridian of the component of $\Lambda_K$ associated to sheet $A$
(respectively $B$), which is the sheet numbered by $k$ (respectively $k+1$) at $s=s_{2l}$ and by $k+1$ (respectively $k$) at $s=s_{2l-2}.$

Let $\bar{\BB}$ denote the set of decorated special chords at $s=s_{2l-2}$:
\[
\bar{\BB}=\left\{\bar{b}_{ij}^{\,\uparrow}, \bar{b}_{ij}^{\,\downarrow} \right\}_{1\leq i,j\leq n,~i\ne j},
\]
and let $\bar{\A}$ denote the
$\Z$-algebra generated by $\mu_A^{\pm 1}$, $\mu_B^{\pm 1}$, and decorated special chords at $s=s_{2l}$:
\[
\bar{\A}=\Z\left\la \mu_A^{\pm 1},\mu_B^{\pm 1},\bar{a}_{ij}^{\,\uparrow}, \bar{a}_{\,ij}^{\downarrow}\right\ra_{1\leq i,j\leq n,~i\ne j}.
\]

Given a tree $\Gamma\in \T^{\pm}(\bar{b}_{ij})$ and a normal $\nu$ for
$\bar{b}_{ij}$, the $1$-jet lift of $\Gamma$ determines a word in
$\bar{\A}$, in a manner that we now describe. Orient the $1$-jet lift
of $\Gamma$ by the flow orientation (see Section~\ref{ssec:gft}).
This orientation induces an ordering $a_{i_1j_1},\ldots,a_{i_qj_q}$ of
the negative punctures of $\Gamma$ so that the $1$-jet lift of
$\Gamma$ consists of a union of oriented paths
$\gamma_0,\ldots,\gamma_q \subset \Lambda_K$, with the beginning point
of $\gamma_0$ and the end point of $\gamma_q$ equal to the ends of
$\bar{b}_{ij}$ (the points in $\Lambda_K$ lying over $\bar{b}_{ij}$ on
sheets $i$ and $j$), and with the beginning point of $\gamma_r$ and
the end point of $\gamma_{r-1}$ equal to the ends of
$\bar{a}_{i_rj_r}$ for $r=1,\ldots,q$.

For $0\leq r\leq q$, define
$n_A^r,n_B^r\in\Z$ to be the intersection numbers of $\gamma_r$
with the preimage of $t=\frac{\pi}{2}$ in sheets $A,B$,
respectively. (In the notation of Section~\ref{sssec:endpointpath},
these numbers count intersections with $\mu_A'$ and
$\mu_B'$.) Also, use vector splitting (see
Section~\ref{ssec:signs}) from the normal vector $\nu$ at
$\bar{b}_{ij}$ to obtain normal vectors $\nu_r$ at each $\bar{a}_{i_rj_r}$.
Finally, define the word
\[
w(\Gamma,\nu) := (\mu_A^{n_A^0}\mu_B^{n_B^0}) \bar{a}_{i_1j_1}^{\nu_1}
(\mu_A^{n_A^1}\mu_B^{n_B^1}) \bar{a}_{i_2j_2}^{\nu_2} \cdots
\bar{a}_{i_qj_q}^{\nu_q} (\mu_A^{n_A^q}\mu_B^{n_B^q}).
\]

We can now define maps $\eta_{\sigma_k^{\pm 1}}\colon\bar{\BB}\to
\bar{\A}$ as follows, with $\nu\in\{\uparrow,\downarrow\}$:
%\begin{equation}\label{eq:defnphi}
\[
\eta_{\sigma_k^{\pm 1}}(\bar{b}_{ij}^{\,\nu})=\sum_{\Gamma\in \T^{\pm}(\bar{b}_{ij})}\tau(\Gamma)q(\Gamma,\nu),
\]
%\end{equation}
where $\tau(\Gamma)=\Pi_t\sigma_{n,\Gamma}(t)\in\{1,-1\}$ with the product running over all trivalent punctures of $\Gamma$, see Equation~\eqref{eq:trivialsign}.

If $b\in \bar{\BB}$ then define $b^{\dagger}$ as the same chord but with the decorating normal reversed. Similarly, for a monomial $q\in \bar{\A}$ define $q^{\dagger}$ as the same monomial of chords but with all decorating normals reversed. Write
\[
\eta_{\sigma_k^{\pm 1}}(b) = \eta_{\sigma_k^{\pm 1}}^{\rm odd}(b)+\eta_{\sigma_k^{\pm 1}}^{\rm even}(b),
\]
where the two terms on the right hand side are summands containing all monomials with odd and even number of variables respectively.
\begin{lma}\label{lma:switchnormal}
For any $b\in \bar{\BB}$, the map $\eta_{\sigma_k^{\pm 1}}$ satisfies
\[
\eta_{\sigma_k^{\pm 1}}\left(b^{\dagger}\right)= \left(\eta_{\sigma_k^{\pm 1}}^{\rm odd}(b)\right)^{\dagger}-\left(\eta_{\sigma_k^{\pm 1}}^{\rm even}(b)\right)^{\dagger}.
\]
\end{lma}

\begin{pf}
Let $n$ be a normal at the positive puncture. Then the vector splittings of $-n$ and $n$ along $\Gamma$ differs by an over all sign.  Since the number of trivalent vertices in $\Gamma$ is one less than the number of negative punctures of $\Gamma$, the lemma follows.
\end{pf}

We next turn to the actual calculation of the maps $\eta_{\sigma_k}$
and $\eta_{\sigma_k^{-1}}$. By Lemma \ref{lma:switchnormal}, it is
sufficient to compute for $\bar b_{ij}^{\,\uparrow}$, $1\le i,j\le n$,
$i\neq j$.

\begin{lma}\label{lma:treetwist}
Let $\eta_{\sigma_k^{\pm1}}\colon \bar{\BB}\to\bar{\A}$ denote the maps associated to a twist representing the braid group generator $\sigma_k^{\pm 1}$ in an $[s_0,s_1]$-slice, as described above. Then
\[
\begin{aligned}
\eta_{\sigma_k}(\bar{b}_{ij}^{\,\uparrow})&=
\bar{a}_{ij}^{\,\uparrow}  \quad & i,j\not= k, k+1\\
\eta_{\sigma_k}(\bar{b}_{k\,k+1}^{\,\uparrow})&=
\bar{a}_{k+1\,k}^{\,\uparrow} \quad & \\
\eta_{\sigma_k}(\bar{b}_{k+1\,k}^{\,\uparrow})&=
\mu_A\bar{a}_{k\,k+1}^{\,\uparrow}\mu_B^{-1}\quad &\\
\eta_{\sigma_k}(\bar{b}_{i\,k+1}^{\,\uparrow})&=
\bar{a}_{ik}^{\,\uparrow}\quad & i\not= k, k+1\\
\eta_{\sigma_k}(\bar{b}_{k+1\, i}^{\,\uparrow})&=
\bar{a}_{ki}^{\,\uparrow}\quad & i\not=k, k+1\\
\eta_{\sigma_k}(\bar{b}_{ik}^{\,\uparrow})&=
\bar{a}_{i\,k+1}^{\,\uparrow} + \bar{a}_{ik}^{\,\uparrow}\bar{a}_{k\, k+1}^{\,\uparrow}\quad& i < k\\
\eta_{\sigma_k}(\bar{b}_{ik}^{\,\uparrow})&=
\bar{a}_{i\,k+1}^{\,\uparrow} +\bar{a}_{ik}^{\,\uparrow} \mu_A\bar{a}_{k\, k+1}^{\,\uparrow}\mu_B^{-1}\quad& i > k+1\\
\eta_{\sigma_k}(\bar{b}_{ki}^{\,\uparrow})&=
\bar{a}_{k+1\,i}^{\uparrow} - \bar{a}_{k+1\,k}^{\,\uparrow}\bar{a}_{ki}^{\,\uparrow}\quad &i \ne k, k+1\\
\end{aligned}
\]
and
\[
\begin{aligned}
\eta_{\sigma_k^{-1}}(\bar{b}_{ij}^{\,\uparrow})&=
\bar{a}_{ij}^{\,\uparrow}  \quad & i,j\not= k, k+1\\
\eta_{\sigma_k^{-1}}(\bar{b}_{k\,k+1}^{\,\uparrow})&=
\mu_B^{-1}\bar{a}_{k+1\,k}^{\,\uparrow}\mu_A \quad & \\
\eta_{\sigma_k^{-1}}(\bar{b}_{k+1\,k}^{\,\uparrow})&=
\bar{a}_{k\,k+1}^{\,\uparrow}\quad &\\
\eta_{\sigma_k^{-1}}(\bar{b}_{ik}^{\,\uparrow})&=
\bar{a}_{i\,k+1}^{\,\uparrow}\quad & i\not= k, k+1\\
\eta_{\sigma_k^{-1}}(\bar{b}_{ki}^{\,\uparrow})&=
\bar{a}_{k+1\,i}^{\,\uparrow}\quad & i\not=k, k+1\\
\eta_{\sigma_k^{-1}}(\bar{b}_{i\,k+1}^{\,\uparrow})&=
\bar{a}_{ik}^{\,\uparrow} - \bar{a}_{i\,k+1}^{\,\uparrow}\mu_B^{-1}\bar{a}_{k+1\, k}^{\,\uparrow}\mu_A\quad& i<k\\
\eta_{\sigma_k^{-1}}(\bar{b}_{i\,k+1}^{\,\uparrow})&=
\bar{a}_{ik}^{\,\uparrow} - \bar{a}_{i\,k+1}^{\,\uparrow}\bar{a}_{k+1\, k}^{\,\uparrow}\quad& i>k+1\\
\eta_{\sigma_k^{-1}}(\bar{b}_{k+1\,i}^{\,\uparrow})&=
\bar{a}_{ki}^{\,\uparrow} + \bar{a}_{k\,k+1}^{\,\uparrow}\bar{a}_{k+1\,i}^{\,\uparrow}\quad &i \not = k, k+1.\\
\end{aligned}
\]

\end{lma}

\begin{pf}
We label the sheets of $\Lambda_K$ in the slice under consideration by
\[
S_1,\dots,S_{k-1},A,B,S_{k+2},\dots,S_n.
\]
Here the sheet $S_j$ corresponds to the $j^{\rm th}$ strand of the
standard trivial braid, and sheets $A$ and $B$ are defined as before.

Consider first the linear terms in the expressions for $\eta_{\sigma_k}$ and $\eta_{\sigma_k^{-1}}$. The trees that give these contributions are simply the (negative) gradient flow lines of positive function differences which end at $\bar{a}_{ij}$. They obviously exist, are unique, and transport normal vectors as claimed. This proves that the linear terms of the equations are correct except for homology coefficients. To see these, note that only the flow lines between sheets $A$ and $B$ intersect $t=\frac{\pi}{2}$. Further, a flow line between sheets $A$ and $B$ intersects this curve only if the twist is positive and it ends at $\bar a_{k\,k+1}$ or if the twist is negative and it ends at $\bar a_{k+1\,k}$. Finally, noting that the component in the upper sheet of the $1$-jet lift is oriented according to the flow orientation and that the component in the lower sheet is oriented opposite to the flow orientation, it follows that the coefficients are as stated.

We next study higher order terms arising from trees with trivalent vertices. Such a vertex arises as follows: a flow line determined by sheets $X$ and $Y$ splits into two flow lines determined by sheets $X$ and $Z$, and by $Z$ and $Y$. Furthermore, since we consider only trees with one positive puncture it is required that all flow lines correspond to positive function differences. In other words, the $z$-coordinate of the sheet $Z$ must lie between the $z$-coordinates of the sheets $X$ and $Y$ at the splitting point.

Given a slice tree $\Gamma$ suppose that $\bar a_{ij}$ is a negative special puncture of $\Gamma$ with $i,j\ne k, k+1$. Then the negative gradient flow $\gamma_{ij}$ of $F_{i j}$ ending at $\bar a'_{ij}$ in $[s_{2l-1}, s_{2l}]$ is disjoint from all the flow lines starting at all the other $\bar a'_{i'j'}$, except the flow line of $F_{k\, k+1}$ ending at $\bar a'_{k\, k+1}$ or  of $F_{k+1\, k}$ ending at $\bar a'_{k+1\, k}$. However we cannot have a $Y_0$ vertex here since $i,j\ne k, k+1$. Also, $\gamma_{ij}$ is disjoint from the other flow lines in the interval $[s_{2l-2}, s_{2l-1}]$. Thus if $\bar a'_{ij}$ is a negative special puncture of $\Gamma$ then $\Gamma$ has just one edge.

If $\bar a'_{i\, k+1}$ is a negative puncture of $\Gamma$ and $i<k$
then the negative gradient flow line of $F_{i\, k+1}$ starting at
$\bar a'_{i\, k+1}$ intersects only the flow line  of $F_{k\, k+1}$,
but since they have the same upper sheets these flow lines cannot
merge at a $Y_0$ vertex. Thus again $\Gamma$ just has one
edge. Similarly $\Gamma$ will only have one edge if $i>k+1$. The same
argument shows that a slice tree with negative puncture at $\bar
a'_{k+1\,i}$ must have only one edge.

It remains to consider slice trees with negative punctures among $\bar
a'_{k\,k+1}$, $\bar a'_{k+1\,k}$, and $\bar a'_{ik}$, $\bar a'_{ki}$
for $i\neq k,k+1$. Besides the flow lines
ending at $\bar a'_{k\,k+1}$ or $\bar a'_{k+1\,k}$, all such slice
trees must have some $\bar a'_{ik}$ or $\bar a'_{ki}$, $i\neq k,k+1$,
as a negative puncture.

\begin{figure}[htb]
\centering
\includegraphics[height=2in]{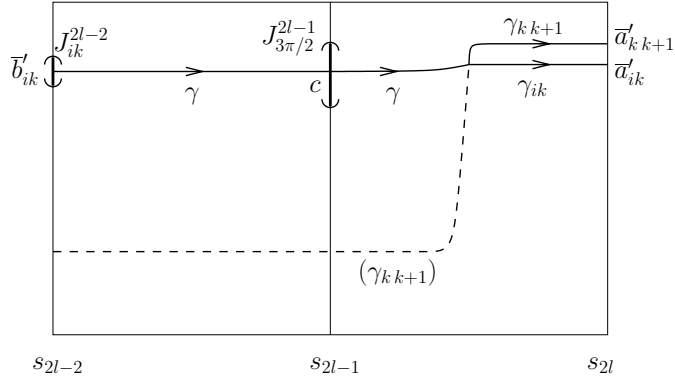}
\caption{A flow tree with one $Y_0$ vertex and negative punctures at
  $\bar a'_{i k}$ and $\bar a'_{k\,k+1}$. The case $i<k$ is shown here.}
\label{fig:braiddetail}
\end{figure}

If $\bar a'_{i k}$ is a negative puncture of $\Gamma$ then the
negative gradient flow line $\gamma_{ik}$ of $F_{i k}$ ending at $\bar
a'_{i k}$ intersects the flow line $\gamma_{k\, k+1}$ of $F_{k\, k+1}$
ending at $\bar a'_{k\, k+1}$. Thus we can have a $Y_0$ splitting of
the flow line from $F_{i\, k+1}$ into these flow lines; see
Figure~\ref{fig:braiddetail} for an illustration when $i>k+1$. Thus there
is a point $c$ in $J^{2l-1}_{\pi/2}$ or $J^{2l-1}_{3\pi/2}$ and a flow
line $\gamma$ of $F_{i\, k+1}$ starting at $c$ so that $\gamma\cup
\gamma_{k\, k+1}\cup \gamma_{ik}$ form a flow tree. Notice assuming
the twist interval is small enough it is clear that $\gamma$ does not
intersect any unstable manifolds of the $a_{ij}$ in $[s_{2l-1},
s_{2l}]$.  We can extend $\gamma$ through the interval $[s_{2l-2},
s_{2l-1}]$ and see that there is a unique point $\bar b'_{i\, k}$ in
$J^{2l-2}_{i\, k}$ that $\gamma$ runs through. (Recall our labeling
conventions from the beginning of the subsection: the sheets on the left side of the interval $[s_{2l-2}, s_{2l}]$ are labeled just as on
the right, but with the role of $k$ and $k+1$ switched.) Moreover
notice that by the lexicographical ordering on the unstable manifolds,
$\gamma$ will not intersect any of the unstable manifolds of the $\bar
a'_{i'j'}$ over the interval $[s_{2l-2}, s_{2l-1}]$.

\begin{figure}
\centering
\includegraphics[height=5in]{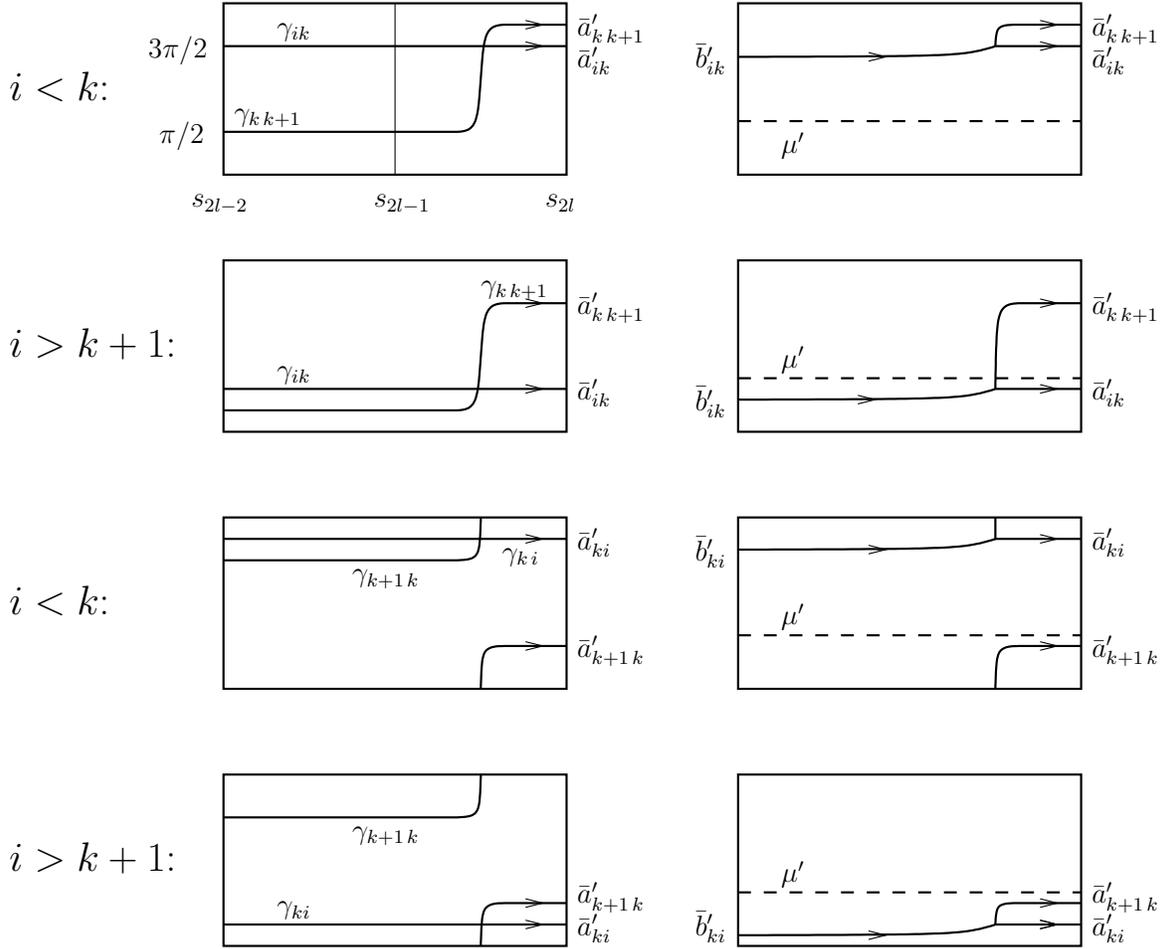}
\caption{The four slice trees that consist of more than one edge
  (right), resulting from the four configurations of intersecting flow
lines (left). The cycle $\mu'$ is also shown on the right, and
intersects only the second slice tree. Up to sign, we can read off the
words associated to these four slice trees, from top to bottom:
$\bar a'_{ik}\bar a'_{k\,k+1}$; $\bar a'_{ik} \mu_A \bar a'_{k\,k+1}
\mu_B^{-1}$; $\bar a'_{k+1\,k} \bar a'_{ki}$; $\bar a'_{k+1\,k} \bar a'_{ki}$.
}
\label{fig:braidtrees}
\end{figure}

Thus we see that if $\bar a'_{ik}$ is a negative puncture of $\Gamma$
then $\Gamma$ either has one edge (this is the trivial flow line
mentioned above) or exactly one $Y_0$ vertex. A similar argument says
the same for $\Gamma$ having a negative puncture at $\bar a'_{k
  i}$. See Figure~\ref{fig:braidtrees} for an illustration of all
slice trees with a $Y_0$ vertex and a negative puncture at $\bar
a'_{ik}$ or $\bar a'_{ki}$ for some $i\neq k,k+1$.

We conclude that all slice trees in $[s_{2l-2}, s_{2l}]$
either have one edge or, if the positive puncture is $\bar b_{ik}$ or
$\bar b_{ki}$, are of the type described above (and shown in
Figure~\ref{fig:braidtrees}) with three edges and one $Y_0$ vertex. This
shows that the quadratic terms of the equations are correct except for
homology coefficients and vector splittings. To see the homology
coefficients simply notice that the only flow trees that intersect
$t=\frac{\pi}{2}$ are the ones containing a flow line of $F_{k\, k+1}$
and having $i>k+1$. Since such a flow tree crosses $\mu'_A$ positively
between $\bar a'_{ik}$ and $\bar a'_{k\,k+1}$, and crosses $\mu'_B$
negatively after $\bar a'_{k\,k+1}$, it contributes
the word $\bar a'_{ik} \mu_A \bar a'_{k\,k+1}
\mu_B^{-1}$ to $\eta_{\sigma_k}(\bar b_{ik})$.

As for vector splittings for the $Y_0$ trees, note that the upward
normal at $\bar b_{ik}'$ or $\bar b_{ki}'$ is split into upward
normals at each of the negative punctures ($\bar a'_{ik}$ and $\bar
a'_{k\,k+1}$ or $\bar a'_{ki}$ and $a'_{k+1\,k}$). For a $Y_0$ tree
$\Gamma$, an easy application of the definitions from
Section~\ref{sssec:signrules} shows that $\tau(\Gamma)$ is $+1$ for
the top two trees in Figure~\ref{fig:braidtrees} and $-1$ for the
bottom two trees. (The difference between these pairs is the relative
placement at the trivalent vertex of the flows labeled 1 and 2 from
Figure~\ref{fig:vectorsplit}.) Thus the arrow decorations and signs
are as given in the statement of the lemma.

The case for $\sigma_k^{-1}$ can be handled similarly.
\end{pf}

% **************************************************
\subsubsection{Decomposing flow trees}
Our expression for the differential derives from the following geometric decomposition of flow trees. Consider a braid $B=\sigma_{k_1}^{\epsilon_1}\cdots\sigma_{k_m}^{\epsilon_m}$, $\epsilon_l=\pm 1$, on $n$ strands. Let $K$ denote the closure of $B$ and consider $\Lambda_K$. We represent $\Lambda_K$ as described above as a sequence of twists slices separated by trivial braid slices in the braiding region, one small slice containing all Reeb chords, and the remaining trivial braid slice which is the complement of these two slices.  Consider an $[s',s'']$-slice as just described in which there are no Reeb chords.
\begin{lma}\label{lma:dividetree}
Any rigid flow tree or $1$-dimensional partial flow tree of $\Lambda_K$ with one positive puncture intersects the $[s',s'']$-slice in a union of slice trees (for appropriately chosen $\bar a'_{ij}$).
\end{lma}
\begin{pf}
Suppose $\Gamma$ is a flow tree for $\Lambda_K$.  We see that $\Gamma$ will intersect the slice immediately preceding the $a_{ij}$-chords of $\Lambda_K$ in a union of slice trees, where the $\bar a'_{ij}=a_{ij}$, by definition. Notice that the slice trees in this slice define $\bar a'_{ij}$ for the next slice. Using the new $\bar a'_{ij}$ we see again that $\Gamma$ will intersect this next slice in a union of slice trees. Continuing by induction we see that $\Gamma$ will intersect each slice in a union of slice trees determined by appropriately chosen $\bar a'_{ij}$.
\end{pf}
\begin{rmk}
We note that any partial flow tree $\Gamma'$ of $\Lambda_K$, with only
the positive puncture $p$ special, can be completed in a unique way to
a flow tree $\Gamma$ by appending a flow line connecting $p$ to a Reeb chord $b_{ij}$.
\end{rmk}

% **************************************************
\subsubsection{Twist morphisms}
The following algebraic construction makes it possible to apply the result in Lemma~\ref{lma:treetwist} inductively. Let $K$ be a link with $r$ components and let $\Lambda_K=\Lambda_{K;1}\cup\dots\cup \Lambda_{K;r}$ be the subdivision of its conormal lift into components. Let $\mu_j,\lambda_j\in H_1(\Lambda_{K;j};\Z)$ be as described above.
As in the Introduction, consider the algebra $\A_n^0$ over $\Z$
generated by $\Z[H_1(\Lambda_K)]$ along with the Reeb chords $a_{ij}$, $1\le i,j\le n$, $i\ne j$. We will define morphisms $\phi_{\sigma_k^{\pm1}}\colon\A_n^0\to\A_n^0$ associated to braid group generators.

Consider a twist corresponding to $\sigma_k^{\pm 1}$. We use the notation $\mu_A$ and $\mu_B$ for homology variables exactly like in Lemma~\ref{lma:treetwist} and in order to connect to that result we make the following identifications:
\begin{equation}\label{eq:identifsource}
\begin{aligned}
\bar{b}_{ij}^{\,\uparrow} = +a_{ij}\quad &\text{if } i>j,\\
\bar{b}_{ij}^{\,\downarrow} =-a_{ij}\quad &\text{if } i>j,\\
\bar{b}_{ij}^{\,\uparrow} = -a_{ij}\quad &\text{if } i<j,\\
\bar{b}_{ij}^{\,\downarrow} =+a_{ij}\quad &\text{if } i<j
\end{aligned}
\end{equation}
in the source $\A_n^0$, and
\begin{equation}\label{eq:identiftarget}
\begin{aligned}
\bar{a}_{ij}^{\,\uparrow} = +a_{ij}\quad &\text{if } i>j,\\
\bar{a}_{ij}^{\,\downarrow} =-a_{ij}\quad &\text{if } i>j,\\
\bar{a}_{ij}^{\,\uparrow} = -a_{ij}\quad &\text{if } i<j,\\
\bar{a}_{ij}^{\,\downarrow} =+a_{ij}\quad &\text{if } i<j
\end{aligned}
\end{equation}
in the target $\A_n^0$. We define a homomorphism $\phi_{\sigma_k}\colon \A_n^0\to\A_n^0$ using this identification in combination with Lemmas~\ref{lma:treetwist} and~\ref{lma:switchnormal}. That is, we define it as follows on generators:
\[
\begin{aligned}
\phi_{\sigma_k}(a_{ij})&=
a_{ij}  \quad & i,j\not= k, k+1;\\
\phi_{\sigma_k}(a_{k\,k+1})&=
-a_{k+1\,k} \quad & \\
\phi_{\sigma_k}(a_{k+1\,k})&=
-\mu_Aa_{k\,k+1}\mu_B^{-1}\quad &\\
\phi_{\sigma_k}(a_{i\,k+1})&=
a_{ik}\quad & i\not= k, k+1;\\
\phi_{\sigma_k}(a_{k+1\, i})&=
a_{ki}\quad & i\not=k, k+1;\\
\phi_{\sigma_k}(a_{ik})&=
a_{i\,k+1} - {a}_{ik}a_{k\, k+1}\quad& i<k;\\
\phi_{\sigma_k}(a_{ik})&=
a_{i\,k+1} - {a}_{ik}\mu_Aa_{k\, k+1}\mu_B^{-1}\quad& i>k+1;\\
\phi_{\sigma_k}(a_{ki})&=
a_{k+1\,i} - a_{k+1\,k}a_{ki}\quad &i \not = k, k+1.\\
\end{aligned}
\]
Similarly, we define $\phi_{\sigma_k^{-1}}\colon \A_0\to\A_0$ as follows on generators:
\[
\begin{aligned}
\phi_{\sigma_k^{-1}}(a_{ij})&=
a_{ij}  \quad & i,j\not= k, k+1;\\
\phi_{\sigma_k^{-1}}(a_{k\,k+1})&=
-\mu_B^{-1}a_{k+1\,k}\mu_A \quad & \\
\phi_{\sigma_k^{-1}}(a_{k+1\,k})&=
-a_{k\,k+1}\quad &\\
\phi_{\sigma_k^{-1}}(a_{ik})&=
a_{i\,k+1}\quad & i\not= k, k+1;\\
\phi_{\sigma_k^{-1}}(a_{ki})&=
a_{k+1\,i}\quad & i\not=k, k+1;\\
\phi_{\sigma_k^{-1}}(a_{i\,k+1})&=
a_{ik} - a_{i\,k+1} \mu_B^{-1}a_{k+1\, k}\mu_A\quad& i<k;\\
\phi_{\sigma_k^{-1}}(a_{i\,k+1})&=
a_{ik} - a_{i\,k+1}a_{k+1\, k}\quad& i>k+1;\\
\phi_{\sigma_k^{-1}}(a_{k+1\,i})&=
a_{ki} - a_{k\,k+1}a_{k+1\,i}\quad &i \not = k, k+1.\\
\end{aligned}
\]
Note that $\phi_{\sigma_k^{-1}} \circ \phi_{\sigma_k} =
\phi_{\sigma_k} \circ \phi_{\sigma_k^{-1}} = id$, since $A,B$ switch
places between the $\sigma_k$ and $\sigma_k^{-1}$ twists.

\begin{rmk}
In any equation above where the sign differs from that of the corresponding equation of the formulas in Lemma~\ref{lma:treetwist} the following holds. If we use Equations~\eqref{eq:identifsource} and~\eqref{eq:identiftarget} to substitute decorated variables $\bar b_{ij}^{\,\uparrow}$ {\em etc.}, exactly one arrow of a decorated chord in the target monomial is oriented differently than all other arrows in the equation.
\end{rmk}

For the braid $B=\sigma_{k_1}^{\pm 1}\cdots\sigma_{k_m}^{\pm 1}$,
define
%\begin{equation}\label{eq:dfnphiB}
\[
\phi_B=\phi_{\sigma_{k_1}^{\pm 1}}\circ\dots\circ\phi_{\sigma_{k_m}^{\pm 1}}.
\]
%\end{equation}
Note that $\phi$ then gives a representation of the braid group into the group of automorphisms of $\A_n^0$.

\begin{rmk}
In order for $\phi_B$ to respect the order of composition we think of
the braid as written in the ``operator order,'' so that $B$ above
should be interpreted as: apply $\sigma_{k_m}^{\pm 1}$ first and
$\sigma_{k_1}^{\pm 1}$ last. Thinking of the braid in the opposite
order, $B\mapsto\phi_B$ would be an anti-homomorphism. See Figure
\ref{fig:dB}.
\end{rmk}

% **************************************************
\subsubsection{Flow trees for $\Lambda_K$ in $J^1(\Lambda)$}\label{sssec:countflowtree}
Since there are no cusps of $\Lambda_K$ in $J^1(\Lambda)$, we know for grading reasons that any rigid flow tree must have its positive puncture at some Reeb chord $b_{ij}$ and its negative punctures at Reeb chords $a_{ij}$. Consequently, Theorem~\ref{thm:gentree} implies that the differential of the Legendrian algebra can be computed as
\[
\pa(b_{ij})=\sum_{\Gamma\in\T(b_{ij})}\epsilon(\Gamma)q(\Gamma),
\]
where $\T(b_{ij})$ denotes the set of all flow trees with positive
puncture at $b_{ij}$ and where, if $\Gamma$ is such a tree,
$q(\Gamma)$ denotes the monomial of its negative punctures and
$\epsilon(\Gamma)$ its sign. To compute this differential we fix
orientation choices as follows, see Figure
\ref{fig:dB}:
\begin{equation}\label{eq:ker(b)}
v^\krn(b_{ij})=\pa_s,\quad\text{for all $1\le i,j\le n$, $i\ne j$}
\end{equation}
and
\begin{equation}\label{eq:coker(b)=ker(a)}
v^\cokrn(b_{ij})=v^\krn(a_{ij})=
\begin{cases}
\pa_t &\text{if }i>j,\\
-\pa_t &\text{if }i<j.
\end{cases}
\end{equation}

\begin{figure}[htb]
\labellist
\small\hair 2pt
\pinlabel $s$ [Bl] at  308 1
\pinlabel $\sigma_{i_2}$ [Bl] at 300 167
\pinlabel $\sigma_{i_1}$ [Bl] at 210 167
\pinlabel $\phi_{\sigma_{i_2}}\circ\phi_{\sigma_{i_1}}$ [Bl] at 230 -10
\pinlabel $v^\krn(a_{ij})$ [Bl] at -37 95
\pinlabel $b^\cokrn(b_{ij})$ [Bl] at 96 95
\pinlabel $b^\krn(b_{ij})$ [Bl] at 122 62
\pinlabel $a_{ij}$ [Bl] at -1 57
\pinlabel $b_{ij}$ [Bl] at 89 57
\pinlabel $\lambda'$ [Bl] at 33 113
\endlabellist
\centering
\includegraphics{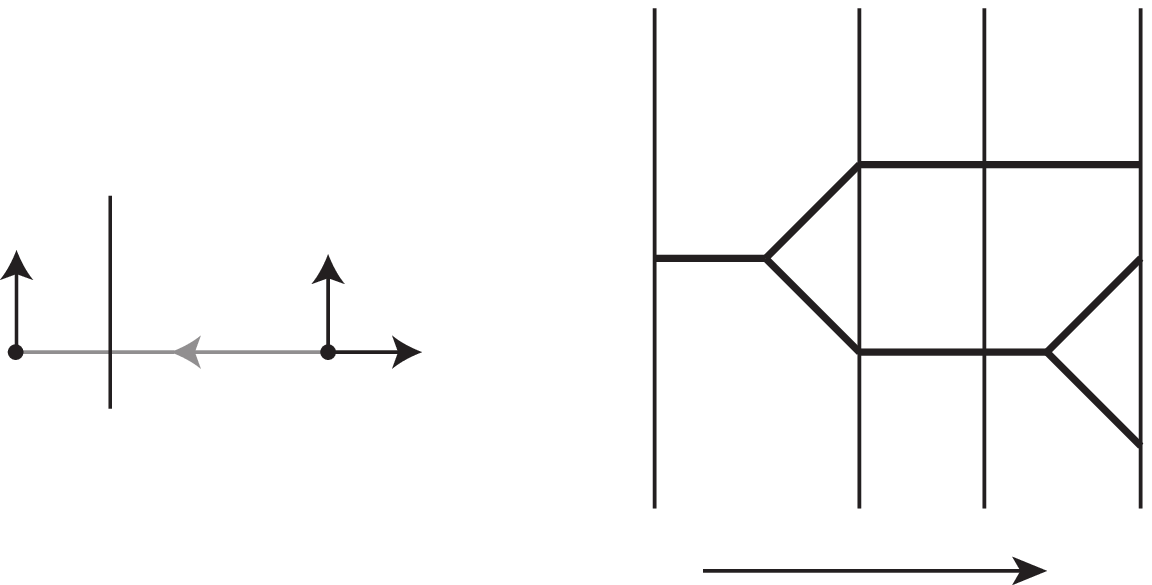}
\caption{Flow trees contributing to $\pa\Bb$.}
\label{fig:dB}
\end{figure}

The following lemma, where we use the matrix notation from Theorem \ref{thm:mainlink}, determines the differential discussed above.

\begin{lma}\label{lma:dB}
With orientation data as in Equations~\eqref{eq:ker(b)} and~\eqref{eq:coker(b)=ker(a)}, the following equation holds:
\[
\pa\Bb=-\Ll^{-1}\cdot \Aa\cdot\Ll+\phi_B(\Aa)=-\Ll^{-1}\cdot \Aa\cdot\Ll+\Phi_B^{L}\cdot\Aa\cdot\Phi_B^{R}.
\]
\end{lma}
\begin{pf}
We will prove the left equality, as the right equality follows from
Proposition~\ref{prop:PhiLPhiR}.
The first term in the right hand side of the left equality comes from
the short flow lines connecting $b_{ij}$ to $a_{ij}$, with flow
orientation given by $-\pa_s$. These flow lines clearly exist and are
unique. The sign $\epsilon$ of one of these flow lines $\Gamma$ is
given, according to Theorem~\ref{thm:combsign-fortrees}, by
\begin{align*}
\epsilon(\Gamma)&=\sign\left(\la v^\cokrn(b_{ij}),v^\flow(\Gamma)\ra\right)\,\sign\left(\la v^\cokrn(b_{ij}),v^\krn(a_{ij})\ra\right)\\
&=(-1)\cdot 1=-1.
\end{align*}
Consider their homology coefficients. Since no endpoint path
intersects $\mu_j'$ or $\lambda_j'$ it is sufficient to consider the
intersection between $\mu_j'$ and $\lambda_j'$ and the $1$-jet lift of
the trees. Clearly, all flow lines under consideration are disjoint
from $\mu_j'$, $j=1,\dots,r$. Furthermore, the $1$-jet lift in the
upper (respectively lower) sheet of flow line passes $\lambda'$ in the
negative (respectively positive) direction. Consequently $a_{ij}$
comes with a homology coefficients if and only if either the sheet
$S_i$ or the sheet $S_j$ is the leading sheet for that link
component. If $S_i$ (respectively $S_j$) is the leading sheet, then
the intersection with $\lambda_{\gamma(i)}'$ (respectively
$\lambda_{\gamma(j)}$) in that sheet contributes the coefficient
$\lambda_{\gamma(i)}^{-1}$ on the left of $a_{ij}$ (respectively
$\lambda_{\gamma(j)}$ on the right). This corresponds to multiplication by the matrix $\Ll^{-1}$ from the right and $\Ll$ from the left as claimed. See the leftmost picture in Figure \ref{fig:dB}.

The second term in the right hand side comes from the flow trees which
end as flow lines in $W^{\mathrm{u}}(a_{ij})$ flowing in the $+\pa_s$
direction. By Lemma~\ref{lma:dividetree} we find that the intersection
of such a flow tree with any twist slice is a twist tree and thus any
flow tree starting at $b_{ij}$ and ending at the $a_{ij}$ along
unstable manifolds oriented in the $\pa_s$ direction will contribute
a term from $\phi_B(a_{ij})$. Moreover any term in $\phi_B(a_{ij})$
gives a flow tree starting at $b_{ij}$. See the rightmost
picture in Figure \ref{fig:dB}.  Note also that no such tree passes
$\lambda'$. Lemma~\ref{lma:treetwist} in combination with the
composition formula for $\phi_B$ and the sign rule in Theorem
\ref{thm:combsign-fortrees} then shows that the second term is
$\phi_B(\Aa)$ (since all the nontrivial terms in the sign rule appear
at the $Y_0$ vertices, which were accounted for in the formula for
$\phi$). 
\end{pf}

% **************************************************
% **************************************************
\subsection{Negative gradient flows in twist slices}\label{ssec:twistslice}
Our goal in this subsection is to construct a braid $B =
\sigma_{k_1}^{\epsilon_1}\cdots \sigma_{k_m}^{\epsilon_m}$,
$\epsilon_l = \pm 1$, in a braiding slice
$[s_0^{\mathrm{br}},s_1^{\mathrm{br}}]$, in such a way that the
properties of unstable manifolds detailed in
Section~\ref{sssec:endpointpath} are satisfied.
We perform this construction twist by twist going from right (larger
$s$) to left (smaller $s$).
To facilitate the formulas we will
change coordinates from $(s,t)$ to $(u,t)$ where
$u=s_1^\mathrm{br}-s$. So the braid region happens over $u\in [0,U]$
where $U=s_1^{\mathrm{br}}-s_0^{\mathrm{br}}$. Throughout the
computation we will also change $u$ by translations, but the key is
that $u$ is always $-s$ up to translations. In $(u,t)$ coordinates we
will build up the braid region, twist by twist, going from left
(smaller $u$) to right (larger $u$); moreover in the $u$ coordinates we consider the reversed word $\sigma_{k_m}^{\epsilon_m}\cdots \sigma_{k_1}^{\epsilon_1}$ so that when we switch back to the $s$ coordinates it is $K$ that is represented.

More specifically we will break $[0,U]$ into $3m$ subintervals $I_1,
\ldots, I_{3m}$, ordered from left to right. For $l=1,\ldots,m$, the
union of the three intervals $I_{3l-2} \cup I_{3l-1} \cup I_{3l}$ is
associated to the braid generator
$\sigma_{k_{m-l+1}}^{\epsilon_{m-l+1}}$ and will be called the {\em
  braid interval} associated to $\sigma_{k_{m-l+1}}^{\epsilon_{m-l+1}}$.

The intervals $I_{3l-2}$, $l=1,\ldots,n$, will contain trivial
braids as discussed in Remark~\ref{rmk:standardtrivbraid} and are used
to adjust the braid to prepare for a twist between two strands of the
braid. These will be called {\em preparatory intervals}. The intervals
$I_{3l-1}$, $l=1,\ldots,n$, will be the intervals over which two
strands of the braid will actually twist. These will be called {\em
  twist intervals}. The intervals $I_{3l}$, $l=1, \ldots, n$,
will contain the trivial braid and will be used to adjust the braid so that
we can more easily count the flow trees. These will be called {\em
  concentration intervals}.

For comparison with Section~\ref{sssec:endpointpath}, we set $s_{2l-2}$
and $s_{2l-1}$ to be the endpoints (in reverse order) of the
concentration intervals
$I_{3(m-l+1)}$ for $l=1,\ldots,m$, and $s_{2m}$ to be the leftmost
(in $u$) endpoint of $I_1$. Then $s_0,\ldots,s_{2m}$ are arranged in
increasing order and $B$ is trivial in each $s\in [s_{2l-2},s_{2l-1}]$
slice (which corresponds to a concentration interval) and comprises the braid generator $\sigma_{k_l}^{\epsilon_l}$ in
each $s\in [s_{2l-1},s_{2l}]$ slice (the union of a preparatory and a
twist interval).

In the rest of this subsection we will describe the braid interval corresponding to the twist $\sigma_{k_l}^{\epsilon_l}$, but before focusing on this we make an observation and some conventions. First, for convenience,
we will think of the function $f_j$ describing the braid as maps  $[0,C]\to \R^{2}$, for some arbitrarily large constant $C$,  rather than $[0,U]\to D^{2}$. Since scaling the variable $u$ and multiplying all the functions by a small constant will not affect the discussion below we will be able to return to the appropriate braid set up once we have constructed our desired functions
$[0,C]\to \R^{2}$. (This step could be avoided by
choosing appropriate scaling constants throughout the argument, but as
these constants would depend on the entire braid it is considerably
simpler to proceed as we do.) Moreover we will start by considering
the trivial braid over $\R_{\geq 0}$, that is, our functions $f_j$ will be maps
$[0,\infty)\to\R^2$. We will then alter the $f_j$ over some interval
$[0,c_1]$ which we call $I_1$, then over $[c_1,c_2]$ which we call
$I_2$, and so on. Once we finish with the interval $I_{3m}$ we let
$[0,C]$ be the union of these intervals and our braid will be
described by the functions $f_j$ restricted to this interval.

Just as in the proof of Lemma~\ref{lma:braidReeb+triv} we will
describe our Legendrian $\Lambda_K$ in two steps. We begin, using the
notation from Lemma~\ref{lma:braidReeb+triv} and its proof, by considering the
standard trivial braid $\{f_1,\ldots, f_n\}$ given by
\[
f_i(u)=\psi_i(u)(h_t(i), i),
\]
where $\psi_i(u)=u+k_i$ for some
constant $k_i>0$. Throughout our construction, as we alter the
functions $\psi_i$, we will always assume that $i\psi_i$ has slope in
the interval $[i-1/2, i+1/2)$. We will also assume the $f_i$ are
exactly equal to $(u+k_i)(h_t(i),i)$ near the endpoints of each of the braid
intervals, though the $k_i$'s will depend on the particular braid
interval. Before we perform any braiding the unstable manifolds
$W^{\rm u}(a_{ij})$ of any Reeb chord $a_{ij}$ intersects $\{u\}\times
S^1$ near $t=\frac{\pi}{2}$ if $i>j$ or $t=\frac{3\pi}{2}$ if $i<j$
and are lexicographically ordered as in
Lemma~\ref{lma:braidReeb+triv}. As we inductively build up our braid
we will assume that these unstable manifolds have this same property
at the boundary of all the $I_l$. (We will see in the construction
that we can make them as near as we like.)

We will now focus on the braid interval associated to the braid
generator $\sigma_k$; the case of $\sigma_k^{-1}$ is completely
similar and is treated at the end of this subsection. We reset our
coordinate $u$ so that the braid interval for $\sigma_k$  is $[0,U]$.
Recall the braid interval consists of three subintervals: the
preparatory interval $I_p$, the twist interval $I_t$ and the
concentration interval $I_c$. In the interval $I_p$ we will alter the
slopes of the curves $i\phi_i(u)$. In particular, we alter the slopes
of the strands over the interval so that near the upper endpoint of
the interval we have that the difference of the slopes of the
$k^\mathrm{th}$ and $(k+1)^\mathrm{st}$ strands  is constant and very
small (that is each slope is near $k+1/2$) and the slope of the
$i^\mathrm{th}$ strand is $i$. Thus the difference in the slopes of
the functions $f_i$ and $f_k$ is greater than $1$  whereas the slope
of the difference function $f_{k+1}-f_k$ is arbitrarily
small. Allowing $u$ to increase sufficiently we can assume that
$|f_j(u)-f_{k}(u)|$ and $|f_j(u)-f_{k+1}(u)|$ are sufficiently large
compared to $|f_{k+1}(u)-f_{k}(u)|$ for each $j\ne k,k+1$ so that a
certain approximation described below is valid. This completes the
description of the braid in $I_p$. Notice that none of these
alterations affect the unstable manifold of the $a_{ij}$.

(Here, as below, it might be useful to consider the situation when
$h_t(y)$ is zero or constant in $y$, and then notice that perturbing
it slightly to another function $h_t(y)$ does not affect the qualitative behavior of the flow. Here this is clear since the unstable manifolds stay far away from the regions near $t=0$ and $\pi$ where $h_t(y)$ actually depends on $t$. Also keep in mind that $h_t(y)$ can be taken to be arbitrarily small.)

Now consider the twist region $I_t$ for a braid generator $\sigma_k$ which interchanges the $k^{\rm th}$ and $(k+1)^{\rm st}$ strands by a $\pi$ rotation of the line segment between them around its midpoint, in the positive direction as $s$ increases.
This means that the strands are also interchanged by a rotation in the positive direction as $u$ increases. For the standard trivial braid under consideration we let $f_k(u)$ and $f_{k+1}(u)$ rotate at a fast rate around the midpoint between $f_k(u)$ and $f_{k+1}(u)$ while still moving slowly away from each other.

For the functions $f_k$ and $f_{k+1}$ we begin by replacing $h_t(k)$
and $h_t(k+1)$ by $0$, though we leave the other $f_i$ as they were.
Specifically before we perform the twist we can assume (after possible
translations) that there are constants $c<c'$ such that
$f_k(u)=f_k^0(u) := (0,
(k+1/2 -\epsilon)u+c)$ and $f_{k+1}(u)=f_{k+1}^0(u) := (0,
(k+1/2+\epsilon)u+c')$ for
some very small $\epsilon$. Now to perform the twist over the interval
$[u', u'']$, choose an increasing surjective function
$\beta:[u',u'']\to [0, \pi]$ that is constant near the endpoints, and let
\begin{align*}
f_k(u) &= \frac{f_k^0(u)+f_{k+1}^0(u)}{2} + \left|
  \frac{f_{k+1}^0(u)-f_k^0(u)}{2} \right| (\sin \beta(u),-\cos\beta(u))
\\
f_{k+1}(u) &= \frac{f_k^0(u)+f_{k+1}^0(u)}{2} - \left|
  \frac{f_{k+1}^0(u)-f_k^0(u)}{2} \right| (\sin
\beta(u),-\cos\beta(u)).
\end{align*}
This describes the half twist between the two strands.

Note that if the distances $|f_j(u)-f_k(u)|$ and
$|f_j(u)-f_{k+1}(u)|$, $j\ne k,k+1$ are sufficiently large compared to
$|f_k(u)-f_{k+1}(u)|$ then the gradient flows of $\pm F_{jk}$ and $\pm
F_{j\,k+1}$ can be made arbitrarily close to the corresponding flows
for the standard trivial braid, {\em i.e.}, the same braid but with non-rotating $f_k(u)$ and $f_{k+1}(u)$. We assume that in the interval $I_p$ we arranged that the other points are sufficiently far away from $f_k(u)$ and $f_{k+1}(u)$ so that these approximations of the gradient flows of $\pm F_{jk}$ and $\pm F_{j\,k+1}$ are valid.

Consider next the gradient flow of $\pm F_{k+1\,k}$ which, in contrast to the flows just discussed, changes drastically. We take the rotation to be supported in a small subinterval $[u',u'']$ of $[0,U]$. We have
\begin{align*}
F_{k+1\,k}(u,t)&=(2\epsilon u + c'-c)
\bigl(\sin(\beta(u)), \cos(\beta(u))\bigr)\bull (\cos t,\sin t)\\
&=(2\epsilon u + c'-c)\sin(t+\beta(u)).
\end{align*}
The gradient is
\begin{align*}
\nabla F_{k+1\,k}=& (2\epsilon \sin(t+\beta(u))+(2\epsilon u + c'-c)\frac{d\beta}{du}\cos(t+\beta(u)))\,\pa_u\\
&+(2\epsilon u + c'-c)\cos(t+\beta(u))\,\pa_t.
\end{align*}
In order to understand relevant negative gradient flow lines of this
vector field we first note that $F_{k+1\,k}$ is positive for
$-\beta(u)<t<\pi-\beta(u)$ and negative in the complementary
region. Moreover, in the limit where $\epsilon=0$, the gradient flow
of $F_{k+1\, k}$ is perpendicular to the level sets $\{t+\beta(u)=a\}$
and flowing towards the curve of critical points
$\{t+\beta(u)=\pi/2\}$.

Now instead let $\epsilon>0$ be small, and
choose $a_m<\pi/2<a_M$ so that $a_M-a_m$ is small. Notice $\nabla
F_{k+1\, k}$ is still transverse to $\{t+\beta(u)=a_m\}$ and
$\{t+\beta(u)=a_M\}$ and pointing into the region $R$ bounded by these
level sets and containing $\{t+\beta(u)=\pi/2\}$. Along the  curve
$I=R\cap \{u=u'\}$, the $u$-component of $\nabla F_{k+1\, k}$ is
positive and the flow is into $R$, while along the curve $I'=R\cap
\{u=u''\}$, the $u$-component of $\nabla F_{k+1\, k}$ is also positive
but now flows out of $R$. Thus since there are no critical points of
$\nabla F_{k+1\, k}$ inside $R$, we see that any flow line starting on
$I$ will exit $R$ along $I'$. It follows that the flow lines in $R$
are approximately equal to the level sets $\{t+\beta(u)=a\}$ inside
$R$.

From this we see that the unstable manifold $W^{\rm u}(a_{k+1\, k})$ exits
the twist region near $3\pi/2$ and by choosing $\epsilon$ small in the
region $I_p$ (which can be done without affecting any essential
feature mentioned above) we can arrange that the $t$-coordinate of
$W^{\rm u}(a_{k+1\, k})\cap \{u=u''\}$ is as close to $3\pi/2$ as we like.
Furthermore, we can choose small intervals $J''$ and $J'$ in the slices
$\{u=u''\}$ and $\{u=u'\}$, containing the respective intersection
points of $W^{\rm u}(a_{k+1\,k})$ with these slices, such that any
flow line of $-\nabla F_{k+1\, k}$ that starts in $J''$ leaves through
$J'$ and is transverse to the flow lines of the other $F_{ij}$'s.
We can also then choose intervals $J''_{ij}$ and $J'_{ij}$ in the
slices $\{u=u''\}$ and $\{u=u'\}$, containing the respective
intersection points of $W^{\rm u}(a_{ij})$ with these slices, so that
all of these intervals are disjoint from each other and from $J''$ and
$J'$,
and so that all flow lines of $-\nabla F_{ij}$ that start in
$J''_{ij}$ are disjoint from each other and transversely intersect the
flow lines of $-\nabla F_{k+1\,k}$ that begin in $J''$.
\begin{figure}[htb]
\labellist
\small\hair 2pt
\pinlabel $t=2\pi$ at 162 146
\pinlabel $t=\frac{3\pi}{2}$ at 162 109
\pinlabel $t=\frac{\pi}{2}$ at 162 38
\pinlabel $t=0$ at 162 2
\endlabellist
\centering
\includegraphics{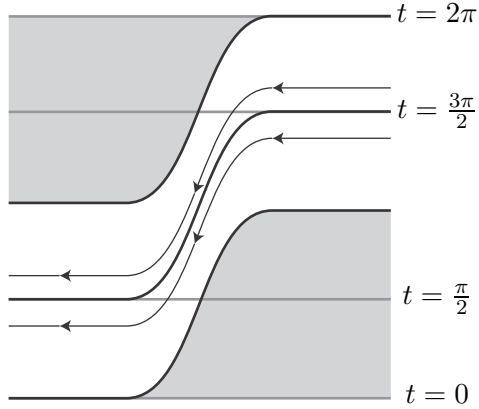}
\caption{The twist interval. The shaded regions are where $F_{k+1\,
    k}$ is negative. The dark lines are level curves of $t+\beta(u)$
  and the thinner lines are approximate flow lines of the negative
  gradient flow of $F_{k+1\, k}$. The left hand end is $u=u'$ and the
  right hand end is $u=u''$.}
\label{fig:twistinterval}
\end{figure}

The above discussion assumes that $h_t(k)=h_t(k+1)=0$, whereas in fact
$h_t(y)$ is a small function that is constant in $t$ outside of small
neighborhoods of $0$ and $\pi$ (see the proof of
Lemma~\ref{lma:braidReeb+triv}). In the argument above, the same
qualitative features of the flow hold if we change $h_t(k)$ and
$h_t(k+1)$ from $0$ to constants in $t$, since the level sets of
$F_{k+1\,k}$ only change slightly. If now we choose $h_t(y)$ to be the
function from the proof of Lemma~\ref{lma:braidReeb+triv},
as is needed to define the trivial braid, then the flow only changes
near $t=0$ and $t=\pi$.
But here by taking $\beta$ to have very large derivative over most of its support we see that the alteration to the gradient of $F_{k+1\, k}$ in the $t$-support of $h_t(y)$ can be thought of as arbitrarily small. Thus again we see that the qualitative features of the flow are unchanged.

Analogously, $W^{\rm u}(a_{k\,k+1})$ lies close to the curve $\{t=\frac{3\pi}{2}-\beta(u)\}$ and we can argue for the same intervals $J'$ and $J''$ here too. This completes the discussion of the twist interval $I_t$ for $\sigma_k$.

Now for the concentration interval $I_c$. The purpose of this interval is to concentrate gradient flow lines near the unstable manifolds. Specifically, in the first part of $I_c$ we alter our $\psi_i$ so they are the standard affine functions again. Notice that in the region where $F_{ij}$ is positive all the flow lines of $\nabla F_{ij}$ converge towards a constant $t$ line near $t=\pi/2$ if $i>j$ or $3\pi/2$ if $i<j$.
Thus choosing the interval $I_c=[a,b]$ large enough we can find intervals  $J_{\pi/2}$ and $J_{3\pi/2}$ on $\{u=a\}$ and for
 each $(i,j)$ intervals $J_{ij}$ on $\{u=b\}$ that are an arbitrarily small neighborhood of $W^u(a_{ij})\cap \{u=b\}$, so that all the unstable manifold intersect $J_{\pi/2}$ or $J_{3\pi/2}$ and any flow line of $\nabla F_{ij}$, $i>j$, that starts on $J_{\pi/2}$ intersects $J_{ij}$ and if $i<j$ then a flow line that starts in $J_{3\pi/2}$ intersects $J_{ij}$.

Similarly, we consider the inverse $\sigma_k^{-1}$ of the braid generator $\sigma_k$ which interchanges the $k^{\rm th}$ and $(k+1)^{\rm th}$ strands by a rotation of magnitude $\pi$ of the line segment between them around its midpoint in the negative direction as $s$ increases. The analysis of this situation is exactly as above, except the unstable manifolds veer up instead of down. See Figure \ref{fig:twistslices} above (where we have returned to $s$ coordinates).

% **************************************************
% **************************************************
% **************************************************
\section{Combinatorial Computation of the Differential}\label{sec:combdiff}
In this section we compute the differential in the Legendrian algebra of $\Lambda_K,$ where $K\subset\R^{3}$ is a link braided around the unknot.
Our computation heavily uses the results of Section~\ref{sec:diffthrflowtree}, where we determined all flow trees of $\Lambda_K$ viewed as a Legendrian submanifold of $J^1(\Lambda)$, where $\Lambda \approx T^2$ is the conormal lift of the unknot. These flow trees give the differential for a subalgebra of the Legendrian DGA of $\Lambda_K \subset J^1(S^2)$.

In Section~\ref{ssec:multiscale}, we introduce the notion of a
multiscale flow tree of $\Lambda_K \subset J^1(S^2)$, which is
essentially a collection of partial flow trees for $\Lambda_K \subset
J^1(\Lambda)$, glued to a flow tree of $\Lambda \subset J^1(S^2)$. By
a result (Theorem~\ref{thm:diskandgentree}) whose proof is deferred to
Section~\ref{sec:diskandgentree}, there is a one-to-one correspondence
between rigid multiscale flow trees and rigid holomorphic disks with
boundary on $\Lambda_K$ and one positive puncture. This allows us to
reduce the computation of the Legendrian DGA of $\Lambda_K \subset
J^1(S^2)$ to a combinatorial enumeration of all rigid multiscale flow
trees. The enumeration is performed in
Sections~\ref{ssec:multiscaleLambda} through~\ref{ssec:MFTcount} (for
signs associated to multiscale flow trees, we use some results whose
proofs are postponed to Section~\ref{sec:orientations})
and completes the proof of the main theorem of this paper, Theorem~\ref{thm:mainlink}.

% **************************************************
% **************************************************
\subsection{Multiscale flow trees}\label{ssec:multiscale}
We begin by discussing multiscale trees. We first recall the basic notation that will be used in this section.
Let $U\subset\R^{3}$ denote the unknot and write $\Lambda=\Lambda_U\subset J^{1}(S^{2})$. Let $K$ be a link given by the closure of an $n$-strand braid around $U$ such that each (local) strand is $C^{2}$-close to $U$. Then
$\Lambda_{K}\subset J^{1}(\Lambda)\subset J^{1}(S^{2})$, and we have the front projection $\Pi_F^{\Lambda}\colon J^{1}(\Lambda)\to \Lambda\times\R$ and the base projection $\pi^{\Lambda}\colon J^{1}(\Lambda)\to \Lambda$. The latter induces an $n$-fold cover $\Lambda_K \to \Lambda$; if $\gamma$ is a path in $\Lambda$, then there are $n$ distinct lifts $\widetilde\gamma$ of $\gamma$ with $\pi^{\Lambda}\circ\widetilde\gamma=\gamma$, which we call {\em neighborhood lifts} of $\gamma$.

If $\Gamma$ is a flow tree of $\Lambda\subset J^{1}(S^{2})$, let $\widetilde \Gamma$ denote its $1$-jet lift.

\begin{defn}\label{def:genflowtree}
A {\em multiscale flow tree} $\Gamma_{\Delta}$ on $\Lambda$ determined by $\Lambda_K$ is a flow tree $\Gamma$ of $\Lambda\subset J^{1}(S^{2})$ and a finite set of partial flow trees $\Delta=\{\Delta_j\}_{j=1}^{m}$ of $\Lambda_{K}\subset J^{1}(\Lambda)$ each with exactly one special puncture $x_j$, $j=1,\dots,m$, such that the following holds.
\begin{itemize}
\item $x_j\in\widetilde\Gamma$, $j=1,\dots,m$;
\item for each component of $\widetilde\Gamma-\{x_1,\dots,x_m\}$ a neighborhood lift to $\Lambda_K$ in $J^1(T^2)$ is specified;
\item the union of the $1$-jet lifts of the flow trees $\Delta_j$, $j=1,\dots,m$, and the specified neighborhood lifts, together with their flow orientation, gives a collection of consistently oriented curves $\widehat\Gamma\subset\Lambda_K$;
\item the curve $\widehat\Gamma$ is such that $\Pi^{\Lambda}(\widehat\Gamma)$ is closed, where $\Pi^{\Lambda}\colon J^{1}(\Lambda)\to T^{\ast}\Lambda$ is the Lagrangian projection.
\end{itemize}
$\Gamma$ is called the {\em big tree} and $\Delta$ the \emph{small tree} part of $\Gamma_\Delta$.
\end{defn}

\begin{rmk}
As we shall see, cf.~Section \ref{ssec:multiscaleLambda}, the partial trees $\Delta_j$ of $\Gamma_{\Delta}$ are of two types: either $\Delta_j$ is constant at one critical point $b$ of Morse index $2$, with both its positive special puncture $x_j$ and its negative puncture lying at $b$, or $\Delta_j$ is non-constant with positive special puncture at $x_j$.
\end{rmk}

The punctures of a multiscale flow tree $\Gamma_{\Delta}$ are the punctures of the flow trees $\Delta_j$ ({\em not} including the special punctures) and the punctures of the tree $\Gamma$. We say that the chord at a positive (respectively negative) puncture of $\Gamma$ connects the sheets determined by the neighborhood lift of the arc oriented toward (respectively away from) the puncture, to the sheet determined by the neighborhood lift of the arc oriented away from (respectively toward) the puncture. A puncture of a multiscale flow tree is positive (respectively negative) if the corresponding puncture of the flow tree $\Gamma$ or $\Delta_j$ is positive (respectively negative). A straightforward application of Stokes' theorem shows that every multiscale flow tree has at least one positive puncture.

Define the formal dimension of a multiscale flow tree $\Gamma_{\Delta}$ as
\[
\dim(\Gamma_{\Delta})=\dim(\Gamma)+\sum_{\Delta_j\in\Delta}(\dim(\Delta_j)-1),
\]
where the (formal) dimension of a (partial) flow tree is given in Equation~(\ref{eq:dimoftrees}), see also  \cite[Definition 3.4]{Ekholm07}.

We say that a multiscale flow tree $\Gamma_{\Delta}$ is {\em rigid} if $\dim(\Gamma_{\Delta})=0$ and if it is transversely cut out by its defining equation.

As discussed above in Lemma~\ref{lma:braidnicechords}, the set of Reeb chords $\ch(\Lambda_K)$ of $\Lambda_K\subset J^{1}(S^{2})$ can be written as follows:
\[
\ch(\Lambda_K)=\ch^{\Lambda}(\Lambda_K)\cup \bigcup_{1\le i,j\le n}\ch(\Lambda)_{ij},
\]
where $\ch^{\Lambda}(\Lambda_K)$ is the set of short Reeb chords (contained in 
$J^{1}(\Lambda) \subset J^1(S^2)$) and 
$\ch(\Lambda)_{ij}$ denotes the set of long Reeb chords of $\Lambda\subset J^{1}(S^{2})$ with endpoint on the $i$-th sheet of $\Lambda_K$ and beginning point on the $j$-th sheet.

\begin{thm}\label{thm:diskandgentree}
For any $\epsilon>0$, there exists an almost complex structure $J$ on $T^{\ast}S^{2}$, regular with respect to holomorphic disks with one positive puncture of dimension $\le 1$, such that: there is a one-to-one correspondence between rigid holomorphic disks with one positive puncture and boundary on $\Lambda_K,$ and rigid multiscale flow trees on $\Lambda$ determined by $\Lambda_K$ with one positive puncture; and the $1$-jet lift of a multiscale flow tree lies in an $\epsilon$-neighborhood of the boundary lift of the corresponding holomorphic disk.
\end{thm}

Theorem \ref{thm:diskandgentree} is proved in Section \ref{sec:diskandgentree}, and constitutes the basic tool in our calculation of the differential in $LA(\Lambda_K)$.

Above we described the front of $\Lambda$, the conormal lift of the unknot, and  its flow trees. With this established we will next classify possible multiscale flow trees determined by a braid closure (Section~\ref{ssec:multiscaleLambda}) and give an algorithm for the sign of such a tree (Section~\ref{ssec:signs}). In Section~\ref{ssec:MFTcount} we then turn to the actual calculation of the Legendrian DGA of $\Lambda_K$ by explicitly computing all multiscale flow trees for $\Lambda_K$ with signs.

% **************************************************
% **************************************************
\subsection{Classification of rigid multiscale flow trees of $\Lambda_K$}\label{ssec:multiscaleLambda}

To apply Theorem~\ref{thm:diskandgentree} to calculate the Legendrian DGA of $\Lambda_K$, we need to classify all possible rigid multiscale flow trees of $\Lambda_K \subset J^1(S^2)$. We give a rough characterization of such trees in this subsection, examine the signs associated to the trees in Section~\ref{ssec:signs}, and present the full classification in Section~\ref{ssec:MFTcount}.

Let $K$ be a link and assume that $\Lambda_{K}$ satisfies Lemma \ref{lma:braidnicechords}. Since the front of $\Lambda_{K}$ in $J^{0}(\Lambda)$ does not have any singularities and since all critical points of positive differences of local defining functions are either maxima or saddle points it follows from Section~\ref{ssec:gft} that for generic functions a rigid tree must have a
\begin{enumerate}
\item positive puncture at a Reeb chord $b_{ij}$ of type $\mathbf{S}_1$,
\item $k-1$ $Y_0$-vertices (trivalent vertices away from (non-existent) cusps of $\Lambda$), and
\item $k$ negative punctures at Reeb chords $a_{ij}$ of type $\mathbf{S}_0$.
\end{enumerate}
Likewise, a partial flow tree of dimension $1$ has the same vertices and punctures, except its positive puncture is a special puncture instead of a maximum. Also the constant partial flow tree with both special positive and negative puncture at $b_{ij}$ of type $\mathbf{S}_1$ will be of importance.

Since the $1$-jet lift of any flow tree of $\Lambda\subset J^{1}(S^{2})$ has codimension $1$ in $\Lambda$, it follows that any tree in the small tree part of any rigid multiscale flow tree is either a $1$-dimensional partial flow tree with positive puncture on the $1$-jet lift of the big tree on $\Lambda$ or it is a constant tree at some $b_{ij}$ of type $\mathbf{S}_1$. Furthermore, the flow tree on $\Lambda$ corresponding to the big tree part must either be rigid, or rigidified by a constant tree (a point condition at some $\pi(b_{ij})$). In the case where the big tree is rigidified by point conditions, its dimension must be equal to the number of point conditions {\em i.e.}, the number of constant trees in the multiscale flow tree. Combining this discussion with Lemma \ref{lma:treesofU} we find that (after small perturbation) there are the following types of rigid multiscale flow trees for $\Lambda$ determined by $\Lambda_{K}$, with notation as in Lemma \ref{lma:treesofU}:
\begin{itemize}
\item[${\rm MT}_{0}:$] A rigid flow tree $\Gamma$ of $\Lambda$ with constrained rigid partial flow trees of $\Lambda_K$ attached. Here the constraint says that the special positive puncture of each partial flow tree must lie on the $1$-jet lift of $\Gamma$.
\item[${\rm MT}_1:$] A constrained rigid flow tree $\Gamma^{\ast}$ with constrained rigid partial flow trees of $\Lambda_K$ attached. Here the constraint of $\Gamma^{\ast}$ is the requirement that the $1$-jet lift passes a point in $\Lambda$ in the fiber where a Reeb chord $b_{ij}$ lies, and the constraint of the partial flow trees is as above.
\item[${\rm MT}_\varnothing:$] A rigid flow tree of $\Lambda_{K}$.
\end{itemize}
\begin{rmk}
In our setting, we can rule out one other ostensible possibility for a rigid multiscale tree: those with big tree a constant rigid flow tree $\Gamma$ of $\Lambda$ and with small tree a constrained rigid flow tree of $\Lambda_K$. Here $\Gamma$ would correspond to a Reeb chord and the constraint would say that the special positive puncture of the partial flow tree must lie on the Reeb chord on the $1$-jet lift of $\Gamma$. If the location of the Reeb chord is generic with respect to the flow determined by $\Lambda_K$ then its endpoints does not lie on $W^{\mathrm{u}}(a_{ij})$ for any $a_{ij}$ or on $b_{ij}$ and such a configuration is rigid only if the flow line ends at minimum. As there are no positive local function differences of $\Lambda_K$ which are local minima no such trees correspond to disks with one positive puncture. (The rigid configurations wit flow line that ends at a negative local minimum correspond to disks with two positive punctures.)
\end{rmk}

% **************************************************
% **************************************************
\subsection{Signs of rigid multiscale flow trees}\label{ssec:signs}%\label{subsec:signsforMT}
In this subsection we describe a combinatorial algorithm for computing
the sign of a rigid multiscale flow tree, which determines its
contribution to the Legendrian algebra differential.
This is the analogue for multiscale flow trees of the discussion in Section~\ref{sssec:signrules}.
We will discuss the
derivation of the combinatorial rule as well as the effect of
orientation choices in detail in Section~\ref{sec:orientations}.

We will use the notation established in Section~\ref{ssec:vectorsplit} for vector splitting along flow trees and signs associated to rigid flow trees of $\Lambda_K$ as well as partial flow trees of $\Lambda_K$ of dimension $1$ with special positive puncture.

Before we can state the combinatorial rule for orienting rigid
multiscale flow trees, we need to discuss signs of rigid trees
determined by $\Lambda$; see Sections~\ref{ssec:flowtreesonU} and
\ref{ssec:flowtreesforLambda} for the notation for these rigid trees. Except for basic orientation choices the signs depend on orientations of determinants of capping operators. We call such choices {\em capping orientations}. Recall that there are two Reeb chords of $\Lambda$, $e$ and $c$, and that if $K\subset \R^{3}$ is a link represented as a closed braid on $n$ strands, then the long Reeb chords of $\Lambda_K$ are $e_{ij}$ and $c_{ij}$, $1\le i,j\le n$, where $c_{ij}$ lies very close to $c$ and $e_{ij}$ lies very close to $e$. In particular, capping orientations for $c$ and $e$ induce capping orientations for $c_{ij}$ and $e_{ij}$ respectively.
\begin{thm}\label{thm:signunknottree}
There is a basic orientation choice and choice of capping orientation for $c$ so that the sign $\epsilon(T)$, for $T$ a rigid flow tree of $\Lambda$, satisfies
\[
\epsilon(I_N)=\epsilon(Y_N)=\epsilon(I_S)=\epsilon(Y_S)=1,
\]
and
\[
\epsilon(E_1)=-\epsilon(E_2).
\]
\end{thm}

\begin{pf}
This is a consequence of Theorem \ref{l:unkotdiff} below.
\end{pf}

Furthermore, the choice of capping orientation of $e$ induces an
orientation of the $1$-dimensional moduli spaces of flow trees such
that the induced orientation at the broken disk $E_j\,\#\, T$ is
$\epsilon(E_j)\epsilon(T)$ for $T\in\{I_N,Y_N,I_S,Y_S\}$. If $\Gamma$ is a flow tree in such a $1$-dimensional moduli space we consider the orientation as a normal vector field $\nu$ along the $1$-jet lift of $\Gamma$.

In order to state the sign rule for multiscale flow trees, we first make some preliminary definitions.
\begin{itemize}
\item[${\rm MT}_{1}:$] Consider a multiscale rigid flow tree $\Theta$
  of type ${\rm MT}_{1}$ with $1$-dimensional big tree $\Gamma$ and a
  negative puncture at $b_{ij}$. Let $v^\flow(\Gamma)$ denote the
  vector field along the $1$-jet lift oriented in the positive
  direction ({\em i.e.}, the $1$-jet lift of each edge is equipped
  with the flow-orientation, defined in Section~\ref{ssec:gft}). The
  sign $\epsilon(\Theta)$ of the rigid tree constrained by $b_{ij}$ of type $\mathbf{S}_1$ is
  defined to be
\[
\epsilon(\Theta)=\sign\left(\la \nu,v^\krn(b_{ij})\ra\la v^\flow(\Gamma),v^\cokrn(b_{ij})\ra\right).
\]
If non-constant flow trees $\Gamma_1,\dots,\Gamma_k$ are attached to $\Theta$ then define $n_j$ to be the normal vector at the special puncture of $\Gamma_j$ with positive inner product with $v^\flow(\Gamma)$.
\item[${\rm MT}_{0}:$]
Consider a multiscale rigid flow tree $\Theta$ of type ${\rm MT}_{0}$ with big tree $\Gamma$.
Let $\epsilon(\Theta)$ equal the sign of $\Gamma$.
If $\Gamma$ has two punctures (positive at $e$, negative at $c$), then
we define $v^\flow(\Gamma)$ as the vector field along the boundary
pointing toward the positive puncture. If $\Gamma$ has only one
puncture (positive at $c$), let $v^\flow(\Gamma)$ point in the positive direction along the boundary. Then take $n_j$ exactly as in ${\rm MT}_1$.
\item[${\rm MT}_{\varnothing}:$]
Consider a flow tree $\Gamma$ of $\Lambda_K$ in $J^{1}(\Lambda)$. Let $n=v^\cokrn(b_{ij})$ be the normal vector of $\Gamma$ at its positive puncture.
\end{itemize}

\begin{thm}\label{thm:combsign}
There exists a choice of basic orientations and of orientations of capping operators for all long Reeb chords such that Theorem~\ref{thm:signunknottree} holds and such that the sign of a rigid multiscale flow tree is as follows.
\begin{itemize}
\item[${\rm MT}_{1}:$] Let $\Gamma$ be a multiscale rigid flow tree of type ${\rm MT}_{1}$ with constrained rigid flow tree $\Theta$ ({\em i.e.}, $\Theta$ has only one negative puncture at some $b_{ij}$) and attached flow trees $\Gamma_1,\dots,\Gamma_k$. Then the sign of $\Gamma$ is
\[
\epsilon(\Gamma)=\epsilon(\Theta)\,\Pi_{j=1}^{k}\sigma(n_j, \Gamma_j).
\]
\item[${\rm MT}_{0}:$] Let $\Gamma$ be a multiscale rigid flow tree of type ${\rm MT}_{0}$ with rigid flow tree $\Theta$ and attached flow trees $\Gamma_1,\dots,\Gamma_k$. Then the sign of $\Gamma$ is
\[
\epsilon(\Gamma)=\epsilon(\Theta)\,\Pi_{j=1}^{k}\sigma(n_j, \Gamma_j).
\]
\item[${\rm MT}_{\varnothing}:$]
Let $\Gamma$ be a flow tree of type ${\rm MT}_{\varnothing}$. Then the sign of $\Gamma$ is
\[
\sigma_{\rm pos}(\Gamma)\,\sigma(n,\Gamma).
\]
\end{itemize}
\end{thm}

\begin{pf}
Theorem \ref{thm:combsign} is proved in Section \ref{ssec:multisigns}.
\end{pf}

% **************************************************
% **************************************************
\subsection{Counting multiscale flow trees} \label{ssec:MFTcount}
In this section we complete the computation of the Legendrian algebra
differential of
$\Lambda_K\subset J^{1}(S^{2})$, and thereby obtain a proof of
Theorem~\ref{thm:mainlink}, by counting all multiscale
flow
trees determined by $\Lambda_K$ and $\Lambda$.
In Section~\ref{sec:diffthrflowtree}, we counted flow
trees of $\Lambda_K\subset J^{1}(\Lambda)$.
This leads to the expression for $\pa\Bb$ in Theorem~\ref{thm:mainlink}.
In this subsection we derive the expression for $\pa\Cc$ and $\pa\Ee$
in Theorem~\ref{thm:mainlink} by counting multiscale flow trees with non-empty
big tree part. Our technique relates these multiscale trees to ordinary
flow trees of a stabilized
braid obtained by adding a trivial noninteracting strand to the given  braid.

For notation used throughout this section see Section~\ref{ssec:thecount}.

% **************************************************
\subsubsection{Multiscale flow trees of type ${\rm MT}_{\varnothing}$}
We first consider the part of the differential in the Legendrian algebra of $\Lambda_K$ which accounts for multiscale flow trees of type ${\rm MT}_{\varnothing}$, {\em i.e.}, the parts which count only trees of the braid localized near $\Lambda$. Such a tree has its positive puncture at some Reeb chord $b_{ij}$ and its negative punctures at Reeb chords $a_{ij}$. Furthermore, a straightforward action argument shows that any multiscale flow tree with its positive puncture at a Reeb chord $b_{ij}$ must lie inside the $1$-jet neighborhood of $\Lambda$. Consequently, flow trees of type ${\rm MT}_\varnothing$ account for the boundary of the Reeb chords $b_{ij}$ of type $\mathbf{S}_1$:
\[
\pa(b_{ij})=\sum_{\Gamma\in\T(b_{ij})}\epsilon(\Gamma)q(\Gamma),
\]
where $\T(b_{ij})$ denotes the set of all flow trees with positive puncture at $b_{ij}$ and where, if $\Gamma$ is such a tree, $q(\Gamma)$ denotes the monomial of its negative punctures and $\epsilon(\Gamma)$ its sign.

In Section~\ref{sssec:countflowtree} above, orientation conventions were picked for the Reeb chords $a_{ij}$ and $b_{ij}$ and the differential was computed in Lemma~\ref{lma:dB}.

% **************************************************
\subsubsection{Multiscale flow trees with positive puncture at $c_{ij}$ of type $\mathbf{L}_1$}\label{sssec:dC}
In this section, we compute the differentials of the $c_{ij}$ Reeb chords. This involves counting rigid multiscale flow trees of type $\rm{MT}_0$, with big-tree component given by one of the four rigid big trees with positive puncture at $c$.

We introduce the following notation for $1$-jet lifts of flow trees with one positive puncture. Each point in the $1$-jet lift which is neither a $1$-valent vertex nor a trivalent vertex belongs to either the upper or the lower sheet of its flow line. We call points on the $1$-jet lift {\em head-points} if they belong to the upper sheet and {\em tail-points} if they belong to the lower. Because of positivity of local function differences of trees with one positive puncture, points in a component of the complement of preimages of $1$-valent and trivalent vertices in the $1$-jet lifts are either all head-points or all tail-points.

There are four rigid flow trees with positive puncture at $c$, denoted
in Section~\ref{ssec:flowtreesforLambda} by $I_N,Y_N,I_S,Y_S$. When we
project the $1$-jet lifts of these trees to the torus $\Lambda$, we
obtain the four curves $\Gamma_{\alpha\beta}(c)$,
$\alpha,\beta\in\{0,1\}$, shown in Figure~\ref{fig:dC}.
We decompose $\Gamma_{\alpha\beta}(c)$ as follows:
\[
\Gamma_{\alpha\beta}(c)=\Gamma_{\alpha\beta}^{{\rm v}+}(c)\cup\Gamma_{\alpha\beta}^{{\rm h}}(c)\cup\Gamma_{\alpha\beta}^{{\rm v}-}(c);
\]
see Figure \ref{fig:dC} for $\Gamma_{00}(c)$. More precisely,
$\Gamma_{\alpha \beta}^{\mathrm{v}+}(c)$ consists of head-points,
$\Gamma_{\alpha\beta}^{\mathrm{v}-}(c)$ consists of tail-points, and
$\Gamma_{\alpha\beta}^{\mathrm{h}}(c)$ is the portion near the cusp edge. For our
purposes, we will assume that $\Gamma_{\alpha\beta}^{\mathrm{h}}(c)$ lies at
$t=\frac{\pi}{2}$ (for $\beta=0$) or
$t=\frac{3\pi}{2}$ (for $\beta=1$), as is the case in the degenerate
picture where $U$
is the unperturbed round unknot (see Section~\ref{ssec:flowtreesonU}).
\begin{figure}[htb]
\labellist
\small\hair 2pt
\pinlabel $2\pi$ [Br] at -5 181
\pinlabel $\frac{3\pi}2$ [Br] at -5 138
\pinlabel $\pi$ [Br] at -5 93
\pinlabel $\frac\pi2$ [Br] at -5 46
\pinlabel $0$ [Br] at -5 -2
\pinlabel $t$ [Br] at 5 190
\pinlabel $0$ [Bl] at -3 -12
\pinlabel $2\pi$ [Bl] at 274 -12
\pinlabel $s$ [Bl] at 282 0
\pinlabel {braiding region} [Bl] at 12 203
\pinlabel {$a_{ij}$ and $b_{ij}$} [Bl] at 205 203
\pinlabel $e^-$ [Bl] at 142 98
\pinlabel $c^-$ [Bl] at 281 98
\pinlabel $c^-$ [Bl] at 5 98
\pinlabel $c^+$ [Bl] at 142 6
\pinlabel $e^+$ [Bl] at 6 6
\pinlabel $c^+$ [Bl] at 142 173
\pinlabel $\Gamma_{01}(c)$ [Bl] at 35 146
\pinlabel $\Gamma_{00}(c)$ [Bl] at 35 34
\pinlabel $\Gamma_{11}(c)$ [Bl] at 222 146
\pinlabel $\Gamma_{10}(c)$ [Bl] at 222 34
\pinlabel $\Gamma_{00}^{{\rm v}+}(c)$ [Bl] at 105 15
\pinlabel $\Gamma_{00}^{{\rm h}}(c)$ [Bl] at 80 52
\pinlabel $\Gamma_{00}^{\mathrm{v}-}(c)$ [Bl] at 6 60
\endlabellist
\centering
\includegraphics{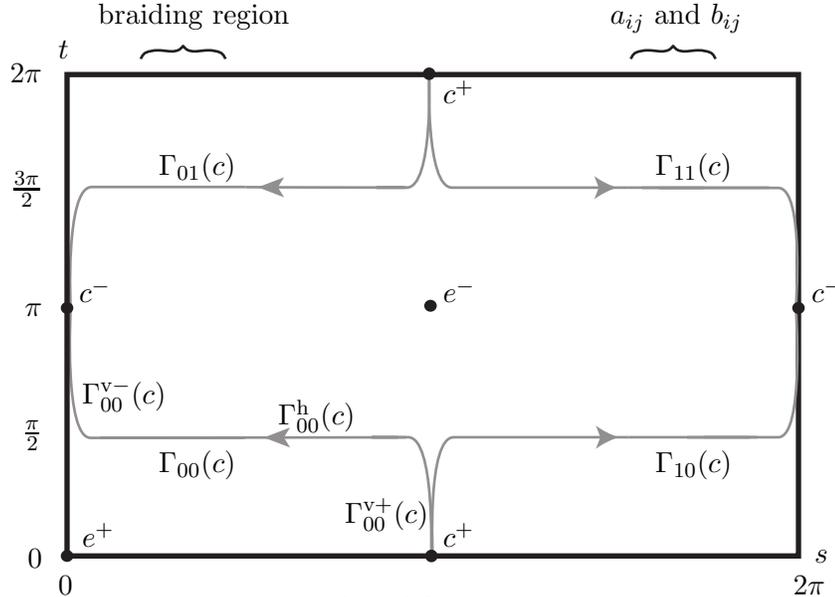}
\caption{The $1$-jet lifts of flow trees with positive puncture at $c$.}
\label{fig:dC}
\end{figure}

The following lemma determines the differential acting on Reeb chords $c_{ij}$. We use the matrix notation of Theorem \ref{thm:mainlink}.
\begin{lma}\label{lma:dC}
With capping operator of Reeb chord $c$ of $\Lambda$ as in Theorem \ref{thm:combsign} and with orientation choices as in \eqref{eq:ker(b)} and \eqref{eq:coker(b)=ker(a)}, the following equation holds:
\[
\pa \Cc = \Aa\cdot\Ll + \Aa\cdot\Phi^{R}_{B}.
\]
\end{lma}
\begin{pf}

We count multiscale flow trees contributing to $\pa c_{ij}$, divided into four cases based on their big tree part, which must be one of the $\Gamma_{\alpha\beta}(c)$'s. For ease of reference, all of these flow trees are pictorially represented in Table~\ref{tbl:dC}.

Before we proceed, note that the braiding region is disjoint from $\Gamma_{10}(c)$ and from $\Gamma_{11}(c)$ and intersects $\Gamma_{00}(c)$ and $\Gamma_{01}(c)$ in tail-points, since the braiding
region has small $s$-coordinate values. Also, from
Section~\ref{ssec:thecount}, the unstable manifolds $W^{\rm
  u}(a_{ij})$ are essentially horizontal and have $t$-coordinate as
follows: just less than $\frac{\pi}{2}$ for $i>j$, ordered
lexicographically by $(i,j)$; just over $\frac{3\pi}{2}$ for $i<j$,
ordered lexicographically by $(j,i)$. Finally, we recall from
Section~\ref{sssec:endpointpath} and Figure~\ref{fig:setup} that we
have cycles $\lambda'$ and
$\mu'$ in $\Lambda$ for the purposes of counting homology classes,
where $\lambda'$ is a vertical line between the $a_{ij}$ and the
$b_{ij}$, and $\mu'$ is a horizontal line just below
$t=\frac{\pi}{2}$. We can in particular choose $\mu'$ to lie above
all of the $W^{\rm u}(a_{ij})$ for $i>j$.

\begin{table}
\centerline{
\begin{tabular}{|c|c|c||c|c|c|}
\hline
Big tree & Multiscale flow tree & Term &
Big tree & Multiscale flow tree & Term  \\
\hline
\multirow{2}{*}{\raisebox{-0.3in}{$\Gamma_{10}(c)$}} & \raisebox{-0.25in}{\includegraphics[height=0.6in]{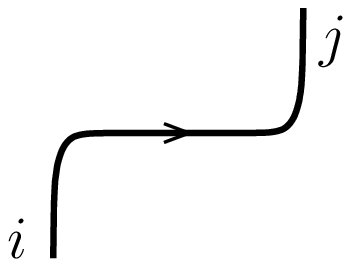}} ~
$(i=j)$&
$\mu_{\alpha(i)}$ &
\multirow{2}{*}{\raisebox{-0.3in}{$\Gamma_{00}(c)$}} & \raisebox{-0.25in}{\includegraphics[height=0.6in]{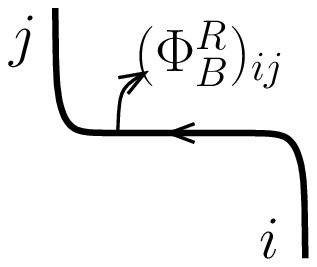}} ~
&
$\mu_{\alpha(i)} (\Phi^R_B)_{ij}$
 \\
\cline{2-3} \cline{5-6}
 & \raisebox{-0.25in}{\includegraphics[height=0.6in]{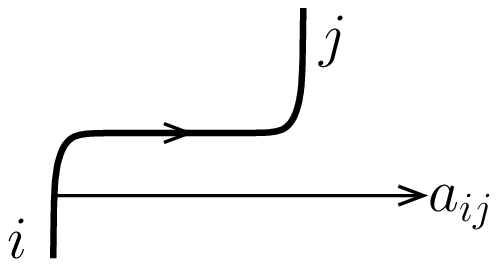}} ~
 $(i>j)$ &
$a_{ij}\mu_{\alpha(j)}$ &
 & \raisebox{-0.25in}{\includegraphics[height=0.6in]{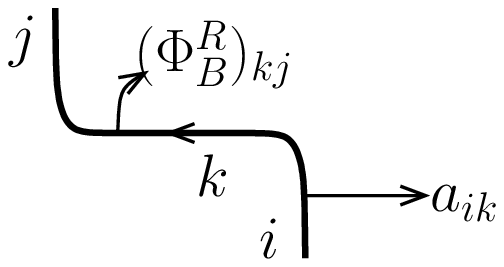}} & $\displaystyle{\sum_{k<i}a_{ik}\mu_{\alpha(k)}(\Phi^R_B)_{kj}}$ \\
\hline
\multirow{2}{*}{\raisebox{-0.3in}{$\Gamma_{11}(c)$}} & \raisebox{-0.25in}{\includegraphics[height=0.6in]{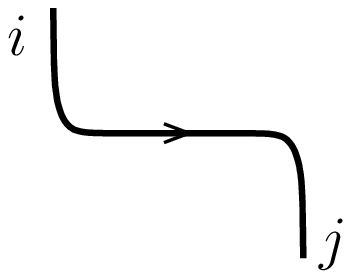}} ~
$(i=j)$ &
$1$ &
\multirow{2}{*}{\raisebox{-0.3in}{$\Gamma_{01}(c)$}} & \raisebox{-0.25in}{\includegraphics[height=0.6in]{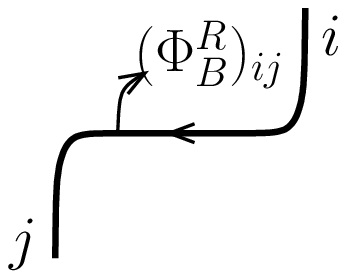}} &
$(\Phi^R_B)_{ij}$
\\
\cline{2-3} \cline{5-6}
 & \raisebox{-0.25in}{\includegraphics[height=0.6in]{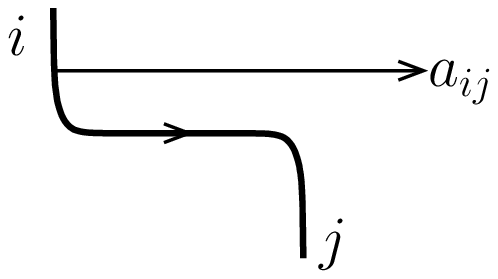}} ~ $(i<j)$ &
$a_{ij}$ &
 & \raisebox{-0.25in}{\includegraphics[height=0.6in]{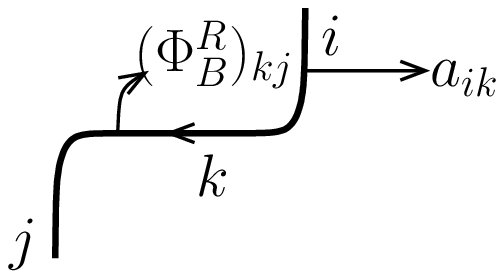}} &
$\displaystyle{\sum_{k>i} a_{ik}(\Phi^R_B)_{kj}}$  \\
\hline
\end{tabular}
\vspace{11pt}
}
\caption{Contributions to $\partial c_{ij}$. For each of the four big
  disks, schematic diagrams for corresponding multiscale flow trees
  are shown, along with the algebraic contribution to $\partial
  c_{ij}$ (with powers of the longitudinal homology classes $\lambda$
  suppressed for simplicity). In the diagrams, along the boundary of
  the big disk, the index of the sheet (one of $i,j,k$) is labeled.}
\label{tbl:dC}
\end{table}

\vspace{11pt}
\noindent \textsc{Case 1}: big tree $\Gamma_{10}(c)$.

If $i=j$, there is a ``trivial'' multiscale flow tree with boundary on sheet $S_i$ that projects to $\Gamma_{10}(c)$. To count other multiscale flow trees corresponding to $\Gamma_{10}(c)$, note that
the curve $\Gamma_{10}^{\mathrm{v}+}(c)$ intersects $W^{\mathrm{u}}(a_{ij})$ for all $i>j$, while the curves $\Gamma_{10}^{\mathrm{h}}(c)$ and $\Gamma_{10}^{\mathrm{v}-}(c)$ are disjoint from $W^{\mathrm{u}}(a_{ij})$ for all $i,j$. As we move along $\Gamma_{10}^{\mathrm{v}+}(c)$ in the sheet $S_i$,
at the intersection between $\Gamma_{10}^{\mathrm{v}+}(c)$ and $W^{\mathrm{u}}(a_{ij})$, $i>j$, a flow line to $a_{ij}$ can split off and the $1$-jet lift of $\Gamma_{10}(c)$ continues to move along sheet $S_j$. Note that no other flow line can split off after this event since, according to the orientation requirement in the definition of multiscale  flow tree, such a flow line could split off only at an intersection with $W^{\mathrm{u}}(a_{jk})$, $j>k$; however, by our choice of ordering, such an intersection precedes the intersection with $W^{\mathrm{u}}(a_{ij})$ and hence no further splitting is possible.

Thus for each $c_{ij}$ with $i\geq j$, there is a multiscale flow tree with big tree $\Gamma_{10}(c)$ and a single negative puncture at $a_{ij}$ (or no negative puncture if $i=j$). We next consider homology coefficients and signs. Note that each tree which is lifted to a leading sheet $S_{\gamma(i)}$ intersects $\lambda_{\gamma(i)}'$ once positively and that each lifted tree intersects $\mu_{\alpha(i)}'$, where  $\gamma(i)$ and $\alpha(i)$ are the indices of the components of the sheets considered.
Furthermore, by Theorem~\ref{thm:combsign}, each tree has sign $+1$: at the splitting point (for $i>j$), the tangent vector to the positively oriented $1$-jet lift is $\pa_t$ which transports to $\pa_t=v^\krn(a_{ij})$ at $a_{ij}$, $i>j$, and the big tree has sign $+1$. Finally, we conclude that the contribution to $\partial c_{ij}$ of Case~1 is:
\[
\begin{cases}
\lambda_{\gamma(i)}\mu_{\alpha(i)} &\text{ if $i=j$ and $S_i$ is leading},\\
\mu_{\alpha(i)} &\text{ if $i=j$ and $S_i$ is not leading}, \\
a_{ij}\lambda_{\gamma(j)}\mu_{\alpha(j)} &\text{ if $i>j$ and $S_j$ is leading},\\
a_{ij}\mu_{\alpha(j)} &\text{ if $i>j$ and $S_j$ is not leading}.
\end{cases}
\]

\vspace{11pt}
\noindent \textsc{Case 2:} big tree $\Gamma_{11}(c)$.

As with $\Gamma_{10}(c)$, if $i=j$ then there is a trivial multiscale tree with boundary on sheet $S_i$ that projects to $\Gamma_{10}(c)$. To count other multiscale trees, notice that $\Gamma_{11}^{\mathrm{v}+}(c)$ intersects $W^{\mathrm{u}}(a_{ji})$, for all $i>j$. As above we find that a $1$-jet lift in sheet $S_j$ can split off a flow line to $a_{jk}$ for $k>j$, then continue along $S_k$, and that no further splittings are possible. We thus find multiscale flow trees with positive puncture at $c_{ji}$ and negative puncture at $a_{ji}$ for all $i>j$. As above the (signed) coefficient equals $+1$ since the normal at the special puncture of the partial tree attached is $-\pa_t$ which agrees with $v^\krn(a_{ij})$, $i<j$. To see homology coefficients, we note that all $1$-jet lifts are disjoint from $\mu'$ and calculate as above. We conclude that the contribution to $\partial c_{ij}$ from Case~2 is: \[
\begin{cases}
\lambda_{\gamma(j)} &\text{ if $i=j$ and $S_j$ is leading},\\
1 &\text{ if $i=j$ and $S_j$ is not leading},\\
a_{ij} \lambda_{\gamma(j)} &\text{ if $i<j$ and $S_j$ is leading},\\
a_{ij} &\text{ if $i<j$ and $S_j$ is not leading}.
\end{cases}
\]

When we combine the contributions of Cases~1 and~2 to $\partial c_{ij}$, we obtain precisely the $(i,j)$ entry of the matrix
$\Aa\cdot \Ll$.

\vspace{11pt}
\noindent \textsc{Case 3}: big tree $\Gamma_{00}(c)$.

Here  we will make use of the stabilized braid $\widehat B$, which is $B$ along with one non-interacting strand labeled $0$. Any multiscale flow tree begins on sheet $S_i$ along $\Gamma_{00}^{\mathrm{v}+}(c)$.
 As above it is possible to split off either no flows and arrive on $S_i$ at the initial point of $\Gamma_{00}^{\mathrm{h}}(c)$, or one flow to $a_{ik}$ (for any $k<i$) and arrive on $S_k$ at the initial point of $\Gamma_{00}^{\mathrm{h}}(c)$. Note further that $\Gamma_{00}^{\mathrm{h}}(c)$ intersects the braiding region in tail-points since the braiding region is to the left of $s= \pi/2;$
 thus, its flow orientation is the same as the oriented lift of the flow line of $\widehat B$ in $W^{\mathrm{u}}(a_{0k})$ for any $k$. Consequently, the flow trees that can split off from $\Gamma_{00}^{\mathrm{h}}(c)$ on sheet $S_k$ as it passes the braiding region, agree with the flow trees that can split off from $W^{\mathrm{u}}(a_{0k})$. It follows that the contribution to
$\pa c_{ij}$ from Case~3, up to sign and homology coefficients, is
\[
\sum_{k<i} a_{ik} (\Phi^R_B)_{kj}.
 \]

We next consider the homology coefficients and signs of these
trees. For homology, note that $\mu'$ is a horizontal line that lies
just below $t=\frac{\pi}{2}$, where $\Gamma_{00}^{\mathrm{h}}(c)$
sits, but just above $W^{\mathrm{u}}(a_{ik})$, where a flow line to
$a_{ik}$ can split off. Thus none of the trees passes $\lambda'$; a
tree that splits off a flow to $a_{ik}$ intersects $\mu_{\alpha(k)}'$
once positively; and a tree that does not split off such a flow
intersects $\mu_{\alpha(j)}'$ once positively.

In order to compute the sign we first note that the sign contribution from the flow line splitting off to $a_{ik}$ is positive: $\pa_t$ is the vector of the boundary orientation as well as $v^\krn(a_{ik})$. Second, we consider sign contributions from trees in the braiding region. At a positive twist the induced normal is $-\pa_s$ which corresponds to a vector splitting of the normal $\pa_t$ along the incoming edge and the sign at the corresponding trivalent vertex of the tree of $\widehat B$ is $+1$. At a negative twist the induced normal is still $-\pa_s$ which now corresponds to splitting of the normal $-\pa_t$ along the incoming edge and the sign at the trivalent vertex of the tree of $\widehat B$ equals $-1$. Thus also in the case of a negative twist the total sign contribution is $(-1)^{2}=1$, which shows that all multiscale flow trees from $\Gamma_{00}(c)$ have sign $+1$.

 In sum, we find that the total contribution to
$\pa c_{ij}$ from Case~3 is
\[
\mu_{\alpha(j)}(\Phi_B^R)_{ij} + \sum_{k<i} a_{ik}\mu_{\alpha(k)}(\Phi_B^{R})_{kj}.
\]

\vspace{11pt}
\noindent \textsc{Case 4}: big tree $\Gamma_{01}(c)$.

As in Case~3, we find that either one or zero flow lines split off along $\Gamma_{01}^{\mathrm{v}+}(c)$. Again $\Gamma_{01}^{\mathrm{h}}$ intersects the braiding region in tail points and is oriented isotopic to $W^{\mathrm{u}}(a_{0i})$. An argument similar to Case~3 shows that the contribution to $\pa c_{ij}$ from Case~4 is
\[
(\Phi_B^R)_{ij} + \sum_{k>i} a_{ik}(\Phi_B^{R})_{kj},
\]
where the difference from the previous case arises because the $1$-jet lift of $\Gamma_{01}(c)$ is disjoint from $\mu_j'$.

When we combine the contributions of Cases~3 and~4 to $\partial c_{ij}$, we obtain precisely the $(i,j)$ entry of the matrix
$\Aa\cdot \Phi_B^R$. The lemma follows.
\end{pf}

% **************************************************
\subsubsection{Multiscale flow trees of type ${\rm MT}_0$ with positive
  puncture at $e_{ij}$ of type $\mathbf{L}_2$}\label{sssec:dEMT0}

To complete the computation of the Legendrian homology of $\Lambda_K$,
we need to compute the differentials of the $e_{ij}$ Reeb
chords. These have contributions from two types of multiscale rigid
flow trees: trees of type ${\rm MT}_{0}$ and trees of type ${\rm
  MT}_1$. In this subsection, we compute the first type; the second
type is computed in the following subsection, Section~\ref{sssec:dEMT1}.

There are two rigid flow trees with positive puncture at $e$ and
negative puncture at $c$, corresponding to $E_1$ and $E_2$ in the
language of Section~\ref{ssec:flowtreesforLambda}. Denote their $1$-jet lifts by $\Gamma_{\alpha}(e;c)$, $\alpha\in\{0,1\}$.  We decompose $\Gamma_{\alpha}(c)$ as follows:
\[
\Gamma_{\alpha}(e;c)=\Gamma_{\alpha}^{\mathrm{he}}(e;c)\cup\Gamma_{\alpha}^{\mathrm{ta}}(e;c),
\]
where $\Gamma_{\alpha}^{\mathrm{he}}(e;c)$ consists of head-points and
$\Gamma_{\alpha}^{\mathrm{ta}}(e;c)$ of tail-points; see Figure \ref{fig:dE_C}.
\begin{figure}[htb]
\labellist
\small\hair 2pt
\pinlabel $2\pi$ [Br] at -5 181
%\pinlabel $\frac{3\pi}2$ [Br] at -5 138
\pinlabel $\pi$ [Br] at -5 93
%\pinlabel $\frac\pi2$ [Br] at -5 46
\pinlabel $0$ [Br] at -5 -2
\pinlabel $t$ [Br] at 5 190
\pinlabel $0$ [Bl] at -3 -12
\pinlabel $2\pi$ [Bl] at 274 -12
\pinlabel $s$ [Bl] at 282 0
\pinlabel {braiding region} [Bl] at 12 203
\pinlabel {$a_{ij}$ and $b_{ij}$} [Bl] at 205 203
\pinlabel $e^-$ [Bl] at 142 99
\pinlabel $c^-$ [Bl] at 281 99
\pinlabel $c^-$ [Bl] at 5 99
\pinlabel $c^+$ [Bl] at 142 7
\pinlabel $e^+$ [Bl] at 6 7
\pinlabel $\Gamma^{\mathrm{he}}_0(e;c)$ [Bl] at 80 10
\pinlabel $\Gamma^{\mathrm{ta}}_1(e;c)$ [Bl] at 80 101
\pinlabel $\Gamma^{\mathrm{ta}}_0(e;c)$ [Bl] at 165 101
\pinlabel $\Gamma^{\mathrm{he}}_1(e;c)$ [Bl] at 165 10
%\pinlabel $c^+$ [Bl] at 142 173
\endlabellist
\centering
\includegraphics{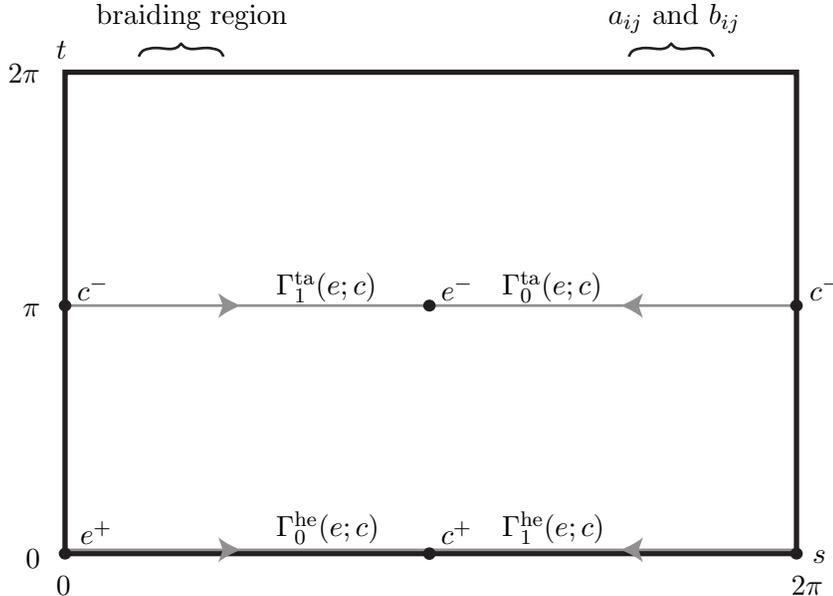}
\caption{The $1$-jet lifts of flow trees with positive puncture at $e$ and negative puncture at $c$.}
\label{fig:dE_C}
\end{figure}

Writing the contribution to $\pa e_{ij}$ from trees of type $\rm{MT}_0$ as $\pa_0 e_{ij}$ and using the matrix notation of Theorem \ref{thm:mainlink}, we have the following result.
\begin{lma}\label{lma:dE;C}
With capping operator of $c$ so that Theorem \ref{thm:combsign} holds and with orientation choices as in \eqref{eq:ker(b)} and \eqref{eq:coker(b)=ker(a)}, there is a choice of capping operator for $e$ so that the following equation holds:
\[
\pa_0 \Ee = -\Phi_B^{L}\cdot \Cc\cdot\Ll^{-1} + \Ll^{-1}\cdot\Cc\cdot(\Phi_B^{R})^{-1}.
\]
\end{lma}
\begin{pf}
First consider the contributions from $\Gamma_{0}(e,c)$:
$\Gamma_0^{\mathrm{ta}}(e;c)$ is disjoint from the braiding region and
from all $W^{\mathrm{u}}(a_{ij})$. Consequently, no flow tree can
split off from $\Gamma_{0}^{\mathrm{ta}}(e;c)$. Also, if
$\Gamma_0^{\mathrm{ta}}(e;c)$ is lifted to $S_i$ then it intersects
$\lambda'_{\gamma(i)}$ once negatively if $S_i$ is leading and not at
all otherwise. The curve of head-points $\Gamma_0^{\mathrm{he}}(e;c)$,
oriented as in Figure~\ref{fig:dE_C}, is isotopic without crossing the
cycle $\mu'$
to the oriented lift of the flow line $W^{\mathrm{u}}(a_{i0})$ of
the stabilized braid $\widehat B$ which lies in the sheet $S_i$. As in
the proof of Lemma~\ref{lma:dC} we conclude that the trees that split
off along $\Gamma$ correspond to the trees which split off from
$W^{\mathrm{u}}(a_{0i})$. As with $\Gamma_{00}(c)$ (Case 3) in the
proof of Lemma~\ref{lma:dC}, we find that the sign contribution from
the split-off trees agrees with the sign of $\phi_B$. Thus the
contribution from $\Gamma_{0}(e,c)$ to $\pa_0 \Ee$ is
\[
\epsilon(\Gamma_0(e;c))\Phi_B^{L}\cdot\Cc\cdot\Ll^{-1}.
\]
Choose the capping operator of $e$ so that
$\epsilon(\Gamma_0(e;c))=-1$, and note that by
Theorem~\ref{thm:signunknottree} this implies
$\epsilon(\Gamma_1(e;c))=+1$.

Next consider the contributions from $\Gamma_{1}(e,c)$: $\Gamma_1^{\mathrm{he}}(e;c)$ is disjoint from the braiding region and from all $W^{\mathrm{u}}(a_{ij})$. Consequently, no flow tree can split off from $\Gamma_{1}^{\mathrm{he}}(e;c)$. Also, if $\Gamma_1^{\mathrm{he}}(e;c)$ is lifted to $S_i$ then it intersects $\lambda'_{\gamma(i)}$ once with negative intersection number if $S_i$ is leading and not at all otherwise. The tail-points curve $\Gamma_1^{\mathrm{ta}}(e;c)$ is  isotopic to the lift of the flow line $W^{\mathrm{u}}(a_{i0})$ of the stabilized braid $\widehat B$ with orientation reversed. Switching the roles of $e$ and $c$ in the Morse-Bott perturbation we would get the following contribution to $\pa'_0 \Cc$ from this disk:
\[
 \Ll\cdot\Ee\cdot\Phi_{B}^{R},
\]
where $\pa'_0$ denotes the analogue of $\pa_0$ with the alternative Morse-Bott perturbation, by a repetition of the argument above for $\Gamma_0(e;c)$.

Observing that the multiscale flow trees we are interested in are
exactly the same as those for the alternative Morse-Bott perturbation
except for the big disk having the opposite orientation, we can view
the equation above as a linear system of equations with coefficients
in $\A_n^0$ and invert it to get the contribution to $\pa_0\Ee$. Thus
we find that the contribution from $\Gamma_{1}(e,c)$ is
\[
\epsilon(\Gamma_0(e;c))\Ll^{-1}\cdot
\Cc\cdot(\Phi_{B}^{R})^{-1}=\Ll^{-1}\cdot \Cc\cdot(\Phi_B^{R})^{-1}
\]
and the lemma holds.
\end{pf}

% **************************************************
\subsubsection{Multiscale flow trees of type ${\rm MT}_1$ with positive
  puncture at $e_{ij}$}\label{sssec:dEMT1}
Finally, we enumerate multiscale flow trees of type ${\rm MT}_1$.
Recalling Lemma~\ref{lma:treesofU} and Remark~\ref{rmk:1dfam}, we see
there are four constrained rigid flow trees with positive puncture at
$e$ and no negative punctures, that are constrained to pass through
some $b_{ij}.$ We denote their $1$-jet lifts $\Gamma_{\alpha
  \beta}(e)$, $\alpha,\beta\in\{0,1\}$, see Figures
\ref{fig:dE_Bfirst} and \ref{fig:dE_Bsecond}. We note that each
constrained tree is a deformation of a broken tree, with
$\Gamma_{\alpha\beta}(c)$ as defined in Section~\ref{sssec:dC} and
$\Gamma_\alpha(e;c)$ as defined in Section~\ref{sssec:dEMT0}:
\[
\begin{aligned}
\Gamma_{00}(e) &\simeq \Gamma_{1}(e;c)\,\#\,\Gamma_{00}(c),\\
\Gamma_{10}(e) &\simeq \Gamma_{1}(e;c)\,\#\,\Gamma_{10}(c),\\
\Gamma_{01}(e) &\simeq \Gamma_{1}(e;c)\,\#\,\Gamma_{01}(c),\\
\Gamma_{11}(e) &\simeq \Gamma_{1}(e;c)\,\#\,\Gamma_{11}(c).
\end{aligned}
\]
\begin{figure}[htb]
\labellist
\small\hair 2pt
%\pinlabel $2\pi$ [Br] at -5 181
%\pinlabel $0$ [Br] at -5 -2
\pinlabel $e^+$ [Bl] at 274 -7
\pinlabel {braiding region} [Bl] at 12 203
\pinlabel $e^-$ [Bl] at 142 99
%\pinlabel $c^-$ [Bl] at 281 99
%\pinlabel $c^+$ [Bl] at 142 7
\pinlabel $b_{ij}$ [Bl] at 230 26
\pinlabel $b_{ij}$ [Bl] at 230 158
\pinlabel $\Gamma_{00}(e)$ [Bl] at 165 43
\pinlabel $\Gamma_{01}(e)$ [Bl] at 165 138
\pinlabel $e^+$ [Bl] at 274 188
\endlabellist
\centering
\includegraphics{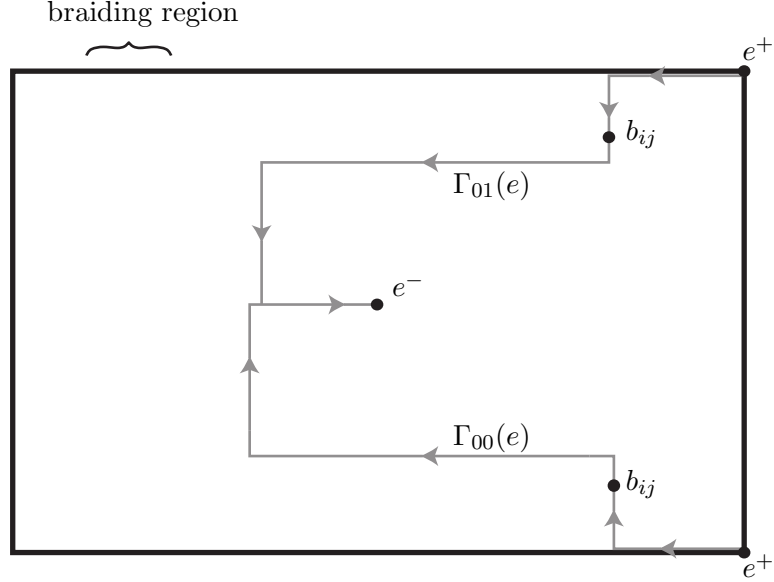}
\caption{$1$-jet lifts of constrained flow trees with positive puncture at $e$ which are disjoint from the braiding region.}
\label{fig:dE_Bfirst}
\end{figure}
\begin{figure}[htb]
\labellist
\small\hair 2pt
%\pinlabel $2\pi$ [Br] at -5 181
%\pinlabel $0$ [Br] at -5 -2
\pinlabel $e^+$ [Bl] at 274 -7
\pinlabel {braiding region} [Bl] at 12 203
\pinlabel $e^-$ [Bl] at 142 99
%\pinlabel $c^-$ [Bl] at 281 99
%\pinlabel $c^+$ [Bl] at 142 7
\pinlabel $b_{ij}$ [Bl] at 230 26
\pinlabel $b_{ij}$ [Bl] at 230 158
\pinlabel $\Gamma_{10}(e)$ [Bl] at 20 43
\pinlabel $\Gamma_{11}(e)$ [Bl] at 20 138
\pinlabel $e^+$ [Bl] at 274 188
\endlabellist
\centering
\includegraphics{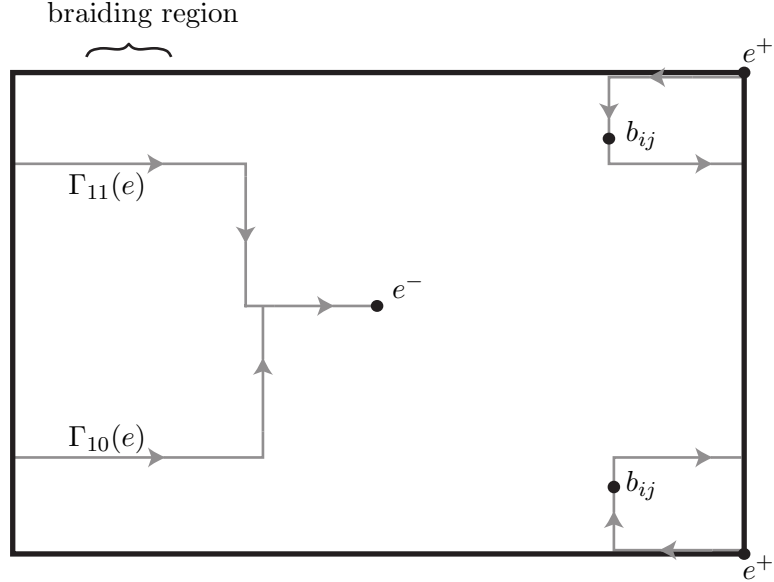}
\caption{$1$-jet lifts of constrained flow trees with positive
  puncture at $e$ which intersect the braiding region. As with
  $\Gamma_{\alpha\beta}(c)$, the horizontal
  segments in the middle of $\Gamma_{10}(e)$ and $\Gamma_{11}(e)$ are
  at $t=\frac{\pi}{2}$ and $t=\frac{3\pi}{2}$, respectively.}
\label{fig:dE_Bsecond}
\end{figure}

Writing the contribution to $\pa e_{ij}$ from trees of type $\rm{MT}_1$ as $\pa_1 e_{ij}$ and using the matrix notation of Theorem \ref{thm:mainlink}, we have the following result.

\begin{lma}\label{lma:dE;B}
With capping operators and orientation data as in Lemma~\ref{lma:dE;C}, the following equation holds:
\[
\pa_1 \Ee = \Bb\cdot(\Phi_B^{R})^{-1} + \Bb\cdot \Ll^{-1}.
\]
\end{lma}

\begin{pf}
Consider first the contributions of the big trees $\Gamma_{00}(e)$ and
$\Gamma_{01}(c)$. Note that
these tree do not intersect the braiding region and that because of the
lexicographic order, exactly as in the proof of Lemma \ref{lma:dC}, a
disk which is constrained at $b_{ij}$ cannot split off flow lines to
any $a_{ik}$. Thus there is exactly one multiscale flow tree
contributing to $\pa_1 e_{ij}$ that arises from
$\Gamma_{00}(e)$ or $\Gamma_{01}(e)$ if $i\neq j$ (from
$\Gamma_{00}(e)$ if $i>j$ or $\Gamma_{01}(e)$ if $i<j$), and it
begins at $e^+$ on sheet $i$, jumps at $b_{ij}$ to
sheet $j$, and remains there until $e^-$. (If $i=j$, then there is no
such multiscale tree.)
Up to sign and homology
classes, we conclude that $\Gamma_{00}(e)$ and
$\Gamma_{01}(c)$ combined contribute the term $b_{ij}$ (or $0$ if $i=j$)
to $\pa e_{ij}$.

As for
homology classes, the $1$-jet lift of this
multiscale tree is disjoint from all $\lambda'$ cycles, with the
exception of $\lambda'_{\alpha(j)}$ (one negative intersection) if $j$
is leading; furthermore, it
is disjoint from all $\mu'$ cycles, with the exception of
$\mu'_{\alpha(j)}$ (one positive intersection) if $i>j$.
Next we determine the signs. The sign of the endpoints of the moduli
spaces of $\Gamma_{00}(e)$ and $\Gamma_{01}(e)$ listed above are both
$+1$ by our choice of capping operators. Thus the orienting vector
field $\nu$ of the moduli space at $b_{ij}$ satisfies $\nu=\pa_s$ in
both cases. Furthermore, $v^\krn(\Gamma_{00}(e))$ equals $\pa_t$ at
$b_{ij}$, $i>j$, and $v^\krn(\Gamma_{01}(e))$ equals $-\pa_t$ at
$b_{ij}$, $i<j$. This shows that the signs are positive.
Collecting homology classes and signs, we conclude that the big trees
$\Gamma_{00}(e)$ and $\Gamma_{01}(e)$ contribute $\Bb\cdot \Ll^{-1}$
to $\pa_1\Ee$ (recall that the $(i,j)$ entry of $\Bb$ is $b_{ij}\mu_j$
if $i>j$, $0$ if $i=j$, and $b_{ij}$ if $i<j$).

Next consider the contributions to $\pa e_{ij}$ from big trees
$\Gamma_{10}(e)$ and $\Gamma_{11}(e)$. The $1$-jet lifts of the
multiscale flow trees
corresponding to these big trees begin at $e^+$ on sheet $i$, switch
to sheet $k$ for some $k<i$ (for $\Gamma_{10}(e)$) or $k>i$ (for
$\Gamma_{11}(e)$) at the constraint $b_{ik}$, and then pass through
the braiding region and end at $e^-$ on sheet $j$. The (horizontal)
portion passing
through the braiding region is located at $t=\frac{\pi}{2}$ or
$t=\frac{3\pi}{2}$ and is thus isotopic to
$\Gamma_1^{\mathrm{ta}}(e;c)$ (which is at $t=\pi$) in the complement
of all the unstable
manifolds $W^{\mathrm{u}}(a_{ij})$ (which are below $t=\frac{\pi}{2}$
or above $t=\frac{3\pi}{2}$) and the cycles $\mu'$ (which is below
$t=\frac{\pi}{2}$).
Thus repeating the computation in Lemma \ref{lma:dE;C}, we find that
$\Gamma_{10}(e)$ and $\Gamma_{11}(e)$ combined give a contribution to
$\pa_1 \Ee$ of
\[
\Bb\cdot(\Phi_B^{R})^{-1}.
\]
Here we find that the sign of this term is $+1$ after observing that
the sign of the underlying restricted rigid disk is again $+1$ by
repeating the calculation above. The lemma follows.
\end{pf}

% **************************************************
\subsubsection{Proof of Theorem \ref{thm:mainlink}}\label{sssec:pfmainlink}
By Theorem~\ref{thm:diskandgentree}, the differential for the
Legendrian DGA $LA(\Lambda_K)$ can be computed in terms of multiscale flow trees determined by $\Lambda_K$ and $\Lambda$.
The contribution from multiscale flow trees of types ${\rm MT}_{\varnothing}$, ${\rm MT}_0$, and ${\rm MT}_1$ are calculated in Lemma \ref{lma:dB}, \ref{lma:dC}, \ref{lma:dE;C}, and~\ref{lma:dE;B}.\qed

% **************************************************
% **************************************************
% **************************************************
\section{Multiscale Flow Trees and Holomorphic Disks}\label{sec:diskandgentree}
The main purpose of this section is to establish Theorem
\ref{thm:diskandgentree}. The proof has several steps. In Section
\ref{ssec:basicft} we establish the connection between holomorphic
disks and flow trees for $\Lambda_U$ following \cite{Ekholm07}. In
Section \ref{ssec:refinedft} we discuss the slightly stronger disk/flow
tree correspondence needed here. Finally, in Sections
\ref{ssec:disktoqtree} and \ref{ssec:qtreetodisk} we establish the correspondence between holomorphic disks of $\Lambda_K$ in $J^1(S^2)$ and multiscale flow trees, {\em i.e.}, holomorphic disks with boundary on $\Lambda_U$ with flow trees of $\Lambda_K\subset J^{1}(\Lambda_U)$ attached along its boundary.

% **************************************************
% **************************************************
\subsection{Basic results on constrained flow trees and disks}\label{ssec:basicft}
In this subsection we give a slight modification of results from \cite{Ekholm07} in the case of a Legendrian surface $\Lambda$ in the $1$-jet space of a surface $S$, $\Lambda\subset J^{1}(S)$. We will apply these results in two cases relevant to this paper; namely when $\Lambda=\Lambda_U$ and $S=S^{2}$, and  $\Lambda=\Lambda_K$ and $S=\Lambda_U$.

For the notion of a flow tree of $\Lambda$ we refer to Section~\ref{ssec:gft} and for a more thorough account \cite[Section 2.2]{Ekholm07}. As explained in \cite[Section 3.1]{Ekholm07}, associated to each flow tree $\Gamma$ is its formal dimension $\dim(\Gamma)$, see Equation~\eqref{eq:dimoftrees},  which is the dimension of the manifold of nearby flow trees for sufficiently generic $\Lambda$. The main result of \cite{Ekholm07} is concerned with the relation between rigid holomorphic disks and rigid flow trees. Here we will need a slight generalization. To this end we introduce \emph{constrained} rigid holomorphic disks and \emph{constrained} flow trees.
If $p_1,\ldots, p_r$ are distinct points in $\Lambda$ then a holomorphic disk \emph{constrained by $p_1,\dots,p_r$} is a holomorphic disk with $r$ extra boundary punctures at which the evaluation map hits $p_1,\dots,p_r$. Similarly, a  flow tree of $\Lambda$ \emph{constrained by $p_1,\dots,p_r$} is a flow tree with $1$-jet lift with $r$ extra marked points where the evaluation map hits $p_1,\dots,p_r$. If $\Gamma'$ is a flow tree (or holomorphic disk) constrained by
$p_1,\dots,p_r\in\Lambda$ and if $\Gamma$ is that flow tree (or holomorphic disk) with the constraining conditions forgotten
then the formal dimension of the constrained flow tree (or holomorphic disk) $\Gamma'$ satisfies
\[
\dim(\Gamma')=\dim(\Gamma)-r,
\]
where $\dim(\Gamma)$ is the formal dimension of $\Gamma$. We say that a flow tree (or holomorphic disk) is {\em
constrained rigid} if its formal dimension equals $0$ and it is transversely cut out by its defining
equations.

For $0<\sigma \leq 1$ consider the map
\[
s_\sigma\colon J^1(S)\to J^1(S),\quad s_\sigma(q,p,z)=(q,\sigma p,\sigma z)
\]
where $q\in S$, $p\in T^*_qS$ and $z\in \R$ and write
\[\Lambda_\sigma=s_\sigma(\Lambda).\]
Notice that since $s_\sigma$ preserves the contact structure,
$\Lambda_\sigma$ is still Legendrian and clearly Legendrian isotopic to $\Lambda$. In order to state the correspondence theorem relating constrained rigid disks and trees we recall the following notation: $D_m$ denotes the unit disk with $m$ boundary punctures and $\Pi\colon J^{1}(S)\to T^{\ast}S$ denotes the Lagrangian projection.

\begin{thm}\label{t:basicdisktree}
Given $\Lambda\subset J^1(S)$ as above there is a small perturbation of $\Lambda$ so that for a generic metric
$g$ on $S$ there exist $\sigma_0>0$, almost complex structures $J_\sigma$, $0<\sigma<\sigma_0$, and perturbations $\widetilde\Lambda_\sigma$ of $\Lambda_\sigma$ with the following properties.
\begin{itemize}
\item
The Legendrian submanifold $\widetilde\Lambda_\sigma$ is obtained from $\Lambda_\sigma$ by a $C^0$-deformation supported near the cusp edges of $\Lambda_\sigma.$
\item
The constrained rigid flow trees defined by $\widetilde\Lambda_\sigma$ have well-defined limits as $\sigma\to 0.$
\item
The (constrained)
rigid $J_\sigma$-holomorphic disk with boundary on $\widetilde\Lambda_\sigma$ with one positive puncture are in
one-to-one correspondence with the (constrained) rigid flow trees of $\widetilde\Lambda_\sigma$ with one positive puncture.
In particular, the following holds for all sufficiently small $\sigma>0$: if
$u_\sigma\colon D_m\to T^\ast S$ is a (constrained) rigid $J_\sigma$-holomorphic disk then
there exists a (constrained) rigid flow tree $\Gamma$ of $\widetilde\Lambda_\sigma$ such that
$u_\sigma(\pa D_m)$ lies in an $\Ordo(\sigma\log(\sigma^{-1}))$ neighborhood of the Lagrangian lift $\bar\Gamma$ of $\Gamma$.
Moreover, outside $\Ordo(\sigma\log(\sigma^{-1}))$-neighborhoods of the $Y_0$-, $Y_1$-vertices and
switches of $\Gamma,$ the curve $u_\sigma(\pa D_m)$ lies at $C^1$-distance
$\Ordo(\sigma\log(\sigma^{-1}))$ from $\bar\Gamma$.
\end{itemize}
\end{thm}

\begin{pf}
The proof is an adaption of results from \cite{Ekholm07}, where Theorems 1.2 and 1.3 gives a version of Theorem \ref{t:basicdisktree} for unconstrained rigid disks and trees. (See the proof of Theorem 1.3 and Lemma 5.13 in \cite{Ekholm07} for the $\Ordo(\sigma\log(\sigma^{-1}))$-estimate). We briefly recall the construction in order to adapt it to the constrained rigid case.

The first step is to fix a Riemannian metric on $S$ such that there are only a finite number of (constrained) flow trees of formal dimension $0$ determined by $\Lambda$ and such that all such flow trees are transversely cut out. A straightforward modification of the unconstrained case, \cite[Proposition 3.14]{Ekholm07}, shows that  the set of such metrics is open and dense.

The second step is to change the metric to $g$, to introduce almost complex structures $J_\sigma$, and to isotope the Legendrian $\Lambda_\sigma$ to a new Legendrian $\widetilde \Lambda_\sigma$. The main features of these objects are the following. The rigid flow trees determined by $g$ and $\widetilde\Lambda_\sigma$ are in 1-1 correspondence with the rigid flow trees of $\Lambda$ and corresponding trees lie very close to each other, see \cite[Lemma 4.4]{Ekholm07}. The submanifold $\widetilde\Lambda_\sigma$ is rounded near its cusps and changed accordingly near its swallowtails. The metric $g$ is flat in a neighborhood of any rigid flow tree and $\Pi(\widetilde\Lambda_\sigma)$ is affine in this neighborhood outside a finite number of regions of diameter $\Ordo(\sigma)$ where it is curved in only one direction, see \cite[Subsection 4.2]{Ekholm07} and Remark \ref{rmk:Lneartrees} below for details. The almost complex structure $J_\sigma$ agrees with the almost complex structure $J_g$ induced by the metric in a neighborhood of all rigid flow trees and swallowtails and outside a neighborhood of the caustic (the locus of fiber tangencies) of $\Pi(\widetilde\Lambda_\sigma)$. Near the points in caustic outside a neighborhood of the swallowtails, $J_\sigma$ is constructed so that both $J_\sigma$ itself and the Lagrangian boundary condition given by $\Pi(\Lambda)$ splits as products with one direction along the cusp edge and one perpendicular to it, see \cite[Section 4.2.3]{Ekholm07}.

In fact, $\Pi(\widetilde \Lambda_\sigma)=\Phi_\sigma^{-1}(\widehat L_\sigma)$ where $\widehat
L_\sigma$ is a totally real immersed submanifold and where $\Phi_\sigma$ is a diffeomorphism with $d_{C^{0}}(\Phi_\sigma,\id)=\Ordo(\sigma)$ and with $d_{C^{1}}(\Phi_{\sigma},\id)$ arbitrarily small (but finite), which is supported in a small neighborhood of the caustic of $\Pi(\widetilde\Lambda_{\sigma})$ and which is equal to the identity in a neighborhood of all rigid flow trees of $\widetilde\Lambda_\sigma$.  We use the almost complex structure $J_\sigma=d\Phi_\sigma\circ J\circ
d\Phi_\sigma^{-1}$ on $T^\ast S$, where $J$ is the almost complex structure induced by a metric
on $S$. Then $J_\sigma$-holomorphic disks with boundary on $\widetilde \Lambda_\sigma$ correspond to
$J$-holomorphic disks with boundary on $\widehat L_\sigma$.

We modify the deformations of the metric and of the Legendrian  discussed above, in order
to deal also with constrained rigid trees. In the presence of point conditions, repeat the construction in
\cite[Subsection 4.2]{Ekholm07}, deforming the metric to $g$ and constructing $\widehat L_\sigma$ in the exact same way as
along rigid flow trees.  Furthermore, this should be
done in such a way that no constraining point is an edge point, see \cite[Subsection 4.2.H]{Ekholm07}. Then take
$J_\sigma=d\Phi_\sigma\circ J\circ d\Phi_\sigma^{-1}$, and let $\widetilde\Lambda_\sigma$ be such that $\Phi_\sigma(\Pi\widetilde \Lambda_\sigma)=\widehat L_\sigma$, where $J$ is the almost complex structure on $T^\ast S$ induced by the
special metric. We use the notation from \cite[Remark 3.8]{Ekholm07}  for vertices of (constrained) rigid flow
trees.

After the modifications of the constructions in \cite[Section 4]{Ekholm07} (first and second steps) described above, the
theorem follows from \cite[Theorem~1.2, Theorem~1.3]{Ekholm07} with the following additions.
Theorem~1.2 shows that any rigid holomorphic disk with one positive puncture converges to a flow
tree. Since the condition that the boundary of a disk passes through a constraining point is
closed, it follows that constrained disks converge to constrained trees. In Theorem~1.3, rigid
$J_\sigma$-holomorphic disks with boundary on $\widehat L_\sigma$ near any rigid flow tree are constructed and
proved to be unique. The corresponding construction and uniqueness proof in the case of constrained
rigid flow trees is completely analogous after the following alteration. If $\Gamma$ is a
constrained rigid flow tree with constraining point $m$ then take the preimage of $m$ (which is not
an edge point) to be a marked point in the domain $\Delta_{p,m}^{(0,0)}(\Gamma,\sigma)$ of the approximately holomorphic disks, see \cite[Subsection 6.2.A]{Ekholm07}, and let $V_\sol(m)=0$ instead of $V_\sol(m)\approx\R$, see \cite[Subsection~6.3.B, Definition~6.15]{Ekholm07}, .
\end{pf}

\begin{rmk}\label{rmk:Lneartrees}
Below we will use the following special features of the metric $g$ and the Legendrian submanifold $\widetilde\Lambda_\sigma$ in Theorem \ref{t:basicdisktree}. (For the construction of the metric and $\widetilde\Lambda_\sigma$ with these properties we refer to \cite[Section 4]{Ekholm07}.)
\begin{itemize}
\item[$(1)$] There exists a neighborhood $X\subset S$ which contain all (constrained) rigid flow trees in which the metric $g$ on $S$ is flat. We write $X_{\widetilde\Lambda_\sigma}\subset\Pi_{F}^{-1}(X) \cap \widetilde\Lambda_\sigma$.
\item[$(2)$] The image of any flow segment in a (constrained) rigid flow tree is a geodesic of the metric $g$.
\item[$(3)$] The two sheets of the Lagrangian projection $\Pi(\widetilde\Lambda_\sigma)$ near each double point consists of two transverse affine Lagrangian subspaces.
The Lagrangian projection $\Pi(\widetilde\Lambda_\sigma)$ is parallel to the $0$-section ({\em i.e.\ }~the graph of a polynomial function in flat local coordinates of degree at most $1$) in a neighborhoods of the following points: trivalent vertices of any (constrained) rigid flow tree,
and intersection points of any (constrained) rigid flow tree.
\item[$(4)$] In Fermi coordinates along any edge the Lagrangian is parallel to the $0$-section in the coordinate perpendicular to the edge ({\em i.e.\ }~the Lagrangian is the differential of the graph of a function of the form $f(x_1)+cx_2$), where $x_1$ is the coordinate along the edge, $x_2$ perpendicular to it, and $c$ a constant.
\item[$(5)$]
Outside $\Ordo(\sigma)$-neighborhoods of a finite number of {\em edge points} in each (constrained) rigid flow tree the Lagrangian projection $\Pi(\widetilde\Lambda_\sigma)$ is affine ({\em i.e.}, the function $f$ in $(4)$ has the form $f(x_1)=a_2x_1^{2}+a_1x_1+a_0$.) For our study of multiscale flow trees below we will also assume that the edge point regions are disjoint from the junction points.
\item[$(6)$] We will also assume that the following extra condition is met: at a finite number of fixed \emph{extra} points the Lagrangian projection has the form mentioned in $(3)$.
\end{itemize}
\end{rmk}

% **************************************************
% **************************************************
\subsection{Refined results on constrained flow trees and disks}\label{ssec:refinedft}
Recall that Theorem \ref{thm:diskandgentree} relates holomorphic disks to multiscale flow trees. This relation is the result of a double degeneration: first the conormal lift of the unknot is pushed to the $0$-section in $J^{1}(S^{2})$ and then the conormal lift of a more general closed braid is pushed toward the almost degenerate conormal lift of the unknot. To deal with this we will stop the first degeneration close to the limit where actual holomorphic disks on the almost degenerate conormal lift of the unknot are close to flow trees. Then we degenerate the conormal lift of a general braid toward the almost degenerate conormal lift of the unknot and show that holomorphic disks near the limit admit a description in terms of quantum flow trees, {\em i.e.}, holomorphic disks with flow trees attached along their boundaries. Since both quantum flow trees and multiscale trees are defined as intersection loci of evaluation maps of holomorphic disks and flow trees, respectively, we need to show that the disks and the flow trees are arbitrarily $C^{1}$-close in order to get the desired relation between quantum flow trees and multiscale flow trees. However the relation between quantum trees and holomorphic disks holds only for almost complex structures with special properties near $\Pi(\Lambda_U)$. In this section we show that there exists almost complex structures with these special properties for which the holomorphic disks are still close to flow trees.

% **************************************************
\subsubsection{Definition of quantum flow trees}
As already mentioned quantum flow trees will be central to establishing the relation between multiscale flow trees and holomorphic disks. We define them as follows.
Consider $\Lambda_K\subset J^{1}(\Lambda_U)$ as above. A \emph{quantum flow} tree $\Xi$ of $\Lambda_K$ is a holomorphic disk $u\colon D_m\to T^{\ast}S^{2}$ with boundary on $\Pi(\Lambda_U)$ and a collection $\Gamma$ of partial flow trees $\Gamma=\{\Gamma_1,\dots,\Gamma_m\}$ with one special positive puncture on the lift $\widetilde u\colon\pa D_m\to \Lambda_U$ of the boundary of $u.$ Note that $\Gamma_j$ could be a constant flow tree at a Reeb chord of $\Lambda_K$. Then the positive special punctures subdivide the boundary of $D_m$ into arcs and we require that there is a lift of these arcs to $\Lambda_K$ which together with the $1$-jet lifts of the trees in $\Gamma$ from closed curve when projected to $T^{\ast} S^{2}$. As for multiscale flow trees we call the points where flow trees are attached to $u$ \emph{junction points}.

% **************************************************
\subsubsection{Modifying the almost complex structure---metric around flow trees}\label{sssec:modifacs}
Consider the degeneration $\sigma\to 0$ in Theorem \ref{t:basicdisktree}. Fix a small $\sigma = \sigma_0 >0$ so that the boundaries of all (constrained) rigid holomorphic
disks are close to the cotangent lifts of their corresponding (constrained) rigid flow trees. More
precisely, we take $\sigma_0$ so that the boundaries of all (constrained) rigid
$J_{\sigma_0}$-holomorphic disks lie well inside the finite neighborhood $X$ of the tree where the
metric is flat and where $\widetilde \Lambda_{\sigma_0}$ is as described in Remark \ref{rmk:Lneartrees}. For simpler notation we write
\[ \Lambda=\widetilde \Lambda_{\sigma_0} = \Lambda_U
\]
where $U$ is the unknot. We continue to use the subscript $\sigma_0$ in $J_{\sigma_0}$ from Theorem~\ref{t:basicdisktree}
since we will modify the almost complex structure some more.

Consider an arbitrary closed braid $K\subset \R^{3}$ lying in a tubular neighborhood of $U.$
If $K$ is sufficiently close to $U$ then $\Lambda_K$ lies in a tubular neighborhood of $\Lambda=\Lambda_U$ which is symplectomorphic to $J^{1}(\Lambda)$. Furthermore, the front projection $\Pi_F^{\Lambda}\colon\Lambda_K\to\Lambda$ is an immersion. Since these properties are preserved under the global scaling by $\sigma$ we consider $\Lambda_K\subset N\subset J^{1}(\Lambda)$, where the neighborhood $N$ of the $0$-section in $J^{1}(\Lambda)$ is identified with a neighborhood of $\Lambda$ in $J^{1}(S^{2})$.

When we compute the Legendrian homology of $\Lambda_K$, we will use an almost complex
structure $J_\eta$ (to be defined after Lemma~\ref{l:goodmetric} for small $\eta >0$) on $T^\ast S^2$ which differs from $J_{\sigma_0}.$
In particular, to relate holomorphic disks with boundary on $\Lambda_K$ with quantum flow trees of $\Lambda_K$, it will be important that $J_\eta$ agrees with the almost complex structure induced by a metric on $\Lambda$ in a neighborhood $N_\eta$ of $\Pi(\Lambda)\subset T^\ast S^2$. Here $N_\eta$ is the image under a symplectic immersion of a small neighborhood of the $0$-section in $T^{\ast}\Lambda$ which extends $\Pi|_{\Lambda}$.  Lemma~\ref{l:goodmetric} below, establishes the existence of such a metric on $\Lambda.$ Specifically,
the metric induces an almost complex structure on $N_\eta$ which (has a push-forward under an immersion which) agrees up to first order with $J_{\sigma_0}$ in the fixed size neighborhood $\Pi(X_{\Lambda})\subset\Pi(\Lambda)$ of the union of all boundaries of constrained rigid flow trees in $\Lambda$.  Our desired almost complex structure $J_\eta$ will interpolate between this push-forward
and $J_{\sigma_0}.$

Recall the special form of $\Pi(\Lambda)$ near its double points, see Remark \ref{rmk:Lneartrees} $(3)$. Choose a metric $g$ on $\Lambda$ flat near its Reeb chord endpoints. Let $J_g$ be the almost complex structure induced by $g.$
Note that for such a metric we can find an immersion $\phi$ defined on the cotangent bundle of the neighborhood of the Reeb chord endpoint which is $(J_g, J_{\sigma_0})$ holomorphic.
In particular, if  $\phi\colon T^{\ast}\Lambda\to T^{\ast} S^{2}$ an immersion which extends $\Pi|_{\Lambda}$ and which has these properties near Reeb chord endpoints then the push-forward of $J_g,$ $\phi_\ast J_g,$ is well-defined. Furthermore, we assume that $(g,\phi)$ satisfies this $(J_g, J_{\sigma_0})$-holomorphicity condition also near other points mentioned in Remark \ref{rmk:Lneartrees} $(3)$ and $(6)$, where we take the extra points to be the junction points of the multiscale trees. We call the points in $(3)$ and the extra points the \emph{distinguished} points. We call a pair of a metric and a symplectic immersion $(g,\phi)$ with properties as above \emph{adapted to $\Lambda$}.

\begin{lma}\label{l:goodmetric}
There exists a neighborhood $N$ of the $0$-section in $T^{\ast}\Lambda$ and a pair $(g,\phi)$ consisting of a metric $g$ and an immersion $\phi\colon N\to T^{\ast} S^{2}$ which extends $\Pi|_{\Lambda}$, which is adapted to $\Lambda$ and such that the following holds on $\phi(N)$.
\begin{itemize}
\item[$(1)$] $J_{\sigma_0}$ and $\phi_\ast J_g$ agree along $\Pi(\Lambda)$,
\item[$(2)$] $J_{\sigma_0}$ and $\phi_\ast J_g$ agree in neighborhoods of distinguished points.
\item[$(3)$] $J_{\sigma_0}$ and $\phi_\ast J_g$ agree to first order in $\Pi(X_{\Lambda})$.
\end{itemize}
\end{lma}

\begin{pf}
It is straightforward to check that statement $(1)$ can be achieved and statement $(2)$ follows by the definition of adapted pair (where the fact that the metric is flat and the Lagrangian affine near distinguished points readily implies existence).

We turn our attention to statement $(3)$. Let $p$ be a point on the $1$-jet lift of a rigid flow tree. Pick normal coordinates $x=(x_1,x_2)$ on $\Lambda$ around $p$ with the $1$-jet lift corresponding to
$\{x_2=0\}$. Since the metric on $S^2$ is flat we can identify it locally with $\C^{2}$ with coordinates $(u_1,v_1,u_2,v_2)$ and we can choose these coordinates so that the flow tree under consideration lies along $\{u_2=0\}$.

Since $\Pi(\Lambda)$ is a product of a curve in the $u_1$-direction and a line segment parallel to the $0$-section in the $u_2$-direction we have the following local parametrization of $\Pi(\Lambda)$:
\[
f(x_1,x_2)=(x_1,f(x_1),x_2,c_2).
\]
Let $y = (y_1,y_2)$ denote the fiber coordinate. Defining the local immersion
\[
\psi(x,y)=f(x_1,x_2)+y_1(-f'(x_1)\pa_{x_1}+\pa_{y_1})+y_2\pa_{y_2}
\]
we find that condition $(1)$ holds and that the Taylor expansion of $\psi^{\ast}J$ ({\em i.e.}, the complex structure of the flat metric on $S^{2}$ which corresponds to the standard complex structure in $u+iv$-coordinates pulled back by $\psi$) with respect to  $y$ is
\[
(\psi^{\ast}J)(x,y)=J_0+B(x_1)y_1+\Ordo(2),
\]
where $J_0$ is the standard complex structure on $\C^2$ with $x+iy$-coordinates which is the complex structure induced by the flat metric on $\Lambda$. Here
$B(x_1)J_0+J_0B(x_1)=0$, and $B(x_1)\pa_{x_2}=0$. A straightforward calculation shows that if
we change the immersion $\psi$ by pre-composing with a diffeomorphism $\Phi$ with the Taylor expansion
\[
\Phi(x,y)=(x,y)+ \frac12 \bigl(B_1(x_1)\pa_{x_1}\bigr) y_1^2,
\]
then the almost complex structures agree to first order, {\em i.e.\ }~if $\phi=\psi\circ\Phi$ then $\phi^{\ast}J_{\sigma_0}$ and $J_0$ agree to first order along $\{y_1=y_2=0\}$.
\end{pf}

\begin{rmk}
For a general ambient almost complex structure $J$ it is not possible to make the push forward agree up to first order. The Taylor expansion above with respect to $(y_1,y_2)$ for a general $J$ is
\[
(\psi^{\ast}J_{\sigma_0})(x,y)=J_0+B_1(x)y_1+B_2(x)y_2+\Ordo(2),
\]
where $B_j$ anti-commutes with $J_0$. One would then look for a map with Taylor expansion of the form
$$
\Phi(x,y)=
(x,y)+\frac12\bigl(B_1\pa_{x_1}\bigr)y_1^2
+\frac12\bigl(B_2\pa_{x_2}\bigr)y_2^2+ C y_1y_2.
$$
However, in order for the almost complex structures to agree up to first order one needs both $C=B_1\pa_{x_2}$ and $C=B_2\pa_{x_1}$. In general $B_1\pa_{x_2}\ne B_2\pa_{x_1}$
so no solution $\Phi$ exists.
\end{rmk}

Let $(g,\phi)$ be as in Lemma~\ref{l:goodmetric}. For small $\eta>0$, write $N_\eta$ for the image under $\phi$ of an $\eta$-neighborhood of the $0$-section in $T^\ast\Lambda$ and let $\Xi_\eta$ denote the image of an $\eta$-neighborhood of $T^{\ast}X_{\Lambda}$.
Theorem \ref{t:basicdisktree} implies that (for $\sigma_0$ small enough) every (constrained) rigid $J_{\sigma_0}$-holomorphic disk intersects a neighborhood of $\Pi(\Lambda)$ inside $\Xi_{\eta'}$ for some $\eta'>0$. Write $M_{\eta}$ for a $\eta$-neighborhood of all (constrained) rigid holomorphic disks.
Let $J_\eta$ denote an almost complex
structure on $T^\ast S^2$ which equals $\phi_\ast J_g$ on $N_{\eta}$, which equals $J_{\sigma_0}$ outside
$N_{2\eta}$, and which interpolate between the two in the remaining region in such a way that
that the following hold.
\begin{align}\label{e:Jgto0}
&|J_{\sigma_0}-J_\eta|_{C^0}\to 0,\text{ as }\eta\to 0,\\\label{e:Jg1to0}
&|J_{\sigma_0}-J_\eta|_{C^1}\to 0\text{ in $\Xi_{\eta_0}\cup M_{\eta_0}$ as $\eta\to 0$, for fixed $\eta_0>0$}\\\label{e:Jg1}
&|J_{\sigma_0}-J_\eta|_{C^2}\le K_1\text{ in $\Xi_{\eta_0}\cup M_{\eta_0}$, for fixed $\eta_0, K_1$, and $\eta < \eta_0$}\\\label{e:Jg2}
&|J_{\sigma_0}-J_\eta|_{C^1}\le K_2,\text{ for fixed $\eta_0,K_2$, and for $\eta < \eta_0$ }.
\end{align}

We make two remarks.
As a consequence of the fact that in general $J_{\sigma_0}$ and $\phi_\ast J_g$ do not agree up to first order outside
$X$ we typically have $|J_{\sigma_0}-J_\eta|_{C^2}\to\infty$ as $\eta\to 0$.
As shown in the $C^0$-convergence portion of
the proof of Lemma~\ref{l:C1conv}  below, we may assume that for sufficiently small $\eta$, any rigid $J_\eta$-holomorphic
disk lies in $M_{\eta_0}.$

We next show that rigid $J_\eta$-holomorphic disks are $C^{1}$-close to rigid $J_{\sigma_0}$-holomorphic disks. Note that it follows from Theorem~\ref{t:basicdisktree} that (for $\sigma_0$ sufficiently small) the almost complex structure $J_{\sigma_0}$ is regular in the sense that all (constrained) $J_{\sigma_0}$-holomorphic disks in $T^{\ast} S^{2}$ with boundary on $\Pi(\Lambda)$ of formal dimension $\le 0$ are transversely cut out.
\begin{lma}\label{l:C1conv}
Let $u_\eta\colon D_m\to T^\ast S^2$, $\eta\to 0$ be a sequence of (constrained) rigid $J_\eta$-holomorphic disks with boundary on $\Lambda$. Then some subsequence of $u_\eta$ $C^1$-converges on compact subsets of $D_m$ to a (constrained) rigid $J_{\sigma_0}$-holomorphic disk. Moreover, for all $\eta>0$ small enough there is a unique (constrained) rigid $J_\eta$-holomorphic disk in a neighborhood of each rigid $J_{\sigma_0}$-holomorphic disk.
\end{lma}

\begin{pf}
For the first statement we use Gromov compactness (for $|J_\eta-J_{\sigma_0}|_{C^0}\to 0$) to conclude that either
$|Du_\eta|$ is uniformly bounded or there is bubbling in the limit, see {\em e.g.\ }\cite{Sikorav94}, and the fact that point constraints are closed.
The case that $u_\eta$ is a sequence of rigid disks bubbling is not possible: all bubbles of the limit have
dimension at least $0$, since $J_{\sigma_0}$ is regular, this implies $\dim(u_\eta)>0$ in contradiction to
$u_\eta$ being (constrained) rigid. We conclude thus that $|Du_\eta|$ is uniformly bounded; thus $u_\eta$ converges uniformly to a (constrained) rigid $J_{\sigma_0}$-holomorphic disk $u$. We must show that it convergence with one derivative as well.

Consider a point $z\in D$ and its image $u(z)$ under $u$. Since $u_\eta(z)\to u(z)$ and since
$|Du_\eta|$ and $|Du|$ are bounded we can find a coordinate neighborhood $W\subset M_{\eta_0}$ of $u(z)$ and a small disk $E$ around $z$ in $D$ so that $u_\eta(E)\subset W$ for all $\eta>0$ and $u(E)\subset W$.

We pick $\C^{2}$-coordinates on $W$ so that $\Pi(\Lambda)$ corresponds to the totally real $\R^2\subset \C^2$. The neighborhood $E$ is either a disk or a half disk, with complex structure $j.$
We find as in Section 2.3 of  \cite{Sikorav94} that in local coordinates $u$ and $u_\eta$ satisfy the equations
\begin{align}
&\bar\pa u + q\,\pa u=0,\\
&\bar\pa u_\eta + q_\eta\,\pa u_\eta=0,
\end{align}
where
\begin{align}
&q(z)=(i+J_{\sigma_0}(u))^{-1}(i-J_{\sigma_0}(u)),\\
&q_\eta(z)=(i+J_\eta(u_\eta))^{-1}(i-J_\eta(u_\eta)).
\end{align}
Letting $u-u_\eta=h_\eta$ we conclude that
\begin{equation}
\bar\pa h_\eta + q_\eta\,\pa h_\eta=(q-q_\eta)\pa u.
\end{equation}
By scaling we may take $|q_\eta|_{C^2}\le \epsilon\ll 1,$ see \cite{Sikorav94}. Moreover, by $C^0$-convergence $u_\eta$ lies in $M_{\eta_0}$ an therefore $|q-q_\eta|_{C^1}\to 0$. A standard bootstrap argument for $h_\eta$ now shows that $|h_\eta|_{C^1}\to 0$ as $\eta\to 0$.
\end{pf}

\begin{rmk}
A general sequence of $J_\eta$-holomorphic disks which does not blow up would $C^0$-converge to a
$J_{\sigma_0}$-holomorphic disk but not necessarily $C^1$-converge.
\end{rmk}

% **************************************************
\subsubsection{Metric and perturbations---metric near flow trees of $\Lambda_K$}\label{sssec:metrpert}
Note that the metric $g$ in Lemma \ref{l:goodmetric} has arbitrary form outside a small neighborhood of the $1$-jet lifts of constrained rigid flow trees of $\Lambda$ and that furthermore it is flat near all distinguished points. In this section we will impose further conditions on $g$ outside this region in order to adapt it to the (partial) flow trees of $\Lambda_K\subset J^{1}(\Lambda)$ that are parts of rigid multiscale flow trees of $\Lambda_K$. Furthermore, we will also deform $\Lambda_K$ itself in a way analogous to how $\Lambda_\sigma$ was deformed into $\widetilde\Lambda_\sigma$. The construction is completely analogous to the construction in \cite[Section 4]{Ekholm07}, although it is simpler in the present situation since $\Lambda_K$ has no front
singularities. The construction in Section \ref{sssec:modifacs} gives a metric in $\Lambda$ which is flat in a neighborhood of the $1$-jet lifts of all (constrained) rigid flow trees whose $\Pi$-projections contain the boundaries of all (constrained) rigid holomorphic disks. Furthermore, we take the distinguished points to include all junction points as well as points where multiscale flow trees intersect constrained rigid disks.
We next extend the region where the metric is flat to contain all rigid flow trees of $\Lambda_K\subset J^{1}(\Lambda)$ as well as all partial flow trees of $\Lambda_K$ which are parts of rigid generalized disks. Note that these regions include the projection of any Reeb chord of $\Lambda_K\subset J^{1}(\Lambda)$. We next deform $\Lambda_K$ slightly so that it has the form described in Remark \ref{rmk:Lneartrees} over all the flow trees just mentioned, see \cite[Section 4.2]{Ekholm07}. (Here we
treat junction points corresponding to positive punctures of special trees like the $3$-valent
vertices of flow trees in \cite{Ekholm07}	and treat the other junction points like the $2$-valent punctures in \cite{Ekholm07}.) In particular, $\Lambda_K$ is affine at Reeb chord endpoints, lifts of flow trees are geodesics in the flat metric, and
the sheets of $\Lambda_{K}$ near a junction point which is a special Reeb chord will be
parallel to the $0$-section in $J^{1}(\Lambda)$ (which in turn is parallel to the $0$-section in $J^{1}(S^{2})$ over a subset $U\subset S^2$ where the metric on $S^2$ is flat). Similarly, the metric on $S^{2}$ is flat near junction points which are Reeb chords, where the sheet of $\Lambda$ is parallel to the $0$ section, and where the sheets of $\Lambda_K$ are affine (and almost parallel to the $0$-section).

Let $0\le \eta\le 1$ and define $\Lambda_{K;\eta} = s_\eta(\Lambda_K)$ to be the image under fiber scaling by $\eta$ in $J^{1}(\Lambda)$. Then as above, along $1$-jet lifts of flow lines that are part of rigid generalized disks, the Lagrangian $\Pi(\Lambda_{K;\eta})$ is a product of a horizontal line segment
and a curve over the distinguished curve. Thus, as in \cite{Ekholm07, EkholmEtnyreSabloff09}, the regions where this curve is not
affine have diameters $\Ordo(\eta)$ as $\eta\to 0$. With this metric we construct the almost
complex structure $J_{\eta}$ of Lemma~\ref{l:C1conv} for some sufficiently small but fixed
$\eta>0$. Note that $J_\eta$ then
agrees with the complex structure induced by the metric on $\Lambda$ in a neighborhood of $\Pi(\Lambda)$
which is the image under an immersion of a small neighborhood of the $0$-section in $T^\ast \Lambda$.

% **************************************************
% **************************************************
\subsection{From disks to quantum flow trees} \label{ssec:disktoqtree}
Now that we have finished modifying our almost complex structures to achieve $J_\eta$ with the desired properties
from the previous subsection, we simplify notation and let
\[
J = J_\eta.
\]

The main result of this section is that any sequence of rigid $J$-holomorphic disks with boundary on
$\Lambda_{K,\eta}= s_\eta(\Lambda_K)$ has a subsequence that converges to a rigid quantum flow tree of $\Lambda_K$ and $\Lambda$. In Section
\ref{sssec:outflowtree} we characterize certain subsets of the domains of any sequence of
$J$-holomorphic disks with boundary on $\Lambda_{K,\eta}$ such that the restrictions of the maps to
these subsets converge to a (partial) flow tree of $\Lambda_K\subset J^1(\Lambda)$. In particular, in case these
subsets constitute the whole domains of the members in the sequence we find that the $J$-holomorphic
disks converge to a flow tree. In Section \ref{sssec:blowup} we show that there can be at most one
disk bubbling off in the limit of a sequence of rigid $J$-holomorphic disks with boundary on
$\Lambda_{K,\eta}$ as $\eta\to 0$ and that by adding a puncture in the domain near the point where
the bubble forms we ensure that the maps in the sequence satisfy a uniform derivative bound. In
Section \ref{sssec:qtreeconv} we prove the main result of the section. After the previous subsections
there are two main points which must be demonstrated. First, we show that the limit has only one
holomorphic disk part which must be a (constrained) rigid disk. Second, we show that our analysis of
the two separate parts (the flow tree- and the disk part) gives a complete description of the limit
objects.

% **************************************************
\subsubsection{Notation}\label{sssec:notatdisktotree}
Consider $\Lambda_K\subset N\subset J^{1}(\Lambda)$ where $N$ is a neighborhood of the $0$-section which is identified with a neighborhood of $\Lambda\subset J^{1}(S^{2})$. In particular, if $0<\eta<1$ and $s_\eta\colon J^{1}(\Lambda)\to J^{1}(\Lambda)$ denotes the fiber scaling then $s_\eta(N)\subset N$ and hence $\Lambda_{K,\eta}$ is a Legendrian submanifold in $N$ which is Legendrian isotopic to $\Lambda_K.$

Consider the Reeb chords of $\Lambda_{K;\eta}$. Recall that these are of two types: short and long chords. We will use the notation for chords introduced in Remark \ref{rmk:Reebchordtypes}.
The {\emph{action}} $\action(c)$ of a Reeb chord $c$ is the positive difference of the $z$-coordinates of its two endpoints.
Stokes' Theorem implies the area of a holomorphic disk is the signed sum of the actions of the Reeb chords at its punctures;
see \cite[Lemma 2.1]{EkholmEtnyreSullivan05a}, for example.
Recall the two Reeb chords $e$ and $c$ of $\Lambda=\Lambda_U$ introduced in Section \ref{ssec:flowtreesonU}.
Chords of $\Lambda_{K;\eta}$ satisfies the following: chords of type $\mathbf{L}_2$ lie close to $e$ and have action $\action(e)+\Ordo(\eta)$, chords of type $\mathbf{L}_1$ lie close to $c$ and have action $\action(c)+\Ordo(\eta)$, and chords of types $\mathbf{S}_0$ and $\mathbf{S}_1$ have action $\Ordo(\eta)$.

For the fixed almost complex structure $J=J_\eta$, Lemma \ref{l:C1conv} holds and $J$ has properties as in Section \ref{sssec:metrpert}. Then, in particular, boundaries of (constrained) rigid disks of $\Lambda$ lie $C^{1}$-close to 1-jet lifts of its corresponding (constrained) rigid trees.
It follows, in particular, that there is a natural one-to-one correspondence between rigid quantum trees of $\Lambda_K$ and rigid generalized trees of $\Lambda_K$. Furthermore, in a neighborhood of $\Pi(\Lambda)$ the almost complex structure $J$ agrees with the one induced by the metric on $\Lambda$ by Lemma \ref{l:goodmetric}.

Below we will discuss $J$-holomorphic disks in $T^{\ast}S^{2}$ with boundary punctures. Throughout we will think of these as maps $u\colon\Delta_m\to T^{\ast} S^{2}$, where the source is a standard domain. For details on standard domains we
refer to \cite[Subsection 2.2.1]{Ekholm07}; here we give a
brief description. Consider $\R^{m-2}$ with coordinates
$\underline{\tau}=(\tau_1,\dots, \tau_{m-2})$. Let $t\in\R$ act on $\R^{m-2}$ by $t\cdot
\underline{\tau}=(\tau_1+t,\dots,\tau_{m-2}+t)$. The orbit space of this action is
the space of conformal structures of the disk with $m$ boundary punctures, one of which is distinguished, $\conf_m\approx\R^{m-3}$. Define a {\em standard domain} $\Delta_{m}(\underline{\tau})$ as the
subset of $\R\times[0,m]$ obtained by removing $m-2$ horizontal slits
of width $\epsilon$, $0<\epsilon\ll1$, starting at $(\tau_j,j)$,
$j=1,\dots,m-2$ and going to $+\infty$. All slits have the same shape,
ending in a half-circle, see Figure~\ref{fig:stdom}.  The points $(\tau_j, j)$ are called the {\emph{boundary minima}}.
\begin{figure}[htb]
\labellist
\small\hair 2pt
\pinlabel $\tau_1$ [Br] at 80 12
\pinlabel $\tau_2$ [Br] at 126 12
\pinlabel $\tau_3$ [Br] at 164 12
\pinlabel $x$ [Br] at 298 13
\pinlabel $y$ [Br] at 156 146
\endlabellist
\centering
\includegraphics{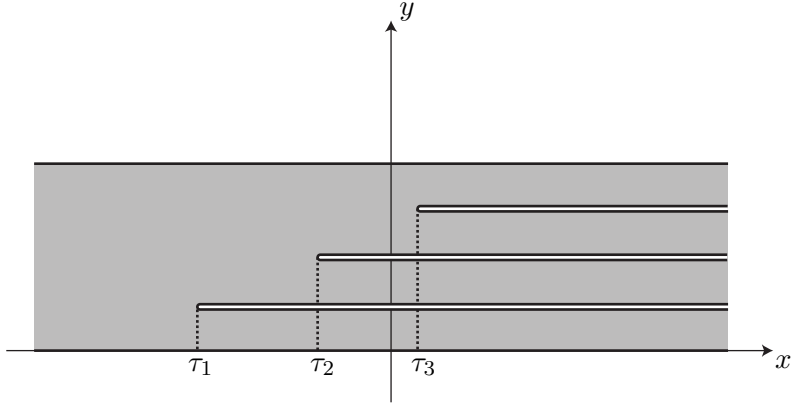}
\caption{The standard domain $\Delta_5(\tau_1,\tau_2,\tau_3)$ with three boundary minima.}
\label{fig:stdom}
\end{figure}

In the case that $u$ has less than two negative punctures we add marked points at
intersections with a small circle around the positive Reeb chord endpoint and puncture the domain there, so that it admits a description as a standard domain. We will often write $\Delta_m$ dropping the precise information about the conformal structure from the notation.

Consider a standard domain $\Delta_m\subset\C$. Let $V_{\beta}=\{x+iy\colon x=\beta\}$. A connected component of the closure of $V_\beta\cap(\Delta_m-\pa\Delta_m)$ in $\Delta_m$ will be called a \emph{vertical segment} in $\Delta_m$.

% **************************************************
\subsubsection{Flow tree convergence}\label{sssec:outflowtree}
Consider a sequence $u_\eta\colon\Delta_m\to T^\ast S^2$ of rigid $J$-holomorphic disks with
boundary on $\Lambda_{K;\eta}$, $\eta\to 0$. As $\eta\to 0$, the actions of Reeb chords of $\Lambda_{K;\eta}$ of type $\mathbf{S}$ satisfy a $\Ordo(\eta)$ bound. Hence, by Stokes' Theorem and the
dimension formula for holomorphic disks, if $\eta>0$ is sufficiently small then a moduli spaces of
$J$-holomorphic disks with one positive puncture, with boundary on $\Lambda_{K,\eta}$, and of formal
dimension $0$ can be non-empty only if it contains disks with punctures of the following types:
\begin{itemize}
\item[$(\mathrm{QT}_{\varnothing})$] The positive puncture is of type ${\bf S}_1$ and all negative punctures are type ${\bf
S}_0$.
\item[$(\mathrm{QT}_0)$] The positive puncture is of type ${\bf L}_1$ and all negative punctures are of type
${\bf S}_0$.
\item[$(\mathrm{QT}_0')$] The positive puncture is of type ${\bf L}_2$, one negative puncture is of type ${\bf
L}_1$, and all other negative punctures are of type ${\bf S}_0$.
\item[$(\mathrm{QT}_1)$] The positive puncture is of type ${\bf L}_2$, one negative puncture is of type ${\bf
S}_1$, and all other negative punctures are of type ${\bf S}_0$.
\end{itemize}

\begin{lma}\label{l:onlytreeconv}
If $u_\eta\colon\Delta_m\to T^\ast S^2$ is a sequence of $J$-holomorphic disks of type
$(\mathrm{QT}_{\varnothing})$ then $u_\eta$ has a subsequence that converges to a flow tree of $\Lambda_K\subset N\subset J^1(\Lambda)$ as $\eta\to0$.
\end{lma}

\begin{pf}
The actions of Reeb chords in $\mathbf{S}_1$ are $\Ordo(\eta)$. Thus the area of $u_\eta$ is $\Ordo(\eta)$ as well. Monotonicity then implies that $u_\eta(\Delta_m)$ must stay inside an $\Ordo(\eta^{\frac12})$-neighborhood of
$\Pi(\Lambda)$. Since $J$ agrees with the complex structure coming from the metric on $\Lambda$ in a finite
neighborhood of $\Pi(\Lambda)$ and since disks lift to the symplectization of $J^{1}(S^{2})$ where $\Lambda\times\R$ is embedded, the lemma follows from \cite[Theorem 1.2]{Ekholm07}.
\end{pf}

We next show that for any sequence of $J$-holomorphic disks $u_\eta\colon\Delta_m\to T^{\ast}S^{2}$ there are neighborhoods of each negative puncture of type ${\bf S}$, where the disk converges to a flow tree (which may be constant). The key to establishing this convergence is an $\Ordo(\eta)$ derivative bound on neighborhoods of the punctures. Consider the inclusion $\Lambda_{K;\eta}\subset N\subset J^{1}(\Lambda)$ and let $z$ denote a coordinate in the $\R$-direction of $J^{1}(\Lambda)\approx T^{\ast}\Lambda\times\R$. If $\gamma$ is a curve in $T^{\ast} S^{2}$ then let $\ell(\gamma)$ denote the length of $\gamma$ in the metric induced by $\omega$ and $J$. Fix $M>0$ larger than the maximum of the function $|z|$ on $\Lambda_K=\Lambda_{K;1}$. A vertical segment $l_\eta\approx [0,1]$
in the domain $\Delta_m$ of $u_{\eta}$ such that
\begin{align}\label{eq:vertlineL}
\ell(u_\eta(l_\eta))\le M\eta,\\\label{eq:vertlineZ}
|z(u_\eta|_{l_\eta}(1))-z(u_\eta|_{l_\eta}(0))|\le M\eta,
\end{align}
will be called an \emph{$\eta$-short} vertical segment. It follows from the asymptotic properties of holomorphic disks near punctures that there exists $\eta$-short vertical segments in a neighborhood of each puncture of $u_{\eta}$ of type $\mathbf{S}$. Note that a vertical segment $l_\eta$ subdivides $\Delta_m$ into two components: $\Delta_m=\Delta_m^{+}(l_\eta)\cup\Delta_m^{-}(l_\eta)$, where $\Delta^{+}(l_\eta)$ contains the positive puncture of $u_\eta$. For $d>0$, let $\Delta^\pm(l_\eta,d)$ denote the subset of points in $\Delta^\pm(l_\eta)$ which are at distance at least $d$ from $l_\eta$.
\begin{lma}\label{l:outest}
For all sufficiently small $\eta>0$ the following derivative bound holds: if $l_\eta$ is an $\eta$-short vertical segment in the domain of  $u_\eta\colon\Delta_m\to T^{\ast}S^{2}$ then
\[
|du_\eta(z)|=\Ordo(\eta),\quad z\in\Delta^-(l_\eta,1).
\]
\end{lma}

\begin{pf}
Let $b_1,\dots,b_r$ denote the negative punctures of $u_\eta$ which lie in $\Delta_m^{-}(l_\eta)$. Then the area $A^{-}_\eta$ of $u_{\eta}(\Delta^-(l_\eta))$ satisfies
\begin{align*}
0 \le A^{-}_{\eta}&=\int_{u_{\eta}(\pa\Delta^{-}(l_\eta))}p\,dq\\
&= \int_{u_\eta(l_\eta)} p\,dq + (z(u_\eta|_{l_\eta}(1))-z(u_\eta|_{l_\eta}(0))) - \sum_{j=1}^{r}\ell(b_j)=\Ordo(\eta).
\end{align*}
We conclude by
monotonicity that $u_\eta(\Delta^-(l_\eta))$ must lie in an
$\Ordo(\eta^{\frac12})$-neighborhood of $\Lambda$ (in particular all chords $b_j$ are of type $\mathbf{S}$). Since $J$ agrees with the almost complex structure induced by the metric on $\Lambda$ in such a neighborhood, \cite[Lemma 5.4]{Ekholm07},  shows that the
function $|p|^2$, where $p$ is the fiber coordinate in $T^\ast \Lambda,$ composed with $u_\eta$ is
subharmonic on $\Delta^-(l_\eta)$ and therefore attains its maximum on the boundary. The
lemma then follows from \cite[Lemma 5.6]{Ekholm07}.
\end{pf}

The derivative bound of Lemma \ref{l:outest} leads to flow tree convergence on $\Delta^{-}(l_\eta)$. Consider a sequence of $\eta$-short vertical segments $l_\eta$ such that each domain $\Delta^{-}(l_\eta)$ contains a puncture mapping to a Reeb chord $\eta\cdot b$ of $\Lambda_{K;\eta}$ for some Reeb chord $b$ of $\Lambda_{K}$ with $\Pi(b)\in N$.
\begin{cor}\label{c:outconv}
There exists a constant $C>0$ such that the sequence of restrictions $u_\eta|_{\Delta^-(l_\eta,C\log(\eta^{-1}))}$ has a subsequence that converges to a flow tree of $\Lambda_K\subset J^1(\Lambda)$. (Note that the flow tree in the limit may be constant.)
\end{cor}

\begin{pf}
By Lemma \ref{l:outest}, the image under $u_\eta$ of any	region in $\Delta_m^{-}(l_\eta,1)$ of diameter $\log(\eta^{-1})$ lies inside a disk of radius $\Ordo(\eta\log(\eta^{-1}))$. Furthermore, along any strip region in
$\Delta^-(l_\eta,\log(\eta^{-1}))$ outside an $\Ordo(\log(\eta^{-1}))$-neighborhood of the
boundary minima in $\Delta^-(l_\eta)$, the map converges to a flow tree by the proof of \cite[Theorem~1.2]{Ekholm07}, see in particular Lemmas~5.12, 5.16, and~5.17 in \cite{Ekholm07}.
\end{pf}

\begin{rmk}
If the limiting flow tree in Corollary~\ref{c:outconv} is constant then it lies at $\Pi(b)\in N$, where $b$ is the Reeb chord of $\Lambda_K$ above. To see this note that $\Delta_m^{-}(l_\eta,C\log(\eta^{-1}))$ always contain
a half infinite strip which is a neighborhood of the puncture mapping to $b$. If the vertical segment bounding this strip does not converge to $\Pi(c)$ for some Reeb chord $c$, then the limiting tree is non-constant. Since $\Pi(b)$ is in the image we find that the tree must lie at $\Pi(b)$.
\end{rmk}

% **************************************************
\subsubsection{Blow-up analysis}\label{sssec:blowup}
We next show that the limit of any sequence of $J$-holomorphic disks $u_\eta\colon\Delta_m\to T^{\ast}S^{2}$ with boundary on $\Lambda_{K;\eta}$ can contain at most one bubble. We also show how to add one puncture consistently to each domain
so that this forming bubble corresponds to some coordinate of the domains of $u_\eta$, which give points in
the space of conformal structures on the disk with $m$ boundary punctures, one distinguished, $\conf_m\approx\R^{m-3}$, approaching $\infty$ (rather than the derivative of $u_\eta$ blowing up).

\begin{lma}\label{l:blowup}
If $u_\eta\colon\Delta_m\to T^\ast S^2$ is a sequence of $J$-holomorphic disks with
boundary on $\Lambda_{K;\eta}$ and with one positive puncture such that $\sup_{\Delta_m}|du_\eta|$ is unbounded as
$\eta\to 0$ then, after adding one puncture in the domains $\Delta_m$ of $u_\eta$, we get an induced sequence $u_\eta\colon\Delta_{m+1}\to T^{\ast}S^{2}$ for which $|du_\eta|$ is uniformly bounded from above.
\end{lma}

\begin{pf}
The proof uses standard blow-up arguments, see {\em e.g.\ }~\cite{Sikorav94}.  Assume that $M_\eta=\sup_{\Delta_m}|du_\eta|$ is not bounded as $\eta\to 0$. Using asymptotic properties of $J$-holomorphic disks near the punctures of $u_\eta$, we find that for $\eta>0$ there exist points $p_\eta\in\Delta_m$ such that $|du_\eta(p_\eta)|=M_\eta$. View $\Delta_m$ as a subset of $\C$ and consider the sequence of maps $g_\eta(z)=u_\eta\left(p_\eta+\frac{z}{M_\eta}\right)$ defined on $U=\bigr\{z\in\C\colon p_\eta+\frac{z}{M_\eta}\in \Delta_m\bigr\}$, where $\Delta_m$ refers to the domain of $u_{\eta}$. Note that the derivative $|dg_\eta|$ of $g_\eta$ is uniformly bounded as $\eta\to 0$. We can thus extract a convergent subsequence, which in the limit $\eta=0$ gives a non-constant holomorphic disk $v\colon H\to T^{\ast}S^{2}$, where $H$ is the upper half plane, with boundary on $\Lambda$, with a positive puncture at infinity, and no other puncture. (There may be other limiting bubble disks,
but at least one of them has a single puncture. In the limit, only the chords of type ${\bf{L}}$ exist.)
Fix an arc $A$ in $\Lambda$ that intersects $v(\pa H)$ transversely at a point far from all Reeb chord endpoints. It follows from the convergence $g_\eta\to v$ that there exist a point $q_\eta$ in the domain $\Delta_m$ for $u_\eta$ with $|p_\eta-q_\eta|\to 0$ as $\eta\to 0$ such that $u_\eta(q_\eta)\in \Pi(\Lambda)$. Adding a puncture at $q_\eta$ in the domain $\Delta_m$ of $u_\eta$ gives a new sequence of maps $u_\eta^{1}\colon\Delta_{m+1}\to T^\ast S^2$. Assume now that   $\sup_{\Delta_{m+1}}|d u_\eta^{1}|$ is unbounded. Repeating the blow up argument sketched above, we would again find a bubble disk $v^{1}\colon H\to T^{\ast}S^{2}$ in the limit with one positive puncture and no other punctures.

Since the area contributions of $u_\eta^{1}$ in a neighborhood of the added puncture $q_\eta$ is uniformly bounded from below and since this neighborhood can be taken to map  to a region far from all Reeb chords, it follows that there are at least two non-constant disks in the limit, both with one positive puncture and no other punctures. The sum of the areas of these two disks is bounded from below by $2\action(c)+\Ordo(\eta)$, where $c$ is the shorter of the two Reeb chords of $\Lambda$. This however contradicts $u_\eta$ having one positive puncture since the lengths of Reeb chords then implies $\area(u_\eta)\le \action(e)+\Ordo(\eta)$ and $2\action(c)>\action(e)$. The lemma follows.
\end{pf}

% **************************************************
\subsubsection{Quantum flow tree convergence}\label{sssec:qtreeconv}
Consider a sequence $u_\eta\colon\Delta_m\to T^{\ast}S^{2}$ of rigid $J$-holomorphic disks with boundary on $\Lambda_{K,\eta}$. Assume, without loss of generality (see Lemma \ref{l:blowup}), that $|du_\eta|$ is uniformly bounded.

\begin{lma}\label{l:someD}
If each of the disks $u_\eta$ has a positive puncture at a Reeb chord of type ${\bf L}$
then $\sup_{\Delta_m}|du_\eta|$ is uniformly bounded from below.
\end{lma}

\begin{pf}
Consider the neighborhood $N$ of $\Pi(\Lambda)$ where $J$ is induced by the metric on $\Lambda$. Since $u_\eta$ maps $\pa\Delta_m$ to an $\Ordo(\eta)$ neighborhood of $\Pi(\Lambda)$ there exists $\epsilon>0$ such that if $|du_\eta|\le\epsilon$ then $u_\eta(\Delta_m)\subset N$ for all sufficiently small $\eta>0$. Lemma~5.4 from  \cite{Ekholm07} then shows that the function $|p\circ u_\eta|^2$ is subharmonic and, consequently, that the sequence $u_\eta$ has a subsequence $u_{\eta'}$ which converges to a flow tree, see \cite[Lemma 5.6]{Ekholm07}. In particular the area of $u_{\eta'}$ is $\Ordo(\eta').$  But since $u_{\eta'}$ has a positive puncture of type ${\bf L}$, there is a uniform bound
from below on the area of $u_{\eta'}$ restricted to a neighborhood of this puncture. See the proof of \cite[Lemma 9.3]{EkholmEtnyreSullivan05b},
for example.
\end{pf}

Lemma \ref{l:someD} leads to a description of the $J$-holomorphic components in the limit of the sequence $u_\eta$ with positive puncture at a chord of type $\mathbf{L}$ as follows. Consider a sequence of points $p_\eta$ in the domains $\Delta_m$ of $u_\eta$ such that $|du_\eta(p_\eta)|\ge\epsilon>0$ for all $\eta>0$. After passing to a subsequence, we may assume that $u_\eta(p_\eta)$ converges in $T^\ast S^2$. Consider coordinates $\tau+it$ on $\C$ and represent the domain $\Delta_m\subset\C$ of $u_\eta$ by letting the $\tau$-coordinate of $p_\eta$ equal $0$. Then as $\eta\to 0$ in the sequence of domains of $u_\eta$, some boundary minima of $\Delta_m$ stay at finite distance from $p_\eta$ and other do not. After passing to a subsequence we may assume that every sequence of boundary minima on a fixed height has a limit which may be finite or infinite and we find a limiting conformal structure on a domain $\Delta_{m_0}$ that contains $p_0=\lim_\eta p_\eta$. It follows in a  straightforward way that $u_\eta$ converges (uniformly on compacts) to a non-constant
$J$-holomorphic disk $v\colon\Delta_{m_0}\to T^{\ast}S^{2}$ with boundary on $\Lambda$. We say that such a disk is a {\em non-constant $J$-holomorphic component} of the limit.

We next consider the role of the choice of $p_\eta$. If the sequence of domains $\Delta_m$ converges and if they contain some point $q_\eta$ such that $|du_\eta(q_\eta)|\ge\delta>0$ and if $|p_\eta-q_\eta|\to\infty$ as $\eta\to 0$ then, repeating the above argument, we extract another non-constant $J$-holomorphic component $v'$ containing $q_0=\lim_{\eta\to 0} q_\eta$ of the limit, which is distinct from the component $v$ that contains $p_0$. Furthermore, if the $\tau$-coordinate of $q_\eta$ approaches $\pm\infty$ in coordinates where the $\tau$-coordinate of $p_\eta$ equals $0$ and if $a$ is the Reeb chord at the positive puncture of $v'$ then $a$ is also the negative puncture of some non-constant $J$-holomorphic component in the limit. Arguing by action it is easy to see that the number of non-constant components in the limiting configuration is finite. And since $\Lambda$ has only two Reeb chords of almost equal actions the number of such components is at most two.

We next show that the flow trees in Lemma \ref{l:outest} fit together with the non-constant components to form a quantum flow tree. Consider a puncture $\zeta_r$ at $+\infty$ in the domains $\Delta_m$ of $u_\eta$ that map to a Reeb chord $r$ of $\Lambda_{K;\eta}$. By asymptotic properties of holomorphic disks at punctures there are $\eta$-short vertical segments which separate $\zeta_r$ from the positive puncture $\zeta_+$ at $-\infty$ of $\Delta_m$. In particular, there is an $\eta$-short vertical segment $l_\eta(\zeta_r)$ of minimal $\tau$-coordinate that separates $\zeta_r$ from $\zeta_+$, we call it the \emph{extremal $\eta$-small} vertical segment of $\zeta_r$. If $v\colon \Delta_{m'}\to T^{\ast}S^{2}$ is a non-constant $J$-holomorphic component in the limit configuration then we write $\pa v=v(\pa\Delta_{m'})$ for the image of $v$ restricted to the boundary.

\begin{lma}\label{l:nogap}
If $\zeta_r$ is any puncture at $+\infty$ in the domains $\Delta_m$ of $u_\eta$ which maps to a Reeb chord $r$ and if $l_\eta$ is the extremal $\eta$-small vertical segment of $\zeta$ then there exists a non-constant component $v$ of the limit such that $u_\eta(l_\eta)$ converges to a point in $\pa v$.
\end{lma}

\begin{pf}
We prove this lemma by contradiction: if the statement of the lemma does not hold, then the area difference between a limit disk and the disks before the limit violates an $\Ordo(\eta)$ bound derived from Stokes' theorem.

Assume the lemma does not hold. Then there exists $\epsilon>0$ such that for any sequence of $l_\lambda$ which satisfies the Inequalities~\eqref{eq:vertlineL} and \eqref{eq:vertlineZ}, some point on $l_\eta$ maps a distance at least $\epsilon>0$ from $\partial v$. Consider a strip region  between two slits,$[-d,d]\times[0,1]\subset \Delta_m,$ for which some point converges to a point at distance $\delta$ from $\partial v$, where $\frac{\epsilon}{4}<\delta<\frac{\epsilon}{2}$. Let $\sup_{[-d,d]\times[0,1]}|du_\eta|=K$. Then $K$ is not bounded by $M\eta$ for any $M>0$. Since the difference between the sum of areas of the non-constant components in the limit and that of $u_\eta$ is $\Ordo(\eta),$ it follows that $|du_\eta|=\Ordo(\eta^{\frac12})$ (by monotonicity and a standard bootstrap estimate). Thus $K=\Ordo(\eta^{\frac12})$. Consider next the scaling of the target by $K^{-1}$ at the image of $(0,0) \in [-d,d]\times[0,1].$ We get a sequence of maps $\widehat u_\eta$ from $[-d,d]\times[0,1]$ with bounded derivative. Note moreover that the scaled boundary condition is $\Ordo(\eta^{\frac12})$ from the $0$-section. Changing coordinates to the standard $(\C^{n},\R^{n})$ respecting the complex structure at the limit point, we find that there are maps $f_\eta\colon[-d,d]\times[0,1]\to\C^{n}$ with the following properties
\begin{itemize}
\item $\sup_{[-d,d]\times[0,1]}|D^{k}f_\eta|=\Ordo(\eta^{\frac12})$, $k=0,1$,
\item $\widehat u_\eta+f_\eta$ satisfies $\R^{n}$ boundary conditions, and
\item $\bar\partial (\widehat u_\eta+f_\eta)=\Ordo(\eta^{\frac12})$.
\end{itemize}
It follows that $\widehat u_\eta + f_\eta$ converges to a
  holomorphic map with boundary on $\R^{n}$, which takes $0$ to $0$
  and which has derivative of magnitude $1$ at $0$. Using solvability
  of the $\bar\partial$-equation in combination with $L^{2}$-estimates
  in terms of area we find that the area of $\widehat u_\lambda$ must be
  uniformly bounded from below by a constant $C$. The area
  contribution to the original disks near the limit is thus at least
  $K^{2}C$.  Since the image of $[-d,d] \times [0,1]$ in $\Pi(\Lambda)$ has diameter
  at most $2Kd$, we may repeat the argument with many disjoint finite
  strips with maximal
  derivatives $K_j$ and with sum of diameters bounded below by $\frac{\epsilon}{100}.$ We find that the area contribution is bounded
  from below by $C \sum K_j^{2}.$ Since the length
  contribution is bounded below, we get:
  \[ 2d \sum K_j\ge \frac{\epsilon}{100}.
  \] Now,
  \[ C \sum K_j^{2}\ge C\inf_j\{K_j\}\sum K_j\ge C'\inf_j\{K_j\}.
  \] For any $M>0$, $\inf_j\{K_j\}\ge M\eta$. To see this assume
  that it does not hold true. Then there is a sequence of vertical
  segments $l_\eta$ such that $|du_\eta|\le 2M\eta$ with the
  property that the distance between $u(l_\eta)$ and $\partial v$
  is at most $\frac34\epsilon$. This however contradicts our
  hypothesis. Consequently, the area contribution from the remaining
  part of the disk is not $\Ordo(\eta)$, which contradicts Stokes'
  theorem.
\end{pf}

As a consequence of the preceding lemmas, we get the following result.
\begin{cor}\label{c:rigidftconv}
Any sequence of rigid holomorphic disks $u_\eta$ with boundary on $\Lambda_{K;\eta}$ has a subsequence which converges to a quantum flow tree. The quantum flow trees which arise as limits of rigid disks are of the following types.
\begin{itemize}
\item[$(\mathrm{QT}_{\varnothing})$] No non-constant $J$-holomorphic components ({\em i.e.}, flow trees), positive puncture of type $\mathbf{S}_1$, and negative punctures of type $\mathbf{S}_1$.
\item[$(\mathrm{QT}_0)$] Rigid non-constant $J$-holomorphic component, positive puncture of type $\mathbf{L}_1$, and negative punctures of type $\mathbf{S}_0$.
\item[$(\mathrm{QT}_0')$] Rigid non-constant $J$-holomorphic component, positive puncture of type $\mathbf{L}_1,$ one negative puncture is of type ${\bf L}_1$, and all other negative punctures are of type ${\bf S}_0$.
\item[$(\mathrm{QT}_1)$] Constrained rigid non-constant $J$-holomorphic component, positive puncture of type $\mathbf{L}_2$, one negative puncture of type $\mathbf{S}_1$ and remaining negative punctures of type $\mathbf{S}_0$.
\end{itemize}
\end{cor}
\begin{pf}
Consider $(\mathrm{QT}_{\varnothing})$, if the positive puncture of $u_\eta$ is of type $\mathbf{S}_1$ flow tree convergence was established in Lemma~\ref{l:onlytreeconv}.
For $(\mathrm{QT}_0)$ and $(\mathrm{QT}_1)$, assume that the positive puncture has type $\mathbf{L}$. Convergence to a quantum flow tree follows from the above: Lemma \ref{l:someD} implies that there is some non-constant component in the limit and the discussion following that lemma shows that the non-constant component is a broken disk with at most two levels, Lemma~\ref{l:nogap} then implies that the flow tree pieces of Corollary~\ref{c:outconv} are attached to the non-constant components. If the positive puncture is of type $\mathbf{L}_1$ then by the dimension formula there can be only one non-constant component in the limit which must be rigid. Since no rigid disk passes through any chord of type $\mathbf{S}$ and since flow trees with negative punctures at chords of type $\mathbf{S}_1$ are constant we conclude that $(\mathrm{QT}_0)$ holds in this case. A similar argument shows that $(\mathrm{QT}_0')$ holds also when the positive puncture is of type $\mathbf{L}_2$ and there is a negative puncture of type $\mathbf{L}_1$. In the case that the positive puncture is of type $\mathbf{L}_2$ and all negative punctures are of type $\mathbf{S}$ it follows from the dimension formula that exactly one negative puncture maps to a chord $b$ of type $\mathbf{S}_1$. Since all flow trees with a negative puncture at $b$ are constant it follows that some non-constant component in the limit passes $\Pi(b)$. Since no rigid disk passes through $\Pi(b)$ it follows that the non-constant component is un-broken and hence constrained rigid. We conclude that $(\mathrm{QT}_1)$ holds in this case.    	
\end{pf}

\begin{rmk}\label{r:domaincontrol}
As in \cite{Ekholm07}, the proof of Corollary~\ref{c:rigidftconv} allows us to control the conformal
structures of the sources of a sequence of rigid disks $u_\eta\colon \Delta_m\to T^\ast S^2$ with boundary on
$\Lambda_{K;\eta}$ in the following way. The distance from a boundary minimum which maps near a
trivalent puncture of the tree to its nearest boundary minimum equals
$c\eta^{-1}+\Ordo(\log(\eta^{-1}))$ where $c$ is a constant determined by the quantum flow
tree. The distance between other boundary minima equals $c'+\ordo(1)$ where $c'$ depends only on the disk component (and its marked points) of the quantum limit tree. Thus, the conformal structure of the big disk part converges to that of the limiting disk and the conformal structures (represented as truncated standard domains) of the flow tree parts converge after rescaling by $\eta^{-1}$.
\end{rmk}

% **************************************************
% **************************************************
\subsection{From quantum flow trees to disks}\label{ssec:qtreetodisk}
In this section we construct rigid $J$-holomorphic disks near any rigid quantum flow tree. Technical results needed for this were already developed in  \cite{Ekholm07} and \cite{EkholmEtnyreSabloff09}. Here we will thus present the main steps together with detailed references to these two papers. The construction follows the standard gluing scheme often used in Floer theory: we associate
an approximately holomorphic disk to each rigid quantum flow tree and use Floer's Picard lemma to show that near each such disk there is a unique actual holomorphic disk. We begin with the following observations.

% **************************************************
\subsubsection{Properties of rigid quantum flow trees}
We start by recalling some of the properties of rigid quantum flow trees of $\Lambda_K\subset N\subset T^{\ast}(S^{2})$ that will be used below. We will write $(u,\Gamma)$ for a quantum flow tree where $u\colon\Delta_m\to T^{\ast}S^{2}$ is the holomorphic disk with boundary on $\Lambda$ part of the quantum flow tree and where $\Gamma$ denotes the flow tree part.
\begin{itemize}
\item Rigid quantum flow trees $(u,\Gamma)$ of $\Lambda_{K}$ are of two main types: those with $u$ constant and those with $u$ non-constant.
\item If $(u,\Gamma)$ is a rigid quantum flow tree with $u$ non-constant then $u$ is either rigid or constrained rigid.
\item If $(u,\Gamma)$ is a rigid quantum flow tree then it consists of a (constrained) rigid disk with a finite number of partial flow trees $\Gamma$ attached along its boundary at junction points. In the case that $u$ is constrained rigid then the constraint is at the image in $\Pi(\Lambda)$ of a Reeb chord of type $\mathbf{S}_1$. The partial flow trees attached to the holomorphic disk have trivalent $Y_0$-vertices, $1$-valent vertices at critical points of index $1$, and no other vertices.
\end{itemize}

Rigid quantum flow trees $(u,\Gamma)$ with $u$ constant are rigid flow trees of $\Lambda_{K;\eta}\subset J^{1}(\Lambda)$ in the sense of \cite{Ekholm07}. By Lemma~\ref{l:onlytreeconv}, holomorphic disks with positive puncture at a Reeb chord of type $\mathbf{S}$ lie in an $\Ordo(\eta)$ neighborhood of $\Pi(\Lambda)$. Consequently existence and uniqueness of rigid holomorphic disks near each local rigid quantum flow tree follows from \cite[Theorem 1.2]{Ekholm07}. (In fact, the convergence rate is $\Ordo(\eta\log(\eta^{-1}))$, see \cite{Ekholm07}.)  Because of this we will mainly focus on quantum flow trees with non-constant $u$ below.

Consider a rigid quantum flow tree $(u,\Gamma)$ and let $\Gamma'$ denote one of the partial flow trees attached to $u$ at the junction point $p$, which is then also the special puncture of $\Gamma'$. Recall that we choose a metric $g$ on $\Lambda$ and an almost complex structure on $T^{\ast}S^{2}$ which agrees with that in a neighborhood of $\Pi(\Lambda)$ and which have the following additional properties:
\begin{itemize}
\item The metric is flat in a neighborhood of $\Gamma'$.
\item $\Pi(\Lambda_{K;\eta})\subset T^{\ast}\Lambda$ is affine outside neighborhoods a finite number of edge points, which are points on the edges of $\Gamma'$, at least one on each edge. We call these neighborhoods \emph{edge point regions}. Furthermore, along any edge $\Pi(\Lambda_{K;\eta})$ is a product of a curve in the direction of the edge and horizontal line segments perpendicular to it.
\item Near each trivalent vertex $\Pi(\Lambda_{K;\eta})\subset T^{\ast}\Lambda$ is parallel to the $0$-section, and near each critical point $\Pi(\Lambda_{K;\eta})\subset T^{\ast}\Lambda$ is affine.
\end{itemize}

% **************************************************
\subsubsection{Local solutions}%\label{sssec:locsol}
The approximately holomorphic disks near the rigid quantum flow tree $(u,\Gamma)$ will be constructed by patching local solutions. The local solutions are as follows:
\begin{itemize}
\item The map $u\colon\Delta_m\to T^{\ast}S^{2}$. Here the standard domain has punctures mapping to double points of $\Pi(\Lambda)$ as well as to each junction point and constraining critical points. See \cite[Equations~(6.9) and~(6.10)]{EkholmEtnyreSabloff09} for explicit forms of $u$ near junction points and double points. The normal form at a constraining puncture is the same as that at a junction point, see \cite[Section~6.5.1~(gd.2)]{EkholmEtnyreSabloff09}.
\item Along each part of an edge between edge points where $\Pi(\Lambda_K)$ is affine we have a local solution $s_\eta\colon [T_1,T_2]\times[0,1]\to T^{\ast} S^{2}$ which is a holomorphic strip with image in the strip region which consists of straight line segments in the fibers connecting the two sheets of the tree. See \cite[Section~6.1.1]{Ekholm07} for details.
\item At each $1$-valent puncture we choose coordinates $(x_1,x_2)$ along $\Pi(\Lambda)$ with the critical point at $0$ and corresponding holomorphic coordinates $(z_1,z_2)=(x_1+iy_1,x_2+iy_2)$ so that the flow line of the tree lies along the $x_1$-direction. The local solution $s_\eta\colon[0,\infty)\times[0,1]$ is then
\[
s_\eta(\tau+it)=(\eta ic_1+c_\eta e^{-\theta_\eta(\tau+it)},0),
\]
where $\pi-\theta_\eta$ is the largest complex (K\"ahler) angle of the intersection point and where $c_\eta$ is chosen so that the distance from $s_\eta(0+it)$ to the nearest edge point is $\Ordo(\eta)$. See \cite[Section~6.1.2]{Ekholm07} for details.
\item At each trivalent vertex we chose $\C^{2}$-coordinates as above with the three sheets corresponding to constant sections with values $i v_j$, $j=1,2,3$, where a flow line of $v_1-v_2$ breaks into flow lines of $v_1-v_3$ and $v_3-v_2$. Let $p_1,p_2,p_3$ denote the punctures of a standard domain $\Delta_3.$ Consider the biholomorphic maps
\[U_j\colon\Delta_3\to\R\times[0,1], \,\,j=2,3
\]
with $U_2(p_1)=-\infty$, $U_2(p_3)=i$, and $U_2(p_2)=\infty$, and $U_3(p_1)=-\infty$, $U_3(p_2)=0$, and $U_3(p_3)=\infty$. Let $a_j\colon\R\times[0,1]\to\C^{2}$, $j=2,3$ be the maps $a_2(z)=(v_3-v_2)z+iv_2$, $a_3(z)=(v_1-v_2)z$. Then $s_\eta$ is a restriction to a subdomain of $\Delta_3$ cut off by vertical segments of the map
\[
\widetilde s_\eta=\eta(a_2\circ U_2+a_3\circ U_3)
\]
such that the vertical segments lie at distance $\Ordo(\eta\log(\eta^{-1}))$ from an $\epsilon$-sphere around $0$. See \cite[Section~6.1.5]{Ekholm07} for details.
\item at each junction point we choose $\C^{2}$-coordinates so that the two sheets of $\Pi(\Lambda_{K;\eta})$ corresponds to $\R^{2}$, and to the section $(i\eta,0)$, respectively. We take
\[
w_\eta^{\rm jun}(z)=(\eta z,0)+\sum_{n} c_ne^{n\pi z},\quad c_n\in\R^{2},
\]
where the latter sum agrees with the Fourier expansion of $u$ in the strip neighborhood of the junction point.
\end{itemize}

% **************************************************
\subsubsection{The domain of a rigid quantum flow tree and approximately holomorphic disks}
Consider a rigid flow tree $(u,\Gamma)$. The local solutions discussed above are associated with domains that are determined by requiring that their vertical boundary segments map close to edge points. There are in particular finite strip regions of length bounded below by $c\eta^{-1}$ and above by $C\eta^{-1}$ associated to each segment of an edge between edge points, finite neighborhoods of the boundary minimum in $\Delta_3$ of the same size around each trivalent vertex, as well as half-infinite strips associated to critical points. These regions are patched together over uniformly finite size rectangles corresponding to the edge points where we also interpolate between local solutions. We thus get a domain $\bar\Delta_{p,m}(\Gamma',\eta)$ for each tree $\Gamma'$ in $\Gamma$. We construct the domains $\Delta_{p,m}(u,\Gamma,\eta)$ by gluing the tree domains to the disk domain in the strip regions corresponding to the junction points. In this way we obtain the desired domain with a map $w_\eta\colon\Delta_{p,m}(u,\Gamma,\eta)\to T^{\ast}S^{2}$ obtained by patching local solutions which is approximately holomorphic in a sense that we will next make precise.

In order to prove existence and uniqueness of holomorphic disks near rigid flow trees we need an appropriate functional analytic setting for Fredholm theory. Here one cannot use standard Sobolev spaces because the domains are degenerating near the limit and derivatives of maps go to $0$ accordingly. For this reason we use weighted Sobolev spaces. The norms of the natural vector fields associated to shifting the local solutions are then unbounded as $\eta\to 0$. To correct this we use instead a subspace of the Sobolev space determined by a vanishing condition at a marked point in the middle of each strip region between edge point and add a finite dimensional space of shifts endowed with the supremum norm. The total configuration space is then obtained by adding conformal variations of the target which corresponds to moving boundary minima and the marked points of the above mentioned vanishing conditions. More precisely, the functional analytic spaces are constructed as follows:
\begin{itemize}
\item The domain $\Delta_{m,p}$ of $u$ is cut off as described in \cite[Section~6.5.1]{EkholmEtnyreSabloff09},
at vertical segments corresponding to junction points and to punctures. In the
finite strip regions near the cut offs corresponding to junction points of the form $[0,d\eta^{-1}]\times[0,1]$ we use an exponential weight peaked in the middle of this strip, see \cite[Equation~(6.12)]{EkholmEtnyreSabloff09}.
\item For each flow tree $\Gamma'$ in $\Gamma$ we take the domain $\overline\Delta_{p',m'}(\Gamma',\eta)$ associated to a partial flow tree as in \cite[Section~6.4.1]{Ekholm07}, cut off at a vertical segment corresponding to the edge point closest to its special puncture, with the weight function constructed there. We attach these domains at the vertical segments of the corresponding junction points. Note that the weight functions match.
\end{itemize}
We denote the resulting standard domain $\Delta_{p,m}^{\eta}$ and the weight function $h\colon \Delta_{p,m}^{\eta}\to[1,\infty)$. Associated to the junction points are spaces of cut off local solutions. We denote the direct sum of these spaces with the supremum norm
\[
V_\sol^{\rm jun},
\]
see \cite[Section~6.5.1, p.~66]{EkholmEtnyreSabloff09}. Associated to each partial flow tree there is a space of cut off solutions we denote the sum of these spaces with the supremum norm
\[
V_\sol^{\Gamma},
\]
see \cite[Section~5.4.1, p.~1209]{Ekholm07}.
Furthermore, conformal variations corresponding to moving boundary minima in the domain of $u$ give conformal variations of $\Delta_{p,m}$ we denote the sum of these spaces with the supremum norm
\[
V_\con^{u},
\]
see \cite[Sections~6.2.3 and~6.3.5]{Ekholm07}. Finally there are also conformal variations in the domains of the trees corresponding to moving boundary minima and marked points. They form a space $V_\con(\Gamma',\eta)$ and we denote the sum over all trees in $\Gamma$
\[
V_\con^{\Gamma}.
\]
Finally, we denote the space of vector fields that vanishes at the marked points in $\Delta_{p,m}^{\eta}$ and that satisfies Lagrangian boundary conditions and are holomorphic along the boundary, and which has two derivatives in $L^{2}$ weighted by $h$
\[
\sblv_{2,\delta},
\]
where $\delta$ denotes the small positive exponential weight that controls the size of $h$. Then, as in \cite{EkholmEtnyreSabloff09}, we view the $\bar\pa$-operator on function in a neighborhood of a tree as a Fredholm map
\[
\bar\pa_J\colon \sblv_{2,\delta}\oplus V_\sol \oplus V_\con\to \sblv_{1,\delta},
\]
where $\sblv_{1,\delta}$ is the Sobolev space of vector fields with one derivative in $L^{2}$ weighted by $h$.
We denote the norm in $\sblv_{1,\delta}$ by $\|\cdot\|_{1,\delta}$ and that in $\sblv_{2,\delta}\oplus V_\sol \oplus V_\con$ by $\|\cdot\|_{2,\delta}$.

The proof now follows the same steps as in \cite{EkholmEtnyreSabloff09}. First we estimate the approximate solution:
\begin{lma}
The function $w_{\eta}$ satisfies
\[
\|\bar\pa_J w_\eta\|_{1,\delta}=\Ordo(\eta^{3/4-\delta}\log(\eta^{-1})).
\]
\end{lma}
\begin{pf}
The restriction of $\bar\pa_Jw_\eta$ to the part of the domain corresponding to flow trees is controlled by \cite[Remark~6.16]{Ekholm07}. The proof is then a repetition of the proof of \cite[Lemma~6.20]{EkholmEtnyreSabloff09}.
\end{pf}
Second we show that the differential of $\bar\pa_J$ is invertible.
\begin{lma}
The differential
\[
L\bar\pa\colon \sblv_{2,\delta}\oplus V_\sol \oplus V_\con\to \sblv_{1,\delta}
\]
is uniformly invertible.
\end{lma}
\begin{pf}
After replacing the vector field $v^{\rm mo}_{\eta}$ in the proof of \cite[Lemma~6.20]{EkholmEtnyreSabloff09} with a sum $v^{\rm mo}=a^{\Gamma}+b^{\Gamma}$ where $a^{\Gamma}$ is a vector field in $\sblv_{2,\delta}$ supported in the part of the domain corresponding to $\Gamma$ and $b^{\Gamma}\in V_\con^{\Gamma}$ and using \cite[Lemma~6.20]{Ekholm07} to control $v^{\rm mo}$, the proof is a word by word repetition of the proof of \cite[Lemma~6.20]{EkholmEtnyreSabloff09}.
\end{pf}	
Third we establish a quadratic estimate for the $\bar\pa_J$-map. We let $w_\eta$ correspond to $0\in\sblv_{2,\delta}\oplus V_\sol \oplus V_\con$
\begin{lma}
There exists a constant $C$ so that
\[
\bar\pa_J(v)=\bar\pa_J(0)+L\bar\pa_J(v) + N(v),
\]
where
\[
\|N(v_1)-N(v_2)\|_{2,\delta}=C(\|v_1\|_{2,\delta}+\|v_2\|_{2,\delta})\|v_1-v_2\|_{2,\delta}.
\]
\end{lma}	
\begin{pf}
See \cite[Lemma~6.22]{EkholmEtnyreSabloff09}.
\end{pf}
With these results established we get the following result as a consequence of Floer's Picard lemma:
\begin{cor}
For all sufficiently small $\eta>0$ there exists a unique rigid holomorphic disk in a finite $\|\cdot\|_{2,\delta}$-neighborhood of $w_\eta$.
\end{cor}
The last lemma needed to show the correspondence is the following.
\begin{lma}
For sufficiently small $\eta>0$, if a holomorphic disk lies in a sufficiently small $C^{0}$-neighborhood of $w_\eta$, then it lies in a $\Ordo(\eta^{\frac12})$ $\|\cdot\|_{2,\delta}$-neighborhood of it.
\end{lma}

\begin{pf}
See \cite[Lemma~6.24]{EkholmEtnyreSabloff09}.
\end{pf}

% **************************************************
% **************************************************
% **************************************************
\section{Orientations}\label{sec:orientations}
The main purpose of this section is to prove Theorems \ref{thm:signunknottree} and \ref{thm:combsign}. To this end we first give an overview of the general  orientation scheme constructed in \cite{EkholmEtnyreSullivan05c} and then interpret this scheme in geometric terms for rigid holomorphic disks near quantum flow trees of $\Lambda_K$.

% **************************************************
% **************************************************
\subsection{The general orientation scheme}\label{ssec:genori}
We first give a rough outline of the orientation scheme that we will employ below.
For simplicity we restrict attention to closed orientable Legendrian surfaces $\Lambda$ inside the $1$-jet space $J^{1}(S)$ of some orientable surface $S$.
Let $\bar\pa$ be the operator on functions $v\colon D\to\C^{n}$, where $D$ is the unit disk in $\C$, such that $v(e^{i\theta})\in L(\theta)$, where $L(\theta)$ is a \emph{trivialized} Lagrangian boundary condition.
The starting point for constructing orientations of moduli spaces of holomorphic disks with boundary on $\Lambda$ is the fact that the index bundle of $\bar\pa$  is orientable, and that a choice of orientation on $\C$ and on $\R^{n}$ determines an orientation on this index bundle. We call this induced orientation the \emph{canonical orientation}, see \cite[Proposition~8.1.4]{FOOOII} (or \cite[Section~3.3]{EkholmEtnyreSullivan05c}).

Consider now a holomorphic disk $u\colon\Delta_m\to T^{\ast}S$ with boundary on $\Lambda$. In order to parameterize the moduli space of holomorphic disks near $u,$ we look at the $\bar\pa_J$-operator as a section of the bundle of complex anti-linear maps over the configuration space with local chart at $u$ given by $\sblv\oplus \conf_m$, where $\sblv$ is a Sobolev space of vector fields along $u$ and where $\conf_m$ denotes the space of conformal structures on the source $\Delta_m$ of $u$.
In this setting the tangent space of the moduli space can be identified with the kernel of the linearized operator $L\bar\pa_J$ acting on $\sblv\oplus\conf_m$. In order to orient the moduli space we choose capping operators at all Reeb chords of $u$ with oriented determinant bundles. Then closing up the boundary condition of $u$ by these capping operators gives a Lagrangian boundary condition on the closed disk and, provided the Lagrangian boundary conditions are compatibly trivialized, we may use the canonical orientation to orient the determinant of the resulting operator. The orientation of the glued problem and orientations on all capping operators induce an orientation of the determinant of the original linearized problem.
For Legendrian homology in general, the exact sequence which relates all the orientations of the different operators was introduced in \cite[Equation~3.17]{EkholmEtnyreSullivan05c}.
Together with an orientation of the space of conformal structures this then gives an orientation on the tangent space to the moduli spaces,
as described in \cite[Remark~3.18]{EkholmEtnyreSullivan05c}.

Appropriate trivializations on the boundary condition of $u$ can be defined provided $\Lambda$ is spin. In order to have the above scheme compatible with disk breaking at the boundary of the compactified moduli space one must choose oriented capping operators at Reeb chords as positive and negative punctures that add to the trivialized boundary condition with the canonical orientation. We note also that when discussing orientations we can stabilize the operator and add oriented finite dimensional spaces to the source or target of an operator or take direct sums with other oriented Fredholm problems as long as we keep track of the orientations that these extra directions carry.

% **************************************************
% **************************************************
\subsection{Basic choices for the canonical orientation}\label{ssec:basicori}
As mentioned in Section~\ref{ssec:genori} the orientation scheme that we use derives from orientation properties of the index bundle over trivialized Lagrangian boundary conditions on the disk. In this section we study some of the details of this construction.
Let $L$ be an $n$-dimensional Lagrangian boundary condition on the unit disk $D\subset \C$ which is trivialized. Consider the $\bar\pa$-operator acting on vector fields $v\colon D\to\C^{n}$ which satisfy the boundary condition given by $L$, $v(e^{i\theta})\in L(e^{i\theta})$. Denote this operator $\bar\pa_L$. Then $\bar\pa_L$ is a Fredholm operator of index
\[
\ind(\bar\pa_L)=n+\mu(L),
\]
where $\mu$ is the Maslov index. Since the boundary condition is trivialized the monodromy of $L$ is orientation preserving and $\mu(L)$ is even.

As mentioned above, an orientation of $\R^{n}$ together with a choice of complex orientation in $\C$ induces a canonical
orientation on the determinant of the operator $\bar\pa_L$. This canonical orientation is obtained by trivializing the complex bundle over almost all of $D$, splitting off a
complex vector bundle over $\C P^{1}$ at the center of the disk in the limit. The split problem in the limit consists of the operator $\bar\pa_{\R^{n}}$ on $D$, where $\R^{n}$ denotes the constant trivialized Lagrangian boundary condition given by $\R^{n}\subset\C^{n}$ with the standard basis, and the $\bar\pa$-operator on a complex vector bundle over $\C P^{1}$. The determinant of the latter operator is a wedge product of complex vector spaces and hence has an orientation induced from the choice of complex orientation on $\C$. The former operator has trivial cokernel and kernel
spanned by constant sections with values in $\R^{n};$ hence, the orientation of $\R^{n}$ gives an orientation on its determinant. For operators near the split limit the kernel and cokernel are isomorphic to the sum of kernel and cokernels of the split problem and we get an induced canonical orientation of the determinant of the original problem by transporting this orientation along the path of the deformation. We call the choice of orientation of $\R^{n}$ and $\C$ \emph{basic orientation choices}.
\begin{rmk}
Changing the choice of basic orientation on $\R^{n}$ clearly changes the canonical orientation of $\det(\bar\pa_L)$ for every $L$. Changing the choice of basic orientation of $\C$ preserves the canonical orientation of $\det(\bar\pa_L)$ for all $L$ with $\ind(\bar\pa_L)-n=\mu(L)$ divisible by $4$ and reverses the canonical orientation of $\det(\bar\pa_L)$ for all $L$ with $\ind(\bar\pa_L)-n=\mu(L)$ not divisible by $4$.
\end{rmk}

In our calculations below the orientation on $\R^{n}$ above will correspond to the choice of an orientation on the
conormal lift of the unknot which we take as fixed once and for all. We will denote the chosen basic orientation on $\C$ by $o_{\C}$.

Furthermore, the trivialization of the boundary conditions of the linearized operators at holomorphic disks with boundary on $\Lambda_K$ are induced from a trivialization of the tangent bundle $T\Lambda_K$. Note that the orientation above depends only on the trivialization modulo $2$, {\em i.e.}, on the corresponding spin structure. In the calculations below we will fix a spin structure on the conormal lift $\Lambda$ of the unknot and pull it back to $\Lambda_K$ under the natural projection in the $1$-jet neighborhood of $\Lambda$.

% **************************************************
% **************************************************
\subsection{Capping operators and orientation data at Reeb chords}
As mentioned in Section \ref{ssec:genori} orientations of moduli spaces are constructed using capping operators at Reeb chords. In this section we discuss capping operators and their orientations for the conormal lift $\Lambda_K$ of a closed braid $K$
by applying \cite[Section~3.3]{EkholmEtnyreSullivan05c} to $\Lambda_K.$

% **************************************************
\subsubsection{Auxiliary directions}
Before we start this discussion we note that the construction of
coherent orientations in \cite{EkholmEtnyreSullivan05c} uses auxiliary
directions. More precisely, the boundary condition of a punctured disk
with boundary on $\Lambda$ is stabilized, {\em i.e.}, multiplied by
boundary conditions for vector fields in $\C^{2}$ with boundary
conditions close to constant $\R^{2}$ boundary conditions, see
\cite[Section~3.4.2]{EkholmEtnyreSullivan05c}. The resulting problem
is split and the operator in the auxiliary direction is an
isomorphism. The reason for introducing these auxiliary directions is
that they enable a continuous choice of capping operators for varying
Legendrian submanifolds and thereby simplify the proof of
invariance. When studying kernels and cokernels of linearized
operators below we need not consider the auxiliary directions. We do
include the auxiliary directions in order to use the capping operators
in \cite{EkholmEtnyreSullivan05c} which we discuss next. 

We first describe capping operators of the Reeb chords of $\Lambda_K$ and then connect the operators to geometry. In Tables \ref{tb:braidextr}--\ref{tb:unknotextr} below we represent the Legendrian submanifold near its Reeb chords. For detailed properties of the capping operators, we refer to \cite[Section~3.3.6]{EkholmEtnyreSullivan05c}. The actual Legendrian $\Lambda_K\subset J^{1}(S^{2})$ is stabilized and appears as the restriction of an embedding $\Lambda_K\times(-\epsilon,\epsilon)^{2}\to J^{1}(S^{2}\times(-\epsilon,\epsilon)^{2})$ to $(0,0)\in(-\epsilon,\epsilon)^{2}$. Reeb chords are generic and hence locally their fronts are determined by the Hessian for the function difference of the local sheets of Reeb chord endpoints. In \cite{EkholmEtnyreSullivan05c}, the stabilization is constructed in such a way that the two eigenvalues of the Hessian of smallest norm are positive and lie in the auxiliary direction. Here we take these eigenvalues to be of largest norm for simpler combinatorics, see Remark \ref{rmk:auxangles}. In Tables \ref{tb:braidextr}--\ref{tb:unknotextr} the corresponding  eigendirections are denoted $\mathrm{Aux}_j$, $j=1,2$, the two remaining eigendirections are denoted $\mathrm{Real}_j$, $j=1,2$, and we will use the following notation:
\begin{itemize}
\item $\delta$, $\delta'$, $\delta''$, $\delta_0$ are numbers such that
$0<\delta'<\delta''<\delta<\delta_0$, and such that $\delta'$, $\delta''$, and $\delta$ all approach
$0$ as the conormal lift of the link approaches the $0$-section.
\item If $\lambda=(\lambda_1,\dots,\lambda_m)$ is a collection of paths of Lagrangian subspaces such
that the endpoint of $\lambda_j$ is transverse to the start point of $\lambda_{j+1}$ then
$\widehat\mu(\lambda)$ denotes the Maslov index of the loop of Lagrangian subspaces obtained by closing
up the collection of paths by rotating the incoming subspace to the outgoing one in the negative
direction. Thus $n+\widehat \mu(\lambda)$ is the index of the $\bar\pa$-operator on the $(m+1)$-punctured disk with
boundary conditions given by $\lambda$, where $n$ is the dimension of the Lagrangian subspaces.
\item The expressions ``Coker, const'' and ``Ker, const'' indicates that the cokernel and the kernel
of some operator can be represented by constant functions, {\em i.e.}, the actual kernel or cokernel functions are approximately constant in the sense that they converge to constant functions on any compact subset as the conormal of the link approaches the $0$-section.
\end{itemize}

\begin{rmk}\label{rmk:auxangles}
As mentioned above, our choice of eigenvalues in the auxiliary directions differs from that in \cite{EkholmEtnyreSullivan05c} and therefore the capping operators differ as well. It is easy to see that this does not affect the DGA: there is a DGA isomorphism relating the two choices that takes each Reeb chord generator $c$ of the algebra to $\pm c$, where the sign depends on the choice of orientation of the capping operator. In fact, the same argument shows that any capping operator would do as long as the negative and positive capping operators glue to an operator with the canonical orientation. Our particular choice of auxiliary directions here was made in order to allow for a uniform treatment of  all kernel and cokernel elements in the auxiliary directions, see Remark \ref{rmk:auxangles2}.
\end{rmk}

\begin{rmk}\label{rmk:auxangles2}
As mentioned above, after the capping operators has been defined, we may disregard the auxiliary directions when studying orientations. The reason for this is that the stabilized boundary conditions of the linearized operator splits into real and auxiliary directions. Here the boundary conditions in the auxiliary directions gives an operator which is an isomorphism over any disk with one positive puncture and boundary on $\Lambda_K$, the capping operators in the auxiliary directions at all negative punctures are isomorphisms as well, and the capping operators in the auxiliary directions at the positive puncture is independent of the particular puncture and has index $-2$ and trivial kernel. Thus using the canonical orientation for the isomorphisms and fixing an oriented basis in the cokernel in the auxiliary directions of the capping operators at positive punctures we get induced orientations of moduli spaces of holomorphic disks with boundary on $\Lambda_K$. If the orientation of the cokernel of the positive capping operator is changed then the orientations of moduli spaces changes by an overall sign. Thus, auxiliary directions affect the differential in the Legendrian contact homology algebra of $\Lambda_K$ only by an overall  sign. We therefore suppress auxiliary directions in our calculations below.
\end{rmk}

\begin{table}[htp]
\begin{tabular}{|l|c|c|c|c|}
\hline
{}&\rule{0pt}{4ex}${\rm Aux}_1$&${\rm Aux}_2$&${\rm Real}_1$ &${\rm Real}_2$\\
\hline
Front
&
\rule{0pt}{8ex}
\includegraphics[scale=0.3]{pics/minfr}
&
\includegraphics[scale=0.3]{pics/minfr}
&
\includegraphics[scale=0.3]{pics/maxfr}
&
\includegraphics[scale=0.3]{pics/maxfr}\\

\hline

Lagrangian
&
\rule{0pt}{5ex}
\includegraphics[scale=0.3]{pics/minlag}
&
\includegraphics[scale=0.3]{pics/minlag}
&
\includegraphics[scale=0.3]{pics/maxlag}
&
\includegraphics[scale=0.3]{pics/maxlag}\\

\hline

Complex angle,
&
\rule{0pt}{5ex}
\includegraphics[scale=0.3]{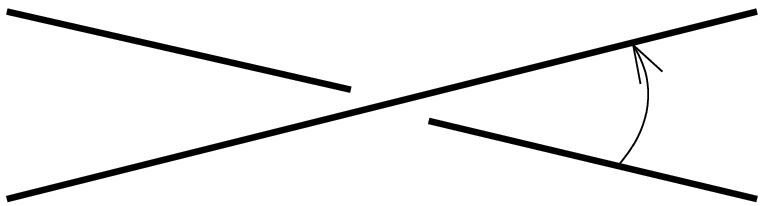}
&
\includegraphics[scale=0.3]{pics/minlag0+}
&
\includegraphics[scale=0.3]{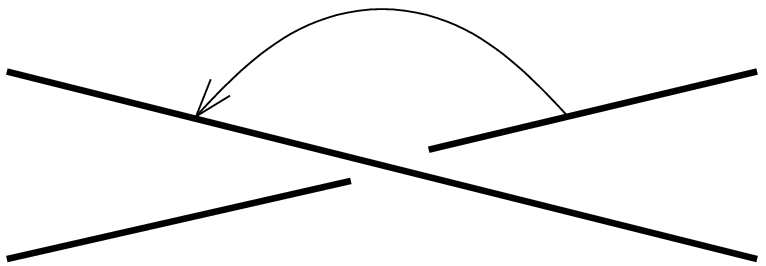}
&
\includegraphics[scale=0.3]{pics/maxlagpi+}\\

positive puncture
&
$\delta$
&
$\delta$
&
$\pi-\delta'$
&
$\pi-\delta''$\\

\hline

Complex angle,
&
\rule{0pt}{5ex}
\includegraphics[scale=0.3]{pics/minlagpi+}
&
\includegraphics[scale=0.3]{pics/minlagpi+}
&
\includegraphics[scale=0.3]{pics/maxlag0+}
&
\includegraphics[scale=0.3]{pics/maxlag0+}\\

negative puncture
&
$\pi-\delta$
&
$\pi-\delta$
&
$\delta'$
&
$\delta''$\\

\hline

Closing rotation,
&
\rule{0pt}{5ex}
\includegraphics[scale=0.3]{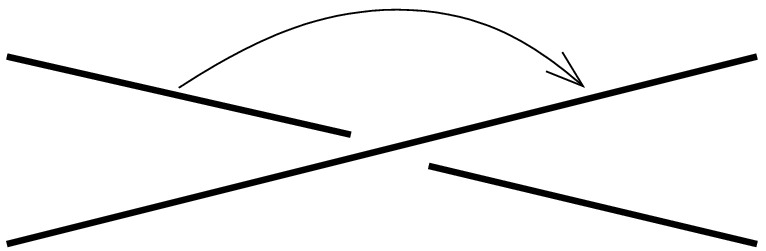}
&
\includegraphics[scale=0.3]{pics/minlagpi-}
&
\includegraphics[scale=0.3]{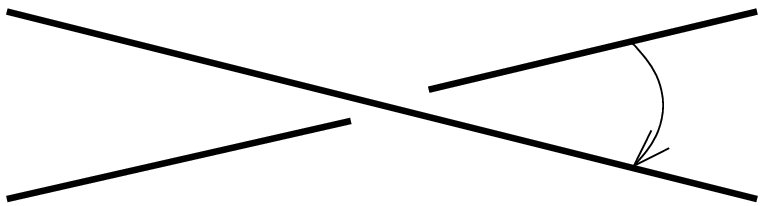}
&
\includegraphics[scale=0.3]{pics/maxlag0-}\\

positive puncture
&
$-(\pi-\delta)$
&
$-(\pi-\delta)$
&
$-\delta'$
&
$-\delta''$\\

\hline

Closing rotation,
&
\rule{0pt}{5ex}
\includegraphics[scale=0.3]{pics/minlag0-}
&
\includegraphics[scale=0.3]{pics/minlag0-}
&
\includegraphics[scale=0.3]{pics/maxlagpi-}
&
\includegraphics[scale=0.3]{pics/maxlagpi-}\\

negative puncture
&
$-\delta$
&
$-\delta$
&
$-(\pi-\delta')$
&
$-(\pi-\delta'')$\\

\hline

Capping operator,
&
\rule{0pt}{5ex}
\includegraphics[scale=0.3]{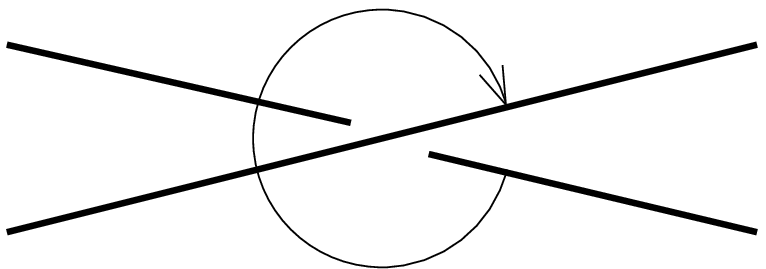}
&
\includegraphics[scale=0.3]{pics/minlag2pi-}
&
\includegraphics[scale=0.3]{pics/maxlag0-}
&
\includegraphics[scale=0.3]{pics/maxlag0-}\\

positive puncture
&
$e^{-i(2\pi-\delta)s}$
&
$e^{-i(2\pi-\delta)s}$
&
$e^{-i\delta's}$
&
$e^{-i\delta''s}$\\

{}
&
$\hat \mu=-2$
&
$\hat\mu=-2$
&
$\hat\mu=-1$
&
$\hat\mu=-1$\\

{}
&
$\ix=-1$
&
$\ix=-1$
&
$\ix=0$
&
$\ix=0$\\

{}
&
Coker, const
&
Coker, const
&
Isomorphism
&
Isomorphism\\

\hline

Capping operator,
&
\rule{0pt}{5ex}
\includegraphics[scale=0.3]{pics/minlag0-}
&
\includegraphics[scale=0.3]{pics/minlag0-}
&
\includegraphics[scale=0.3]{pics/maxlag0+}
&
\includegraphics[scale=0.3]{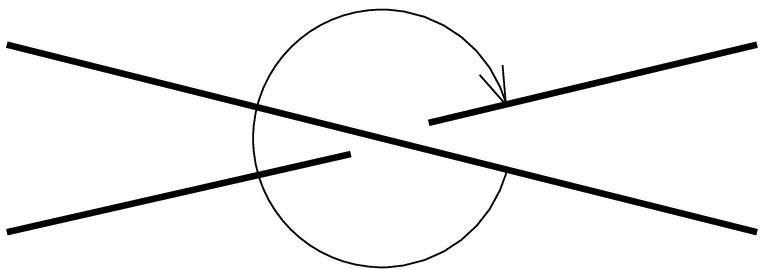}\\

negative puncture
&
$e^{-i\delta s}$
&
$e^{-i\delta s}$
&
$e^{i\delta's}$
&
$e^{-i(2\pi-\delta'')s}$\\

{}
&
$\hat \mu=-1$
&
$\hat\mu=-1$
&
$\hat\mu=0$
&
$\hat\mu=-2$\\

{}
&
$\ix=0$
&
$\ix=0$
&
$\ix=1$
&
$\ix=-1$\\

{}
&
Isomorphism
&
Isomorphism
&
Ker, const
&
Coker, const\\
\hline
\end{tabular}

\vspace{10pt}

\caption{Capping operator at a chord of type $\mathbf{S}_1$. The direction corresponding to the angle $\delta'$ is along the flow line in the direction of the unknot parameter. The direction corresponding to the angle $\delta''$ is the direction of the fiber of the conormal bundle.}\label{tb:braidextr}
\end{table}

\begin{table}[htp]
\begin{tabular}{|l|c|c|c|c|}
\hline
{}&\rule{0pt}{4ex}${\rm Aux}_1$&${\rm Aux}_2$&${\rm Real}_1$ &${\rm Real}_2$\\
\hline
Front
&
\rule{0pt}{8ex}
\includegraphics[scale=0.3]{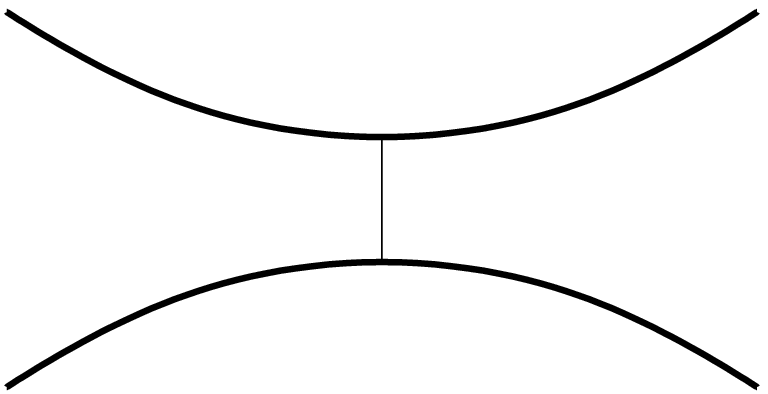}
&
\includegraphics[scale=0.3]{pics/minfr.eps}
&
\includegraphics[scale=0.3]{pics/minfr.eps}
&
\includegraphics[scale=0.3]{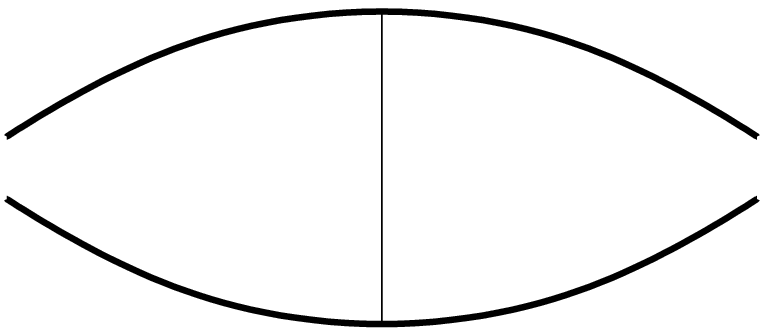}\\

\hline

Lagrangian
&
\rule{0pt}{5ex}
\includegraphics[scale=0.3]{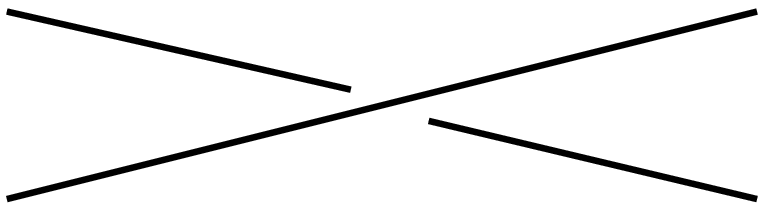}
&
\includegraphics[scale=0.3]{pics/minlag.eps}
&
\includegraphics[scale=0.3]{pics/minlag.eps}
&
\includegraphics[scale=0.3]{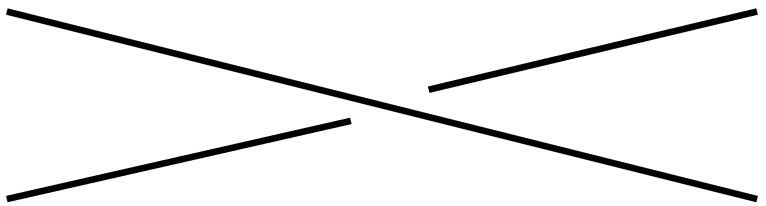}\\

\hline

Complex angle,
&
\rule{0pt}{5ex}
\includegraphics[scale=0.3]{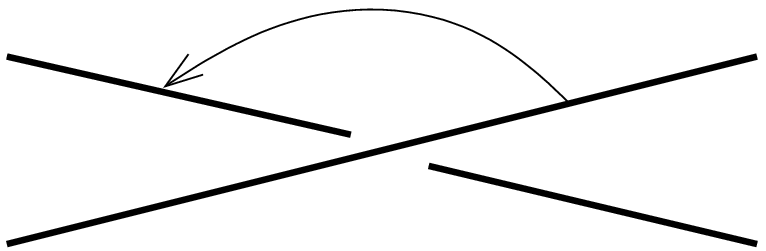}
&
\includegraphics[scale=0.3]{pics/minlagpi+.eps}
&
\includegraphics[scale=0.3]{pics/minlagpi+.eps}
&
\includegraphics[scale=0.3]{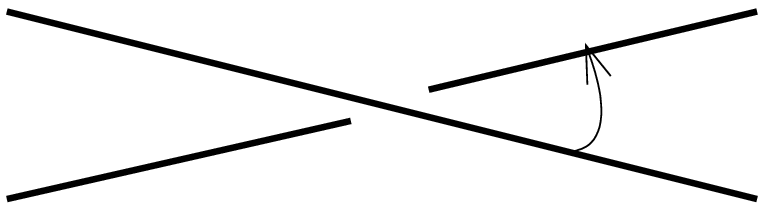}\\

negative puncture
&
$\pi-\delta$
&
$\pi-\delta$
&
$\pi-\delta'$
&
$\delta''$\\

\hline

Closing rotation,
&
\rule{0pt}{5ex}
\includegraphics[scale=0.3]{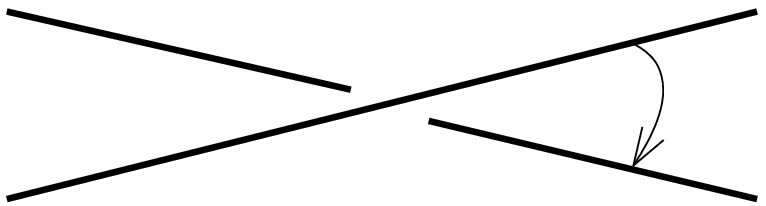}
&
\includegraphics[scale=0.3]{pics/minlag0-.eps}
&
\includegraphics[scale=0.3]{pics/minlag0-.eps}
&
\includegraphics[scale=0.3]{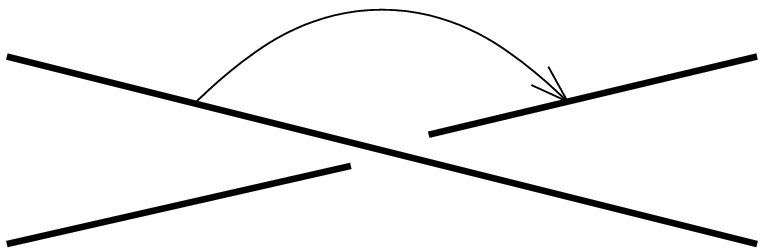}\\

negative puncture
&
$-\delta$
&
$-\delta$
&
$-\delta'$
&
$-(\pi-\delta'')$\\

\hline

Capping operator,
&
\rule{0pt}{5ex}
\includegraphics[scale=0.3]{pics/minlag0-.eps}
&
\includegraphics[scale=0.3]{pics/minlag0-.eps}
&
\includegraphics[scale=0.3]{pics/minlag0-.eps}
&
\includegraphics[scale=0.3]{pics/maxlag0+.eps}\\

negative puncture
&
$e^{-i\delta s}$
&
$e^{-i\delta s}$
&
$e^{-i\delta's}$
&
$e^{i\delta''s}$\\

{}
&
$\hat \mu=-1$
&
$\hat\mu=-1$
&
$\hat\mu=-1$
&
$\hat\mu=0$\\

{}
&
$\ix=0$
&
$\ix=0$
&
$\ix=0$
&
$\ix=1$\\

{}
&
Isomorphism
&
Isomorphism
&
Isomorphism
&
Ker, const\\
\hline
\end{tabular}

\vspace{10pt}

\caption{Local data at a chord of type $\mathbf{S}_0$.}\label{tb:braidsaddle}
\end{table}

\begin{table}[htp]
\begin{tabular}{|l|c|c|c|c|}
\hline
{}&\rule{0pt}{4ex}${\rm Aux}_1$&${\rm Aux}_2$&${\rm Real}_1$ &${\rm Real}_2$\\
\hline
Front
&
\rule{0pt}{8ex}
\includegraphics[scale=0.3]{pics/minfr}
&
\includegraphics[scale=0.3]{pics/minfr}
&
\includegraphics[scale=0.3]{pics/minfr}
&
\includegraphics[scale=0.3]{pics/maxfr}\\

\hline

Lagrangian
&
\rule{0pt}{5ex}
\includegraphics[scale=0.3]{pics/minlag}
&
\includegraphics[scale=0.3]{pics/minlag}
&
\includegraphics[scale=0.3]{pics/minlag}
&
\includegraphics[scale=0.3]{pics/maxlag}\\

\hline

Complex angle,
&
\rule{0pt}{5ex}
\includegraphics[scale=0.3]{pics/minlag0+}
&
\includegraphics[scale=0.3]{pics/minlag0+}
&
\includegraphics[scale=0.3]{pics/maxlag0+}
&
\includegraphics[scale=0.3]{pics/maxlagpi+}\\

positive puncture
&
$\delta$
&
$\delta$
&
$\delta'$
&
$\pi-\delta_0$\\

\hline

Complex angle,
&
\rule{0pt}{5ex}
\includegraphics[scale=0.3]{pics/minlagpi+}
&
\includegraphics[scale=0.3]{pics/minlagpi+}
&
\includegraphics[scale=0.3]{pics/minlagpi+}
&
\includegraphics[scale=0.3]{pics/maxlag0+}\\

negative puncture
&
$\pi-\delta$
&
$\pi-\delta$
&
$\pi-\delta'$
&
$\delta_0$\\

\hline

Closing rotation,
&
\rule{0pt}{5ex}
\includegraphics[scale=0.3]{pics/minlagpi-}
&
\includegraphics[scale=0.3]{pics/minlagpi-}
&
\includegraphics[scale=0.3]{pics/minlagpi-}
&
\includegraphics[scale=0.3]{pics/maxlag0-}\\

positive puncture
&
$-(\pi-\delta)$
&
$-(\pi-\delta)$
&
$-(\pi-\delta')$
&
$-\delta_0$\\

\hline

Closing rotation,
&
\rule{0pt}{5ex}
\includegraphics[scale=0.3]{pics/minlag0-}
&
\includegraphics[scale=0.3]{pics/minlag0-}
&
\includegraphics[scale=0.3]{pics/minlag0-}
&
\includegraphics[scale=0.3]{pics/maxlagpi-}\\

negative puncture
&
$-\delta$
&
$-\delta$
&
$-\delta'$
&
$-(\pi-\delta_0)$\\

\hline

Capping operator,
&
\rule{0pt}{5ex}
\includegraphics[scale=0.3]{pics/minlag2pi-}
&
\includegraphics[scale=0.3]{pics/minlag2pi-}
&
\includegraphics[scale=0.3]{pics/minlagpi-}
&
\includegraphics[scale=0.3]{pics/maxlag0-}\\

positive puncture
&
$e^{-i(2\pi-\delta)s}$
&
$e^{-i(2\pi-\delta)s}$
&
$e^{-i(\pi-\delta')s}$
&
$e^{-i\delta_0s}$\\

{}
&
$\hat \mu=-2$
&
$\hat\mu=-2$
&
$\hat\mu=-1$
&
$\hat\mu=-1$\\

{}
&
$\ix=-1$
&
$\ix=-1$
&
$\ix=0$
&
$\ix=0$\\

{}
&
Coker, const
&
Coker, const
&
Isomorphism
&
Isomorphism\\

\hline

Capping operator,
&
\rule{0pt}{5ex}
\includegraphics[scale=0.3]{pics/minlag0-}
&
\includegraphics[scale=0.3]{pics/minlag0-}
&
\includegraphics[scale=0.3]{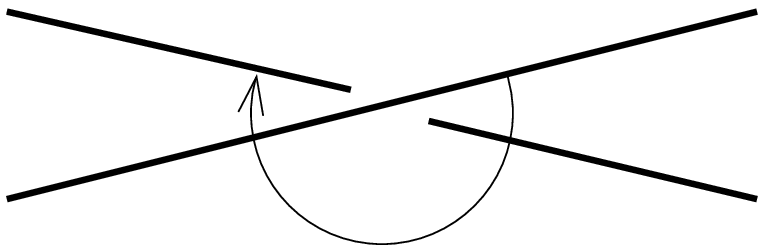}
&
\includegraphics[scale=0.3]{pics/maxlag0+}\\

negative puncture
&
$e^{-i\delta s}$
&
$e^{-i\delta s}$
&
$e^{-i(\pi+\delta')s}$
&
$e^{i\delta_0s}$\\

{}
&
$\hat \mu=-1$
&
$\hat\mu=-1$
&
$\hat\mu=-2$
&
$\hat\mu=0$\\

{}
&
$\ix=0$
&
$\ix=0$
&
$\ix=-1$
&
$\ix=1$\\

{}
&
Isomorphism
&
Isomorphism
&
Coker, const
&
Ker, const\\
\hline
\end{tabular}

\vspace{10pt}

\caption{Capping operator at a chord of type $\mathbf{L}_1$. The direction corresponding to $\delta'$ is along the equator (along the parameter of the unknot). The direction corresponding to $\delta_0$ is perpendicular to the equator (along the fiber).}\label{tb:unknotsaddle}
\end{table}

\begin{table}[htp]
\begin{tabular}{|l|c|c|c|c|}
\hline
{}&\rule{0pt}{4ex}${\rm Aux}_1$&${\rm Aux}_2$&${\rm Real}_1$ &${\rm Real}_2$\\
\hline
Front
&
\rule{0pt}{8ex}
\includegraphics[scale=0.3]{pics/minfr}
&
\includegraphics[scale=0.3]{pics/minfr}
&
\includegraphics[scale=0.3]{pics/maxfr}
&
\includegraphics[scale=0.3]{pics/maxfr}\\

\hline

Lagrangian
&
\rule{0pt}{5ex}
\includegraphics[scale=0.3]{pics/minlag}
&
\includegraphics[scale=0.3]{pics/minlag}
&
\includegraphics[scale=0.3]{pics/maxlag}
&
\includegraphics[scale=0.3]{pics/maxlag}\\

\hline

Complex angle,
&
\rule{0pt}{5ex}
\includegraphics[scale=0.3]{pics/minlag0+}
&
\includegraphics[scale=0.3]{pics/minlag0+}
&
\includegraphics[scale=0.3]{pics/maxlagpi+}
&
\includegraphics[scale=0.3]{pics/maxlagpi+}\\

positive puncture
&
$\delta$
&
$\delta$
&
$\pi-\delta'$
&
$\pi-\delta_0$\\

\hline

Closing rotation,
&
\rule{0pt}{5ex}
\includegraphics[scale=0.3]{pics/minlagpi-}
&
\includegraphics[scale=0.3]{pics/minlagpi-}
&
\includegraphics[scale=0.3]{pics/maxlag0-}
&
\includegraphics[scale=0.3]{pics/maxlag0-}\\

positive puncture
&
$-(\pi-\delta)$
&
$-(\pi-\delta)$
&
$-\delta'$
&
$-\delta_0$\\

\hline

Capping operator,
&
\rule{0pt}{5ex}
\includegraphics[scale=0.3]{pics/minlag2pi-}
&
\includegraphics[scale=0.3]{pics/minlag2pi-}
&
\includegraphics[scale=0.3]{pics/maxlag0-}
&
\includegraphics[scale=0.3]{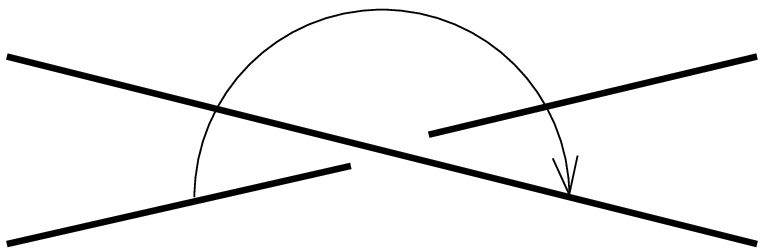}\\

positive puncture
&
$e^{-i(2\pi-\delta)s}$
&
$e^{-i(2\pi-\delta)s}$
&
$e^{-i\delta's}$
&
$e^{-i(\pi+\delta_0)s}$\\

{}
&
$\hat \mu=-2$
&
$\hat\mu=-2$
&
$\hat\mu=-1$
&
$\hat\mu=-2$\\

{}
&
$\ix=-1$
&
$\ix=-1$
&
$\ix=0$
&
$\ix=-1$\\

{}
&
Coker, const
&
Coker, const
&
Isomorphism
&
Coker, const\\

\hline

\end{tabular}

\vspace{10pt}

\caption{Capping operator at a chord of type $\mathbf{L}_2$. The direction of $\delta'$ is along the equator. The direction of $\delta_0$ perpendicular to it.}\label{tb:unknotextr}
\end{table}

% **************************************************
\subsubsection{Details for the orientation data at a chord of type $\mathbf{S}_1$}\label{sssec:oriS_1}
As indicated in Corollary~\ref{c:rigidftconv} a Reeb chord  $b$ of $\Lambda_K$ of type $\mathbf{S}_1$ can appear
as a positive puncture for a disk of type $(\mathrm{QT}_{\varnothing})$ and a negative puncture for a disk of type $(\mathrm{QT}_1).$
Let $\bar\pa_{b^+}$ (respectively, $\bar\pa_{b^-}$) denote the capping operator associated to the  Reeb chord  $b$ of $\Lambda_K$ when it
appears as a positive (respectively, negative) puncture of the
$J_\eta$-holomorphic (see Equation~\eqref{e:Jgto0}) disk $u_\eta.$
Although as Lemma~\ref{lma:braidnicechords} indicates, there are many such chords of type $\mathbf{S}_1$,
each is a parallel translate of the other;
thus, we can consider their capping operators simultaneously.

At a Reeb chord $b$ of $\Lambda_K$ of type $\mathbf{S}_1,$ $\bar\pa_{b^-}$ splits into two 1-dimensional problems. Recall that auxiliary directions are disregarded. Because the grading of $b$ is odd, the conventions set in \cite[Subsection~3.3.6]{EkholmEtnyreSullivan05c} imply one 1-dimensional problem has index $1$ with $1$-dimensional kernel and the other has index $-1$ with $1$-dimensional cokernel.
This is indicated in the two left ``Real" columns at the bottom row of Table~\ref{tb:braidextr}.

Consider first the index $1$ component. This operator is an operator on a Sobolev space of complex valued functions on the disk with one boundary puncture. As the parameter $\delta'\to 0$ in Table~\ref{tb:braidextr} we continue the operator family continuously to the limit by introducing a small negative exponential weight in a strip neighborhood of the puncture. In the limit, the kernel is spanned by a constant real valued function. By continuity, it follows that solutions near the limit are close to constant functions and in particular, the $L^{2}$ pairing with a kernel function is close to the $L^{2}$-pairing of the corresponding constant function in the limit. We thus fix an orientation of the kernel of the index $1$-component of the capping operator at $b$ by fixing a vector
\[
v^\krn(b)\in T_b\Lambda
\]
parallel to the direction of the knot.
To see that $v^\krn(b)$ should be chosen parallel to knot, note that the direction of the knot is the direction of the Bott manifold of Reeb chords of $\Lambda_U$ and hence corresponds to the smaller of the two eigenvalues of the Hessian.

Similarly, elements in the cokernel of the index $-1$ component are solutions to a dual boundary value problem for the $\pa$-operator. The cokernel functions converges to constant real valued functions as $\delta''\to 0$ and we fix an orientation in the cokernel by fixing a vector
\[
v^\cokrn(b)\in T_b\Lambda,
\]
perpendicular to the direction of the knot, which is the direction corresponding to the positive eigenvalue of the Hessian along the Bott manifold of Reeb chords of $\Lambda_U$ and hence correspond to the largest eigenvalue after perturbation. The basis
\[
\left(v^{\cokrn}(b),v^{\krn}(b)\right)
\]
determines the orientation of the operator $\bar\pa_{b^-},$ and constant functions with values in the lines spanned by the basis vectors are approximate solutions.

When $b$ appears as a negative puncture, we see from its parity and the conventions of \cite[Subsection~3.3.6]{EkholmEtnyreSullivan05c}
that both the kernel and cokernel are trivial.

% **************************************************
\subsubsection{Orientation data at a chord of type $\mathbf{S}_0,\mathbf{L}_1$ or $\mathbf{L}_2$}%\label{sssec:oriS_0}
The discussion for the other types of chords is similar to that of chords of type $\mathbf{S}_1$ in Section \ref{sssec:oriS_1},
so we discuss them only briefly here.
Recall that while a chord of type $\mathbf{L}_1$ can appear as a positive or negative puncture,
chords of the other two types only occur as punctures of one sign.

A Reeb chord $a$ of $\Lambda_K$ of type $\mathbf{S}_0$ has even parity and only appears as a negative puncture.
Thus, as indicated in Table~\ref{tb:braidsaddle},
the capping operator $\bar\pa_{a^-}$  splits into two $1$-dimensional problems, one of index $0$ which is an isomorphism and one of index $1$ with $1$-dimensional kernel. As in Section~\ref{sssec:oriS_1} the kernel functions of the index $1$ component are approximately constant and we fix an orientation of the capping operator by fixing a vector
\[
v^\krn(a)\in T_a\Lambda
\]
perpendicular to the direction of the knot.
We can assume  $v^\krn(a)$ is perpendicular to the knot because the second real coordinate
(the last column in Table~\ref{tb:braidsaddle}) represents the unstable manifold of $a,$ thought of as a braid saddle point.
Constant functions with values in the line spanned by this vector are approximate solutions of $\bar\pa_{a^-}.$

Noting the parity of the grading of a chord $c$ of  type $\mathbf{L}_1$ and a chord $e$ of type $\mathbf{L}_2,$
the conventions of \cite[Subsection~3.3.6]{EkholmEtnyreSullivan05c} imply the two capping operators each split into
two 1-dimensional problems with  kernels
and cokernels as indicated in Tables~\ref{tb:unknotsaddle} and~\ref{tb:unknotextr}.
As above, we fix orientations on these problems
by fixing vectors $v^\krn(c), v^\cokrn(c)$ and $v^\cokrn(e).$

% **************************************************
% **************************************************
\subsection{Signs in the unknot differential---proof of Theorem \ref{thm:combsign}}\label{ssec:pfcombsign}
Fix the Lie group spin structure on $\Lambda$. The corresponding trivialization of $T \Lambda$ is then the canonical trivialization of
the tangent bundle of the $2$-torus coming from the identification $T^2=\R^2/\Z^2$. Here, we take coordinates on $\Lambda$ as described in Section \ref{sec:combdiff}.

Recall for the unknot there are four rigid flow trees $I_N$, $I_S$, $Y_N$, and $Y_S$ with one puncture which
is at $c$ that contribute to $\pa c$, and there are two rigid strips with positive puncture at $e$ and negative puncture at $c$. Consequently, by Lemma \ref{t:basicdisktree}, for $\sigma>0$ sufficiently small there are four corresponding rigid holomorphic disks with positive puncture at $c$ and two corresponding rigid holomorphic strips $E_1$ and $E_2$ connecting $e$ to $c$. We next compute their signs.
\begin{thm}\label{l:unkotdiff}
For any choice of basic orientations there is a choice of capping operator at $c$ such that the signs of the rigid disks satisfies
\begin{align*}
&\epsilon(I_S)=\epsilon(I_N)=\epsilon(Y_S)=\epsilon(Y_N)=1,\\
&\epsilon(E_1)=-\epsilon(E_2).
\end{align*}
\end{thm}

\begin{pf}
We first show that $\epsilon(I_S)=\epsilon(I_N)$. To this end consider a geodesic arc in $S^{2}$ which passes through $\Pi(c)$ and which contains both poles. Let $x_1$ be a coordinate along this arc with $\Pi(c)$ corresponding to $0$ and let $x_2$ be a coordinate perpendicular to the arc. Then $\{x=(x_1,x_2)\colon |x_1|\le \pi+\delta,\;|x_2|<\delta\}$ parametrizes a disk $D$ in $S^{2}$ and we find a complex trivialization of the tangent bundle of $T(T^{\ast}S^{2})$ over $D$ by noting that the metric is flat and using coordinates $(x+iy)$.

Let $\bar\pa_c$ denote the positive capping operator at $c$, let $\bar\pa_I$ denote the linearized boundary condition of any one of the four disks with positive puncture at $c$, and let $\widehat I$ denote the problem on the closed disk obtained by gluing these two. Then the following gluing sequence is used to orient moduli spaces:
\[
\begin{CD}
0 @>>> \krn(\pa_{\widehat I}) @>>> \krn(\bar\pa_{I}) @>>> 0,
\end{CD}
\]
see \cite[Equation~(3.17)]{EkholmEtnyreSullivan05c} and Table~\ref{tb:unknotsaddle}. Note first that the trivialized boundary conditions of $I_N$ and $I_S$ agree. The signs of the disks are then obtained by comparing the oriented kernel of $\bar\pa_I$ with the orientation induced by conformal automorphism. Note furthermore that if $u\colon \Delta_1\to T^{\ast}S^{2}$ is a parametrization of $I_N$ then $-u$ parametrizes $I_S$, where $-(x+iy)=(-x-iy)$ in the coordinates discussed above. Since the automorphism group of $\Delta_1$ is $2$-dimensional the signs of the two disks agree, $\epsilon(I_N)=\epsilon(I_S)$. An identical argument shows that $\epsilon(Y_N)=\epsilon(Y_S)$.

After noting that the orientation of the capping operator at $c$ determines the sign in the orienting isomorphism above, it remains only to show that $\epsilon(I_N)=\epsilon(Y_N)$ to complete the proof of the first equation. To this end we compare the boundary condition of $\bar\pa_{I_N}$ and $\bar\pa_{Y_N}$. Note that the boundary conditions of the disks are arbitrarily close to the boundary conditions of the corresponding trees  and that the trees $I_N$ and $Y_N$ are identical except near the north pole, $(x_1,x_2)=(\pi,0)$. Using the trivialization  (over the disk $D$ above) of the $(x+iy)$-coordinates around the north pole, the Lie group spin of the torus $\Lambda$ is $(\pa_s,\pa_t)$ and induces the trivialized boundary condition
\[
(\cos t\, \pa_{x_1}+\sin t\,\pa_{x_2}\,\,,\,\, -\sin t\,\pa_{y_1} + \cos t\,\pa_{y_2})
\]
on the $I_N$-disk and
\[
(\cos t\, \pa_{x_1}-\sin t\,\pa_{x_2}\,\,,\,\, \sin t\,\pa_{y_1} + \cos t\,\pa_{y_1})
\]
on the $Y_N$-disk.

Now, the homotopy of Lagrangian boundary conditions which are given by acting by the complex matrix
\[
\left(\begin{matrix}
1 & 0\\
0 & e^{i\theta}
\end{matrix}\right),\qquad 0\le\theta\le\pi
\]
takes one trivialization to the other and we conclude that multiplication by the matrix at $\theta=\pi$ takes the positively oriented kernel of $\bar\pa_{I_N}$ to that of $\bar\pa_{Y_N}$. Comparing this orientation to the orientation induced by source isomorphisms, for example by evaluation at a point where the disks agree, we find that the signs of the two disks agree.

The argument for finding signs of the two disks $E_1$ and $E_2$ is similar to the above arguments: both disks come from flow lines and their boundary conditions are identical. Again the disks are related by multiplication by $-1$ in suitable coordinates. Here however, the kernel of the linearized operator and the automorphism group are both $1$-dimensional and it follows that multiplication by $-1$ reverses orientation. The lemma follows.
\end{pf}

In our computations of the signs for the differential in the Legendrian algebra of $\Lambda_K$ we will use the capping operators of chords of type $\mathbf{L}_2$ and $\mathbf{L}_1$ which correspond to the capping operators of $e$ and $c$, respectively, for which Lemma~\ref{l:unkotdiff} holds.

% **************************************************
% **************************************************
\subsection{Conformal structures}\label{ssec:conformal}
Conformal parameters for flow trees are best represented as moving boundary minima in standard domains, see Section~\ref{sssec:notatdisktotree}. In the general orientation scheme of \cite{EkholmEtnyreSullivan05c} the space of conformal structures was represented as the location of boundary punctures on the unit disk in the complex plane. The main purpose of this section is to relate these two representations in order to allow for the representation best adapted to trees to be used in computations.

Consider first the representation of conformal structures $\conf_{m}$, used in \cite{EkholmEtnyreSullivan05c}, on the (unit) disk $D_m$ in $\C$
with $m\ge 3$ boundary punctures $p_0, \ldots, p_{m-1}.$ Recall punctures are ordered counter-clockwise.
Fix the (distinguished) puncture $p_0$ at $1,$ $p_1$ at $i,$ and $p_{m-1}$
at $-i$. Then the locations of the remaining punctures in the boundary arc between $i$ and $-i$ determine the conformal structure uniquely. Thus the space $\conf_m$ of conformal structures on the disk with $m$ boundary punctures on, one of which is distinguished, is an $(m-3)$-dimensional simplex. We write $b_j$ for the tangent vector that corresponds to moving the  $j^{\rm th}$ puncture $p_j$ in the positive direction and keeping all other punctures fixed. Then $b_2,\dots,b_{m-2}$ is a basis in $T\conf_m$.

Consider second the representation of $\conf_{m}$ using standard domains $\Delta_m.$ Recall that a standard domain is a strip with slits of fixed width, that a standard domain determines a conformal structure on the disk with $m$ boundary punctures, and that two standard domains determine the same conformal structure if and only if they differ by an overall translation. Assume that $m>3$ and let $t_j\in T\conf_m$ denote the tangent vector which is the first order variation that corresponds to moving the $j^{\rm th}$ boundary minimum toward $-\infty$ and keeping all other boundary minima fixed.

\begin{lma}\label{l:boundminandmovpct}
Let $m>3$. Then
\[
t_j=\sum_{k=2}^{j} \xi_k b_k +\sum_{k=j+1}^{m-2} \eta_k b_k\in T\conf_m,
\]
where  $\xi_k>0$, $k=2,\dots,j$ and $\eta_k<0$, $k=j+1,\dots,m-2$. In particular, we can represent $T\conf_m$ as the vector space generated by $t_1,\dots,t_{m-2}$ divided by the $1$-dimensional subspace generated by the vector
\[
t_1+t_2+\dots+t_{m-2},
\]
and the orientation given by the basis $b_2,\dots,b_m$ agrees with that induced by $t_2,\dots,t_m$.
\end{lma}

\begin{pf}
Consider the map which takes a neighborhood of $+\infty$ in an infinite strip to a neighborhood of the origin in the upper half plane:
\[
w= -e^{-\pi z}.
\]
Under such a change of coordinates the vector field $1$ on the $w$-plane corresponds to the vector field $\pi^{-1} e^{\pi z}$ on the strip since
\[
1= \left(\frac{dw}{dz}\right)\pi^{-1}e^{\pi z}.
\]

Let $\sblv_{-\delta;k}$ denote the Sobolev space of vector fields along $\Delta_m$ which are tangent to $\Delta_m$ along the boundary and with a small negative weight at each puncture, {\em i.e.}, a weight function of the form $e^{-\delta|\tau|}$ in a strip region, $\tau+it\in[0,\infty)\times[0,1]$ or $\tau+it\in(-\infty,0]\times[0,1]$, and with $k$ derivatives in $L^{2}$. The $\bar\pa$-operator $\bar\pa\colon \sblv_{-\delta;k}\to\sblv_{-\delta;k-1}$ has index
\[
1-(m-2)=-(m-3),
\]
where $1=\dim(\R)$ and $-(m-2)$ is the Maslov index of the boundary condition with a negative half turn at each boundary minimum. The exact degree of regularity of the vector fields we use will be of no importance and we will be dropped from the notation.
Let $b_j'$ denote cut-off versions of the vector fields $e^{\pi z}$ supported in the $j^{\rm th}$ strip end. Then we can think of $T\conf_{m}$ as the quotient space
\[
\bar\pa (\sblv_{-\delta}\oplus\la b_2',\dots, b_{m-2}'\ra)/\bar\pa(\sblv_{-\delta})
\]

In this setting, we can interpret $t_j$ as follows, see \cite[Section~2.1.1]{Ekholm07}. Consider $\Delta_m\subset \C$ and let $z=x+iy$ be the standard complex coordinate on $\C$. Let $T_j$ denote a vector field on $\C$ supported in a small ball $B_{r}$ centered at the $j^{\rm th}$ boundary minimum and equal to $-\pa_x$ in $B_{r/2}$ and tangent to the boundary of $\Delta_m$ in $B_{r}-B_{r/2}$. The conformal variation $t_j$ is then represented by $\bar\pa T_j\in\sblv_{-\delta}.$  (To see this, linearize the comparison of conformal structures, $\kappa$ and
$d \phi^{-1} \kappa d\phi$ where $\phi\colon \Delta_m \rightarrow \Delta_m$ is a small diffeomorphism associated to the vector field $T_j,$
To first order, $d \phi^{-1} \kappa d\phi \approx \kappa + \kappa \bar\pa T_j.$)
Because  $\{\pa\bar b_k'\}_k$ spans the cokernel of $\pa \bar,$
\[
\bar\pa T_j = \sum_{k=2}^{m-2}\alpha_k(\bar\pa b_k') + \bar\pa v,
\]
for some $v\in\sblv_{-\delta}$ and real constants $\alpha_2,\dots,\alpha_{m-2}.$
We first show that $\alpha_k\ne 0$, $k=2,\dots,m-2$. Define the vector field $w\colon \Delta_m\to\C$ as $w=T_j-\sum_k\alpha_k b'_k-v$. Then $w$ satisfies a Lagrangian boundary condition of Maslov index $-(m-3)$ and lies in a Sobolev space with exponential weight $-\delta$ in the strip like end around the puncture at $-\infty$, at the first and last punctures at $+\infty$, and at all punctures for which $\alpha_k=0$, at punctures where $\alpha_k\ne 0$ the weight is $-\pi-\delta$. If the number of punctures with $\alpha_k=0$ is $N$ then the index of the $\bar\pa$-operator on the Sobolev space with Lagrangian boundary condition and weights as just explained equals
\[
\ix(\bar\pa)=1-(m-3)+(m-3-N)=1-N.
\]
Note that $w \ne 0$ since $T_j$ is not tangent at the boundary while  $v+ \sum_{k=2}^{m-2}\alpha_k b_k'$ is tangent.
Since $\bar\pa w=0$ and $w \ne 0,$ it follows by automatic transversality in dimension $1$ ({\em i.e.}, the argument principle) that $N=0$, {\em i.e.}, $\alpha_k\ne 0$, all $k$.

To determine the signs in the expression for $t_j$, consider the limit as the shift of the boundary minimum goes to $-\infty$. In this limit the disk $\Delta_m$ splits into three components: a three punctured disk containing the puncture at $-\infty$ and two punctures at $+\infty$ where two standard domains $\Delta_{m'}$ and  $\Delta_{m''}$ are attached. We choose notation so that the punctures in $\Delta_m$ at $+\infty$ to the left of the moving slit ends up in $\Delta_{m'}$ and those to the right in $\Delta_{m''}$. From the point of view of the representations of conformal structures via boundary punctures on the closed disk, the punctures in $\Delta_{m'}$ collides at $-1$ and those in $\Delta_{m''}$ collides at $1$. As the coefficients $\alpha_k$, $k=2,\dots,m-2$  are non-zero for the infinitesimal deformation $t_j$ at each instance of this total deformation it follows that $\alpha_k>0$ for $k<j$ and $\alpha_k<0$ for $k\ge j$.
\end{pf}

\begin{rmk}
Lemma~\ref{l:boundminandmovpct} also has an intuitive justification using harmonic measure.
Suppose the conformal structure changes by slightly decreasing the $j^{\rm th}$ boundary minimum.
Then the harmonic measure of the $j^{\rm th}$ slit (the probability of a Brownian motion particle
first hitting the boundary of $\Delta_m$ at that slit) increases while the measures of all other boundary components
decrease. Harmonic measure is preserved under a conformal map from $\Delta_m$ to $D_m.$
Thus, the corresponding changes in measures of the boundary arcs of $D_m$ can only occur if puncture $p_i$ moves
in the negative direction for $i = 2 ,\ldots j$ and the positive direction for $i=j+1, \ldots, m-2.$
\end{rmk}

% **************************************************
% **************************************************
\subsection{Signs of rigid quantum flow trees---proof of Theorem~\ref{thm:combsign}}\label{ssec:multisigns}
In this subsection we compute the signs of rigid quantum flow trees determined by $\Lambda_K\subset J^{1}(S^{2})$ and thereby prove Theorem~\ref{thm:combsign}.
Recall from Corollary~\ref{c:rigidftconv} that there are four types of rigid quantum flow trees: $(\mathrm{QT}_\varnothing), (\mathrm{QT}_0),
(\mathrm{QT}_0')$ and $(\mathrm{QT}_1).$
We will consider each case separately.

Let $A$ denote the Lagrangian boundary condition, suppressing auxiliary directions,
on the domain $\Delta_m$ of the linearized $\bar\pa$-problem corresponding to the
$J_\eta$-holomorphic  disk $u_\eta.$
Recall  that we think of $\eta>0$ as small, see Equation~\eqref{e:Jgto0}.
As in \cite[Section~6]{EkholmEtnyreSullivan05b}, we must close up the boundary conditions at each puncture using a ``negative $J$-twist."
This is illustrated in Tables~\ref{tb:braidextr}--\ref{tb:unknotextr} under the ``Closing rotation" row(s). Note that
this may contribute to the index of $\bar\pa_A.$ Let $\bar\pa_A$ denote the linearized problem with this boundary.
Let $\widehat{A}$ denote the boundary condition after adding the appropriate capping operators. Define $\bar\pa_{\widehat{A}}$
similarly.
For each of the four cases above, we must compute the exact sequence \cite[Equation~3.17]{EkholmEtnyreSullivan05c}.

% **************************************************
\subsubsection{The sign of a quantum flow tree of type $(\mathrm{QT}_\varnothing)$}
A rigid holomorphic disk near the limit in a neighborhood of a quantum flow tree of type $(\mathrm{QT}_\varnothing)$ lies in a small neighborhood of a rigid flow tree in $\Lambda\subset J^{1}(S^{2})$ determined by $\Lambda_{K}$. Let $\Gamma$ be such a rigid flow tree with  positive puncture $b$ and negative punctures $a_1,\dots,a_{m-1}$. Recall that, since the front of $\Lambda_K\subset J^{1}(\Lambda)$ has no singularities, all vertices of such a rigid tree are trivalent $Y_0$-vertices except for $1$-valent vertices at Reeb chords. Let $t_1,\dots,t_{m-2}$ denote the trivalent vertices of $\Gamma$. Note that each trivalent vertex corresponds to a boundary minimum in the domain $\Delta_m$ of the holomorphic disks $u_\eta$ which corresponds to $\Gamma$ for small $\eta$. We number the trivalent vertices according to the order of the corresponding boundary minima in the vertical direction of the complex plane. We write $\tau_j$ for the boundary minimum corresponding to the trivalent vertex $t_j$.

\begin{lma}\label{l:signbraidtree}
There exists a choice of basic complex orientation $o_\C$ such that if $\eta>0$ is sufficiently small and if $u_\eta$ is a rigid holomorphic disk in a neighborhood of the rigid flow tree $\Gamma$ then the sign of $u_\eta$ is given by
\[
\epsilon(u_\eta)=\epsilon(\Gamma)=\sigma_{\rm pos}(\Gamma)\sigma(n, \Gamma),
\]
where $n=v^{\cokrn}(b)$, see Section~\ref{sssec:signrules} for notation.
\end{lma}

\begin{pf}
Since $\eta$ is small, $u_\eta$ lies close to $\Gamma.$
Using the trivialization of $T(T^{\ast}S^{2})$ in a neighborhood of $\Lambda$ induced by the trivialization of $T(T^{\ast}\Lambda)$,
the boundary condition $A$ is very close to constant $\R^{2}$ boundary conditions (for $\C^{2}$-valued vector fields) on $\Delta_m$.
For such disks with $m$ punctures, using closing rotation angles from Tables~\ref{tb:braidextr} and~\ref{tb:braidsaddle}, we compute
\[
\ix(\bar\pa_A)= \mu(A) + 2 = (m-1)\times (-1) + 1 \times 0 + 2 = -(m-3).
\]

Adding capping operators at the punctures to $A$, see Table~\ref{tb:braidextr} for the positive puncture and Table~\ref{tb:braidsaddle} for the negative punctures, we get a boundary condition $\widehat A$ on the closed disk which also has boundary conditions very close to constant. It follows that $\krn(\bar\pa_{\widehat A})$ is $2$-dimensional with kernel spanned by almost constant sections, and that $\cokrn(\bar\pa_{\widehat A})$ is $0$-dimensional. By definition of the canonical orientation, see Section \ref{ssec:basicori}, the positive orientation of the determinant of $\bar\pa_{\widehat A}$ is represented by a basis of its kernel which converges to constant solutions that form a positively oriented basis of $T\Lambda$ (and a positive sign on its $0$-dimensional cokernel).

Consider first the case $m=2$. In this case the tree is simply a flow line, the operator $\bar\pa_{A}$ has index $1$ and $1$-dimensional kernel spanned by an almost constant solution that converges to $v^\flow(\Gamma)$ as $\eta\to 0$, see Section \ref{sssec:signrules} for notation. The exact gluing sequence that determines the orientation is then, see \cite[Equation~(3.17)]{EkholmEtnyreSullivan05c},
\[
\begin{CD}
0 @>>> \krn(\bar\pa_{\widehat A}) @>>> \krn(\bar\pa_{a_1-})\oplus\krn(\bar\pa_{A}) @>>> 0,
\end{CD}
\]
Here $\krn(\bar\pa_{a_1-})$ is spanned by an approximately constant section which converges to $v^{\krn}(a_1)$, which is perpendicular to $\Gamma$, see Table~\ref{tb:braidsaddle}. It follows that the sign of the disk agrees with the orientation sign of the basis $(v^{\krn}(a_1),v^\flow(\Gamma))$ of $T\Lambda$, where $v^\flow(\Gamma)$ is the vector field induced by the automorphism of the strip, times the sign of the determinant of the capping operator of a positive puncture at $b$, which is an isomorphism.

Gluing the positive and negative capping operators  $\bar\pa_{b+}$ and $\bar\pa_{b-}$ at $b$ gives an operator $\bar\pa_{b0}$ of index $0$ with $\dim(\krn\bar\pa_{b0})=\dim(\cokrn(\bar\pa_{b0}))=1$. Here the kernel is spanned by the constant solution $v^{\krn}(b)$ and the cokernel by the constant solution $v^{\cokrn}(b)$ of the dual problem. Note that the canonical orientation of $\det(\bar\pa_{b0})$ changes with $o_\C$.  Choose $o_\C$ so that $v^{\krn}(b)\wedge v^{\cokrn}(b)$ represents the positive orientation of $\det(\bar\pa_{b0})$ if $(v^{\cokrn}(b),v^{\krn}(b))$ is a positively oriented basis of $T\Lambda$. Then the sign of the disk can be expressed as
\[
\sign\left(\la v^\flow(\Gamma),v^{\krn}(b)\ra\cdot\la v^{\krn}(a),v^{\cokrn}(b)\ra\right)=
\sigma_{\rm pos}(\Gamma)\sigma(n, \Gamma)
\]
as claimed.

Consider the case $m\ge 3$. Here the exact gluing sequence that gives the orientation is
\begin{equation}\label{eq:gluseqtree}
\begin{CD}
0 @>>> \krn(\bar\pa_{\widehat A}) @>{\alpha}>> \krn(\bar\pa_-) @>{\beta}>> \cokrn(\bar\pa_A) @>>> 0,
\end{CD}
\end{equation}
where we write
\[
\bar\pa_-=\oplus_{j=1}^{m-1}\bar\pa_{a_j-}.
\]
Since $u_\eta$ is a rigid disk and $\ix(\bar\pa_A)=-(m-3)$, we get that $\dim(\cokrn(\bar\pa_A))=m-3$, and that
$\cokrn(\bar\pa_{A})$ is spanned by linearized conformal variations which we represent as motion of the boundary minima in the domain, see Lemma~\ref{l:boundminandmovpct}. Here $\krn(\bar\pa_{\widehat A})$ is endowed with the canonical orientation and $\krn(\bar\pa_-)$ with the orientation from capping operators. The sequence then induces an orientation on $\cokrn(\bar\pa_A)$ which gives a sign when compared to the orientation induced from the space of conformal variations, as indicated in \cite[Remark~3.18]{EkholmEtnyreSullivan05c}.

Applying Lemma~3.1 and Remark~3.3 of \cite{EkholmEtnyreSullivan05c},
the map $\alpha$ is defined as follows. An element in $\krn(\bar\pa_{a_j-})$ is a solution of the $\bar\pa$-equation with boundary condition given by the negative capping operator at $a_j$. This solution is cut off and thereby defines an element in the space of sections over the closed disk. In this way, we identify $\krn(\bar\pa_-)$ with an $(m-1)$-dimensional subspace of the domain of the operator $\bar\pa_{\widehat A}$. The map $\alpha$ is then given by $L^{2}$-projection of $\krn(\bar\pa_{\widehat A})$ to this subspace. Likewise, we identify $\cokrn(\bar\pa_A)$ with a subspace of the target space of $\bar\pa_{\widehat A}$ by cutting off solutions of the dual problem and define the map $\beta$ as $\bar\pa_{\widehat A}$ followed by $L^{2}$ projection.

The boundary condition of $u_\eta$ is very close to constant $\R^{2}$ boundary conditions and the complex (K\"ahler) angles at the punctures are close to either $0$ or $\pi$. Thus there is a deformation of $\bar\pa_A$ which takes the boundary conditions to constant $\R^{2}$ boundary conditions and which introduces a small negative exponential weight where the complex angle is close to $0$ and a small positive exponential weight where it is close to $\pi$, which is sufficiently small so that the kernel and cokernel of $\bar\pa_A$ undergoes a continuous deformation (in particular dimensions of kernel and cokernel do not change). The capping operators can be deformed accordingly. We will use these deformed operators to determine the sign in the gluing sequence above. For simplicity, we will keep the notation $\bar\pa_A$ and $\bar\pa_{\widehat A}$ for the deformed operators.

We next introduce notation for parts of the tree $\Gamma$ as well as for corresponding parts of the domain $\Delta_m$. We write $E_{ij}$ for the edge connecting the $i^{\rm th}$ trivalent vertex $t_i$ to the $j^{\rm th}$, $t_j$, and $R_{ij}$ for the (finite) strip region in $\Delta_m$ corresponding to $E_{ij}$, where we think of the boundary of $R_{ij}$ as vertical line segments located below the $i^{\rm th}$ boundary minimum $\tau_i$ and above the $j^{\rm th}$, $\tau_j$. We write $E_0$ for the edge ending at the positive puncture and we take $R_0$ to be the half strip with a slit in $\Delta_m$ with boundary given by the two vertical segments bounding $R_{ji}$ and $R_{ji'}$, where $\tau_j$ is the minimal boundary minimum. We write $E_l$, $l=1,\dots,m-1$ for edges that end at negative punctures and $R_l$ for the the corresponding half infinite strip which is a neighborhood of the $l^{\rm th}$ negative puncture in $\Delta_m$.

Then
\[
\Delta_m-\left(\bigcup_{E_i\subset\Gamma}R_i \;\;\cup\;\bigcup_{E_{ik}\subset\Gamma,\;i<k} R_{ik}\right)
\]
is a disjoint union
\[
\bigcup_{1\le i\le m-2,\,i\ne j} V_i,
\]
where $V_i$ is a neighborhood of $\tau_i$. Consider the vertical segment $l$ through $\tau_i$. The boundary $\pa l$ of $l$ lies in the boundary of $\Delta_m$, if the boundary segment containing the lower endpoint of $l$ lies in the lower boundary on $\Delta_m$ we define $i_-=0$ and if it lies on a boundary segment containing $\tau_k$ we define $i_-=k$. Likewise, if the upper endpoint lies in the upper boundary segment of $\Delta_m$ then we define $i_+=m-1$ and if it lies in a boundary segment containing $\tau_{k'}$ we define $i_+=k'$. Note that $i_-<i<i_+$.

In order to deal with  $\cokrn(\bar\pa_A),$ we introduce below the space $V_\con$ of conformal variations of $\Delta_m$. Write $\widetilde v_i^{\con}$ for the conformal variation supported in $V_i$. Then $\bar\pa_A(\widetilde v_i^{\con})=du_\eta(\bar\pa(\widetilde\pa_\tau))$, where $\widetilde\pa_\tau$ is a cut-off of the constant vector field $\pa_\tau$ supported in $V_i$. See Section~\ref{ssec:conformal}.
Thus, as $\eta \to 0$, $\bar\pa_A(\widetilde v_i^{\con})$ is supported in three rectangular regions $R^{i}_s$, $s=0,1,2$, containing the vertical segments in the boundary $\pa V_i$ and lying in the strip regions corresponding to the incoming edge $E_0^{i}$ and the outgoing edges $E_1^{i}$ and $E_2^{i}$, respectively, at $t_i$. Then in $R^{i}_s$, $s=0,1,2$, $\bar\pa_A(\widetilde v_i^{\con})$ approaches $\bar\pa T^{i}_s$, where $\bar\pa T^{i}_s$ is a cut-off constant vector field tangent to $E^{i}_s$, $s=0,1,2$, directed towards the positive puncture. For the purpose of calculating signs we thus replace $\widetilde v_i^{\con}$ with $v_i^{\con}$ where $\bar\pa_A v_i^{\con}=\bar\pa(T^{i}_0+T^{i}_1+T^{i}_{2})$ and think of $V_\con$ as the vector space spanned by the $m-3$ conformal variations $v_i^{\con}$, $i\ne j$.

Let $\sblv_{A}$ and $\sblv_{A}'$ (respectively $\sblv_{\widehat A}$ and $\sblv_{\widehat A}'$) denote the spaces of vector fields on the closed disk $D$ (respectively the punctured disk $\Delta_m$) which are the domain and target, respectively, of $\bar\pa_A$ (respectively of $\bar\pa_{\widehat A}$). Recall that
\[
\bar\pa_A\colon \sblv_{A}\oplus V_\con\to\sblv_{A}'
\]
is an isomorphism because $\krn(\bar\pa_A)$ is trivial and $\cokrn(\bar\pa_A)$ is mapped onto by $\bar\pa_A(V_\con).$ Viewing $\Delta_m$ as a subset of $D$, we define $\bar\pa_{\widehat A}(v_i^{\con})=\bar\pa_{A}(v_i^{\con})$ for $v_i^{\con}\in V_\con$ and write
\[
\bar\pa_{\widehat A,\,\con}\colon \sblv_{\widehat A}\oplus V_\con\to\sblv_{\widehat A}'
\]
for the operator with extended domain. Then $\bar\pa_{\widehat A,\,\con}$
is an operator of index $2+(m-3)$ which has a $(m-1)$-dimensional kernel.

We will define a map $\psi\colon V_\con\to\krn(\bar\pa_-)$ such that $\beta\circ\psi\colon V_\con\to\cokrn(\bar\pa_A)$, see Equation~\eqref{eq:gluseqtree}, is an isomorphism which induces the same orientation on $\cokrn(\bar\pa_A)$ as the isomorphism $\bar\pa_A\colon V_\con\to\cokrn(\bar\pa_A)$. It then follows that the sign of the disk $u_\eta$ equals the sign of the  determinant of the isomorphism
\[
\begin{CD}
\krn(\bar\pa_{\widehat A,\,\con})\approx \krn(\bar\pa_{\widehat A})\oplus V_\con @>{\alpha+\psi}>> \krn(\bar\pa_-),
\end{CD}
\]
between oriented vector spaces, where $\alpha$ is as in Equation~\eqref{eq:gluseqtree}. To finish the proof we must thus first define $\psi$ and then compute $\alpha$ and the resulting determinant.

We introduce certain vector fields on $\Delta_m$ which are supported in neighborhoods of the strip regions of the form $R_{l}$ or $R_{lk}$ in $\Delta_m$ associated to edges in $\Gamma$ as explained above. We will call these vector fields \emph{edge solutions}.  More precisely, we take $n_{0}^{1}$ and $n_0^{2}$ to be constant sections supported in $R_0,$ cut off in a neighborhood of its boundary where $n_0^{1}$ is tangent to the second outgoing edge at $t_j$ and $n_0^{2}$ tangent to the first, see Figure~\ref{fig:n_0}.

\begin{figure}[htb]
\labellist
\small\hair 2pt
\pinlabel $\tau_j$ [Br] at 73 158
\pinlabel {$\text{supp}(n_0^k), k=1,2$} [Br] at 110 116
\pinlabel $t_j$ [Br] at 76 48
\pinlabel $n_0^2$ [Br] at 37 56
\pinlabel $n_0^1$ [Br] at 37 22
\pinlabel $\Gamma_2$ [Br] at 129 85
\pinlabel $\Gamma_1$ [Br] at 129 -8
\pinlabel $\Gamma$ [Br] at 1 39
\endlabellist
\centering
\includegraphics{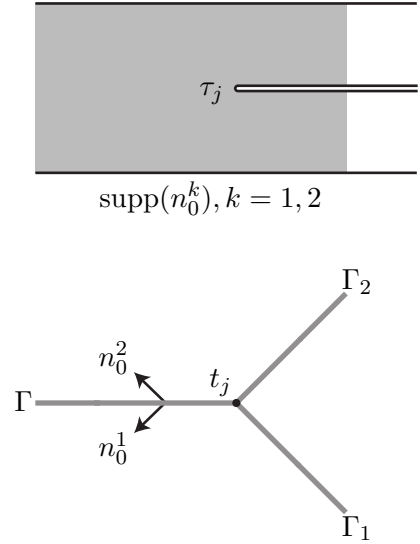}
\caption{Vector fields supported in $R_0$.}
\label{fig:n_0}
\end{figure}

Along edges $E_{ik}$ (respectively $E_i$) we define two cut-off constant vector fields: $\nu_{ik}$ (respectively $w_i$) perpendicular to the edge and $\zeta_{ik}$ (respectively $\zeta_i$) tangent to the edge. Here $\zeta_{ik}$ (resp.~$\zeta_i$) has support in a neighborhood of $R_{ik}$ (respectively $R_i$), whereas $\nu_{ik}$ (respectively $w_i$) has support in a neighborhood of $R_i\cup V_i$, see Figure~\ref{fig:suppedgesol}. We call $\nu_{ik}$, $\zeta_{ik}$, and $\zeta_{i}$ \emph{interior edge solutions} and $w_i$ \emph{exterior edge solutions}.

\begin{figure}[htb]
\labellist
\small\hair 2pt
\pinlabel $\text{supp}(\zeta_{ik})$ [Br] at 45 92
\pinlabel $\text{supp}(w_i)$ [Br] at 160 78
\pinlabel $\tau_i$ [Br] at 39 159
\pinlabel $\tau_k$ [Br] at 131 126
\pinlabel $a_i$ [Br] at 255 24
\endlabellist
\centering
\includegraphics{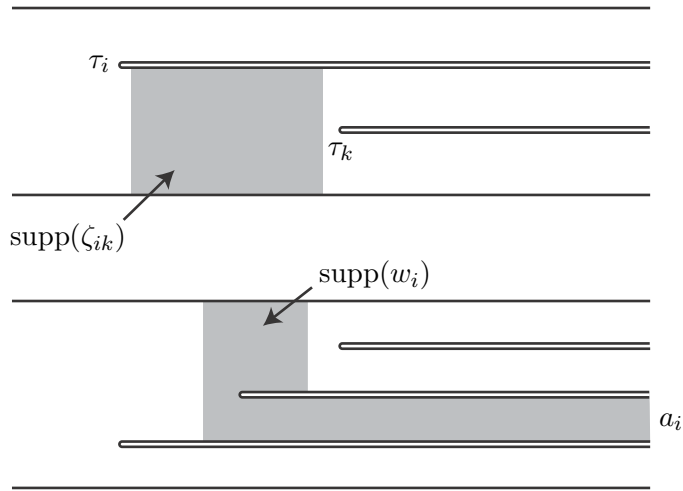}
\caption{Supports of edge solutions.}
\label{fig:suppedgesol}
\end{figure}

Furthermore, we assume that the elements $v_{i}^{\con}\in V_{\con}$ are chosen in such a way that the following holds: in any component $C$ of the support of $\bar\pa_A v_i^{\con}$ lying in an edge region $R_{lk}$ or $R_{l}$ the corresponding edge solution $\zeta_{lk}$ (or $\zeta_l$) satisfies the matching condition $\bar\pa_A\zeta_{lk}=\bar\pa_A v_i^{\con}$ (or $\bar\pa_A\zeta_{l}=\bar\pa_A v_i^{\con}$).

We can now say that the $(m+1)$-dimensional kernel of  $\bar\pa_{\widehat A,\,\con}$ is spanned by $m+1$ linear
independent sections
\begin{align*}
\kappa_0^{1} &=n_{0}^{1} + (w_1 +\dots + w_{j}) + \mathbf{E}_{j}^{1},\\
\kappa_0^{2} &=n_0^{2} + (w_{j+1}+\dots+w_m) + \mathbf{E}_{j}^{2},\\
\kappa_i &= (-(w_{i_-+1}+\dots+w_i)+(w_{i+1}+\dots+w_{i_+}))\\
&- \sigma_{n_0^{1},\Gamma^{1}}(t_i)v_i^{\con}+ \mathbf{E}_{i},\quad 1\le i\le j-1,\\
\kappa_i &= (-(w_{i_-+1}+\dots+w_i)+(w_{i+1}+\dots+w_{i_+}))\\
&- \sigma_{n_0^{2},\Gamma^{2}}(t_i)v_i^{\con}+ \mathbf{E}_{i},\quad j+1\le i\le m-1,
\end{align*}
where $\Gamma^{1}$ and $\Gamma^{2}$ are the partial flow trees obtained by cutting $\Gamma$ at the first and second outgoing edge at $t_j$, respectively, and where $\sigma_{n,\Gamma}$ is as in Equation~\eqref{eq:trivialsign}.
Here $\mathbf{E}^{i}$, $i\ne j$ (resp.~$\mathbf{E}_j^{\alpha}$, $\alpha=1,2$) are some linear combinations of interior edge solutions and conformal variations, respectively, that are supported in the component of $\Delta_{m}-s$, where $s$ the vertical segment through $\tau_i$ (resp.~$\tau_j$), that contains punctures at $+\infty$, see Figure~\ref{fig:addedgesol}.
(The matching conditions for edge solutions imply that linear combinations $\mathbf{E}_{i}$ and $\mathbf{E}_j^{\alpha}$, $\alpha=1,2,$ exists so that the sections indeed lie in $\krn(\bar\pa_{\widehat A,\,\con}.)$

\begin{figure}[htb]
\labellist
\small\hair 2pt
\pinlabel $t_k$ [Br] at 88 60
\pinlabel $\text{supp}(\mathbf{E}_{k})$ [Br] at 214 202
\pinlabel $\tau_k$ [Br] at 84 158
\pinlabel {$\mathbf{E}_k$ supported here} [Br] at 175 -10
\endlabellist
\centering
\includegraphics{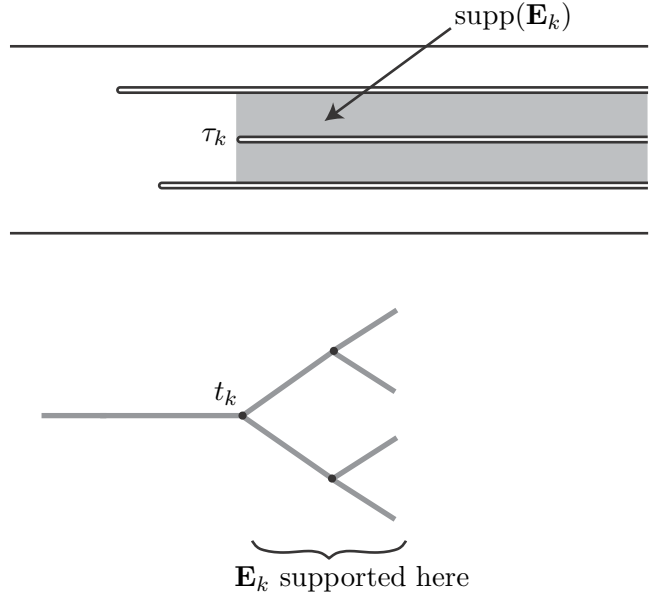}
\caption{Support of additional interior edge solutions.}
\label{fig:addedgesol}
\end{figure}

Using these equations we define the map $\psi$ as follows:
\[
\psi(v_i)=\sigma_{n_0^{s},\Gamma^{s}}(t_i)(-(w_{i_{-}+1}+\dots+w_i)+(w_{i+1}+\dots+w_{i_{+}})+ \mathbf{E}_{i}),
\]
where $s=1$ if $i<j$ and $s=2$ if $i>j$.
Since $\bar\pa \kappa_i=0$ we have, by construction of $\kappa_i$, $\bar\pa\psi(v_i)=\bar\pa v_i$. By definition, the orientation on $\cokrn(\bar\pa_A)$ induced by conformal variations is given by $L^{2}$-projection of $\bar\pa_A(V_\con)$ (with orientation on $V_{\con}$ as in Section~\ref{ssec:conformal}). Thus $\beta\circ\psi$ induces the correct orientation on $\cokrn(\bar\pa_A)$.

In order to compute the sign we first note that the projection of the subspace spanned by $\kappa_1,\dots\kappa_{j-1},\kappa_{j+1},\dots,\kappa_{m-3}$ to $V_\con$ is an isomorphism and that the map from its complement spanned by $\kappa_{0}^{1},\kappa_{0}^{2}$ given by evaluation at the positive puncture gives an isomorphism to $\krn(\bar\pa_{\widehat A})$ which consists of constant solutions. It follows that the sign of the rigid disk $u_\eta$ is given by
\[
\epsilon(u_\eta)=s_1s_2s_3,
\]
where $s_k$, $k=1,2,3$ are as follows. First, $s_1$ equals the sign of the orientation given by
\[
\xi_{0}^{1}\wedge\xi_{0}^{2}\wedge\xi_{1}\wedge\dots\wedge\xi_{j-1}\wedge\xi_{j+1}\wedge\xi_{m-2}
\]
 on $\krn(\bar\pa_-)$ where
\begin{align*}
\xi_{0}^{1}&=w_1 +\dots + w_{j},\\
\xi_{0}^{2}&=w_{j+1}+\dots+w_{m-1},\\
\xi_1&=-w_{1}+(w_{2}+\dots+w_{1_{+}}),\\
\xi_2&=-(w_{2_{-}+1}+w_2)+(w_{3}+\dots+w_{2_{+}}),\\
&\qquad\vdots\\
\xi_{j-1}&=-(w_{(j-1)_{-}+1}+\dots+w_{j-1})+w_{j},\\
\xi_{j+1}&=-(w_{j+1})+(w_{j+2}+\dots+w_{(j+1)_{+}}),\\
&\qquad\vdots\\
\xi_{m-1}&=-(w_{(m-2)_{-}+1}+\dots+w_{m-2})+w_{m-1}.
\end{align*}
Second, $s_2$ equals the sign of the orientation on $V_{\con}$ given by
\begin{align*}
&(\sigma_{n_0^{1},\Gamma^{1}}(t_1)v_1^{\con})\wedge\dots\wedge (\sigma_{n_0^{1},\Gamma^{1}}(t_{j-1})v_{j-1}^{\con})\wedge\\
&(\sigma_{n_0^{2},\Gamma^{2}}(t_{j+1})v_{j+1}^{\con})\wedge
\dots\wedge(\sigma_{n_0^{2},\Gamma^{2}}(t_{m-2})v_{m-2}^{\con}).
\end{align*}
Third, $s_3$ equals the sign of the orientation on $\krn(\pa_{\widehat A})$ given by
\[
\bar n_{0}^{1}\wedge \bar n_{0}^{2},
\]
where $\bar n_0^{s}$ is a constant solution agreeing with $n^{s}_{0}$ in the region where $n^{s}_{0}$ is constant, $s=1,2$.

We have
\begin{align*}
&\xi_{0}^{1}\wedge\xi_{0}^{2}\wedge\xi_{1}\wedge\dots\wedge\xi_{j-1}\wedge\xi_{j+1}\wedge\xi_{m-2}\\
&=
(-1)^{j-1}
\xi_{0}^{1}\wedge\xi_{1}\wedge\dots\wedge\xi_{j-1}\wedge\xi_{0}^{2}\wedge\xi_{j+1}\wedge\xi_{m-2}\\
&=(-1)^{j-1}w_1\wedge\dots\wedge w_{m-1}\\
&=(-1)^{j-1}\;\Pi_{s=1}^{m-1}\la w_s,v^{\krn}(a_s)\ra\;v^{\krn}(a_1)\wedge\dots\wedge v^{\krn}(a_{m-1}),
\end{align*}
and thus
\[
s_1=(-1)^{j-1}\;\Pi_{s=1}^{m-1}\la w_s,v^{\krn}(a_s)\ra.
\]
Next,
\[
v_1\wedge\dots\wedge v_{j-1}\wedge v_{j+1}\wedge\dots\wedge v_{m-2}=
(-1)^{j-1} \pa_{q_3}\wedge\dots\wedge\pa_{q_{m}},
\]
by Lemma \ref{l:boundminandmovpct}, where $\pa_{q_3}\wedge\dots\wedge\pa_{q_{m}}$ is the standard orientation of $\conf_m$ corresponding to moving the last $m-3$ punctures counter clockwise along the boundary of the disk, and thus
\[
s_2=(-1)^{j-1}\;\Pi_{s=1}^{j-1}\sigma_{n_0^{1},\Gamma^{1}}(t_s)\;\Pi_{s=j+1}^{m-1}\sigma_{n_{0}^{2},\Gamma^{2}}(t_s).
\]
Finally, recall that $o_\C$ was chosen so that $v^{\cokrn}(b)\wedge v^{\krn}(b)$ represents the positive orientation on $T\Lambda$. Since $v^{\krn}(b)$ is parallel to the vector $v^{\con}(t_j)$ related to the conformal variation at $\tau_j$, we find that
\begin{align*}
\bar n^{1}_{0}\wedge \bar n^{2}_0&=(n_0^{1}+n_0^{2})\wedge(-n_{0}^{1}+n_{0}^{2})\\
&=\la(n_0^{1}+n_0^{2}),v^{\cokrn}(b)\ra\la(-n_0^{1}+n_0^{2}),v^{\krn}(b)\ra\; v^{\cokrn}(b)\wedge v^{\krn}(b)\\
&=\la(n_0^{1}+n_0^{2}),v^{\cokrn}(b)\ra\sigma(t_j)\la v^\flow_\Gamma, v^{\krn}(b)\ra\;v^{\cokrn}(b)\wedge v^{\krn}(b).
\end{align*}
Thus, if $n_0^{1}$ and $n_0^{2}$ are the vector splittings of $v^{\cokrn}(b)$ then $s_3=\sigma(t_j)$, and
\[
s_1s_2s_3=\sigma_{\rm pos}(\Gamma)\sigma(n, \Gamma).
\]
\end{pf}

% **************************************************
\subsubsection{The sign of a quantum flow tree of type $(\mathrm{QT}_0')$}\label{sssec:stripwtrees}
Let $\Xi$ be a quantum flow tree whose big disk part is a rigid strip $\Theta$. Assume that $\Xi$ has $m$ punctures. Let $e$ denote its positive puncture (a Reeb chord of type $\mathbf{L}_2$), let $c$ denote its negative puncture (of type $\mathbf{L}_1$) and denote remaining punctures by $a_1,\dots,a_{m-2}$ (all of type $\mathbf{S}_0$). It then follows that $\Theta$ is a strip with positive puncture at a Reeb chord of $\Lambda$ close to $e$ and negative puncture at a Reeb chord of $\Lambda$ close to $c$. Let $t_1,\dots,t_{m-1}$ denote the junction points ({\em i.e.}, the points on the boundary of $\Theta$ where trees are attached) and the trivalent vertices in the trees attached. Then each $t_j$ corresponds to a unique boundary minimum $\tau_j$ in the domain $\Delta_m$ of a holomorphic disk $u_\eta$ corresponding to $\eta$ and we number the points $t_j$ according to the vertical coordinate of the boundary minima $\tau_j$ in $\Delta_m$. We write $\II$ for the set of junction points of $\Xi$ and for $t_j\in\II$ we write $\Gamma_j$ for the tree attached at $t_j$ and $n_j$ for the vector at $t_j$ which is tangent to the boundary of $\Theta$ and points toward the positive puncture $e$.

\begin{lma}\label{l:signstrip}
The sign of the rigid disk $u_\eta$ corresponding to $\Xi$ is given by
\[
\epsilon(u_\eta)=\epsilon(\Xi)=
\epsilon(\Theta)\Pi_{t_j\in\II}\,\sigma(n_j,\Gamma_j),
\]
see Sections~\ref{sssec:signrules} and~\ref{ssec:signs} for notation.
\end{lma}

\begin{pf}
Consider first the case $m=2$, {\em i.e.}, when there are no flow trees attached to the disk. In this case the linearized operator $\bar\pa_{A}$ as well as the capping operators are small deformations of the linearized operator and capping operators of the corresponding disk with boundary on $\Lambda$. It follows that the signs of the two disks agree and hence $\epsilon(u_\eta)=\epsilon(\Theta)$ as claimed.

In order to prepare for the case $m>2$ we write down the gluing sequence for $m=2$, see \cite[Equation~(3.17)]{EkholmEtnyreSullivan05c}, that gives the sign explicitly. Using Tables~\ref{tb:braidsaddle}, \ref{tb:unknotsaddle} (as a negative puncture) and~\ref{tb:unknotextr} we have
\[
\begin{CD}
0 @>>> \krn(\bar\pa_{c-})\oplus\krn(\bar\pa_{A}) @>{\beta}>> \cokrn(\bar\pa_{e+})\oplus\cokrn(\bar\pa_{c-}) @>>> 0,
\end{CD}
\]
where we use the fact that the glued operator $\bar\pa_{\widehat A}$ is an isomorphism. (To see this note that $\pa_{\widehat A}$ splits into a direct sum of two $1$-dimensional operators both of index $0$.)

Noting that the boundary conditions of $\bar\pa_A$ are close to $\R^{2}$ conditions we deform them to constant $\R^{2}$ boundary conditions inserting weights as determined by the complex angles, exactly as in the proof of Lemma~\ref{l:signbraidtree}. The solutions in the above sequence can then be thought of as cut-off constant solutions (which may be extended on a sufficiently large domain so that the supports of $v^{\krn}(c-)$ and $v^{\cokrn}(e+)$ overlap in order for the sequence to be exact) and the sign is given by
\[
\epsilon(u_\eta)=\epsilon(\Gamma)=\la v^{\krn}(c),v^{\cokrn}(e)\ra\,\la v^\flow(\Gamma),v^{\cokrn}(c)\ra.
\]

Consider next the case $m>2$. As above we use the associated weighted problem corresponding to $A$ with constant boundary conditions and exponential weights. Again $\bar\pa_{\widehat A}$ is an isomorphism and the gluing sequence which determines the sign is
\[
\begin{CD}
0 \rightarrow \krn(\bar\pa_-)\oplus\krn(\bar\pa_{c-})
\rightarrow \oplus\cokrn(\bar\pa_{e+})\oplus\cokrn(\bar\pa_{c-})\oplus\cokrn(\bar\pa_A) \rightarrow 0,
\end{CD}
\]
where $\krn(\bar\pa_-)=\bigoplus_{j=1}^{m-2}\krn(\bar\pa_{a_j-})$ is spanned by cut off (constant) solutions of the capping operators at the negative punctures $a_1,\dots,a_{m-2}$, where $\krn(\bar\pa_{c-})$ and $\cokrn(\bar\pa_{e+})$ are as above, where the orientation of $\krn(\bar\pa_-)\oplus\krn(\bar\pa_{c-})$ is induced by the order of the punctures, and where $\cokrn(\bar\pa_{A})$ is equipped with the orientation induced by the space of conformal variations of $\Delta_m$.

As in the proof of Lemma~\ref{l:signbraidtree} we stabilize the operator $\bar\pa_{\widehat A}$ by adding the finite dimensional space $V_{\con}$ spanned by conformal variations supported near all boundary minima except $\tau_r$, where we take $\tau_r$ to be the boundary minimum corresponding to the junction point $t_r$ immediately following the negative puncture $c$, or if there is no junction point after $c$, the junction point $t_r$ immediately preceding $c$. Here the conformal variation $v^{\con}_j$ near a boundary minimum $\tau_j$ such that $t_j\notin \II$ is defined exactly as the elements of $V_{\con}$ in the proof of Lemma~\ref{l:signbraidtree}. Conformal variations supported near boundary minima $\tau_i$ with $t_i\in\II$ have the form $\bar\pa_A( \widetilde v_i^{\con})=du_\eta(\bar\pa(\widetilde\pa_\tau))$, where $\widetilde\pa_\tau$ is a cut-off of the vector field $\pa\tau$ in the domain of the holomorphic disk part of $\Gamma$ which is tangent to the boundary and directed towards the positive puncture and continued constantly into the domain corresponding to the tree attached, see Figure~\ref{fig:vconjun}.
\begin{figure}[htb]
\labellist
\small\hair 2pt
\pinlabel $\tau_j$ [Br] at 41 47
\pinlabel $t_j$ [Br] at 71 158
\pinlabel $\Theta$ [Br] at 142 122
\pinlabel $\text{supp}(v_j^\con)$ [Br] at 170 82
\pinlabel $\Gamma_j$ [Br] at 100 171
\endlabellist
\centering
\includegraphics{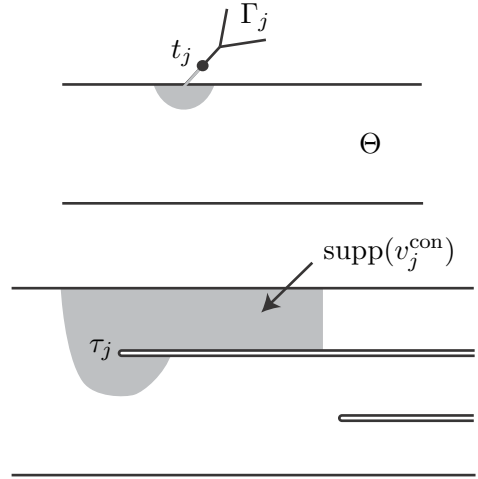}
\caption{Support of conformal variation at junction point.}
\label{fig:vconjun}
\end{figure}
As $\eta\to 0$, $\bar\pa_A\widetilde v_j^{\con}$ converges to a $\bar\pa \widetilde v^{\flow}$ in the part of the domain corresponding to the holomorphic disk part of $\Gamma$, where $\widetilde v^{\flow}$ is a cut-off of the constant vector field $v^{\flow}$ pointing toward the positive puncture, and to $\widetilde n_i$ in the part of the domain corresponding to the tree (near the junction point), where $\widetilde n_i$ is a cut-off of the constant vector field $n_i.$ In analogy with $v_j^{\con}$ for $t_j\notin \II$, we define $v_i^{\con}$ for the deformed boundary conditions so that $\bar\pa_A v_i^{\con}$ agrees with the operator acting on these cut off constant vector fields in the two components of its support. This gives the stabilized operator
\[
\bar\pa_{\widehat A,\;\con}\colon \sblv_{\widehat A}\oplus V_{\con}\to \sblv_{\widehat A}'
\]
and the new sequence
\[
\begin{CD}
0 @>>> \krn(\bar\pa_{\widehat A,\,\con}) @>>> \krn(\bar\pa_-)\oplus\krn(\bar\pa_{c-}) \\
@>>> \cokrn(\bar\pa_{e+})\oplus\cokrn(\bar\pa_{c-}) @>>> 0,
\end{CD}
\]
which by an argument similar to that used for the stabilized sequence in Lemma~\ref{l:signbraidtree} gives the correct sign.

In addition to conformal variations we also introduce \emph{internal} and \emph{external edge solutions} corresponding to edges of the flow trees $\Gamma_j$, $t_j\in\II$ attached exactly as in Lemma~\ref{l:signbraidtree}, where we start the construction for the tree $\Gamma_j$ with the vector $n_j$ at $t_j\in\II$. We also introduce an extra cut-off solution, which compensates for the lack of conformal variation at  $\tau_r$: let $v^{\sol}=v^\flow+v_r$ where $v^\flow$ is the constant vector field along $\Theta$ cut off near each junction point and near $c$, and where $v_r$ is a conformal variation of the usual type supported near $\tau_r$, see Figure~\ref{fig:vsol}.
\begin{figure}[htb]
\labellist
\small\hair 2pt
\pinlabel $\Theta$ [Br] at 143 53
\pinlabel $\text{supp}(v^{\sol})$ [Br] at 82 53
\pinlabel $\Gamma_1$ [Br] at 39 4
\pinlabel $\Gamma_2$ [Br] at 98 8
\pinlabel $\Gamma_3$ [Br] at 126 101
\pinlabel $\Gamma_4$ [Br] at 63 101
\pinlabel $c$ [Br] at 176 53
\pinlabel $e$ [Br] at -4 53
\endlabellist
\centering
\includegraphics{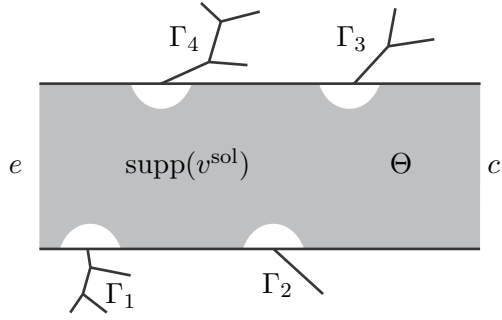}
\caption{The support of $v^{\sol}.$ Here $\Gamma_3$ is the first partial tree to follow the puncture $c$ in the positive direction of the boundary
of the domain.}
\label{fig:vsol}
\end{figure}

As in the proof of Lemma~\ref{l:signbraidtree} we assume that cut-off tangential edge solutions and $v^{\sol}$ agree with conformal variations on intersections of supports of their derivatives.

Define $V_{\sol}$ as the $1$-dimensional space spanned by $v^{\sol}$ and stabilize the operator one more time:
\[
\bar\pa_{\widehat A,\,\con,\,\sol}\colon \sblv_{\widehat A}\oplus V_{\con}\oplus V_{\sol}\to\sblv'_{\widehat A}.
\]
Noting that the projection of $\bar\pa v^{\sol}$ to $\cokrn(\bar\pa_{c-})$ equals the corresponding projection of $\bar\pa v^\flow$ to $\cokrn(\bar\pa_{c-})$ for the disk $\Theta$ we can determine the sign in the exact sequence for the stabilized operator $\bar\pa_{\widehat A,\,\con,\,\sol}$, which reads
\[
\begin{CD}
0 @>>> \krn(\bar\pa_{\widehat A,\,\con,\sol}) @>>> \krn(\bar\pa_-)\oplus\krn(\bar\pa_{c-}) @. \\
@. @>>>\cokrn(\bar\pa_{e+}) @>>> 0,
\end{CD}
\]
in combination with the sign of $\la v^{\sol},v^{\cokrn}(c-)\ra$.

In analogy with the sign calculation in Lemma~\ref{l:signbraidtree}, we find a basis in the $(m-2)$-dimensional kernel of $\bar\pa_{\widehat A,\,\con,\,\sol}$ given by $\kappa_1,\dots,\kappa_{m-2}$, where
\begin{align*}
\kappa_i &=-(w_{i-+1}+\dots+w_i)+(w_{i+1}+\dots w_{i+})
- \sigma_{n_{s(i)},\Gamma_{s(i)}}(t_i)v_i^{\con} + \mathbf{E}_i,\\
&\text{ if $t_i\notin \II$,}\\
\kappa_j &= (w_{j_-+1}+\dots+w_{j}) - v_j^{\con} + \mathbf{E}_j,\\
&\text{ if $t_j\in\II$ and $j\ne r$},\\
\kappa_r&= (w_{r_-+1}+\dots+w_{r}) - v^{\sol} -\sum_{j\ne r} v_j^{\con}+ \mathbf{E}_j.
\end{align*}
where $\Gamma_{s(i)}$ is the attached flow tree in which $t_i\notin\II$ lies, and where $\mathbf{E}_j$ denotes linear combinations of interior edge solutions of edges below $t_j$, exactly as in the proof of Lemma~\ref{l:signbraidtree}.

It follows that
$\epsilon(u_\eta)=s_1s_2s_3,$
where the signs $s_j$, $j=1,2,3$ are as follows. First, $s_1$ equals the orientation on $\krn(\bar\pa_-)$ given by
\[
(-1)^{r-1}\;\left(\Pi_{t_i\notin\II}\sigma_{n_{s(i)},\Gamma_{s(i)}}(t_i)\right)\;\xi_1\wedge\xi_{r-1}\wedge\xi_{r+1}\wedge\dots\wedge\xi_{m-1},
\]
where
\begin{align*}
\xi_i&=
-(w_{i-+1}+\dots+w_j)+(w_{i+1}+\dots w_{i+}),\text{ if $t_i\notin\II$,}\\
\xi_j&=(w_{j_-+1}+\dots+w_{j}),\text{ if $t_j\in\II$}.
\end{align*}
Second, up to a positive factor,
\[
s_2=\la v^\sol, v^{\cokrn}(c-)\ra = \la v_r + v^\flow, v^{\cokrn}(c-)\ra = \la v^\flow, v^{\cokrn}(c-)\ra.
\]
Third, $s_3=(-1)^{r-1}\la v^{\krn}(c),v^{\cokrn}(e)\ra$, by the orientation conventions for sequences, see \cite[Section~3.2.1]{EkholmEtnyreSullivan05c}. We conclude that, with notation as in Theorem~\ref{thm:combsign},
\[
\epsilon(u_\eta)=\epsilon(\Xi)=\epsilon(\Theta)\Pi_{j=1}^{l}\sigma(n_j,\Gamma_j).
\]	
\end{pf}

% **************************************************
\subsubsection{The sign of a quantum flow tree of type $(\mathrm{QT}_0)$}
Let $\Xi$ be a quantum flow tree with holomorphic component $\Theta$, a rigid disk with one puncture. Assume that $\Xi$ has $m$ punctures. Let $c$ denote its positive puncture (a Reeb chord of type $\mathbf{L}_1$) and  $a_1,\dots,a_{m-1}$ denote its negative punctures (all Reeb chords of type $\mathbf{S}_0$). We use notation as in Section~\ref{sssec:stripwtrees} for junction points and trivalent vertices of $\Xi$.

We mark two points $q_0'$ and $q_0$ on the boundary of the rigid disk $\Theta$ right after the positive puncture using two parallel oriented hypersurfaces near the positive puncture. Let $\widetilde v^\flow$ denote the holomorphic vector field along the disk which vanishes at the positive puncture and at the second
marked point and which is directed toward the positive puncture along the boundary. At each junction point $t_j\in\II$, let $n_j$ denote the value of $\widetilde v^{flow}$ at $t_j$.

\begin{lma}\label{l:signrigdisk}
The sign of the rigid disk $u_\eta$ corresponding to $\Xi$ is given by
\[
\epsilon(u_\eta)=\epsilon(\Xi)=
\epsilon(\Theta)\Pi_{j=1}^{l}\sigma(n_j,\Gamma_j),
\]
see Sections~\ref{sssec:signrules} and~\ref{ssec:signs} for notation.
\end{lma}

\begin{pf}
Let $\bar\pa_A$ denote the linearized operator corresponding to the rigid disk $\Theta$ with domain thought of as a strip with punctures at the positive puncture and at $q_0$. Then the linearized operator
\[
\bar\pa_A\colon\sblv_{A}\to\sblv_{A}'
\]
with $m-1 = 0$ negative punctures has index $1.$ The $1$-dimensional kernel for an appropriate choice of orientation of the hypersurface makes $\widetilde v^\flow$ give the positive orientation to the moduli space. The proof then follows from an argument similar to the proof of Lemma~\ref{l:signstrip} so we just sketch the details.

Assume that $u_\eta$ has $m$ punctures. Adding capping operators we get an operator $\bar\pa_{\widehat A}$ on the closed disk of index $2$ with two dimensional kernel corresponding to linearized conformal automorphism. Adding the vanishing condition at the marked point $q_0$ for the vector fields in $\sblv_{\widehat A}$ we get a new operator $\bar\pa_{\widehat A^{\ast}}$ with domain a codimension one subspace $\sblv_{\widehat A^{\ast}}\subset\sblv_{\widehat A}$. The restriction of the operator gives an operator
\[
\bar\pa_{\widehat A^{\ast}}\colon\sblv_{\widehat A^{\ast}}\to\sblv_{\widehat A}'
\]
of index $1$ with $1$ dimensional kernel. We then consider the stabilized problem
\[
\bar\pa_{\widehat A^{\ast},\,\con}\colon\sblv_{\widehat A^{\ast}}\oplus V_{\con}\to\sblv_{\widehat A}'
\]
which has index $(m-1)$ and find that the sign of $u_\eta$ is given by the sign of the map
\[
\begin{CD}
0 @>>> \krn(\bar\pa_{\widehat A^{\ast},\,\con}) @>>> \krn(\pa_-) @>>> 0,
\end{CD}
\]
where $\krn(\pa_-)=\bigoplus_{j=1}^{m-1}\krn(\bar\pa_{a_j-})$ is the sum of cut off kernel functions of the capping operators at all the negative punctures of $u_\eta$. The equation $\epsilon(u_\eta)=\epsilon(\Theta)\Pi_{j=1}^{l}\sigma(n_j,\Gamma_j)$ then follows from a slightly simpler version of the analogous calculation in the proof of Lemma~\ref{l:signrigdisk}.
\end{pf}

% **************************************************
\subsubsection{The sign of quantum flow tree of type $(\mathrm{QT}_1)$}
Consider a quantum flow tree $\Xi$ with $m$ punctures and with big disk part a once punctured
constrained rigid disk $\Theta$ which is constrained to pass through $\Pi(b)$ where $b$ is a Reeb chord of type $\mathbf{S}_1$. Let the positive puncture of $\Xi$ be $e$ (a Reeb chord of type $\mathbf{L}_2$), let the negative
punctures be $b$ (a Reeb chord of type $\mathbf{S}_1$), and $a_1,\dots,a_{m-2}$ (Reeb chords of type $\mathbf{S}_0$). We use notation for trivalent vertices and junction points of $\Xi$ as in Section~\ref{sssec:stripwtrees}. Let $\widetilde v^\flow$ denote the holomorphic vector field
along $\Theta$ which vanishes at $e$ and $b$ and which is directed toward the positive puncture
$e$ along the boundary. Let $n_j$ be the value of $\widetilde v^{\flow}$ at $t_j\in\II$.

\begin{lma}\label{l:sign1dimdisk}
The sign of the rigid disk $u_\eta$ corresponding to $\Xi$ is given by
\[
\epsilon(u_\eta)=\epsilon(\Xi)=
\epsilon(\Theta)\Pi_{j=1}^{l}\sigma(n_j,\Gamma_j).
\]
see Sections~\ref{sssec:signrules} and~\ref{ssec:signs} for notation.
\end{lma}

\begin{pf}
The proof follows along the lines of previous lemmas in the section and detailed calculations will be omitted. We use notation as before.
Start with the case of no negative punctures.
Consider first the orientation sequence for the moduli space containing the $1$-dimensional disk. We write $\bar\pa_{E}$ for the corresponding operator. Then $\bar\pa_E$ has $3$-dimensional kernel whereas the capped off operator $\bar\pa_{\widehat E}$ has $2$-dimensional kernel corresponding to linearized automorphism of the disk with one puncture and we get the gluing sequence
\[
\begin{CD}
0 @>>> \krn(\bar\pa_{\widehat E}) @>>> \krn(\bar\pa_{E}) @>>>\cokrn(\bar\pa_{e+}) @>>> 0,
\end{CD}
\]
thus $\cokrn(\bar\pa_{e+})$ together with conformal automorphisms orient the moduli space and we can identify $\cokrn(\pa_{e+})$ with the vector field $\nu$ which is the positively oriented tangent vector of the $1$-dimensional moduli space, see Section~\ref{ssec:signs}.

Consider first the case $m=2$. The operator $\bar\pa_{\widehat A}$ is an isomorphism and the gluing sequence is
\[
\begin{CD}
0 @>>> \krn(\bar\pa_{b-})\oplus\krn(\bar\pa_{A}) @>>> \cokrn(\bar\pa_{e+})\oplus\cokrn(\bar\pa_{b-}) @>>> 0
\end{CD}
\]
and we find that
\[
\epsilon(u_\eta)=\la \nu,v^{\krn}(b)\ra\la v^\flow(\Gamma),v^{\cokrn}(b)\ra=\epsilon(\Theta)
\]
as claimed. In the case $m>2$ the gluing sequence is
\[
\begin{CD}
0 @>>> \krn(\bar\pa_-)\oplus\krn(\bar\pa_{b-})\\
@>>> \cokrn(\bar\pa_{b-})\oplus\cokrn(\bar\pa_{e+})\oplus\cokrn(\bar\pa_A) @>>> 0,
\end{CD}
\]
as in Lemma~\ref{l:signstrip} we stabilize to an operator
\[
\bar\pa_{\widehat A,\,\con,\,\sol}\colon \sblv_{\widehat A}\oplus V_{\con}\oplus V_{\sol}\to\sblv_{\widehat A}'
\]
of index $m-2$, where $V_\con$ is spanned by conformal variations and where $V_{\sol}$ is a $1$-dimensional space spanned by $v^{\sol}$, which is a sum of an extra conformal variation at the fixed boundary minimum and a cut-off of $v^{\flow}$ of $\widetilde v^{\flow}$ and which maps non-trivially to $v^{\cokrn}(b)$. The resulting map which determines the sign is then
\[
\begin{CD}
0 @>>>  \krn(\bar\pa_{\widehat A,\,\con,\,\sol}) @>>> \krn(\bar\pa_-)\oplus\krn(\bar\pa_{b-})\\
 @>>> \cokrn(\bar\pa_{e+})@>>> 0,
\end{CD}
\]
and the lemma follows from computation similar to those in the proof of Lemma \ref{l:signbraidtree}.
\end{pf}

\end{document}